%% file: arXiv.tex
\documentclass[letter,11pt]{article}

\usepackage[utf8]{inputenc}
\usepackage[dvipsnames,table,xcdraw]{xcolor}
\usepackage{amssymb}
\input{fast-draft-macro}

\title{Testing Approximate Stationarity Concepts \\ for Piecewise Affine Functions\footnote{This work will be presented in part at the ACM-SIAM Symposium on Discrete Algorithms (SODA), 2025.}}

\def\coloneqq{:=}
\def\epsilon{\varepsilon}

\def\ext{\mathop{\textnormal{ext}}}

\def\conv{\mathop{\textnormal{conv}}}
\def\aff{\mathop{\textnormal{aff}}}
\def\cone{\mathop{\textnormal{cone}}}
\def\parr{\mathop{\textnormal{par}}}

\def\argmin{\mathop{\textnormal{argmin}}}
\def\argmax{\mathop{\textnormal{argmax}}}

\def\eargmax{\mathop{\epsilon\textnormal{-argmax}}}
\def\lB{\left\{}
\def\rB{\right\}}

\def\cl{\mathop{\textnormal{cl}}}
\def\intt{\mathop{\textnormal{int}}}

\def\dist{\textnormal{dist}}
\def\sgn{\textnormal{sgn}}

\def\spn{\textnormal{span}}
\def\rnk{\textnormal{rank}}
\def\diag{\mathop{\textnormal{Diag}}}
\def\ri{\mathop{\textnormal{ri}}}

\newcommand{\sz}[1]{\langle #1 \rangle}

\def\leq{\leqslant}
\def\geq{\geqslant}

\def\cP{\mathsf{P}}
\def\cNP{\mathsf{NP}}
\def\MC{\textnormal{MC}}
\def\PA{\textnormal{PA}}
\def\DC{\textnormal{DC}}
\def\MaxMin{\textnormal{Max-Min}}
\def\NAS{\textnormal{NAS}} 

\def\nref{\labelcref}

\usepackage{mathtools}

\usepackage{tablefootnote}

\usepackage{eufrak}

\usepackage{tcolorbox}

\usepackage{multirow}

\usepackage{enumitem}

\definecolor{ye}{HTML}{FFFFC7}
\definecolor{airforceblue}{rgb}{0.36, 0.54, 0.66}
\definecolor{amber}{rgb}{1.0, 0.75, 0.0}
\definecolor{burntorange}{rgb}{0.8, 0.33, 0.0}
\definecolor{dimgray}{rgb}{0.41, 0.41, 0.41}

\usepackage{tikz}
\usetikzlibrary{shadows}
\usetikzlibrary{arrows,positioning}

\usepackage{framed}

\numberwithin{equation}{section}

\usepackage[disable]{todonotes}

\date{January 7, 2025}

\author{Lai~Tian \and Anthony~Man-Cho~So 
}
\author{Lai~Tian\thanks{Department of Systems Engineering and Engineering Management, The Chinese University of Hong Kong, Sha Tin, N.T., Hong Kong SAR. E-mail: \href{mailto:tianlai@se.cuhk.edu.hk}{\tt tianlai.cs@gmail.com}} \and
Anthony~Man-Cho~So\thanks{Department of Systems Engineering and Engineering Management, The Chinese University of Hong Kong, Sha Tin, N.T., Hong Kong SAR. E-mail: \href{mailto:manchoso@se.cuhk.edu.hk}{\tt manchoso@se.cuhk.edu.hk}} 
}

\usepackage{fullpage}

\newcommand{\notimplies}{\;\;\not\!\!\!\!\implies}

\providecommand{\keywords}[1]
{
  \small	
  \textbf{Keywords:} #1
}

\begin{document}
\maketitle
\begin{abstract}

We study the basic computational problem of detecting approximate stationary points for continuous piecewise affine (PA) functions. Our contributions span multiple aspects, including complexity, regularity, and algorithms. Specifically, we show that testing first-order approximate stationarity concepts, as defined by commonly used generalized subdifferentials, is computationally intractable unless $\cP=\cNP$. To facilitate computability, we consider a polynomial-time solvable relaxation by abusing the convex subdifferential sum rule and establish a tight characterization of its exactness. Furthermore, addressing an open issue motivated by the need to terminate the subgradient method in finite time, we introduce the first oracle-polynomial-time algorithm to detect so-called near-approximate stationary points for PA functions.

A notable byproduct of our development in regularity is the first necessary and sufficient condition for the validity of an equality-type (Clarke) subdifferential sum rule. Our techniques revolve around two new geometric notions for convex polytopes and may be of independent interest in nonsmooth analysis. Moreover, some corollaries of our work on complexity and algorithms for stationarity testing address open questions in the literature. To demonstrate the versatility of our results, we complement our findings with applications to a series of structured piecewise smooth functions, including $\rho$-margin-loss SVM, piecewise affine regression, and nonsmooth neural networks.

\vspace{.2cm}
\noindent%
\!\!\keywords{Stationary Points; Piecewise Affine Functions; Computational Complexity; Nonsmooth Analysis; Subdifferential Calculus; Nonsmooth Optimization; Termination Criterion}

\end{abstract}

\pagenumbering{gobble}%
\clearpage

\newpage
{\hypersetup{linkcolor=black}
\setcounter{tocdepth}{2}
\tableofcontents
}

\pagenumbering{gobble}%
\clearpage

\newpage
\clearpage
\pagenumbering{arabic}
\section{Introduction}

For a continuously differentiable function $f:\mathbb{R}^d\rightarrow \mathbb{R}$, a point $\bm{x} \in \mathbb{R}^d$ is called stationary (or critical) if $\nabla f(\bm{x}) = \bm{0}$. However, the situation is much more complicated when the function is nondifferentiable. 
Indeed, there are different stationarity concepts for nonsmooth functions; we refer the reader to \cite{li2020understanding,cui2021modern} for a gentle introduction.
For a locally Lipschitz function $f$, a classic notion of (Clarke) generalized subdifferential of $f$ at $\bm{x}$ can be defined as (see \Cref{def:subd})
\[
	\partial f(\bm{x}) := \conv\big\{\bm{s}:\exists \bm{x}_n\to \bm{x}, \nabla f(\bm{x}_n) \textnormal{ exists}, \nabla f(\bm{x}_n)\to\bm{s}\big\}.
\]
In this paper, we consider the complexity of and robust algorithms for checking whether a given point is approximately stationary with respect to a piecewise affine (PA) function. 
The study of PA functions forms the foundation for investigating the analytic approximation of more general piecewise differentiable functions. 
Specifically, given a \PA{} function $f$ and a point $\bm{x}$, we want to test whether
$
\bm{0} \in \partial f(\bm{x})
$ 
and its various approximate versions.
We emphasize that ``detecting'' (or ``testing'') and ``finding'' (or ``searching'') are two different computational problems. While the co-$\cNP$-hardness of detecting the local optimality of a given point in nonconvex optimization was shown by \citet{murty1987some}, the complexity of finding a local minimizer was an open question proposed by \citet{pardalos1992open}, and is recently settled by \citet{ahmadi2022complexity}. Furthermore, in contrast to the high efficiency of detecting stationary points for smooth function, finding such points can be difficult in general \cite{fearnley2022complexity,hollender2023computational}.

\paragraph{Motivation.} The gradient method and its variants have been the workhorse in nonconvex smooth optimization. For a lower-bounded function $f$ with a Lipschitz gradient, the standard descent lemma shows that the gradient method computes a point $\bm{x}$ satisfying $\|\nabla f(\bm{x})\| \leq \epsilon$ in $O(1/\epsilon^2)$ steps. While convexity and differentiability have long been considered desirable, problems lacking both properties have recently emerged in machine learning, operations research, and statistics. For such applications, the subgradient method is the \textit{de facto} approach employed in practice. However, its convergence behavior was not clear until very recently.
\begin{Fact}[cf.~{\cite[Corollary 5.9]{davis2020stochastic}}]\label{fct:ddtame}
Let a \PA{} function $f:\mathbb{R}^d\to\mathbb{R}$ be given. Consider the iterates $\{\bm{x}_n\}_n$ produced by the subgradient method. Under mild assumptions on the step-size and boundedness of $\{\bm{x}_n\}_n$, every limit point, say $\bm{x}^*$, of the iterates $\{\bm{x}_n\}_n$ satisfies $\bm{0} \in \partial f(\bm{x}^*)$.
\end{Fact}

These types of results \cite{benaim2005stochastic,majewski2018analysis,davis2020stochastic} are asymptotic in nature, without any \emph{a priori} complexity guarantees. This stands in stark contrast to the smooth and convex cases, where problem complexity is central to evaluating method efficiency. It turns out that such absence is fundamental in the sense that any \emph{a priori} complexity guarantee for the subgradient method is impossible.

\begin{Fact}[cf.~{\cite[Theorem 2]{tian2022no}}]
For any $T \in \mathbb{N}_+$ and any chosen step-size and initial point, there exists a $6$-Lipschitz, lower-bounded \PA{} function $f:\mathbb{R}^2\to \mathbb{R}$ such that the iterates $\{\bm{x}_n\}_n$ produced by the subgradient method satisfy
\[\dist\Big(\bm{0}, \conv \partial f(\bm{x}_n+0.25\mathbb{B})\Big) > 0.12,\qquad\text{for all $n \leq T$.}\footnote{For a set $S$, we write $\partial f(S)$ for the set $\cup_{\bm{y} \in S}\partial f(\bm{y})$.}
\]
\end{Fact}

Because proving an \emph{a priori} guarantee is impossible, we focus on \emph{a posteriori} analysis for the subgradient method. Given a lower-bounded \PA{} function $f:\mathbb{R}^d\to \mathbb{R}$, by the asymptotic convergence of the subgradient method, for any $\epsilon\geq 0, \delta > 0$, there exists a finite $T\in\mathbb{N}_+$ such that 
\begin{equation}\label{eq:intro-stat}
	\bm{0}\in \partial f(\bm{x}_T + \delta\mathbb{B}) + \epsilon\mathbb{B},
\end{equation}
ensuring that the subgradient method will eventually bypass approximate stationary points in finite time.\footnote{Notably, when considering algorithms beyond subgradient method, \citet{kornowski2021oracle} establish an exponential oracle complexity lower bound for any local algorithm that computes approximate stationary points in the sense of \eqref{eq:intro-stat}.} This naturally raises the following basic question:
\begin{tcolorbox}[notitle,boxrule=.7pt,colback=white,colframe=black, sharp corners]
\begin{center}
\textit{How can we confidently stop the subgradient method?}	
\end{center}
\end{tcolorbox}
Unlike the smooth case, where verifying $\|\nabla f(\bm{x}_n)\| \leq \epsilon$ is rarely  a problem, deciding whether a point $\bm{x}_n$ is approximately stationary in the sense of \nref{eq:intro-stat} is highly nontrivial and barely explored.
The main goal of this paper is to initiate the study of testing concepts of approximate stationarity for nonconvex, nonsmooth functions.

	\subsection{Our Results}
	
\begin{table}[]
\setlength\arrayrulewidth{0.7pt}
\centering
\begin{tabular}{c|cc|c}
                                             & \multicolumn{2}{c|}{degree-$p$ polynomials}                                                                                                              &                                                 \\ \cline{2-3}
\multirow{-2}{*}{\textbf{solution concepts}} & \multicolumn{1}{c|}{$p\leq 3$}                         & $p\geq 4$                                                                                       & \multirow{-2}{*}{piecewise affine (\PA{})}      \\ \hline
                                             & \multicolumn{1}{c|}{$\cP$}                             & strongly co-$\cNP$-hard\tablefootnote{See the footnote of \cite[p.~4, Table 1]{ahmadi2022AIM}.} & \cellcolor[HTML]{FFFFC7}strongly co-$\cNP$-hard \\
\multirow{-2}{*}{local minimum}              & \multicolumn{1}{c|}{\cite[Theorem 3.3]{ahmadi2022AIM}} & \cite[Theorem 2]{murty1987some}                                                                & \cellcolor[HTML]{FFFFC7}Theorems~\ref{thm:hard-dc}\ref{item:thm:hard-dc-a},~\ref{thm:hard-general}\ref{item:thm:hard-general-a}   \\ \hline
                                             & \multicolumn{1}{c|}{$\cP$}                             & $\cP$                                                                                           & \cellcolor[HTML]{FFFFC7}strongly $\cNP$-hard    \\
\multirow{-2}{*}{stationary point}           & \multicolumn{1}{c|}{(folklore)}                        & (folklore)                                                                                      & \cellcolor[HTML]{FFFFC7}Theorems~\ref{thm:hard-dc}\ref{item:thm:hard-dc-b},~\ref{thm:hard-general}\ref{item:thm:hard-general-b}  
\end{tabular}
\caption{Complexity of deciding whether a given point belongs to a certain solution type.}
\label{tab:complexity}
\end{table}
	
We start with a classic fact about the representation of \PA{} functions.
\begin{nFact}[DC form; cf.~{\Cref{fct:pa-form}}, {\cite[Proposition 4]{melzer1986expressibility}~and~\cite{kripfganz1987piecewise}}]
Any \PA{} function $f:\mathbb{R}^d\to \mathbb{R}$ can be written as the difference of two convex \PA{} functions $h,g:\mathbb{R}^d\to \mathbb{R}$, i.e., $f=h-g$.
\end{nFact}
In this paper, for a given point $\bm{x}\in \mathbb{R}^d$ and two convex \PA{} functions $h,g:\mathbb{R}^d\to\mathbb{R}$, our main goal is to check $\bm{0}\in\partial (h-g)(\bm{x})$ 
in both  exact and approximate senses.

\subsubsection{Negative Results}

\paragraph{Hardness.} 
Given a sum or composition of smooth elemental functions, testing the (approximate) stationarity of a point relies on the applicability of classic gradient calculus rules. In modern computational environments, this can be implemented efficiently by using Algorithmic Differentiation (AD) \cite{griewank2008evaluating} software, such as PyTorch and TensorFlow. A natural question that arises is whether testing stationarity concepts for a piecewise smooth function (e.g., the loss of a ReLU network) can be handled as efficiently as its smooth counterpart. In sharp contrast, we show (in \Cref{thm:hard-dc}) that such testing for even \PA{} functions is already computationally intractable unless $\cP=\cNP$.

\begin{nTheorem}[Informal; cf.~{\Cref{thm:hard-dc}}] Fix any $\epsilon \in [0,1/2)$. Let two convex \PA{} functions $h,g:\mathbb{R}^d\to\mathbb{R}$ and a point $\bm{x}\in\mathbb{R}^d$ be given.
The following hold:
\begin{itemize}
	\item Checking whether $\bm{0}\in \widehat{\partial} (h-g)(\bm{x}) + \epsilon\mathbb{B}$ is co-$\cNP$-hard.
	\item Checking whether $\bm{0}\in \partial (h-g)(\bm{x}) + \epsilon\mathbb{B}$ is $\cNP$-hard.
\end{itemize}
\end{nTheorem}
In the above theorem, we use the notation $\widehat{\partial}f(\bm{x})$ to denote the Fr\'echet subdifferential (see \Cref{def:subd-f}) of the function $f$ at point $\bm{x}$, which coincides with $\{\nabla f(\bm{x})\}$ when $f$ is continuously differentiable at $\bm{x}$. Notably, for a \PA{} function $f$, a point $\bm{x}$ satisfies $\bm{0} \in \widehat{\partial}f(\bm{x})$ if and only if $\bm{x}$ is a local minimum (see \Cref{fct:frechetloca}).
To put the above hardness results in perspective, let us make the following remarks:
\begin{itemize}
	\item The co-$\cNP$-hardness of detecting local minima for degree-$p$ polynomials, where $p \geq 4$, has been established in \cite[Theorem 2]{murty1987some}. This result is confirmed to be tight in terms of degree $p$ by \cite[Theorem 3.3]{ahmadi2022AIM}. In \Cref{thm:hard-dc}\ref{item:thm:hard-dc-a}, we show that for piecewise degree-$p$ polynomials, the co-$\cNP$-hardness of checking local minimality emerges even when $p=1$, that is, for piecewise affine functions; see \Cref{tab:complexity}.
\item \citet{nesterov2013gradient} shows that deciding whether a given point is a local minimizer for a \PA{} function is co-$\cNP$-hard. However, the reduction in \citep{nesterov2013gradient} is from the weakly $\cNP$-complete problem of 2-PARTITION leaving open the possibility of the existence of a pseudo-polynomial time algorithm. By contrast, our co-$\cNP$-hardness result in \Cref{thm:hard-dc}\ref{item:thm:hard-dc-a} is in the strong sense, thereby ruling out the aforementioned possibilities.
\item For smooth functions, there is little doubt about the high efficiency in detecting stationary points.  In \Cref{thm:hard-dc}\ref{item:thm:hard-dc-b}, we demonstrate that for a relatively simple class of nonsmooth functions, even approximately testing for (Clarke) stationary points is already $\cNP$-hard. 
Note that a Clarke stationary point of a \PA{} function is not necessarily local minimal.
To the best of our knowledge, this is the first hardness result concerning the testing of a (possibly) non-minimizing first-order optimality condition.
	\item For \PA{} functions given in a natural multi-composite form, we highlight the complexity distinction between testing for a local minimal point (i.e., co-$\cNP$-hardness for $\bm{0} \in \widehat{\partial}(h-g)(\bm{x})$) and for a (possibly) non-minimizing stationary point (i.e., $\cNP$-hardness for $\bm{0} \in \partial(h-g)(\bm{x})$). This distinction appears to be fundamental. In fact, we show in \Cref{thm:mem-dc} that detecting a Clarke stationary point (resp.~a local minimum) cannot be co-$\cNP$-hard (resp.~$\cNP$-hard) unless $\cNP=\text{co-}\cNP$ and the Polynomial Hierarchy (\textsf{PH}) collapses to the second level.
\end{itemize}

\paragraph{Corollaries.} We mention three notable corollaries here. 
\begin{itemize}
\item DC-criticality, which characterizes the points to which DCA-type algorithms converge \cite{de2020abc}, has been studied for over 30 years \cite{le2018dc}. It is well-known that DC-criticality represents a weaker notion than Clarke stationarity. In \Cref{coro:dc-critical}, we demonstrate that determining whether a given DC-critical point is Clarke stationary is $\cNP$-hard. Therefore, distinguishing between these two solution concepts is computationally intractable.
\item In \Cref{coro:abs-norm-hard}, we prove that testing so-called first-order minimality (FOM) for the abs-normal form of piecewise differentiable functions is co-$\cNP$-complete, confirming a conjecture of \citet[p.~284]{griewank2019relaxing}.
\item In \Cref{coro:nnhard}, we show that detecting a (Clarke) stationary point for the loss of Convolutional Neural Networks (CNNs) is $\cNP$-hard. CNNs are one of the most popular network architectures for image classification.
\end{itemize}

\paragraph{Max-Min Representation.} 
Just like convex polytopes can be described either as the convex hull of extreme points or as the intersection of finitely many halfspaces, \PA{} functions also have different representations. Besides being written as the difference between two convex \PA{} functions, any \PA{} function can also be written as the maximin of finitely many affine functions; see \Cref{fct:pa-form}. Compared with the DC form, which could have an exponentially large number of affine pieces, a function given in \MaxMin{} form can only have a polynomially many pieces (in the sense of \citet{scholtes2012introduction}). Nevertheless, in \Cref{thm:hard-general,thm:mem-general}, we show that similar hardness results still hold for the \MaxMin{} representation. The \MaxMin{} form appears to be less popular in real-world applications than its DC counterpart, perhaps due to its limited expressive power and the inconvenience of encompassing sum and multi-composite structures.

\subsubsection{Positive Results}\label{sec:postive}
We have just shown that even approximate stationarity testing is already $\cNP$-hard. A natural strategy to proceed is to solve a relaxation of the original computational problem and isolate conditions under which such a relaxation is tight. In this paper, we consider a very intuitive relaxation by abusing the convex subdifferential sum rule\footnote{For convex functions $h,g:\mathbb{R}^d\to \mathbb{R}$ and any point $\bm{x}\in\mathbb{R}^d$, we always have $\partial (h+g)(\bm{x})=\partial h(\bm{x}) + \partial g(\bm{x})$; see \cite[Theorem 23.8]{rockafellar1970convex}.} to nonconvex, nonsmooth functions. Specifically, we propose to check the condition $\bm{0} \in \partial h(\bm{x}) - \partial g(\bm{x})$, rather than the $\cNP$-hard one $\bm{0} \in \partial (h - g)(\bm{x})$, as follows:

\begin{framed}
\vspace{-.5cm}
\paragraph{Sum Rule Relaxation (SRR).} 
Given convex \PA{} functions $h,g$ and a point $\bm{x}$, 
we check the ``$\epsilon$-stationarity'' of a \PA{} function $h-g$ by running the following procedure:
\begin{itemize}
	\item Compute the shortest vector $\bm{g}$ in the polytope $\partial h(\bm{x}) - \partial g(\bm{x})$.
	\item If $\|\bm{g}\|\leq \epsilon$: return {\sf (True)}; else return {\sf (False)}.
\end{itemize}
\vspace{-.1cm}
\end{framed}

For convex \PA{} functions $h$ and $g$, the convex subdifferentials $\partial h(\bm{x})$ and $\partial g(\bm{x})$ are both non-empty polytopes. The projection in above procedure can be efficiently performed via solving a convex QP.
However, despite its efficiency, such a relaxation procedure cannot always guarantee its correctness. The reason for this turns out to be quite fundamental in the field of nonsmooth analysis.

\paragraph{Subdifferential Sum Rule.}
The problem here is on the failure of an exact (equality-type) subdifferential sum rule. For a locally Lipschitz function, subdifferential sum rule is only known to hold in the weak form of set inclusions rather than equalities.\footnote{For a quick example, consider $\{0\}=\partial(|\cdot| - |\cdot|)(0) \subsetneq \partial |\cdot|(0) - \partial |\cdot|(0)=[-2,2]$.}

\begin{Fact}[cf.~{\cite[Proposition 2.3.3]{clarke1990optimization}}]\label{fct:wsum}
 Let $f_1,f_2:\mathbb{R}^d\to\mathbb{R}$ be locally Lipschitz functions. Then, for every $\bm{x} \in \mathbb{R}^d$, we have $\partial (f_1+f_2)(\bm{x}) \subseteq \partial f_1(\bm{x}) + \partial f_2(\bm{x}).$
\end{Fact}

 This weak form may hinder one from computing the subdifferential set of the concerned function.   Thus, to facilitate the tractability of stationarity testing, it is of interest to establish a condition under which an equality-type sum rule holds. This would allow for efficient characterization of the subdifferential set and guarantee the correctness of our SRR method.

For convex \PA{} functions $h,g$ and a given point $\bm{x}$, we establish the first necessary and sufficient condition for the validity of an equality-type subdifferential sum rule $\partial (h-g)(\bm{x})=\partial h(\bm{x}) - \partial g(\bm{x})$. Our new condition is built on a new geometric property, termed \emph{compatibility}, concerning a pair of convex polytopes.

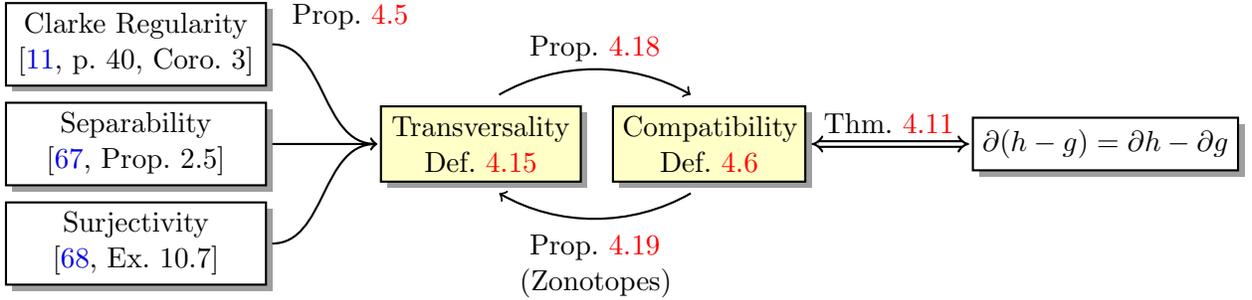
\begin{figure*}
\center
\begin{tikzpicture}[
	block/.style={
		thick,
		draw,
		drop shadow={opacity=0.75},
		fill=white,
		minimum height=0.7cm,
		text centered
	},
	yblock/.style={
		thick,
		draw,
		drop shadow={opacity=0.75},
		fill=ye,
		minimum height=0.7cm,
		text centered
	},
	pil/.style={
		thick,
		->,
		in=180,
		out=0,
		shorten <=2pt,
		shorten >=1pt
	},
	npil/.style={
		thick,
		in=180,
		out=0,
		shorten <=2pt,
		shorten >=1pt
	},
	epil/.style={
		thick,
		in=180,
		out=0,
		shorten <=2pt,
		shorten >=1pt,
		implies-implies,
		double equal sign distance
	},
]
\node[block,text width=3.2cm,align=center] (reg) at (4.5,1.5) {Clarke Regularity \\ \cite[p.~40, Coro.~3]{clarke1990optimization}};
\node[block,text width=3.2cm,align=center,below=.2cm of reg] (sep) {Separability \\ \cite[Prop.~2.5]{rockafellar1985extensions}};
\node[block,text width=3.2cm,align=center,below=.2cm of sep] (suj) {Surjectivity \\ \cite[Ex.~10.7]{rockafellar2009variational}};

\node[yblock,right=1.5cm of sep,text width=2.4cm,align=center] (PI) {Transversality \\ Def.~\ref{def:transv}};

\node[yblock,right=0.4cm of PI,text width=2.3cm,align=center] (comp) {Compatibility \\ Def.~\ref{def:comp-poly}};

\node[block,right=2.2cm of comp] (subd) {$\partial (h-g)=\partial h - \partial g$};

\draw (reg) edge[pil]  (PI);
\path (sep) edge[pil]  (PI);
\path (suj) edge[pil]  (PI);

\path (PI.north) edge[thick, ->, bend left,shorten <=8pt, shorten >= 8pt] node[thick,above] {Prop.~\ref{prop:par-sufficient}} (comp.north);
\path (comp.south) edge[thick, ->, bend left,shorten <=8pt, shorten >= 8pt,] node[thick,below=.1cm, text width=2cm,align=center] {Prop.~\ref{prop:zonotop}\\(Zonotopes)} (PI.south);

\path (comp) edge[epil] node[thick,above] {Thm.~\ref{thm:general-sum-Clarke-geometric}} (subd);

\node[right=.2cm of sep, yshift=1.7cm]{Prop.~\ref{prop:all-are-transveral}};

\end{tikzpicture}
\caption{Interrelations of various conditions validating the exact sum rule for \PA{} functions.}\label{fig:relation}
\end{figure*}

\begin{nDefinition}[Compatibility; cf.~\Cref{def:comp-poly}]%
	Two polytopes $A$ and $B$ in $\mathbb{R}^d$ are called compatible if for any vectors $\bm{a}\in A$ and $\bm{b} \in B$ such that $\bm{a} - \bm{b} \in \ext(A-B)$, we have $\bm{a} + \bm{b} \in \ext(A+B)$.\footnote{The set $\ext(A)$ denotes the set of extreme points of a convex set $A$; see the beginning of \Cref{sec:prel}.}
\end{nDefinition}

One of the main technical contributions of this paper is the following full characterization of the validity of exact subdifferential sum rule for \PA{} functions. To the best of our knowledge, despite being studied for decades in the nonsmooth analysis literature \cite{clarke1975generalized,rockafellar1979directionally,rockafellar1985extensions,clarke1990optimization,Mordukhovich.18}, this is the first time a nontrivial condition validating the (Clarke) subdifferential sum rule has been identified as simultaneously necessary and sufficient.

\begin{nTheorem}[cf.~{\Cref{thm:general-sum-Clarke-geometric}}]
For convex \PA{} functions $h,g:\mathbb{R}^d\to \mathbb{R}$ and any $\bm{x}\in\mathbb{R}^d$, we have $\partial (h - g)(\bm{x})=\partial h(\bm{x}) - \partial g(\bm{x})$ if and only if the polytopes $\partial h(\bm{x})$ and $\partial g(\bm{x})$ are compatible.
\end{nTheorem}

For the intuitive SRR method, 
the compatibility of convex subdifferentials provides a tight characterization of its correctness. Specifically, if the polytopes $\partial h(\bm{x})$ and $\partial g(\bm{x})$ are incompatible, then, by \cite[Exercise 8.8(c)]{rockafellar2009variational}, there exists a vector $\bm{v} \in \mathbb{R}^d$ such that the  relaxation method SRR gives an incorrect answer for the \PA{} function $h'-g$, where $h'(\bm{x}):=\bm{x}^\top \bm{v} + h(\bm{x})$.
However, the time required to verify the compatibility of $\partial h(\bm{x})$ and $\partial g(\bm{x})$ can be exponential. This is because the concerned polytopes  may contain an exponentially large number of vertices.
To alleviate computational difficulties, we introduce a new polynomial-time verifiable sufficient condition, termed transversality,\footnote{See \Cref{rmk:transversality} for connections with existing transversality-type conditions in the literature.} which, when applied to \PA{} functions, generalizes classical notions of regularity in the nonsmooth analysis literature; see \Cref{fig:relation}.

\begin{nDefinition}[Transversality; cf.~\Cref{def:transv}]%
	Given two convex \PA{} functions $h,g:\mathbb{R}^d\to \mathbb{R}$, we say that the functions $h$ and $g$ are transversal at a point $\bm{x} \in \mathbb{R}^d$ if 
	\[
	\parr(\partial h(\bm{x})) \cap \parr(\partial g(\bm{x})) = \{\bm{0}\}.\footnote{The set $\parr(S)$ denote the subspace parallel to the set $S$; see the beginning of \Cref{sec:prel}.}
	\]
\end{nDefinition}

A remarkable property of the notion of transversality, beyond its polynomial-time verifiability and sufficiency in general, is its simultaneous necessity and sufficiency when the polytopes $\partial h(\bm{x})$ and $\partial g(\bm{x})$ are zonotopes.\footnote{Zonotopes are a subclass of convex polytopes; see \Cref{def:zonotope} and \Cref{fig:zonotopes}.}

\begin{nProposition}[Informal; cf.~{\Cref{prop:par-sufficient,prop:zonotop}}] The following hold:
\begin{itemize}
	\item Transversality implies compatibility.
	\item If $\partial h(\bm{x})$ and $\partial g(\bm{x})$ are zonotopes, then compatibility implies transversality.
\end{itemize}
\end{nProposition}

We note in passing that the subdifferentials of key elements in the loss of two-layer ReLU networks \cite{yun2018efficiently}, loss of penalized deep nonsmooth networks \cite{zeng2019global}, and ramp-loss SVMs \cite{brooks2011support} are of zonotope type.
Additionally, our results on regularity, which validate the sum rule, are geometric and independent of the representation of the concerned functions.

\paragraph{Rounding and Finite Termination.}
Up to this point, our focus has been exclusively on the \emph{exact} $\epsilon$-stationarity testing problem, which involves verifying whether $\bm{0} \in \partial (h-g)(\bm{x}) + \epsilon\mathbb{B}$ holds for a point $\bm{x}$. However, in practice, exact nondifferentiable points are almost impossible to reach, primarily due to randomization or finite-precision limitations. Therefore, it is desirable to have a \emph{robust} stationarity testing algorithm that works for any point sufficiently \emph{close} to a stationary (and possibly nondifferentiable) one.
 In other words, we are interested in testing so-called $(\epsilon,\delta)$-Near-Approximate Stationarity ($(\epsilon,\delta)$-NAS; see \Cref{def:stat-concept}), that is, verifying the condition $\bm{0}\in\partial (h-g)(\bm{x} + \delta\mathbb{B}) + \epsilon\mathbb{B}$ for a given point $\bm{x}$.  See \cite{tian2022finite} and the references therein for discussion and algorithms for computing $(\epsilon,\delta)$-NAS points. 
 
 In our new robust testing algorithm, we only need to call the exact $\epsilon$-stationarity testing procedure, i.e., checking whether $\bm{0}\in\partial (h-g)(\bm{x})+\epsilon\mathbb{B}$, in a black-box manner. Therefore, we start by abstracting this procedure for exact testing into the following oracle.

\begin{nDefinition}[cf.~\Cref{def:oracle}]%
	Given a \PA{} function $f:\mathbb{R}^d\to \mathbb{R}$ and point $\bm{x} \in \mathbb{R}^d$, for $\epsilon \geq 0$, the oracle decides whether $\bm{0}\in \partial f(\bm{x}) + \epsilon\mathbb{B}$  or not.
\end{nDefinition}

The main algorithmic contribution of this paper is a  geometric scheme to certify or refute $(\epsilon,\delta)$-NAS for a given point. The following result guarantees the correctness and efficiency of our new robust testing approach.

\begin{nTheorem}[Informal; cf.~{\Cref{thm:robust-test} and \Cref{rmk:uniform-lb}}]
Let convex \PA{} functions $h,g:\mathbb{R}^d\to \mathbb{R}$ and $\epsilon \geq 0$ be given. Suppose that $\bm{x}_n\to \bm{x}^*$ for some unknown $\bm{x}^*\in\mathbb{R}^d$ satisfying $\bm{0}\in\partial (h-g)(\bm{x}^*) + \epsilon\mathbb{B}$ and $\|\bm{x}_n\| \leq B$ for all $n\in\mathbb{N}$. There exists an oracle-polynomial-time algorithm that, for any $n\in\mathbb{N}$ and any $\delta > 0$, certifies either
\[
\bm{0}\in \partial(h-g)(\bm{x}_n + \delta\mathbb{B}) + \epsilon\mathbb{B},\qquad\text{or}\qquad
\|\bm{x}_n - \bm{x}^*\| > \min\{\delta, \delta^*\},
\]
where $\delta^* > 0$ is a constant only dependent on $h,g, \epsilon$, and $B$.%
\end{nTheorem}

One notable application of such a NAS test is to obtain an efficient termination criterion for algorithms that only have asymptotic convergence results. 
For example, given a lower-bounded \PA{} function, every limiting point of the sequence generated by the subgradient method is a Clarke stationary point (see \Cref{fct:ddtame}). %
However, it is still unclear when to terminate the algorithm and how to certify  the obtained point is at least close to some Clarke stationary point, as the norm of any vector in the subdifferential is almost surely lower bounded away from zero along the entire trajectory  (consider running the subgradient method on $x\mapsto |x|$). 
 Our NAS testing approach offers an algorithm-independent stopping rule, effectively transforming asymptotically convergent algorithms into finite-time ones.
The following immediate corollary highlights this application.

\begin{nCorollary}[Informal; cf.~{\Cref{coro:stop}}]
Let two convex \PA{} functions $h,g:\mathbb{R}^d\to\mathbb{R}$ be given. Consider the iterates $\{\bm{x}_n\}_n$ produced by the subgradient method on the function $h-g$. 
There exists an oracle-polynomial-time algorithm such that, for any $\epsilon\geq 0$ and $\delta > 0$, the stopping criterion $\bm{0}\in\partial (h-g)(\bm{x}_T + \delta\mathbb{B}) + \epsilon\mathbb{B}$ can be certified for a finite $T\in\mathbb{N}_+$ by the algorithm.
Consequently, for \PA{} functions, the subgradient method can be terminated confidently in finite time.
\end{nCorollary}

Notably, when specialized to neural networks with ReLU activation functions (see \Cref{sec:2relu}), the above corollary resolves the open problem proposed in \cite[Section 5]{yun2018efficiently}.
	We also mention that, for black-box \PA{} functions in the sense of \cite{nemirovskij1983problem}, robust testing (i.e., detecting NAS points) is impossible to implement in general \cite[Proposition 2]{tian2022no}.

\subsection{Our Techniques} 
We now discuss the main challenges and new ideas in our development.

\subsubsection{Computational Complexity}

\paragraph{DC Representation.} 
For \PA{} functions, the set of Fr\'echet stationary points exactly matches the local minima. However, the existing hardness result for detecting local minima of polynomials cannot be applied here mainly due to two reasons. First, the proof for degree-$p$ polynomials in \cite[Problem 11]{murty1987some} requires $p \geq 4$, which is tight according to \cite[Theorem 3.3]{ahmadi2022AIM}. Despite the piecewise nature, \PA{} functions are merely piecewise degree-$1$ polynomials. Second, in sharp contrast to \PA{} functions, the set of Fr\'echet stationary points of a smooth function $f$, such as polynomials, coincides with the set $\{\bm{x}: \nabla f(\bm{x}) = \bm{0}\}$, whose detection is known to be in the class $\cP$ \cite{ahmadi2022AIM}. 
As for the complexity of detecting a Clarke stationary point, let convex \PA{} functions $h,g:\mathbb{R}^d\to \mathbb{R}$ be given. The subdifferentials $\partial h(\bm{0})$ and $\partial g(\bm{0})$ are both convex polytopes. Indeed, we can verify the condition $\bm{0} \in \partial h(\bm{0}) - \partial g(\bm{0})$ by solving an LP feasibility problem in polynomial time, provided a proper description of the polytopes $\partial h(\bm{0})$ and $\partial g(\bm{0})$. Therefore, the construction for proving hardness of checking whether $\bm{0} \in \partial (h-g)(\bm{0})$ must fundamentally exploit  the failure of the exact subdifferential sum rule (i.e., $\partial (h-g)(\bm{0}) \subsetneq \partial h (\bm{0})- \partial g(\bm{0})$). Besides, while being popular in computation, a Clarke stationary point could be non-minimizing in a meaningful sense, unlike Fr\'echet stationarity. To our knowledge, there is no hardness result on verifying non-minimizing first-order necessary conditions; therefore, we need some new ideas for our proof.

To obtain the desired hardness results, 
for the Fr\'echet case, we construct two simple convex \PA{} functions $h_{\sf F}$ and  $g_{\sf F} $. Our reduction is from maximizing the $\ell_1$-norm over a centered parallelotope, called PAR$\{-1,0,1\}$MAX$_1$ problem; see~{\cite[Theorem 12]{bodlaender1990computational}. We relate the complement problem of checking whether $\bm{0} \in \widehat{\partial} (h_{\sf F} - g_{\sf F})(\bm{0}) + \epsilon\mathbb{B}$ to a polytope containment problem, and then to the strongly $\cNP$-hard norm-maximization PAR$\{-1,0,1\}$MAX$_1$ problem. 
To prove the hardness of detecting Clarke stationarity, the key construction is a seesaw-type gadget composed of two convex \PA{} functions $h_{\sf C}$ and  $g_{\sf C}$. The gadget uses the Fr\'echet stationarity status of $h_{\sf F} - g_{\sf F} $ as its trigger. Specifically, when $\bm{0} \in \widehat{\partial} (h_{\sf F}-g_{\sf F})(\bm{0})$, the gadget function $h_{\sf C} - g_{\sf C}$ is affine near the point $\bm{0}$ with $\|\nabla (h_{\sf C} - g_{\sf C})(\bm{0})\| \geq 1/2$. When $\bm{0} \notin \widehat{\partial} (h_{\sf F}-g_{\sf F})(\bm{0})$, the gadget will open a flat passage near the point $\bm{0}$ resulting $\bm{0} \in \partial (h_{\sf C} - g_{\sf C})(\bm{0})$.

The distinction in complexity between testing whether $\bm{0} \in \widehat{\partial}(h-g)(\bm{0})$ and testing whether $\bm{0} \in \partial(h-g)(\bm{0})$ can be demonstrated by showing membership in their own complexity classes. Unlike problems with explicit discrete structure, the completeness for some problems related to continuous optimization can be anything but trivial (see, e.g., \cite{fearnley2022complexity}), especially when we allow a rather rich and expressive class of convex \PA{} components. 
The crux to our completeness results is a protocol to certify membership to certain index set, termed \emph{essentially active index set} \cite[p.~92]{scholtes2012introduction}.
Assuming the separation between $\cNP$ and co-$\cNP$, we conclude that the distinction in computational complexity between testing Fr\'echet and Clarke stationarities is fundamental.

\paragraph{Max-Min Representation.} %
The main challenge in proving hardness for the \MaxMin{} form lies in its inefficiency in expressiveness. Indeed, any \PA{} function represented in \MaxMin{} form can have only a polynomial number of affine pieces with respect to its input size. In our proof, we build a polynomial-time reduction from 3SAT; see \cite[Section 3.1.1]{garey1979computers}. For the Fr\'echet case, we construct a \PA{} function $f_{\sf F}$ given in \MaxMin{} form and demonstrate that the subdifferential $\widehat{\partial} f_{\sf F}(\bm{0})$ contains a small vector if and only if a given instance of 3SAT is unsatisfiable.  
For the Clarke case, we adapt the seesaw-type gadget originally developed for the \DC{} representation to translate the hardness of testing Fr\'echet stationarity to that of Clarke stationarity.

\subsubsection{Subdifferential Sum Rule} 
For convex \PA{} functions $h,g:\mathbb{R}^d\to \mathbb{R}$, the validity of the exact sum rule $\partial (h-g)(\bm{x}) = \partial h(\bm{x}) - \partial g(\bm{x})$ at a given point $\bm{x}\in\mathbb{R}^d$ is central to understanding the correctness of the natural SRR method  for stationarity testing in \Cref{sec:postive}. Moreover, the equality-type sum rule has numerous applications in nonsmooth optimization and analysis; see, e.g., \cite[Chapter 10]{rockafellar2009variational}, \cite[Section 2.4]{Mordukhovich.18}. 
The main obstacle and technical contribution in our study are mostly conceptual, focusing on identifying the correct regularity condition. 

The calculus of (Clarke) subdifferentials is a well-developed area, pioneered by giants in convex and variational analysis.
Monographs by \citet{clarke1990optimization}, \citet{rockafellar2009variational}, and \citet{Mordukhovich.18} have collected and consolidated decades of development in this direction. 
After examining some known regularity conditions validating the sum rule for the \PA{} function $h-g$, we are left with a vague sense that some form of separability between $h$ and $g$ seems necessary; see \cite[p.~119]{melzer1986expressibility} for similar remark. This observation leads us to consider a transversality-type condition in \Cref{def:transv}, which can be seen as a generalized separability condition encompassing three known regularities as special cases when applied to \PA{} functions (see \Cref{prop:all-are-transveral}). However, its unnecessity for the sum rule was initially unclear to us. Meanwhile, the discovery of the simultaneous necessity and sufficiency of transversality for zonotopes (see \Cref{prop:zonotop}) appears as strong evidence of its necessity for the general \PA{} function $h-g$.

This belief is eventually disproven by a nontrivial example in $\mathbb{R}^4$ (see \Cref{exam:comp}), where the subdifferential sum rule holds without transversality. Consequently, the new notion of compatibility (see \Cref{def:comp-poly}) emerges as the correct regularity condition that fully characterizes the validity of the sum rule, with transversality serving as an appealing, polynomial-time verifiable sufficient condition.
The main technical challenge in our investigation lies in characterizing how the geometry of the polytopes $\partial h(\bm{x})$ and $\partial g(\bm{x})$ affects the structure of the analytically defined polytope $\partial (h-g)(\bm{x})$.
The goal is to identify a simple, practical, and nontrivial regularity condition that uses only  information from $\partial h(\bm{x})$ and $\partial g(\bm{x})$ to ensure $\partial (h-g)(\bm{x})=\partial h(\bm{x})-\partial g(\bm{x})$. We find the new notion of compatibility, in some sense, neat and elegant. Moreover, both the concepts of transversality and compatibility are purely geometric and independent of the concrete representation of the convex functions $h$ and $g$.

\subsubsection{Rounding and Finite Termination}

Stopping criterion is important in the design of iterative algorithms for optimization problems.
Our rounding algorithm is inspired by the finite-termination strategies for solving LP with interior-point methods (IPMs); see \cite[Chapter 6]{grotschel2012geometric} and \cite{ye1992finite,mehrotra1993finding,spielman2003smoothed}. However, our task differs significantly from that of terminating IPMs.
 LP is convex, and IPMs are guaranteed to converge to a global optimal solution. In contrast, our interest lies in terminating an algorithm that converges to a stationary point of a nonconvex objective function. Hence, neither convexity nor optimality gap can be exploited in algorithm design for finite termination.
	Moreover, the existing techniques in \cite{ye1992finite,mehrotra1993finding,spielman2003smoothed} crucially rely on the strict complementarity of the solution that many IPMs converge to (see \cite{guler1993convergence}). However, in our stationarity testing setting, it is unrealistic to make any assumptions about the stationary point to which the algorithm is converging. %
	Another closely related line of research is on \emph{active manifold/constraints identification}; see \cite{hare2004identifying,lewis2011identifying} and the references therein. These works typically require the concerned function to be Clarke regular or even amenable, ruling out general \PA{} functions. Moreover, they can only guarantee identification if the sequence of subgradients evaluated along the iterates approaches zero, which is almost impossible when applying the subgradient method to \PA{} functions.

	Similar to LP and many numerical algorithms, we need a standard algebraic representation for the convex functions $h$ and $g$. According to \cite[Theorem 2.49]{rockafellar2009variational}, every convex \PA{} function can be written as the pointwise maximum of finitely many affine functions. However, this representation is somewhat inefficient, as the number of affine pieces grows only polynomially with respect to the input size. This lack of expressive power renders many computational problems trivial and cannot adequately address many real-world problems (e.g., the empirical loss of a shallow ReLU neural network), where a natural parameterization results in at least exponentially many pieces. In this paper, we consider a multi-composite  (\MC) form of convex \PA{} functions; see \Cref{def:mc}. Simply put, the \MC{} representation allows for arbitrarily deep compositions and summations of pointwise maximum over affine functions. An \MC{} function is, by definition, convex, piecewise affine, and can have exponentially many affine pieces.
	As discussed in \cite[Proposition 2.2.3]{scholtes2012introduction}, every \PA{} function $f:\mathbb{R}^d \to \mathbb{R}$ defines a corresponding polyhedral subdivision of $\mathbb{R}^d$. The \PA{} function $f$ is actually affine on every element of the polyhedron subdivision. Given a point $\bm{w} \in \mathbb{R}^d$, there can be more than exponentially many polyhedra in the subdivision near $\bm{w}$, whose enumeration is clearly computational intractable. The key component of our algorithm is a family of convex polyhedra $\{P^\delta \}_\delta$, parameterized by a positive scalar $\delta$, which governs the process of searching for the unknown stationary point (or certifying its absence).

Let us informally sketch the idea behind our new robust testing algorithm. Let  convex \PA{} functions $h,g:\mathbb{R}^d\to \mathbb{R}$ and a point $\bm{w} \in \mathbb{R}^d$ be given. Fix any desired precision $\epsilon \geq 0$ and $\delta > 0$. Our goal is to identify an unknown stationary point $\bm{w}^* \in \mathbb{R}^d$ near the point $\bm{w}$ or certify its absence. To convey the intuition informally, imagine we are standing at the point $\bm{w}$, swinging a butterfly net in an attempt to capture the elusive $\bm{w}^*$. The length of the net is adjustable, set by a positive scalar $\delta_0:=\delta$, and the shape of its opening is also adjustable, polyhedral, and parameterized by $\delta_0$. Let $P^{\delta_0} \subseteq \mathbb{R}^d$ represent the opening. We swing the net with $\delta_0$ by projecting the point $\bm{w}$ onto the polyhedral $P^{\delta_0}$. Suppose that the projection is $\widehat{\bm{w}} \in P^{\delta_0}$. If $\|\bm{w} - \widehat{\bm{w}}\| \leq \delta$ and $\bm{0} \in \partial (h-g)(\widehat{\bm{w}}) + \epsilon\mathbb{B}$, which is efficiently verifiable by the $\epsilon$-stationarity testing oracle, then we succeed by confirming $\bm{w}$ as an $(\epsilon,\delta)$-\NAS{} point with the certificate $\widehat{\bm{w}}$. If not, we halve $\delta_0$ to $\delta_0/2$ and repeat the capturing process. This loop will not continue indefinitely; we will stop the iteration after at most a polynomial number of steps. The crux of the correctness of our algorithm is that if an identification condition is satisfied, 
then we are guaranteed to capture the point $\bm{w}^* \in P^{\delta_0}$ for some positive $\delta_0$, and the projection $\widehat{\bm{w}}$ satisfies
$\partial (h-g)(\widehat{\bm{w}}) = \partial (h-g)(\bm{w}^*)$. This enables us to verify the $(\epsilon,\delta)$-\NAS{} status of $\bm{w}$ without precisely pinpointing the original target $\bm{w}^*$. 
On the contrary, if the identification condition is not satisfied before termination, then we can certify that the candidate point $\bm{w}$ is at least $\min\{\delta, \delta^*\}$ away from any $\epsilon$-stationary point, where $\delta^*$ is a constant independent of the point $\bm{w}$.

\subsection{Related Work and Organization}

\paragraph{Complexity of Testing Solutions.}

There have been many works on the complexity of deciding whether a given point belongs to a certain solution type. It is shown in \citep{murty1987some,pardalos1988checking} that checking whether a given point is local minimum, or strict local minimum for unconstrained or simply constrained problem are both co-$\cNP$-hard. For low-degree polynomials, the work \citep{ahmadi2022AIM} shows that it is possible to efficiently determine the existence and membership of local minima. The work \citep[Theorem 4]{bolte2022complexity} demonstrates that certifying the singleton of the Clarke subdifferential set is $\cNP$-hard. 
The work \citep{nesterov2013gradient} shows that deciding whether a given point is a local minimizer for a \PA{} function is (weakly) co-$\cNP$-hard. However, the result in \citep{nesterov2013gradient} only applies to the exact testing of local minimizers, and it is unclear whether the construction is DC-representable with constant-layer $\MC$ components. Most importantly, the reduction in \citep{nesterov2013gradient} is from the 2-PARTITION problem, leaving open the possibility of detecting local minimizers in pseudo-polynomial time.
For any piecewise smooth function representable by the so-called abs-normal form, the work \cite{griewank2016first} shows that a first-order optimality condition, also called first-order minimality (FOM) in \cite{griewank2019relaxing}, can be verified efficiently under a linear independence constraint qualification (LICQ)-type condition, termed linear independence kink qualification (LIKQ). When applied to \PA{} functions, this LIKQ condition has a close relation to a surjectivity-type assumption.
Another related work to ours is the one by \citet{yun2018efficiently}. They consider the empirical loss of a two-layer ReLU network and introduce an algorithm to check Clarke stationarity, which is essentially similar to the SRR method presented in \Cref{sec:postive}. A limitation of the work \cite{yun2018efficiently} (discussed in \cite[Section 5]{yun2018efficiently}) is that the algorithm therein can only perform \emph{exact} stationarity testing. That is to say, if the objective function is $x\mapsto |x|$, then the algorithm in \cite{yun2018efficiently} will certify approximate stationarity only if the given point is exactly equal to zero. Our results resolve the open problem on \emph{robust} testing mentioned in \cite[Section 5]{yun2018efficiently} for more general \PA{} functions. %

\paragraph{Nonsmooth Optimization.}
The convergence of subgradient-type methods to Clarke stationary points is primarily analyzed from a continuous-time perspective, as studied in \cite{benaim2005stochastic,majewski2018analysis,davis2020stochastic}. These results are only asymptotic, lacking any oracle complexity guarantees. The work \cite{tian2022no} shows that this absence is fundamental; even for \PA{} functions, the trajectory of the subgradient method cannot bypass an approximate stationary point in any \emph{a priori} finite time. 
For the nonasymptotic aspect, the development of iterative methods for solving nonconvex nonsmooth optimization problems is still in its early stages. The work \citep{davis2019stochastic} shows that for weakly convex, Lipschitz functions, $(\epsilon, \delta)$-NAS points can be obtained using a subgradient-type method. However, the requirement of weak convexity is somewhat stringent and excludes general \PA{} functions.
Recently, there has been a surge of research aimed at obtaining oracle complexity results for general Lipschitz functions; see, e.g., \citep{zhang2020complexity,tian2022finite,davis2021gradient}. An important notion of approximate stationarity adopted in these works is the so-called Goldstein approximate stationarity (GAS). Formally, a point $\bm{x}$ is called an $(\epsilon, \delta)$-GAS point for a function $f$ if $\bm{0} \in \conv \partial f(\bm{x} + \delta\mathbb{B}) + \epsilon\mathbb{B}$.
It is easy to see that $(\epsilon,\delta)$-NAS, in the sense of \eqref{eq:intro-stat}, is a more stringent, and therefore more desirable, solution concept than $(\epsilon,\delta)$-GAS. However, the work \cite{kornowski2021oracle} establishes an exponential lower bound on the oracle complexity for computing $(\epsilon,\delta)$-NAS points by local algorithms. In comparison, the works \citep{zhang2020complexity,tian2022finite,davis2021gradient} propose randomized algorithms that compute $(\epsilon,\delta)$-GAS points with dimension-free oracle complexity. Interestingly, the works \cite{tian2022no,jordan2023deterministic} show that no deterministic algorithm can achieve the same for computing $(\epsilon,\delta)$-GAS points.
By restricting the function class with a so-called nonconvexity modulus, the work \cite{kong2023cost} presents a deterministic algorithm that computes $(\epsilon,\delta)$-GAS points with dimension-free oracle complexity.

\paragraph{Organization.} We introduce terminologies and preliminaries in \Cref{sec:prel}. In \Cref{sec:hardness}, we present our results on computational complexity. The regularity conditions for the validity of subdifferential sum rule for \PA{} functions and their relationships are discussed in \Cref{sec:cr}. We present the robust algorithms for testing near-approximate stationarity concepts in \Cref{sec:robust}. In \Cref{sec:app}, we discuss the applicability of our new results to a series of structured piecewise smooth functions. The paper concludes with remarks in \Cref{sec:concl}. %
Missing proofs and discussions are supplemented in the appendices.

\section{Preliminaries}\label{sec:prel}

Throughout this paper, scalars, vectors and matrices are denoted by lowercase letters, boldface lower case letters, and boldface uppercase letters, respectively.
The notation used in this paper is mostly standard: $\mathbb{R}_+:=[0,\infty)$ and $\mathbb{R}_{++}:=(0,\infty)$; $\bm{0}_n\in\mathbb{R}^n$ and $\bm{0}_{n\times m}\in\mathbb{R}^{n\times m}$ denote a vector and a matrix with all elements equal to zero (we may write $\bm{0}$ when the dimensionality is clear from the context);  $x_i$ denotes the $i$-th coordinate of vector $\bm{x}$; $A \times B\coloneqq \{(a,b):a \in A, b\in B\}$ denotes the Cartesian product of two sets $A$ and $B$; $\|\cdot\|_p$ denotes the $\ell_p$ norm (the subscript may be omitted when $p = 2$); $\mathbb{B}_\epsilon(\bm{x})\coloneqq\{\bm{v}:\|\bm{v} - \bm{x}\|\leq \epsilon\}$ with $\mathbb{B}\coloneqq \mathbb{B}_1(\bm{0})$; $[\bm{x},\bm{y}]\coloneqq \{\gamma\bm{x}+(1-\gamma)\bm{y}:\gamma \in [0,1]\}$; $\dist(\bm{x},S)\coloneqq\inf_{\bm{v}\in S} \|\bm{v}-\bm{x}\|$ for a closed set $S$;
the set $\ext(A)$ denotes the set of extreme points of a convex set $A$; the Minkowski addition of sets $A$ and $B$ is $A+B:=\{\bm{a}+\bm{b}: \bm{a} \in A, \bm{b} \in B\}$ (we may write $A-B$ as a shorthand of $A+(-B)$);
the sets $\spn(S),\conv(S)$, $\intt(S)$, $\cone(S), \parr(S):=\spn(S-S)$ denote the linear hull of, convex hull of, interior of, conic hull of, and subspace parallel to the set $S$, respectively; $\bm{e}_i$ denotes the $i$-th column of identity matrix; $[a] := \{1,\ldots,a\}$ for any integer $a\ge1$; for a set $S$ and a set-valued mapping $M$, $M(S) := \cup_{\bm{z} \in S} M(\bm{z})$;  $(x_i)_{i\in[n]}:=(x_1, \dots, x_n) \in \mathbb{R}^n$ for any integer $n \geq 1$; 
 for a directionally differentiable function $f$, $f'(\bm{x}; \bm{d})$ denotes its directional derivative at point $\bm{x}$ in direction $\bm{d}$; 
 given a vector $\bm{x}\in\mathbb{R}^n$ and $i \in [n]$, eliminating the $i$-th coordinate from $\bm{x}$ yields the vector $\bm{x}_{-i} \in \mathbb{R}^{n-1}$; $\mathcal{T}_C(\bm{x})$ and $\mathcal{N}_C(\bm{x})$ denote the tangent cone and normal cone of a convex set $C$ at the point $\bm{x}$, respectively.%

\subsection{Nonsmooth Analysis} 
 Let us review some notions on generalized differentiation for locally Lipschitz functions.
\begin{Definition}[Clarke subdifferential; cf.~{\citep[Theorem 2.5.1]{clarke1990optimization}}]\label{def:subd} Given a point $\bm{x} \in \mathbb{R}^d$, the Clarke subdifferential of a locally Lipschitz function $f:\mathbb{R}^d \to \mathbb{R}$ at $\bm{x}$ is defined by
	\[
	\partial f(\bm{x}) := \conv\big\{\bm{s}:\exists \bm{x}_n\to \bm{x}, \nabla f(\bm{x}_n) \textnormal{ exists}, \nabla f(\bm{x}_n)\to\bm{s}\big\}.
	\]
	We call $\bm{x}$ a (Clarke) stationary point of $f$ if $\bm{0} \in \partial f(\bm{x})$.
\end{Definition}

\begin{Definition}[cf.~{\cite[p.~299]{rockafellar2009variational}}] For $\tau > 0$ and $\bm{x} \in \mathbb{R}^d$, the difference quotient functions $\Delta_\tau f(\bm{x}):\mathbb{R}^d \to \mathbb{R}$ of a real-valued function $f:\mathbb{R}^d\to\mathbb{R}$ at $\bm{x}$ is defined by
	\[
	\Delta_\tau f(\bm{x})(\bm{d}):=\frac{f(\bm{x} + \tau\bm{d}) - f(\bm{x})}{\tau}.
	\]
\end{Definition}

\begin{Fact}[{cf.~\cite[p.~27]{clarke1990optimization} and \cite[Proposition 2.2.4, Proposition 2.2.7]{clarke1990optimization}}]\label{fct:clarke-dd}
	The Clarke generalized directional derivative of a locally Lipschitz function $f:\mathbb{R}^d\to \mathbb{R}$ at $\bm{x} \in \mathbb{R}^d$ in the direction $\bm{d}\in \mathbb{R}^d$, denoted $f^\circ (\bm{x}; \bm{d})$, is defined as 
	\[
	f^\circ (\bm{x}; \bm{d}):=\limsup_{\substack{\bm{x}'\to \bm{x}\\ t \searrow 0}} \Delta_t f(\bm{x}')(\bm{d})=\limsup_{\substack{\bm{x}'\to \bm{x}\\ t \searrow 0}}\frac{f(\bm{x}'+t\bm{d})-f(\bm{x}') }{t}.
	\]
	The sublinear function $f^\circ(\bm{x}; \cdot): \mathbb{R}^d \to \mathbb{R}$ is the support function of the set $\partial f(\bm{x})$; i.e.,   
	\[
	\partial f(\bm{x}) = \left\{\bm{s}: \bm{s}^\top \bm{d} \leq f^\circ (\bm{x}; \bm{d}) \text{ for all } \bm{d} \in \mathbb{R}^d\right\}.
	\]
	Moreover, when $f$ is continuously differentiable at $\bm{x}$, we have $\partial f(\bm{x})=\{\nabla f(\bm{x})\}$. If $f$ is convex, the set $\partial f(\bm{x})$ coincides with the subdifferential in the sense of convex analysis\footnote{The subdifferential in the sense of convex analysis of a convex function $f$ at point $\bm{x}$ is defined as $\partial f(\bm{x}):=\left\{\bm{g}: f(\bm{y}) \geq f(\bm{x}) + \bm{g}^\top (\bm{y} - \bm{x}),\forall \bm{y}\right\}$; see \cite[Definition D.1.2.1]{hiriart2004fundamentals}.} of $f$ at $\bm{x}$.
\end{Fact}

For a locally Lipschitz function, the Clarke subdifferential is always non-empty, convex, and compact \cite[Proposition 2.1.2(a)]{clarke1990optimization}. Now, recall the directional derivative $f'$ of a directionally differentiable function $f$ at $\bm{x}$ in the direction $\bm{d}$ is defined as 
\[
f' (\bm{x}; \bm{d}):=\lim_{t \searrow 0} \Delta_t f(\bm{x})(\bm{d})=\lim_{t \searrow 0} \frac{f(\bm{x}+t\bm{d})-f(\bm{x}) }{t}.%
\]
While our main focus in this work is on the Clarke subdifferential and stationarity, some of our results have implications for, and our analysis requires, another commonly used generalized subdifferential, which we discuss below.
\begin{Definition}[Fr\'echet subdifferential; cf.~{\citep[Exercise 8.4]{rockafellar2009variational}}]\label{def:subd-f}
	Given a point $\bm{x} \in \mathbb{R}^d$, the Fr\'echet subdifferential of a locally Lipschitz and directionally differentiable function $f:\mathbb{R}^d \to \mathbb{R}$ at $\bm{x}$ is defined by
	\[
	\widehat{\partial} f (\bm{x}) := \left\{\bm{s}: \bm{s}^\top \bm{d} \leq f'(\bm{x};\bm{d}) \text{ for all } \bm{d} \in \mathbb{R}^d \right\}.
	\]
	We call $\bm{x}$ a Fr\'echet stationary point of $f$ if $\bm{0} \in \widehat{\partial} f(\bm{x})$.
\end{Definition}

In the following, we record a generalized Fermat's rule for optimality conditions and the relationship among the aforementioned two subdifferentials.
\begin{Fact}[Fermat's rule; cf.~{\cite[Theorem 8.6, 8.49, 10.1]{rockafellar2009variational}}]\label{fact:relation-fermat}
Given a locally Lipschitz and directionally differentiable function $f:\mathbb{R}^d\rightarrow\mathbb{R}$ and a point $\bm{x}\in\mathbb{R}^d$, we always have
 $\widehat{\partial} f(\bm{x}) \subseteq  \partial f(\bm{x})$. Moreover, if the point $\bm{x}$ is a local minimizer of the function $f$, it holds that $\bm{0} \in \widehat{\partial} f(\bm{x}) \subseteq  \partial f(\bm{x})$.
\end{Fact}
Consequently, various first-order stationarity concepts can be defined using different generalized subdifferential constructions. The set of Clarke (resp.~Fr\'echet) stationary points of the function $f$ can be written as $\{\bm{x}: \bm{0} \in \partial f(\bm{x})\}$ (resp.~$\{\bm{x}: \bm{0} \in \widehat{\partial} f(\bm{x})\}$). Unless explicitly mentioned, all (approximate) stationarity concepts in this paper are defined by Clarke subdifferential.
However, under certain regularity conditions, the Clarke subdifferential coincides with the Fr\'echet subdifferential,
so that the corresponding stationarity concepts are also equivalent. The following condition, named Clarke regularity, is classic to ensure this equivalence.
\begin{Definition}[Clarke regularity; cf.~{\cite[Definition 2.3.4]{clarke1990optimization}}]\label{def:regular}
	A locally Lipschitz function $f:\mathbb{R}^d \to \mathbb{R}$ is Clarke regular at $\bm{x}\in\mathbb{R}^d$ if for every $\bm{d} \in \mathbb{R}^d$, the directional derivative $f'$ exists and coincides with the Clarke generalized one:
	\[
	f'(\bm{x}; \bm{d}) = f^\circ(\bm{x}; \bm{d}).
	\]
	If $f$ is regular at every $\bm{x}\in\mathbb{R}^d$, then we say $f$ is regular.
\end{Definition}

With finite computational resources, for a continuous optimization problem, an algorithm usually computes an approximate solution rather than an exact one,\footnote{With algorithms for linear programming as a remarkable exception; see for instance \cite{wright1997primal}.} though the precision can sometimes be arbitrarily high. When the objective function $f$ is smooth, a point $\bm{x}$ is called $\epsilon$-stationary of function $f$ if $\|\nabla f(\bm{x}) \| \leq \epsilon$.
Similar to the notion of differentiation, if the function $f$ is nonsmooth, the notion of approximation becomes very subtle because different approximation schemes have varying degrees of computability.
\begin{Definition}[Approximate stationarity]\label{def:stat-concept}
	We call a point $\bm{x}$ an $\epsilon$-stationary point of $f$ if 
	\[
	\bm{0} \in \partial f(\bm{x}) + \epsilon\mathbb{B}. 
	\]
	If there exists a point $\bm{y} \in \mathbb{B}_\delta(\bm{x})$ such that $\bm{y}$ is $\epsilon$-stationary, we call $\bm{x}$ an $(\epsilon, \delta)$-near-approximate stationary (NAS) point. In other words, a point $\bm{x}$ is $(\epsilon, \delta)$-\NAS{} if
	\[
	\bm{0}\in \partial f(\bm{x} + \delta\mathbb{B}) + \epsilon\mathbb{B} = \bigcup\left\{ \partial f(\bm{y}) + \epsilon\mathbb{B}: \bm{y} \in \mathbb{B}_\delta(\bm{x})\right\}.
	\]
\end{Definition}

\begin{Remark}
Another commonly used generalized subdifferential, namely limiting (or Mordukhovich) subdifferential, is defined in \cite[Definition 8.3(b)]{rockafellar2009variational}. Limiting subdifferential is a robust version of Fr\'echet subdifferential. However, unlike these of Fr\'echet and Clarke, the limiting subdifferential set can be nonconvex or even discrete. Therefore, it might spark less interest in detecting the stationarity concept defined by the limiting subdifferential due to the evident computational challenges.
\end{Remark}

\subsection{Piecewise Differentiable/Affine Functions} 
We mostly follow the terminology in \cite[Section 4.1]{scholtes2012introduction}. A continuous function $f:\mathbb{R}^d \to \mathbb{R}$ is called \emph{piecewise differentiable} (or \emph{piecewise smooth}), if
there exist a finite number $N\in\mathbb{N}$ and a finite set of continuously differentiable functions $f_i:\mathbb{R}^d\to \mathbb{R}$ for $i\in[N]$, such that $f(\bm{w}) \in \{f_1(\bm{w}), \dots, f_N(\bm{w})\}$ holds for any point $\bm{w} \in \mathbb{R}^d$.  
In particular, we call a continuous function $f:\mathbb{R}^d \to \mathbb{R}$ \emph{piecewise affine} (see \cite[Section 2.2]{scholtes2012introduction}), if
there exist a finite number $N\in\mathbb{N}$ and vectors $(\bm{x}_i, a_i)\in \mathbb{R}^{d+1}$ for all $i\in[N]$, such that $f(\bm{w}) \in \{\bm{x}_1^\top\bm{w} + a_1, \dots, \bm{x}_N^\top\bm{w} + a_N\}$ holds for any point $\bm{w} \in \mathbb{R}^d$.  Every real-valued \PA{} function can be represented in some standard forms, which we describe below.

\begin{Fact}[Representations of PA functions; cf.~{\cite[Proposition 4]{melzer1986expressibility}, \cite{kripfganz1987piecewise},~and~\cite[Proposition 2.2.2]{scholtes2012introduction}}]\label{fct:pa-form}
	Let a \PA{} function $f:\mathbb{R}^d\to \mathbb{R}$ be given.	\begin{enumerate}[label=\textnormal{(\alph*)}]
		\item 
		There exist two convex \PA{} functions $h,g:\mathbb{R}^d \to \mathbb{R}$ such that for any $\bm{w} \in \mathbb{R}^d$, it holds
		\[
		f(\bm{w}) = h(\bm{w}) - g(\bm{w}). \tag{DC}
		\]
We call above representation a \DC{} form of the function $f$. \label{item:fct:pa-form-a}
\item There exist finite numbers $k,l \in \mathbb{N}$, data $(\bm{x}_j, a_j) \in \mathbb{R}^{d+1}$ for $j \in [k]$, and index sets $\mathcal{M}_i \subseteq [k]$ for  $i \in [l]$ such that for any $\bm{w} \in \mathbb{R}^d$, it holds
		\[
		f(\bm{w}) = \max_{1\leq i \leq l} \min_{j \in \mathcal{M}_i} \bm{w}^\top \bm{x}_j + a_j. \tag{\MaxMin}
		\]
		We call above representation a \MaxMin{} form of the function $f$.\label{item:fct:pa-form-b}
	\end{enumerate}
\end{Fact}

When discussing convex \PA{} functions, we consider the following representation with an $n$-layer multi-composite structure, termed $n$-\MC{} form. 

\begin{Definition}[Multi-composite (MC) function] \label{def:mc}
We call a convex \PA{} function $f:\mathbb{R}^d\to\mathbb{R}$ an $n$-\MC{} function if it is given by
	\[
f(\bm{w}):=\sum_{1\leq j_n \leq J_n} \max_{1\leq i_{n} \leq I_n} \cdots \sum_{1\leq j_1 \leq J_1} \max_{1\leq i_1 \leq I_1} \Big(\bm{w}^\top \bm{x}_{i_1,j_1,\dots,i_n,j_n} + a_{i_1,j_1,\dots,i_n,j_n}\Big).
\]
\end{Definition}
From \cite[Theorem 2.49]{rockafellar2009variational}, every convex \PA{} function can be written in $1$-\MC{} form.
By modeling with a more general $n$-\MC{} form, we can compactly represent convex \PA{} functions. Indeed, even when the depth $n$ is fixed, an $n$-\MC{} function may still have an exponentially large number of pieces.

Besides, the directional derivative of the \PA{} function enjoys favorable properties.

\begin{Fact}[cf.~{\cite[Proposition 4.4.3(c), Proposition 4.1.2(a)]{cui2021modern}}]\label{fct:pa-dd}
Let a \PA{} function $f:\mathbb{R}^d\to \mathbb{R}$ with \DC{} components $h,g:\mathbb{R}^d\to \mathbb{R}$ be given. For any $\bm{w}\in\mathbb{R}^d$ and any $\bm{w}'$ sufficiently near $\bm{w}$, we have
$f(\bm{w}') = f(\bm{w}) + f'(\bm{w}; \bm{w}' - \bm{w})$ and $f'(\bm{w}; \cdot) = h'(\bm{w}; \cdot)-g'(\bm{w}; \cdot)$.
\end{Fact}

The following important fact elucidates the continuity and directional differentiability of piecewise differentiable functions.
\begin{Fact}[cf.~{\cite[Corollary 4.1.1, Proposition 4.1.3]{scholtes2012introduction}}]
Every piecewise differentiable function is locally Lipschitz continuous and directionally differentiable. 	
\end{Fact}

\paragraph{Stationarity Concepts.} Here, we discuss some properties and give examples of different kinds of stationary points of piecewise smooth functions. First, for \PA{} functions, an important fact is that Fr\'echet stationary points are local minima.

\begin{Fact}[cf.~{\cite[Proposition 6.1.1]{cui2021modern}}]\label{fct:frechetloca}
For a \PA{} function $f:\mathbb{R}^d\to \mathbb{R}$ and a point $\bm{x}\in\mathbb{R}^d$, we have $\bm{0} \in \widehat{\partial} f(\bm{x})$ if and only if the point $\bm{x}$ is a local minimizer of the function $f$. 	
\end{Fact}

As shown in \Cref{fct:pa-form}, every \PA{} function can be written as the difference of two convex \PA{} functions. A solution notion tolerated for \DC{} functions, termed \DC-criticality, is classic in the \DC{} programming literature; see \cite{le2018dc,de2020abc}.

\begin{Definition}[DC-criticality; cf.~{\cite[p.~8]{le2018dc}}]\label{def:dc-crit}
 Let two convex functions $h,g:\mathbb{R}^d \to \mathbb{R}$ be given.
	We call a point $\bm{x} \in \mathbb{R}^d$ a \DC-critical point of $h-g$ if 
	\[
	\bm{0} \in \partial h(\bm{x}) - \partial g(\bm{x}), \qquad\text{or equivalently,} \qquad
	\emptyset\neq \partial h(\bm{x})\cap \partial g(\bm{x}).
	\]
\end{Definition}

One can readily see from the weak form sum rule in \Cref{fct:wsum} that \DC-criticality is a weaker notion of stationarity than that of Clarke. Besides, the definition of \DC-criticality depends on the \DC-decomposition $h,g$, which is not unique. In contrast, Clarke stationarity is a purely analytic notion and  independent of the \DC{} representation. To provide intuition, we give examples of different kinds of stationary points as follows.

\begin{Example}\label{ex:stat}
	In the left subfigure of \Cref{fig:stat-notion}: the set of local minima is $\{x_1,x_3\}$; the set of Fr\'echet stationary points is $\{x_1,x_3\}$; the set of Clarke stationary points is
		$\{x_1,x_2,x_3,x_4\}$. In the right subfigure, we consider a simple \DC{} function 
	$f_2(x):=h(x) - h(-x) = x$ with convex component $h(x):=\max\{x,0\}$. Then the point $x_5=0$ is \DC-critical while it is obviously not stationary in any conventional sense, as $f_2$ is continuously differentiable at $0$ with $\nabla f_2(0) = 1$.
\end{Example}

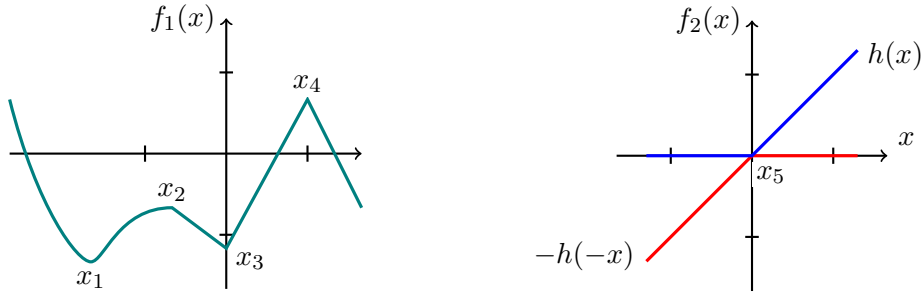
\begin{figure}[thb]
	\centering
\begin{tikzpicture}[scale=0.36] %
	\draw[thick, black, ->] (-8, 0) -- (5, 0);
	\foreach \x in {-3, 3}
	\draw[thick] (\x, 0.3) -- (\x, -0.3);
	\draw[thick, black, ->] (0, -5) -- (0, 5)
	node[left] {$f_1(x)$};
	\foreach \y in {-3, 3}
	\draw[thick] (-0.25, \y) -- (0.25, \y);
	\draw[very thick, teal, -] (-8, 2) .. controls (-7,-2) and (-5.5,-4) .. (-5,-4) .. controls (-4.5,-4) and (-4,-2) .. (-2, -2)--(0,-3.5)--(3,2)--(5,-2);
	
	\node[anchor=north west,fill=white, inner sep=0.5mm] at (0.2,-3.5) {\textcolor{black}{$x_3$}};
	\node[anchor=south,fill=white, inner sep=0.5mm] at (-2,-1.9) {\textcolor{black}{$x_2$}};
	\node[anchor=south,fill=white, inner sep=0.5mm] at (3,2.1) {\textcolor{black}{$x_4$}};
	\node[anchor=north,fill=white, inner sep=0.5mm] at (-5,-4.2) {\textcolor{black}{$x_1$}};

\end{tikzpicture}  
\hspace{2cm}
\begin{tikzpicture}[scale=0.36] %
	
	\draw[thick, black, ->] (-5, 0) -- (5, 0)
	node[anchor=south west] {$x$};
	\foreach \x in {-3, 3}
	\draw[thick] (\x, 0.3) -- (\x, -0.3);
	\draw[thick, black, ->] (0, -5) -- (0, 5)
	node[left] {$f_2(x)$};
	\foreach \y in {-3, 3}
	\draw[thick] (-0.25, \y) -- (0.25, \y);
	\node [fill=white,inner sep=0.5mm] at (0.7,-0.7) {\textcolor{black}{$x_5$}} ;
	\draw[very thick, color=red, -] (-3.9, -3.9) --(0, 0) --(3.9, 0);
	\node[fill=white, inner sep=0.1mm] at (-6.2,-3.7) {$-h(-x)$};
	
	\draw[very thick, blue, -] (3.9, 3.9) --(0, 0)--(-3.9, 0);
	\node[fill=white,inner sep=0.1mm] at (5.3,3.7) {$h(x)$};
	
\end{tikzpicture} 
\caption{Illustration of various stationarities (better viewed in color); see \Cref{ex:stat}. 
}\label{fig:stat-notion}
\end{figure}

\subsection{Facts from Polyhedral Geometry}

A set in $\mathbb{R}^d$ is called (convex) \emph{polyhedral} if it can be represented by the intersection of a finite family of closed halfspaces. 
If a polyhedral set is bounded, we call it a (convex) \emph{polytope}. 
Similar to PA functions,
two equivalent representations of polytopes are popular in the literature.
	We call a set \emph{$\mathcal{V}$-polytope} if it is represented as the convex hull of a finite set of points in $\mathbb{R}^d$.
	An \emph{$\mathcal{H}$-polyhedron} is defined as an intersection of finitely many closed halfspaces in $\mathbb{R}^d$. When an $\mathcal{H}$-polyhedron is bounded, we call it an \emph{$\mathcal{H}$-polytope}. 
	By the Minkowski-Weyl Theorem for polytopes (see \cite[Theorem 2.15]{ziegler2012lectures}), a polytope is a set in $\mathbb{R}^d$ which can be represented either as a $\mathcal{V}$-polytope or (equivalently) as an $\mathcal{H}$-polytope.

\emph{Zonotopes}, as an especially interesting subclass of polytopes, have close relation with linear hyperplane arrangement \cite[Section 7]{ziegler2012lectures}.
  A zonotope $Z\subseteq \mathbb{R}^d$ can be described as the Minkowski sum of line segments or the projection of a cube. %
\begin{Definition}[Zonotope; cf.~{\cite[Definition 7.13]{ziegler2012lectures} and \cite[p.~323]{grunbaum2003convex}}]\label{def:zonotope}
	A \emph{zonotope} is the image of a cube under an affine projection, that is, a polytope $Z\subseteq \mathbb{R}^d$ of the form
	\[
	Z := \bm{V} [-1,1]^p + \bm{z}  =  \left\{ \bm{z} + \sum_{i=1}^p x_i\cdot \bm{v}_i: x_j \in [-1,1], j \in [p] \right\},
	\]
	for some integer $p\in\mathbb{N}$, vector $\bm{z}\in\mathbb{R}^d$, and matrix $\bm{V} = (\bm{v}_1, \dots, \bm{v}_p) \in \mathbb{R}^{d\times p}$. 
	When the matrix $\bm{V}$ has full column rank, we call the zonotope $Z$ a parallelotope.
\end{Definition}

When dealing with generalized differentiations of piecewise affine functions, we need to characterize the solvability of some system of linear inequalities. To this end, the following Farkas-type theorem of alternatives is useful.
\begin{Lemma}[Gordan; {cf.~\cite[Exercise 4.26]{bertsimas1997introduction}}]\label{lem:gordan}
	Let a matrix $\bm{A}\in\mathbb{R}^{n\times m}$ be given. Then, exactly one of the following statements is true:
	\begin{itemize}
		\item There exists an $\bm{x}\in \mathbb{R}^m$ such that $\bm{Ax} < \bm{0}$.
		\item There exists a $\bm{y} \in \mathbb{R}^n$ such that $\bm{A}^\top \bm{y} = \bm{0}$ with $\bm{y} \geq \bm{0}, \bm{y}\neq \bm{0}$.
	\end{itemize}
\end{Lemma}

Subdifferential sum rule pertains to the geometry of the Minkowski sum of subdifferentials of elemental functions. For polytopes, the following result gives a characterization of the extreme points by \emph{Minkowski decomposition} of Minkowski summands.
\begin{Lemma}[cf.~{\cite[Proposition 2.1, Corollary 2.2]{fukuda2004zonotope}}, {\cite[Lemma 2.1.2]{scholtes2012introduction}}]\label{lem:fukuda}
Let $P_1, \dots, P_k$ be polytopes in $\mathbb{R}^d$ and let $P := \sum_{i=1}^k P_i$. We have $\bm{v} \in \ext(P)$ if and only if $\bm{v}$ has a Minkowski decomposition $\bm{v} = \sum_{i=1}^k \bm{v}_i$ with $\bm{v}_i\in \ext(P_i)$ and there exists a vector $\bm{h} \in \mathbb{R}^d$ such that
$
\{\bm{v}_i\} = \argmax_{\bm{z} \in P_i} \langle \bm{h}, \bm{z} \rangle
$  uniformly
over all $i\in[k]$. Moreover, this Minkowski decomposition is unique for any $\bm{v} \in \ext(P)$.
\end{Lemma}

\section{Computational Complexity}\label{sec:hardness}

In this section, we discuss the computational complexity of testing stationarity concepts for different representations of \PA{} functions. %
The proofs of the main results are collected in \Cref{sec:prf-dc,sec:prf-maxmin}.

\subsection{DC Representation}\label{sec:hardness-DC}

If the \PA{} function is provided in form of \DC{} representation (see \Cref{fct:pa-form}\ref{item:fct:pa-form-a}), the primary result of this subsection is the following.
\begin{Theorem}
	[DC-hardness]\label{thm:hard-dc} 
Fix any $n\geq 2$. Let two convex \PA{} function $h,g:\mathbb{R}^d\rightarrow \mathbb{R}$ be given in $n$-\MC{} form
with integer data.  The following hold:
\begin{enumerate}[label=\textnormal{(\alph*)}]
		\item Fix any $\epsilon \in [0, \infty)$, checking whether $\bm{0}\in \widehat{\partial} (h-g)(\bm{0}) + \epsilon\mathbb{B}$ is strongly co-$\cNP$-hard.\label{item:thm:hard-dc-a}
		\item Fix any $\epsilon \in [0, 1/2)$, checking whether $\bm{0} \in \partial (h-g)(\bm{0})+ \epsilon\mathbb{B}$ is strongly $\cNP$-hard. \label{item:thm:hard-dc-b}%
	\end{enumerate}
	\end{Theorem}
	
	We compare \Cref{thm:hard-dc} with the classic hardness result of \citet{murty1987some}. In \cite{murty1987some}, checking the local optimality of a point for a simply constrained indefinite quadratic problem \cite[Problem 1]{murty1987some}  and for an unconstrained quartic polynomial objective \cite[Problem 11]{murty1987some} are both co-$\cNP$-complete. However, these hardness results are inapplicable for checking first-order necessary conditions. %
In \Cref{thm:hard-dc}, we show that, for unconstrained optimization problems with \PA{} objectives, even approximately testing the first-order necessary condition at a specific point $\bm{x}=\bm{0}$, with $h$ and $g$ provided in fixed-depth \MC{} form, is computationally intractable unless $\cP = \cNP$. To the best of our knowledge, this is the first hardness result concerning the testing of a non-minimizing first-order optimality condition.

The technique for proving \cite[Theorem 2]{murty1987some}---the hardness result for detecting local minima of polynomials---cannot be applied here mainly due to two reasons. First, the proof in \cite[Theorem 2]{murty1987some} requires degree-$p$ polynomials for $p \geq 4$, which is tight according to \cite[Theorem 3.3]{ahmadi2022AIM}. Despite the piecewise nature, \PA{} functions are merely piecewise degree-$1$ polynomials. Second, when considering Fr\'echet stationary points of smooth functions $f$, such as polynomials, they coincide with the set $\{\bm{x}: \nabla f(\bm{x}) = \bm{0}\}$, whose verification is known to be in the class $\cP$ \cite{ahmadi2022AIM}.

Notably, \citet{nesterov2013gradient} presents a \PA{} construction demonstrating that detecting a descent direction is co-$\cNP$-hard, which implies that testing local minimizers of a \PA{} function is generally computationally intractable. However, the reduction used in \cite[Lemma 1]{nesterov2013gradient} is from the $2$-PARTITION problem, which is known to be only weakly $\cNP$-hard and solvable in pseudo-polynomial time via a dynamic programming algorithm. In contrast, in \Cref{thm:hard-dc}\ref{item:thm:hard-dc-a}, we show that detecting such local minimality is actually strongly co-$\cNP$-hard, thereby ruling out the possibility of any pseudo-polynomial-time algorithms.

	When functions $h$ and $g$ are represented in the $n$-\MC{} form, we highlight the complexity distinction between testing for a local minimum, i.e., co-$\cNP$-hardness for $\bm{0} \in \widehat{\partial}(h-g)(\bm{x})$, and for a stationary point, i.e., $\cNP$-hardness for $\bm{0} \in \partial f(\bm{x})$. This distinction appears to be fundamental. In fact, for $\epsilon=0$, even when the depth $n$ is treated as an input (rather than fixed), we have the following membership results.\footnote{In our hardness result (\Cref{thm:hard-dc}), we restrict the functions $h$ and $g$ to have fixed depth, whereas in the membership result (\Cref{thm:mem-dc}), we allow the depth of $h$ and $g$ to be treated as an input. These choices only strengthen both the positive and negative results.}

\begin{Theorem}
	[DC-completeness]\label{thm:mem-dc} 
Let two convex \PA{} function $h,g:\mathbb{R}^d\rightarrow \mathbb{R}$ be given  in $n$-\MC{} form
with integer data and the depth $n$ as an input. %
The following hold:	
\begin{enumerate}[label=\textnormal{(\alph*)}]
		\item Checking whether $\bm{0}\in \widehat{\partial} (h-g)(\bm{0})$ is co-$\cNP$-complete. \label{item:thm:mem-dc-a} %
		\item Checking whether $\bm{0} \in \partial (h-g)(\bm{0})$ is $\cNP$-complete. \label{item:thm:mem-dc-b} %
	\end{enumerate}
	\end{Theorem}

An immediate  corollary of \Cref{thm:mem-dc} is that verifying whether $\bm{0}\in\partial (h-g)(\bm{0})$ is unlikely to be co-$\cNP$-hard. Suppose, conversely, that it were proven that checking $\epsilon$-stationarity is co-$\cNP$-hard. Then, due to its $\cNP$-completeness from \Cref{thm:mem-dc}\ref{item:thm:mem-dc-b}, every problem in co-$\cNP$ can be reduced to a problem in $\cNP$ in polynomial time, hence implying co-$\cNP\subseteq \cNP$. Consequently, we would arrive at co-$\cNP= \cNP$, which implies the collapse of the Polynomial Hierarchy (\textsf{PH}) to the second level, a scenario many believe to be unlikely.

\begin{Remark}
The $n$-\MC{} structure of the \DC{} components $h$ and $g$ provides highly efficient expressiveness, allowing a wide range of constructions in the \PA{} function $h-g$. This rich expressiveness only strengthens our membership results in \Cref{thm:mem-dc}. Meanwhile, we emphasize that our hardness result in \Cref{thm:hard-dc} holds for fairly simple and non-pathological functions; see \Cref{prop:DCF,prop:DCC} in \Cref{sec:hardness-DC}.
Notably, the source of hardness does not stem from difficulties in characterizing the convex subdifferentials $\partial h(\bm{0})$ and $\partial g(\bm{0})$. In fact, the sets $\partial h(\bm{0})$ and $\partial g(\bm{0})$ are representable by polynomial-size LPs, and verifying $\bm{0} \in \partial h(\bm{0})-\partial g(\bm{0}) + \epsilon\mathbb{B}$ can be done in polynomial time; see \Cref{sec:nMC}. Instead, the hardness arises intrinsically from the complex interaction between the convex and concave parts of $h-g$, which prevents efficient computation of $\partial (h-g)(\bm{0})$.
In \Cref{sec:cr}, we will examine the subdifferential formula for $\partial (h-g)(\bm{0})$ and provide a complete characterization of the case where $\partial (h-g)(\bm{0})=\partial h(\bm{0})-\partial g(\bm{0})$.
\end{Remark}

\subsection{Max-Min Representation}

If the \PA{} function is provided in form of \MaxMin{} representation (see \Cref{fct:pa-form}\ref{item:fct:pa-form-b}), we have the following hardness results.
\begin{Theorem}
	[Max-Min-hardness]\label{thm:hard-general} 
Let a \PA{} function $f:\mathbb{R}^d\rightarrow \mathbb{R}$ in the form of \MaxMin{} representation
with integer data be given. %
The following hold:	
\begin{enumerate}[label=\textnormal{(\alph*)}]
		\item Fix any $\epsilon \in [0, \infty)$, checking whether $\bm{0}\in \widehat{\partial} f(\bm{0}) + \epsilon\mathbb{B}$ is strongly co-$\cNP$-hard. \label{item:thm:hard-general-a}%
		\item Fix any $\epsilon \in [0, 1/2)$, checking whether $\bm{0} \in \partial f(\bm{0})+ \epsilon\mathbb{B}$ is strongly $\cNP$-hard. \label{item:thm:hard-general-b}%
	\end{enumerate}
	\end{Theorem}

	It is notable that a \PA{} function presented in \MaxMin{} form can have only polynomially many affine pieces, which necessitates techniques different from those used in the proof of \Cref{thm:hard-dc}. Recall the notation used in the \MaxMin{} representation in \Cref{fct:pa-form}\ref{item:fct:pa-form-b}. Fix $ i \in [k] $ and a point $\bm{w}$. Determining whether the selection function \(\bm{w} \mapsto \bm{w}^\top\bm{x}_i + a_i\) is essentially active \cite[p.~92]{scholtes2012introduction} at $\bm{w}$ is \(\cNP\)-hard. Otherwise, we could decide the entire essentially active index set and consequently determine the set \(\partial f(\bm{0})\) in polynomial time, which would contradict \Cref{thm:hard-general}.
Another interpretation of the source of hardness arises from the perspective of nonsmooth analysis. For convex functions $\{f_i: \mathbb{R}^d \to \mathbb{R}, i \in [l]\}$, by the exact chain rule \cite[Corollary D.4.3.2]{hiriart2004fundamentals}, we have  
\[
\partial \left[\max_{1 \leq i \leq l} f_i\right](\bm{w}) = \conv \bigcup \left\{\partial f_j(\bm{w}) : \max_{1 \leq i \leq l} f_i(\bm{w}) = f_j(\bm{w}), j \in [l]\right\}.
\]
Although the \MaxMin{} form has a similar form \( f = \max_{1 \leq i \leq l} g_i \) with \( g_i(\bm{w}) := \min_{1 \leq j \leq M_i} \bm{w}^\top \bm{x}_j + a_j \), the functions $\{g_i: \mathbb{R}^d \to \mathbb{R}, i \in [l]\}$ are concave rather than convex. Hence, the exact chain rule in \cite[Corollary D.4.3.2]{hiriart2004fundamentals} does not apply. Therefore, while \(\partial g_i(\bm{w})\) can be computed easily, deriving an efficiently computable formula for \(\partial f(\bm{w})\) is impossible unless \(\cP = \cNP\).

	Similar to the \DC{} representation, the distinction in complexity between testing for a local minimum and a stationary point remains fundamental for the \MaxMin{} form, as illustrated by the following membership results.

\begin{Theorem}
	[\MaxMin-completeness]\label{thm:mem-general} 
Let a \PA{} function $f:\mathbb{R}^d\rightarrow \mathbb{R}$ in the form of \MaxMin{} representation
with integer data be given. %
The following hold:	
\begin{enumerate}[label=\textnormal{(\alph*)}]
		\item Checking whether $\bm{0}\in \widehat{\partial} f(\bm{0})$ is  co-$\cNP$-complete.%
		\item Checking whether $\bm{0} \in \partial f(\bm{0})$ is $\cNP$-complete.%
	\end{enumerate}
	\end{Theorem}

We end this subsection by quoting \cite[p.~23]{scholtes2012introduction} that ``the construction of a handy max-min form is often very difficult, if at all possible.''
 
 \subsection{Discussion}
 
 We present three noteworthy corollaries of \Cref{thm:hard-general,thm:hard-dc}.

 \paragraph{DC-criticality.} %
 The DC representation naturally provides a \DC{} decomposition of the concerned \PA{} function. Therefore, it is legitimate to discuss the concept of DC-criticality (see \Cref{def:dc-crit}), a solution notion tailored for DC programming \cite{le2018dc,de2020abc}.
 It is well-known that DC-criticality represents a weaker notion than Clarke stationarity. In \Cref{coro:dc-critical}, for the first time, it is shown that determining whether a DC-critical point is Clarke stationary is $\cNP$-hard. Therefore, distinguishing between these two solution concepts is computationally intractable.
 
\begin{Corollary}\label{coro:dc-critical}
There exists a family of convex \PA{} functions, with $h,g:\mathbb{R}^d\to\mathbb{R}$ being two instances, such that the following hold.
\begin{itemize}
	\item Both $\partial h(\bm{0})$ and $\partial g(\bm{0})$ are efficiently LP representable.
	\item We always have $\bm{0} \in \partial h(\bm{0}) - \partial g(\bm{0})$.
	\item Checking whether $\bm{0} \in \partial (h-g)(\bm{0})$ is $\cNP$-hard.
\end{itemize}
\end{Corollary}
\begin{proof}
See \Cref{sec:prf-dc-critical}.
\end{proof}

 Importantly, this corollary highlights that the emergence of $\cNP$-hardness does not stem from difficulties in representing $\partial h(\bm{0})$ or $\partial g(\bm{0})$. Instead, the hardness arises from the complicated interaction between the convex and concave parts.

 \paragraph{Abs-Normal Form.}
 Nonsmooth functions in real-world applications usually contain structures that can be exploited in theoretical analysis and algorithmic design. A subclass of piecewise differentiable functions, termed $C^d_{\textnormal{abs}}$ or functions representable in abs-normal form \cite[Definition 2.1]{griewank2019relaxing}, and defined as the composition of smooth functions and the absolute value function, is introduced by \citet{griewank2013stable}; see \Cref{sec:prf-abs} for a quick and informal review.
An important corollary of our hard construction concerns the complexity of checking an optimality condition, called first-order minimality \cite[Equation (2)]{griewank2019relaxing}, for functions in $C^d_{\textnormal{abs}}$. The following result gives an affirmative answer to a conjecture of \citet[p.~284]{griewank2019relaxing}:

\begin{Corollary}[FOM of abs-normal form]\label{coro:abs-norm-hard}
	Testing First-Order Minimality (FOM) for a piecewise differentiable  function given in the abs-normal form is co-$\cNP$-complete. 
\end{Corollary}
\begin{proof}%
	See \Cref{sec:prf-abs}.
\end{proof}

\paragraph{Convolutional Neural Networks.}
Another notable corollary is on the complexity of detecting an approximate (Clarke) stationary point for the loss of a Convolutional Neural Networks (CNNs). CNNs are one of the most popular network architectures for image classification.
\begin{Corollary}[CNNs]\label{coro:nnhard}
Let $f:\mathbb{R}^d\rightarrow\mathbb{R}$ be the loss function of a shallow CNN with \textnormal{ReLU} activation function and max-pooling operator. Fix any $\epsilon \in [0, 1/2)$.
 Then, testing the (Clarke) $\epsilon$-stationarity $\bm{0} \in \partial f(\bm{\theta})+  \epsilon\mathbb{B}$ for given network parameters $\bm{\theta}\in\mathbb{Q}^d$ is $\cNP$-hard. %
\end{Corollary}

\begin{proof}%
	See \Cref{sec:cnn-hard}.
\end{proof}

\Cref{coro:nnhard} shows a computational tractability separation for the stationarity test between smooth and nonsmooth networks. In the smooth setting, given the gradient of every component function, we can compute the gradient norm of the loss function by iteratively applying the calculus rule. But in the nonsmooth case, while the (Clarke) subdifferential of every elemental function can be computed easily, the validity of the subdifferential calculus rule is not justified, which turns out to be a serious computational hurdle in the stationarity test (strong $\cNP$-hardness).

\subsection{Proofs of \Cref{thm:hard-dc,thm:mem-dc} (DC)}\label{sec:prf-dc}

We formulate the following decision problems for the DC representation.

 \begin{Problem}[PAR$\{-1,0,1\}$MAX$_1$; cf.~{\cite[Theorem 12]{bodlaender1990computational}}]\label{prop:PCP}
 	Suppose that positive integers $n, m \leq n, \alpha$, and $m$ linearly independent vectors $\bm{y}_1, \dots, \bm{y}_m$ in $\{-1,0,1\}^n$ are given. 
 	Determine whether there exist $\epsilon_1, \dots, \epsilon_m \in \{-1,1\}$ such that
 	\[
 	\left\|\sum_{i=1}^m \epsilon_i\bm{y}_i \right\|_1 \geq \alpha.
 	\]
 \end{Problem}
 
 \begin{Problem}\label{prop:DCF}%
 	Fix $\epsilon \in [0, \infty)$. Suppose the input data $r\in\mathbb{N}$ and $\{\bm{y}_i\}_{i=1}^m \subseteq \{-1,0,1\}^n$ are given. Let us define two convex PA functions $h_{\sf F},g_{\sf F}:\mathbb{R}^n \to \mathbb{R}$ as 
 	\[
 	h_{\sf F}(\bm{d}):= r\|\bm{d}\|_\infty,\qquad g_{\sf F}(\bm{d}):=
 	\max\left\{r\|\bm{d}\|_\infty, \sum_{i=1}^m \left| \bm{d}^\top \bm{y}_i \right| \right\}.
 	\]
  Determine whether $\bm{0} \in \widehat{\partial} (h_{\sf F}-g_{\sf F})(\bm{0})+\epsilon \mathbb{B}$.
 \end{Problem}
 \begin{Problem}\label{prop:DCC} 
	Fix $\epsilon \in [0, 1/2)$. Suppose the input data $r\in\mathbb{N}$ and $\{\bm{y}_i\}_{i=1}^m \subseteq \{-1,0,1\}^n$ in the statement of \Cref{prop:DCF} are given. Let us define two convex \PA{} functions $h_{\sf C}, g_{\sf C}:\mathbb{R}^n \to \mathbb{R}$ as
	\[
	h_{\sf C}(\bm{d}):= \frac{\bm{d}^\top \bm{e}_1}{2} + \max\left\{h_{\sf F}(\bm{d}) + \frac{|\bm{d}^\top \bm{e}_1|}{2}, g_{\sf F}(\bm{d})\right\},\qquad
	g_{\sf C}(\bm{d}):= g_{\sf F}(\bm{d}) + \frac{|\bm{d}^\top \bm{e}_1|}{2},
	\]
	where the convex \PA{} functions $h_{\sf F},g_{\sf F}:\mathbb{R}^n \to \mathbb{R}$ are defined in the \Cref{prop:DCF}.
	Determine whether  
	$
	\bm{0} \in \partial (h_{\sf C} - g_{\sf C})(\bm{0}) + \epsilon\mathbb{B}.
	$
\end{Problem}

\subsubsection{Hardness}

The proof of hardness results in \Cref{thm:hard-dc} is built on the following two lemmas.

\begin{Lemma}\label{lem:dc-np-hard}
	\Cref{prop:DCF} is strongly co-{\sf NP}-hard.
\end{Lemma}
\begin{proof}
From \cite[Theorem 12]{bodlaender1990computational} and \cite[p.~220]{bodlaender1990computational}, we know that \Cref{prop:PCP} is strongly $\cNP$-hard.
We present a polynomial-time reduction with polynomially bounded magnitude from PAR$\{-1,0,1\}$MAX$_1$ in \Cref{prop:PCP} to $\overline{\text{\Cref{prop:DCF}}}$.
We start with some preparatory moves.
	Let vectors $\bm{y}_1, \dots, \bm{y}_m$ of \Cref{prop:DCF} be these in \Cref{prop:PCP}. Set $r:=\alpha -1$ and we know $r \in \mathbb{N}$ from positivity of $\alpha$.
 Let us define polytopes $X:=\{\bm{d}:\|\bm{d}\|_1 \leq r\}$ and $Y:=\sum_{i=1}^m [-1,1]\bm{y}_i$. By linear independence of vectors $\bm{y}_1, \dots, \bm{y}_m$, we see that $Y$ is a parallelotope in $\mathbb{R}^n$. Note that
 \begin{align*}
 \sum_{i=1}^m \left| \bm{d}^\top \bm{y}_i \right| &= \sum_{i=1}^m \max\{ \bm{d}^\top \bm{y}_i, -\bm{d}^\top \bm{y}_i\} \\
 & = \sum_{i=1}^m \max_{\bm{y} \in Y_i:=[-1,1]\bm{y}_i} \bm{d}^\top \bm{y} \tag{\cite[Proposition C.2.2.1]{hiriart2004fundamentals}} \\
 & = \max_{\bm{y} \in Y=\sum_{i=1}^m  Y_i}\bm{d}^\top \bm{y}. \tag{\cite[Corollary 16.4.1]{rockafellar1970convex}}
 \end{align*}
Using the duality between  $\|\cdot\|_\infty$ and $\|\cdot\|_1$, we can rewrite the \PA{} function $h_{\sf{F}} - g_{\sf{F}}$ as
 \[
 \begin{aligned}
  h_{\sf{F}}(\bm{d}) - g_{\sf{F}}(\bm{d}) 
  &= \min\left\{0, r\|\bm{d}\|_\infty - \sum_{i=1}^m \left| \bm{d}^\top \bm{y}_i \right| \right\}\\
  &= \min\left\{0, \left(\max_{\bm{x} \in X} \bm{d}^\top \bm{x}\right) - \left(\max_{\bm{y} \in Y} \bm{d}^\top \bm{y} \right)\right\} \\
  &= \min\left\{0, \sigma_X(\bm{d})- \sigma_Y(\bm{d})\right\},
 \end{aligned}
 \]
 where $\sigma_X,\sigma_Y:\mathbb{R}^n\to \mathbb{R}$ are support functions of polytopes $X$ and $Y$. Fix some $\epsilon \in [0,\infty)$.
In the sequel, we show that $\bm{0} \notin \widehat{\partial} (h_{\sf{F}} - g_{\sf{F}})(\bm{0}) + \epsilon\mathbb{B}$ if and only if there exist $\epsilon_1, \dots, \epsilon_m \in \{-1,1\}$ such that $\left\|\sum_{i=1}^m \epsilon_i\bm{y}_i \right\|_1 \geq \alpha=r+1$.

	(``If'') Let $\bm{y}' = \sum_{i=1}^m \epsilon_i\bm{y}_i$ be such that $\|\bm{y}'\|_1 \geq r+1$ and $\epsilon_i \in \{-1,1\}$.  Note that $\bm{y}' \in Y$. Let $\bm{d}':=\sgn(\bm{y}') \in \{-1,1\}^n$, where we set $d'_i \in \{-1,1\}$ arbitrarily if $y'_i = 0$. Compute
	\[
	(h_{\sf{F}} - g_{\sf{F}})(\bm{d}')=\min\left\{0,r-\max_{\bm{y} \in Y} \bm{d}'^\top \bm{y} \right\} \leq \min\left\{0,r - \bm{d}'^\top \bm{y}'\right\} = \min\{0,r - \|\bm{y}'\|_1\}\leq -1.
	\]
	Note that $(h_{\sf{F}} - g_{\sf{F}})(\bm{d}) = (h_{\sf{F}} - g_{\sf{F}})'(\bm{0}; \bm{d})=h_{\sf{F}}'(\bm{0}; \bm{d}) - g_{\sf{F}}'(\bm{0}; \bm{d})$ for any $\bm{d} \in\mathbb{R}^n$. 
	We proceed to prove that $\widehat{\partial} (h_{\sf{F}} - g_{\sf{F}})(\bm{0})=\emptyset$, hence $\bm{0} \notin \widehat{\partial} (h_{\sf{F}} - g_{\sf{F}})(\bm{0}) + \epsilon\mathbb{B}$ for any $\epsilon \in [0,\infty)$.
	Suppose conversely that there exists a vector $\bm{g} \in \widehat{\partial} (h_{\sf{F}} - g_{\sf{F}})(\bm{0})$, so that, by definition, $\bm{g}^\top \bm{d} \leq (h_{\sf{F}} - g_{\sf{F}})'(\bm{0}; \bm{d}) = (h_{\sf{F}} - g_{\sf{F}})(\bm{d})$ for any $\bm{d}$. Examining $\bm{d}:=\bm{g}$, we find that $\|\bm{g}\|^2 \leq (h_{\sf{F}} - g_{\sf{F}})(\bm{g}) \leq 0$. It follows $\bm{g}=\bm{0}$, which implies $(h_{\sf{F}} - g_{\sf{F}})(\bm{d}) \geq \bm{g}^\top \bm{d} = 0$ for any $\bm{d}$, a contradiction to $(h_{\sf{F}} - g_{\sf{F}})(\bm{d}') \leq -1$. 
	
	(``Only if'') Note that $\bm{0} \notin \widehat{\partial} (h_{\sf{F}} - g_{\sf{F}})(\bm{0}) + \epsilon\mathbb{B}$ for some $\epsilon \in [0,\infty)$ implies the existence of $\bm{d}' \in \mathbb{R}^n$ such that $(h_{\sf{F}} - g_{\sf{F}})(\bm{d}')<0$, as otherwise $(h_{\sf{F}} - g_{\sf{F}})(\bm{d})=0$ for all $\bm{d}$, so that $\{\bm{0}\} = \widehat{\partial} (h_{\sf{F}} - g_{\sf{F}})(\bm{0}).$ Therefore, we have $\sigma_X(\bm{d}') - \sigma_Y(\bm{d}') < 0$ from above reformulation. Using \cite[Corollary 13.1.1]{rockafellar1970convex}, one has $Y \not\subseteq X$. Thus, there exists $\bm{y} \in \ext(Y)$ such that $\bm{y}\notin X$. In other words, $\|\bm{y}\|_1 > r$. From \cite[Corollary 18.3.1]{rockafellar1970convex}, we know $\ext(Y) \subseteq \{\sum_{i=1}^m \epsilon_i\bm{y}_i: \epsilon_i\in\{-1,1\}\}$. As $\bm{y}_i\in\{-1,0,1\}^n$ for any $i \in [m]$, by integrality, we have $\|\bm{y}\|_1\geq r+1$. This yields the existence of $\epsilon_1, \dots, \epsilon_m \in \{-1,1\}$ such that 
	$\left\|\sum_{i=1}^m \epsilon_i\bm{y}_i \right\|_1 = \|\bm{y}\|_1 \geq r+1 = \alpha$, as desired.	
\end{proof}

\begin{Lemma}\label{lem:dc-np-hard-C}
	\Cref{prop:DCC} is strongly {\sf NP}-hard.
\end{Lemma}
\begin{proof}
By using the construction of convex \PA{} functions $h_{\sf F},g_{\sf F}$ in the proof of \Cref{lem:dc-np-hard}, we present a polynomial-time reduction with polynomially bounded magnitude from the strongly $\cNP$-hard problem PAR$\{-1,0,1\}$MAX$_1$ in \Cref{prop:PCP} to \Cref{prop:DCC}. %

Fix some $\epsilon \in [0,1/2)$.
In the sequel, we show that $\bm{0} \in \partial (h_{\sf{C}} - g_{\sf{C}})(\bm{0}) + \epsilon\mathbb{B}$ if and only if there exist $\epsilon_1, \dots, \epsilon_m \in \{-1,1\}$ such that $\left\|\sum_{i=1}^m \epsilon_i\bm{y}_i \right\|_1 \geq \alpha=r+1$. We can write $h_{\sf{C}} - g_{\sf{C}}$ as
\[
(h_{\sf{C}} - g_{\sf{C}})(\bm{d})= 
\frac{\bm{d}^\top \bm{e}_1}{2} + \max\left\{h_{\sf F}(\bm{d}) - g_{\sf F}(\bm{d}), -\frac{|\bm{d}^\top \bm{e}_1|}{2}\right\}.
\]

(``If'') If the answer to \Cref{prop:PCP} is ``yes'', by the argument in the (``If'') proof of \Cref{lem:dc-np-hard}, we know that there exists a vector $\bm{d} \in \{-1,1\}^n$ such that $(h_{\sf F} - g_{\sf F})(\bm{d}) \leq -1$. Note that $|\bm{d}^\top \bm{e}_1| = |d_1| = 1$ by construction. It follows that
\[
0>-\left|\frac{\bm{d}^\top \bm{e}_1}{2}\right| = -\frac{1}{2} > -1 \geq  (h_{\sf F} - g_{\sf F})(\bm{d}).
\]
Let $\bm{y}:= \sgn(d_1)\cdot \bm{d} \in \{-1,1\}^n$. By $(h_{\sf F} - g_{\sf F})(\bm{d})=(h_{\sf F} - g_{\sf F})(-\bm{d})=(h_{\sf F} - g_{\sf F})(\bm{y}) < -1/2$, we have 
\[
(h_{\sf C} - g_{\sf C})(\bm{y}) = \frac{\bm{y}^\top \bm{e}_1 - |\bm{y}^\top \bm{e}_1|}{2}  = \frac{\sgn(d_1)\cdot d_1 - |d_1|}{2}  = 0.
\]
From the continuity of $h_{\sf F} - g_{\sf F}$, we know that $-\left|\bm{y}'^\top \bm{e}_1\right|/2 > (h_{\sf F} - g_{\sf F})(\bm{y}')$ for any $\bm{y}'$ near $\bm{y}$, so that $(h_{\sf C} - g_{\sf C})(\bm{y}')=0$ and $h_{\sf C} - g_{\sf C}$ is continuously differentiable at $\bm{y}$ with $\nabla (h_{\sf C} - g_{\sf C})(\bm{y}) = \bm{0}$. By positive homogeneity of $h_{\sf C} - g_{\sf C}$ and considering $t \bm{y}$ with $t \searrow 0 $, we get 
\[
\bm{0} = \lim_{t \searrow 0}\nabla (h_{\sf C} - g_{\sf C})(t\bm{y}) \in \partial (h_{\sf C} - g_{\sf C})(\bm{0}).
\]

(``Only if'') Conversely, if the answer to \Cref{prop:PCP} is ``no'', by the conclusion in the (``Only if'') proof of \Cref{lem:dc-np-hard}, we know $(h_{\sf{F}} - g_{\sf{F}})(\bm{d})=0$ for all direction $\bm{d} \in \mathbb{R}^n$. Therefore, $0=(h_{\sf{F}} - g_{\sf{F}})(\bm{d})>-|\bm{d}^\top\bm{e}_1|/2=-\frac{1}{2}$ always, so that 
\[
\dist\Big(\bm{0}, \partial (h_{\sf C} - g_{\sf C})(\bm{0}) \Big) = \frac{\|\bm{e}_1\|}{2} = \frac{1}{2} > \epsilon,
\]
as desired.
Hence, \Cref{prop:DCC} is strongly $\cNP$-hard.
\end{proof}

Now, we are ready for the proof of the hardness results in \Cref{thm:hard-dc}.

\begin{proof}[Proof of \Cref{thm:hard-dc}]
	Using \Cref{lem:dc-np-hard,lem:dc-np-hard-C}, we only need to show that the functions $h_{\sf F},g_{\sf F},h_{\sf C}$, and $g_{\sf C}$ in \Cref{prop:DCF,prop:DCC} have polynomial-size, polynomial-time computable representations in $2$-\MC{} form, with magnitudes that are also polynomially bounded in the input size of \Cref{prop:DCF,prop:DCC}, respectively. The proof is straightforward.
\end{proof}

\subsubsection{Membership}

We write the two given functions $h,g:\mathbb{R}^d\to \mathbb{R}$ in the $n$-\MC{} form as
\[
\begin{aligned}
h(\bm{w})&:=\sum_{1\leq j_n \leq J_n^h} \max_{1\leq i_{n} \leq I_n^h} \cdots \sum_{1\leq j_1 \leq J_1^h} \max_{1\leq i_1 \leq I_1^h} \Big(\bm{w}^\top \bm{x}_{i_1,j_1,\dots,i_n,j_n} + a_{i_1,j_1,\dots,i_n,j_n}\Big), \\
g(\bm{w})&:=\sum_{1\leq j_n \leq J_n^g} \max_{1\leq i_{n} \leq I_n^g} \cdots \sum_{1\leq j_1 \leq J_1^g} \max_{1\leq i_1 \leq I_1^g} \Big(\bm{w}^\top \bm{y}_{i_1,j_1,\dots,i_n,j_n} + b_{i_1,j_1,\dots,i_n,j_n}\Big).
\end{aligned}
\]
Our membership proof will rely on the following technical lemma:
\begin{Lemma}\label{lem:normal-cone-inters-poly}
Let two vectors $\bm{v}_h \in \ext \partial h(\bm{0})$ and $\bm{v}_g \in \ext \partial g(\bm{0})$ be given and suppose that the open set $\intt\left(\mathcal{N}_{\partial h(\bm{0})}(\bm{v}_h)\right)\cap \intt\left( \mathcal{N}_{\partial g(\bm{0})}(\bm{v}_g)\right)$ is non-empty. We can find, in polynomial time, a vector $\bm{p}\in\mathbb{Q}^d$ with polynomial size such that
$
\bm{p} \in \intt\left(\mathcal{N}_{\partial h(\bm{0})}(\bm{v}_h)\right)\cap \intt\left( \mathcal{N}_{\partial g(\bm{0})}(\bm{v}_g)\right).$
\end{Lemma}

\begin{proof}
We will show that $\bm{p}$ can be taken as a vertex solution of certain linear inequality system.
We begin by focusing on the function $h$.
From \eqref{eq:nMC-standard-LP} in \Cref{sec:nMC}, there exists an $\mathcal{H}$-polyhedron $H \subseteq \mathbb{R}^{d+t}$ defined by two polynomial-size matrices $\bm{M} \in \mathbb{Q}^{d \times m}, \bm{N} \in \mathbb{Q}^{d\times t}$, and a vector $\bm{c} \in \mathbb{Q}^m$ such that $\partial h(\bm{0}) = \{\bm{g}: (\bm{g}, \bm{t}) \in H\}$. The polyhedron $H$ is given in the $\mathcal{H}$-form as
$
H := \left\{(\bm{g}, \bm{t}): \bm{M}^\top \bm{g} + \bm{N}^\top\bm{t} \leq \bm{c} \right\}.
$
By \cite[Lemma 7.10, Proposition 2.3]{ziegler2012lectures} and $\bm{v}_h \in \ext \partial h(\bm{0})$, there exists $\bm{t}_h$ such that $(\bm{v}_h, \bm{t}_h) \in \ext H$, so that the encoding length of the vector $(\bm{v}_h, \bm{t}_h)$ is polynomial and can be found by computing a vertex solution of a linear feasibility problem in polynomial time.
From \cite[Example A.5.2.6(b)]{hiriart2004fundamentals}, the normal cone of the $\mathcal{H}$-polyhedron $H$ at the point $(\bm{v}_h, \bm{t}_h)$ is given by 
\[
\mathcal{N}_{H}\left(\begin{bNiceMatrix}
\bm{v}_h \\ \bm{t}_h
\end{bNiceMatrix}\right)=\cone\left\{
\begin{bNiceMatrix}
\bm{m}_i \\ \bm{n}_i
\end{bNiceMatrix}: i \in \mathcal{I}\right\}, \quad \text{where}\quad
 \mathcal{I} := \left\{j \in [m]:\bm{m}_j^\top \bm{v}_h + \bm{n}_j^\top \bm{t}_h = c_j\right\},
\]
the vectors $\bm{m}_i$ and $\bm{n}_i$ are the $i$-th column of $\bm{M}$ and $\bm{N}$, respectively. Using the calculus rule for normal cone in \cite[Proposition A.5.3.1(iii)]{hiriart2004fundamentals}, we know that
\[
\mathcal{N}_{\partial h(\bm{0})}(\bm{v}_h)
=
\left\{ \bm{g}:  
\begin{bNiceMatrix}
\bm{g} \\ \bm{0}
\end{bNiceMatrix} \in 
\mathcal{N}_{H}\left(\begin{bNiceMatrix}
\bm{v}_h \\ \bm{t}_h
\end{bNiceMatrix}\right)
\right\}=
\left\{
\sum_{i \in \mathcal{I}} \theta_i \bm{m}_i:  \bm{\theta} \in \Theta
\right\},
\]
where $\Theta:=\{\bm{\theta} \in \mathbb{R}_+^m:
\sum_{i \in \mathcal{I}} \theta_i \bm{n}_i=\bm{0}, (\forall i \in [m]\backslash\mathcal{I})\ \theta_i = 0\}$ is a convex cone.
To slightly tighten the representation, we study the intersection of the set $\Theta$ and the strictly positive hyperoctant $\mathbb{R}_{++}^m$ by considering the following index set:
\[
\mathcal{I}':=\left\{j \in \mathcal{I}: (\exists \bm{\theta}\in \Theta) \ \bm{e}_j^\top \bm{\theta} \geq 1\right\}.
\]
Note that the index set $\mathcal{I}'$ is computable by solving at most $m$ linear feasibility problems. 
Moreover, it is easy to see that 
\[
\Theta = \left\{\bm{\theta}: \sum_{i \in \mathcal{I}'}\theta_i \bm{n}_i=\bm{0}, (\forall j \in \mathcal{I}')\ \theta_j \geq 0
, (\forall i \in [m]\backslash\mathcal{I}')\ \theta_i =0\right\},\ \ 
\mathcal{N}_{\partial h(\bm{0})}(\bm{v}_h)=
\left\{
\sum_{i \in \mathcal{I}'} \theta_i \bm{m}_i:  \bm{\theta} \in \Theta
\right\}.
\]
Let us define the following LP representable sets
\[
\Theta_h:=\left\{\bm{\theta}: \bm{\theta} \in \Theta, (\forall j \in \mathcal{I}')\ \theta_j \geq 1\right\},\quad
S_h:=\left\{
\sum_{i \in \mathcal{I}'} \theta_i \bm{m}_i:  \bm{\theta} \in \Theta_h
\right\}.
\]
Similarly, for the function $g$, we can define analogous LP representable sets $\Theta_g$ and $S_g$ by repeating the arguments above.
We observe that: 
\begin{enumerate}[label=\textnormal{(\alph*)}]
	\item $S_h\cap S_g \neq \emptyset$;
	\item $S_h\cap S_g \subseteq \intt\left(\mathcal{N}_{\partial h(\bm{0})}(\bm{v}_h)\right)\cap \intt\left( \mathcal{N}_{\partial g(\bm{0})}(\bm{v}_g)\right)$.
\end{enumerate}
Therefore, the desired $\bm{p} \in \mathbb{Q}^d$ can be taken as a vertex solution of a linear feasibility problem over $S_h\cap S_g$, which has polynomial encoding length.

Now, we sketch the idea to prove (a) and (b). By hypothesis, there exist $\bm{z} \in \mathbb{R}^d$ and $\eta > 0$ such that 
$
\mathbb{B}_\eta(\bm{z}) \subseteq \mathcal{N}_{\partial h(\bm{0})}(\bm{v}_h)\cap \mathcal{N}_{\partial g(\bm{0})}(\bm{v}_g).
$
For any $j\in\mathcal{I}'$, by definition, let $\bm{\theta}_j \in \Theta$ be some vector such that $\bm{e}_j^\top \bm{\theta}_j \geq 1$. Define $\bm{\theta}^h := \sum_{j\in\mathcal{I}'} \bm{\theta}_j$ and linear mapping $T^h:=(\bm{\theta} \mapsto \sum_{i \in \mathcal{I}'} \theta_i \bm{m}_i)$. Since $\theta^h_i \geq \bm{e}_j^\top \bm{\theta}_j \geq 1$ for all $i \in \mathcal{I}'$ and $\theta^h_i =0$ otherwise, we have $\bm{\theta}^h \in \Theta_h$. Similarly, we define $\bm{\theta}^g \in \Theta_g$ and linear mapping $T^g:\Theta_g \to \mathbb{R}^d$.
For sufficiently small $t_h > 0$ and $t_g > 0$ such that $\bm{p}_h:=t_hT^h(\bm{\theta}_h) + (1-t_h)\bm{z} \in \mathbb{B}_\eta(\bm{z}),\bm{p}_g:=t_gT^g(\bm{\theta}_g) + (1- t_g)\bm{z} \in \mathbb{B}_\eta(\bm{z})$, and $\bm{p}:=\tfrac{1}{2}(\bm{p}_h + \bm{p}_g) \in \mathbb{B}_\eta(\bm{z})$, it can be shown that $t\bm{p} \in S_h\cap S_g$ for a sufficiently large $t > 0$, so that (a) holds. To show (b), fix any $\bm{g} \in S_h \cap S_g$. 
Note that $\Theta_h \subseteq \ri({\Theta})$.  Using \cite[Proposition A.2.1.12]{hiriart2004fundamentals}, we have
\[
S_h = T^h (\Theta_h) \subseteq T^h(\ri({\Theta})) = \ri(T^h(\Theta))
=\ri\left(\mathcal{N}_{\partial h(\bm{0})}(\bm{v}_h)\right)
=\intt\left(\mathcal{N}_{\partial h(\bm{0})}(\bm{v}_h)\right),
\]
where the last equality uses $\intt(\mathcal{N}_{\partial h(\bm{0})}(\bm{v}_h)) \neq \emptyset$. Similarly, it holds that $S_g \subseteq \intt(\mathcal{N}_{\partial g(\bm{0})}(\bm{v}_g))$, which completes the proof.
\end{proof}

We are ready for the proof of the membership in \Cref{thm:mem-dc}.
\begin{proof}[Proof of \Cref{thm:mem-dc}] 

To avoid introducing active sets and without loss of generality,\footnote{To see it, note that for any $\bm{w}'$ near $\bm{w}$, we have $(h-g)(\bm{w}') - (h-g)(\bm{w})=(h-g)'(\bm{w}; \bm{w}'-\bm{w})$.} we assume $a_{i_1,j_1, \dots, i_n, j_n}=0$ and $b_{i_1,j_1, \dots, i_n, j_n}=0$. Hence, all pieces are active at the point $\bm{0}$ and $h(\bm{0}) = g(\bm{0}) = 0$.

(a)
Note that $\bm{0}\notin \widehat{\partial}(h-g)(\bm{0})$ if and only if there exists a vector $\bm{d}\in\mathbb{R}^d$ such that $(h-g)'(\bm{0}; \bm{d}) < 0$. Checking whether $(h-g)'(\bm{0}; \bm{d}) < 0$ can be done, using \Cref{fct:pa-dd}, by computing $(h-g)'(\bm{0}; \bm{d}) = (h-g)(\bm{d})$ in polynomial time. As $h'(\bm{0};\cdot)$ and $g'(\bm{0};\cdot)$ are support functions of polytopes $\partial h(\bm{0})$ and $\partial g(\bm{0})$, respectively, if $\bm{0}\notin \widehat{\partial}(h-g)(\bm{0})$, then there exists $\bm{d}\in\mathbb{R}^d$ such that  $\sigma_{\partial h(\bm{0})}(\bm{d}) < \sigma_{\partial g(\bm{0})} (\bm{d})$, so that there is a point $\bm{g}\in \partial g(\bm{0})\backslash \partial h(\bm{0})$ by \cite[Corollary 13.1.1]{rockafellar1970convex}. Moreover, such a point can be chosen as $\bm{g} \in (\ext \partial g(\bm{0}))\backslash\partial h(\bm{0})$. There are only a finite number of vertices of polytope $\partial g(\bm{0})$, and they are all rational vectors of polynomial length relative to the input size; see \nref{eq:nMC-standard-LP} in \Cref{sec:nMC}. Given a certificate $\bm{g} \in \mathbb{Q}^d$, a verifier can check $\bm{g} \in \partial g(\bm{0})$ and $\bm{g} \notin \partial h(\bm{0})$ by solving LPs in polynomial time; see \Cref{prop:nMC-polynomiality}. 
Note that $\partial g(\bm{0}) \nsubseteq \partial h(\bm{0})$ implies, by \cite[Corollary 13.1.1]{rockafellar1970convex}, the existence of $\bm{d}\in\mathbb{R}^d$ such that $(h-g)'(\bm{0}; \bm{d}) < 0$, so that $\bm{0}\notin \widehat{\partial}(h-g)(\bm{0}).$
Hence, the original problem is in the class co-$\cNP$.

(b) 
As the function $h-g$ is piecewise affine, by definition and hypothesis, there exist a finite number $N\in\mathbb{N}$ and vectors $\bm{g}_1, \dots, \bm{g}_N \in \mathbb{R}^d$ such that $(h-g)(\bm{w}) \in \{\bm{w}^\top \bm{g}_1, \dots, \bm{w}^\top \bm{g}_N\}$ holds for any point $\bm{w} \in \mathbb{R}^d$.  
In other words, the function $h-g$ is a continuous selection from the set defined on the right-hand side.
We will construct short certificates for $\bm{0}\in \partial(h-g)(\bm{0})$ using the set of essentially active indices; see \cite[p.~92]{scholtes2012introduction}. According to \cite[Propsition 4.3.1]{scholtes2012introduction}, we have the following characterization of $\partial (h-g)(\bm{0})$:
\[
\partial (h-g)(\bm{0})=\conv\left\{\bm{g}_{i}: i \in \mathcal{I}_{h-g}^e(\bm{0})\right\},
\]
where $\mathcal{I}_{h-g}^e(\bm{0}) $ is the set of essentially active indices of function $h-g$ at the point $\bm{0}$, defined by
\[
\mathcal{I}_{h-g}^e(\bm{0}):=\left\{ i \in [N]: \bm{0} \in \cl\left(\intt\left\{\bm{d}: (h-g)(\bm{d}) = \bm{d}^\top \bm{g}_{i}\right\}\right) \right\}.
\]
Suppose that $\bm{0} \in \partial (h-g)(\bm{0}).$ By Carath\'eodory's Theorem \cite[Theorem 17.1]{rockafellar1970convex}, there exist $d+1$ indices $\mathcal{I}_1, \dots, \mathcal{I}_{d+1} \in \mathcal{I}_{h-g}^e(\bm{0})$ such that $\bm{0} \in \conv\left\{\bm{g}_{\mathcal{I}_1}, \dots, \bm{g}_{\mathcal{I}_{d+1}}\right\}$. Thus, we only need to prove the existence of short certificates for $\bm{g}_{\mathcal{I}_1}, \dots, \bm{g}_{\mathcal{I}_{d+1}} \in \{\bm{g}_i: i \in \mathcal{I}_{h-g}^e(\bm{0})\}$. Now, we focus on $\bm{g}_{\mathcal{I}_1}$. By the definition of $\mathcal{I}_1 \in \mathcal{I}_{h-g}^e(\bm{0})$, it follows that
\[
\bm{0} \in  \cl\left(\intt\left\{\bm{d}: (h-g)(\bm{d}) = \bm{d}^\top \bm{g}_{\mathcal{I}_1}\right\}\right).
\]
From the non-emptiness of above set and positive homogeneity of $h-g$, there exists $\bm{z}\in\mathbb{R}^d$ and $\eta>0$ such that $(h-g)(\bm{z}+\bm{d})=(\bm{z}+\bm{d})^\top \bm{g}_{\mathcal{I}_1}$ and $(h-g)(\bm{z}+\bm{d}) = (h-g)'(\bm{0}; \bm{z}+\bm{d})$, for any $\|\bm{d}\| \leq \eta$. To proceed, for any $\bm{d} \in  \eta\mathbb{B}$, we compute
\begin{align*}
(\bm{z}+ \bm{d})^\top \bm{g}_{\mathcal{I}_1} 
&= (h-g)'(\bm{0};\bm{z}+\bm{d}) \\
&= h'(\bm{0};\bm{z}+\bm{d}) - g'(\bm{0};\bm{z}+\bm{d}) \tag{\Cref{fct:pa-dd}} \\
&= \max_{\bm{g}_h \in \ext \partial h(\bm{0})} (\bm{z}+\bm{d})^\top \bm{g}_h - \max_{\bm{g}_g \in \ext \partial g(\bm{0})} (\bm{z}+\bm{d})^\top \bm{g}_g.
\end{align*}
Without loss of generality, we assume that neither $\partial h(\bm{0})$ nor $\partial h(\bm{0})$ is a singleton, as otherwise the \PA{} function $h-g$ (or $g-h$) would be locally convex around the point $\bm{0}$.
Let $\bm{g}_h^1, \bm{g}_h^2 \in \ext \partial h(\bm{0})$ be two distinct extreme points. We know that $(\bm{z}+\bm{d})^\top \bm{g}_h^1 = (\bm{z}+\bm{d})^\top \bm{g}_h^2$ if and only if $\bm{z}+\bm{d}\in\spn\left(\bm{g}_h^1 - \bm{g}_h^2\right)^\bot$. By finiteness $|\ext \partial h(\bm{0})| < \infty$ and $|\ext \partial g(\bm{0})| < \infty$ from polyhedrality, the interior of 
\[\textstyle
S:=(\bm{z}+\eta\mathbb{B})\backslash\left[\left( \bigcup_{\bm{g}_h^1 \neq \bm{g}_h^2 \in \ext \partial h(\bm{0})} \spn\left(\bm{g}_h^1 - \bm{g}_h^2\right)^\bot\right)\cup \left( \bigcup_{\bm{g}_g^1 \neq \bm{g}_g^2 \in \ext \partial g(\bm{0})} \spn\left(\bm{g}_g^1 - \bm{g}_g^2\right)^\bot\right)\right]
\]
is non-empty, so that there exist $\bm{z}' \in \bm{z} + \eta\mathbb{B}$ and $\eta' \in (0, \eta)$ such that $\bm{z}'+\eta'\mathbb{B}\subseteq\intt(S)$; see \Cref{fig:mem-prf}.
For some $\bm{d}' \in \eta'\mathbb{B}$, consider optimal solution sets:
\begin{equation}\label{eq:mem-dc-unique-H}
\mathcal{H}_{\mathcal{I}_1}:=\argmax_{\bm{g}_h \in \ext \partial h(\bm{0})} (\bm{z}'+\bm{d}')^\top \bm{g}_h, \qquad
\mathcal{G}_{\mathcal{I}_1}:=\argmax_{\bm{g}_g \in \ext \partial g(\bm{0})} (\bm{z}+\bm{d}')^\top \bm{g}_g.
\end{equation}
By continuity and definition of $\intt(S)$, we get $|\mathcal{H}_{\mathcal{I}_1}|=|\mathcal{G}_{\mathcal{I}_1}|=1$. Let $\mathcal{H}_{\mathcal{I}_1}=\{\bm{g}_h^*\}$ and $\mathcal{G}_{\mathcal{I}_1}=\{\bm{g}_g^*\}$. 
Since $\bm{z}'+\eta'\mathbb{B}\subseteq S \subseteq \bm{z} + \eta\mathbb{B}$, we see that $(\bm{z}'+\bm{d}')^\top (\bm{g}_{\mathcal{I}_1} - \bm{g}_h^* +\bm{g}_g^*)=0$ for all $\bm{d}' \in \eta'\mathbb{B}$, and consequently $\bm{g}_{\mathcal{I}_1} = \bm{g}_h^* - \bm{g}_g^*$.
Besides, from the first-order optimality conditions of \eqref{eq:mem-dc-unique-H}, we know that
$\mathbb{B}_{\eta'}(\bm{z}') \subseteq \mathcal{N}_{\partial h(\bm{0})}(\bm{g}_h^*)\cap \mathcal{N}_{\partial g(\bm{0})}(\bm{g}_g^*)$. Therefore, the open set $\intt\left(\mathcal{N}_{\partial h(\bm{0})}(\bm{g}_h^*)\cap \mathcal{N}_{\partial g(\bm{0})}(\bm{g}_g^*)\right) = \intt(\mathcal{N}_{\partial h(\bm{0})}(\bm{g}_h^*))\cap \intt( \mathcal{N}_{\partial g(\bm{0})}(\bm{g}_g^*))$ is non-empty.

\begin{figure}[t]
	\centering
	\includegraphics[width=.58\textwidth]{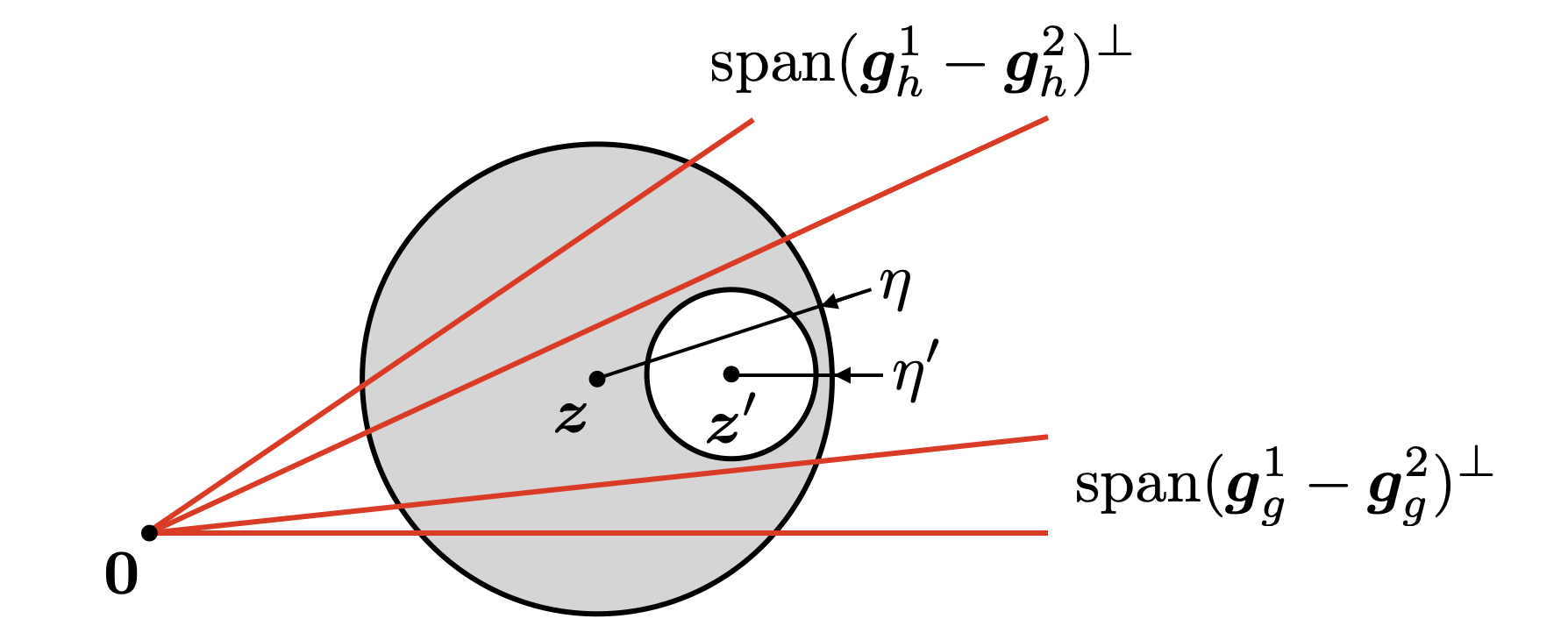}
	\caption{Illustration the proof of \Cref{thm:mem-dc}\ref{item:thm:mem-dc-b}.}\label{fig:mem-prf}
\end{figure}

Rational certificates $\bm{g}_h^*$ and $\bm{g}_g^*$ are of polynomial length as they are vertices of rational polytopes $\partial h(\bm{0})$ and $\partial g(\bm{0})$, whose encoding length is polynomial to the input size of function $h$ and $g$; see \nref{eq:nMC-standard-LP} in \Cref{sec:nMC}. As the length of the vector $\bm{z}'$ may not be polynomial, the task is now to find a polynomial length certificate $\bm{p}\in\mathbb{Q}^d$ such that for any $\bm{p}'$ near $\bm{p}$,
\begin{equation}\label{eq:mem-dc-unique}
	\bm{g}_h^* \in \argmax_{\bm{g}_h \in \ext \partial h(\bm{0})} (\bm{p}')^\top \bm{g}_h,\qquad
\bm{g}_g^* \in \argmax_{\bm{g}_g \in \ext \partial g(\bm{0})} (\bm{p}')^\top \bm{g}_g.
\end{equation}
The existence of the short certificate $\bm{p} \in \mathbb{Q}^d$ will certify, by continuity, polyhedrality,  and positive homogeneity, for any $\bm{p}'$ near $\bm{p}$ and $t>0$, that
\[
(h-g)(t\bm{p}') = (h-g)'(\bm{0}; t\bm{p}')=\max_{\bm{g}_h \in \ext \partial h(\bm{0})} (t\bm{p}')^\top \bm{g}_h - \max_{\bm{g}_g \in \ext \partial g(\bm{0})} (t\bm{p}')^\top \bm{g}_g
=(\bm{g}_h^*-\bm{g}_g^*)^\top(t\bm{p}')=\bm{g}_{\mathcal{I}_1}^\top (t\bm{p}'),
\]
confirming $\mathcal{I}_1 \in \mathcal{I}_{h-g}^e(\bm{0})$. To find such a certificate, by the first-order optimality conditions of \nref{eq:mem-dc-unique}, it suffices to locate a point $\bm{p}\in \mathbb{Q}^d$ in the set $\intt(\mathcal{N}_{\partial h(\bm{0})}(\bm{g}_h^*))\cap \intt( \mathcal{N}_{\partial g(\bm{0})}(\bm{g}_g^*))$, which has been shown to be non-empty.
Then, \Cref{lem:normal-cone-inters-poly} yields the existence of such a $\bm{p}$ with short encoding length and ensures that such a certificates can be computed from $\bm{g}_h^*, \bm{g}_g^*$.
Given certificates $\bm{g}_h^*, \bm{g}_g^* \in \mathbb{Q}^d$, a verifier can certify \nref{eq:mem-dc-unique} by solving LPs in polynomial time using \Cref{lem:normal-cone-inters-poly}. This verification can be done by the verifier for all $\bm{g}_{\mathcal{I}_1},\dots, \bm{g}_{\mathcal{I}_{d+1}}$ and, eventually, certify $\bm{0} \in \conv\left\{\bm{g}_{\mathcal{I}_1}, \dots, \bm{g}_{\mathcal{I}_{d+1}}\right\}$ by solving an LP in polynomial time. Therefore, the original problem is in the class $\cNP$.
\end{proof}

\subsection{Proofs of \Cref{thm:hard-general,thm:mem-general} (Max-Min)}\label{sec:prf-maxmin}

Let us first formulate the following decision problems.
\begin{Problem}[3SAT]\label{prob:3sat}
	Given a collection of clauses $\{C_i(\bm{x})\}_{i=1}^n$ on Boolean variables $\bm{x} \in \{0,1\}^m$ such that clause $C_i(\bm{x})$ is limited to a disjunction of three literals (e.g., $C_i(\bm{x})=x_p \vee x_q \vee \neg x_k$) for any $1 \leq i \leq n$. Let the formula of $C(\bm{x})$ in conjunctive normal form 
	$
	C(\bm{x}) \coloneqq \bigwedge_{i=1}^n  C_i(\bm{x})
	$
	be given.
	Determine whether $C(\bm{x}) = 1$ can be realize for some $\bm{x} \in \{0,1\}^m$. 
\end{Problem}

\begin{Problem}\label{prob:plt} Fix $\epsilon \in [0,\infty)$. Suppose that 
 the input data $\{\bm{y}_i\}_{i=1}^{3n} \subseteq \{-1,0,1\}^{m}$ are given. Let us define a PA function $f_{\sf{F}}: \mathbb{R}^m \rightarrow \mathbb{R}$ as
	\[
	f_{\sf{F}}(\bm{d})\coloneqq \max_{1\leq i \leq n}- \sum_{j=1}^3 \max\left\{ \bm{d}^\top \bm{y}_{3(i-1)+j},0 \right\}.
	\]
	Determine whether  
	$
	\bm{0} \in \widehat{\partial}f_{\sf{F}}(\bm{0}) + \epsilon\mathbb{B}.
	$
\end{Problem}

\begin{Problem}\label{prop:PLST}
	Fix $\epsilon \in [0, 1/2)$. Suppose that the input data $\{\bm{y}_i\}_{i=1}^{3n} \subseteq \{-1,0,1\}^m$ in the statement of \Cref{prob:plt} are given. Consider a function $f_{\sf C}:\mathbb{R}^m \to \mathbb{R}$ defined as
	\[
	f_{\sf C}(\bm{d}):=\frac{\bm{d}^\top \bm{e}_1}{2} + \max\left\{ f_{\sf F}(\bm{d}) + f_{\sf F}(-\bm{d}), -\left|\frac{\bm{d}^\top \bm{e}_1}{2}\right|  \right\},
	\]
	where the function $f_{\sf F}$ is defined in \Cref{prob:plt}. Determine whether  
	$
	\bm{0} \in \partial f_{\sf{C}}(\bm{0}) + \epsilon\mathbb{B}.
	$

\end{Problem}

\subsubsection{Hardness}
The main ingredients for the proof of \Cref{thm:hard-general} are the following two lemmas.

\begin{Lemma}\label{lem:plt-np-hard}
	\Cref{prob:plt} is strongly co-$\cNP$-hard.
\end{Lemma}
\begin{proof}
It is well known that the 3SAT in \Cref{prob:3sat} is (strongly) $\cNP$-complete \cite{garey1979computers}. We give a polynomial-time reduction with polynomially bounded magnitude from 3SAT to $\overline{\text{\Cref{prob:plt}}}$. Given any instance of 3SAT, we get clauses $\{C_i(\bm{x})\}_{i=1}^n$ for $\bm{x}\in \{0,1\}^m$. Literals in $C_t(\bm{x})$ are referred by their positions. For example, given $C_t(\bm{x}) = x_i \vee \neg x_j \vee x_k$, we say the literal $x_i$ occurs in $C_t(\bm{x})$ at position $1$, the literal $\neg x_j$ occurs in $C_t(\bm{x})$ at position $2$, and the literal $x_k$ occurs in $C_t(\bm{x})$ at position $3$. We construct the data $\{\bm{y}_i\}_{i=1}^{3n} \subseteq \{-1,0,1\}^m$ by
\begin{equation}\label{eq:hard-f-y}
	\bm{y}_i :=     \left\{ \begin{array}{rcl}
         \bm{e}_k & \mbox{if}
         & \textnormal{literal } x_k \textnormal{ occurs in } C_{\lfloor (i-1)/3 \rfloor+1}(\bm{x}) \textnormal{ at position } i - 3\lfloor (i-1)/3 \rfloor \\ 
         -\bm{e}_k & \mbox{if}
         & \textnormal{literal } \neg x_k \textnormal{ occurs in } C_{\lfloor (i-1)/3 \rfloor+1}(\bm{x}) \textnormal{ at position } i - 3\lfloor (i-1)/3 \rfloor 
                \end{array}\right.{}.
\end{equation}
Fix some $\epsilon \in [0,\infty)$.
 Note that, by definition, $\bm{0} \in \widehat{\partial}f_{\sf{F}}(\bm{0}) + \epsilon\mathbb{B}$ if and only if there exists a vector $\bm{g} \in\epsilon\mathbb{B}$ such that $f_{\sf{F}}(\bm{d}) = f_{\sf{F}}'(\bm{0};\bm{d}) \geq \langle \bm{g}, \bm{d}\rangle$ for any $\bm{d} \in \mathbb{R}^m$.
 In the sequel,
we show that the function $f_{\sf{F}}$ with input data $\{\bm{y}_i\}_{i=1}^{3n}$ defined in \nref{eq:hard-f-y} satisfies $\bm{0} \notin \widehat{\partial}f_{\sf{F}}(\bm{0}) + \epsilon\mathbb{B}$ if and only if the instance of 3SAT can be satisfied.

(``Only if'')	From the hypothesis, for any $0\leq \|\bm{g}\| \leq \epsilon$, there exists $\bm{d} \in \mathbb{R}^m$ such that $f_{\sf{F}}(\bm{d}) < \langle \bm{g}, \bm{d} \rangle$. We will exhibit an $\bm{x} \in \{0,1\}^m$ such that the given instance of 3SAT can be satisfied with $\bm{x}$.
Let $\bm{g} = \bm{0}$ and there  exists $\bm{d} \in \mathbb{R}^m$ such that $f_{\sf{F}}(\bm{d}) < 0$.
 For any $i \in [m]$, let
\[
x_i = \left\{ \begin{array}{rcl}
         1 & \mbox{if}
         & d_i > 0 \\ 
         0 & \mbox{if}
         & d_i \leq 0
                \end{array}\right..
\]
We claim $C(\bm{x})=1$. By $f_{\sf{F}}(\bm{d}) < 0$, for any $i \in [n],$ we get 
$
\sum_{j=1}^3 \max\left\{ \bm{d}^\top \bm{y}_{3(i-1)+j},0 \right\} > 0.
$
This implies that there exists a $j'_i \in \{1,2,3\}$ such that $\bm{d}^\top \bm{y}_{3(i-1)+j'_i} > 0$ for any $i \in [n]$. Fix $i \in [n]$.
Let the index of the Boolean literal occurs in  $C_i(\bm{x})$ at position $j'_i$ be $k_i$. Now we consider two cases. If $x_{k_i}$ occurs in  $C_i(\bm{x})$ at position $j'_i$, then $\bm{y}_{3(i-1)+j'_i} = \bm{e}_{k_i}$. We get $\bm{d}^\top \bm{y}_{3(i-1)+j'_i} = \bm{d}^\top\bm{e}_{k_i} = d_{k_i} > 0$. So, by definition, $x_{k_i} = 1$ which implies $C_i(\bm{x}) = 1$. Otherwise, if $\neg x_{k_i}$ occurs in  $C_i(\bm{x})$ at position $j'_i$, then $\bm{y}_{3(i-1)+j'_i} = -\bm{e}_{k_i}$. We get $\bm{d}^\top \bm{y}_{3(i-1)+j'_i} = -\bm{d}^\top\bm{e}_{k_i} = -d_{k_i} > 0$, so that $\neg x_{k_i} = 1$ by definition and it implies $C_i(\bm{x}) = 1$. It is clear that $C(\bm{x})=\bigwedge_{i=1}^n  C_i(\bm{x}) = 1$ and the given instance of 3SAT is satisfied with $\bm{x}$.

(``If'') Conversely, we show that if there exists a vector $\bm{g}$ such that $0\leq \|\bm{g}\| \leq \epsilon$ and  $f_{\sf{F}}(\bm{d}) \geq \langle \bm{g}, \bm{d} \rangle$ for any $\bm{d} \in \mathbb{R}^m$, then the given instance of 3SAT cannot be satisfied. Suppose to the contrary that there exists $\bm{x} \in \{0,1\}^m$ such that $C(\bm{x}) = 1$. We have $0 \geq f_{\sf{F}}(\bm{g}) \geq \|\bm{g}\|^2$ by setting $\bm{d}=\bm{g}$, so that we know $\bm{g}=\bm{0}$.
For any $i \in [m]$, let 
\[
d_i = \left\{ \begin{array}{rcl}
         1 & \mbox{if}
         & x_i = 1 \\ 
         -1 & \mbox{if}
         & x_i = 0
                \end{array}\right..
\]
As $\bigwedge_{i=1}^n  C_i(\bm{x}) = 1$, for any $i \in [n]$, there exists a literal of clause $C_i(\bm{x})$ that is satisfied. Let the index of this literal be $k'_i$ and the position of it in $C_i(\bm{x})$ be $j'_i$. We consider two cases. If literal $x_{k'_i}$ occurs in $C_i(\bm{x})$ at position $j'_i$, then $\bm{y}_{3(i-1)+j'_i}=\bm{e}_{k'_i}$. As $C_i(\bm{x})=1$ due to literal $x_{k'_i}$, we get  $x_{k'_i} = 1$ and $d_{k'_i}=1$ by definition. Then, for such $i\in[n]$, we get
$\sum_{j=1}^3 \max\left\{ \bm{d}^\top \bm{y}_{3(i-1)+j},0 \right\} \geq \max\left\{ \bm{d}^\top \bm{y}_{3(i-1)+j'_i},0 \right\} = 
  \max\{d_{k'_i},0\} = 1.
$ 
Otherwise, if literal $\neg x_{k'_i}$ occurs in $C_i(\bm{x})$ at position $j'_i$, then $\bm{y}_{3(i-1)+j'_i}=-\bm{e}_{k'_i}$. As $C_i(\bm{x})=1$ due to literal $\neg x_{k'_i}$, we get  $x_{k'_i} = 0$ and $d_{k'_i}=-1$ by definition. Then, for any $i\in[n]$, we get
$
\sum_{j=1}^3 \max\left\{ \bm{d}^\top \bm{y}_{3(i-1)+j},0 \right\} \geq \max\left\{ \bm{d}^\top \bm{y}_{3(i-1)+j'_i},0 \right\} = 
  \max\{-d_{k'_i},0\} = 1.
$ 
This gives 
$
0 = \langle \bm{g}, \bm{d} \rangle \leq f_{\sf{F}}(\bm{d}) \leq -1, 
$
a contradiction. 

Hence \Cref{prob:plt} is in the class co-$\cNP$-hard. Besides, as the numerical values in the input data are elements in $\{-1,0,1\}$, the largest numerator or denominator of them are clearly upper bounded by the number of bits. \Cref{prob:plt} is automatically strongly co-$\cNP$-hard; see \cite[p.~96]{garey1979computers} and \cite[Observation 5]{garey1978strong}.
\end{proof}

\begin{Lemma}\label{lem:plst-np-hard}
\Cref{prop:PLST} is strongly $\cNP$-hard.
\end{Lemma}
\begin{proof}
We present a polynomial-time reduction with polynomially bounded magnitude  from 3SAT to \Cref{prop:PLST}  using part of the construction in the proof of \Cref{lem:plt-np-hard}. Fix $\epsilon \in [0,1/2)$. For any given instance of 3SAT, we first construct the corresponding $f_{\sf F}$ with input data $\{\bm{y}_i\}_{i=1}^{3n}$ defined in \nref{eq:hard-f-y}.  We show that the induced function $f_{\sf{C}}$ in \Cref{prop:PLST} satisfies  $\bm{0} \in \partial f_{\sf{C}}(\bm{0}) + \epsilon\mathbb{B}$ if and only if the instance of 3SAT can be satisfied. The argument is similar to the proof of \Cref{lem:dc-np-hard-C}.

(``If'') If the instance of 3SAT can be satisfied, by the argument in the (``If'') proof of \Cref{lem:plt-np-hard}, we know that there exists a vector $\bm{d} \in \{-1,1\}^m$ such that $f_{\sf F}(\bm{d}) \leq -1$, so that $f_{\sf F}(\bm{d}) + f_{\sf F}(-\bm{d}) \leq -1$ due to $f_{\sf F}(\cdot) \leq 0$. Note that $|\bm{d}^\top \bm{e}_1| = |d_1| = 1$ by construction. Then, it follows that
\[
0>-\left|\frac{\bm{d}^\top \bm{e}_1}{2}\right| = -\frac{1}{2} > -1 \geq f_{\sf F}(\bm{d}) + f_{\sf F}(-\bm{d}).
\]
Let $\bm{y}:= \sgn(d_1)\cdot \bm{d}$. Note that $f_{\sf F}(\bm{y}) + f_{\sf F}(-\bm{y}) = f_{\sf F}(\bm{d}) + f_{\sf F}(-\bm{d}) < -|\bm{d}^\top \bm{e}_1|/2<0$ by symmetry. Therefore, we have 
\[
f_{\sf C}(\bm{y}) = \frac{\bm{y}^\top \bm{e}_1 - |\bm{y}^\top \bm{e}_1|}{2}  = \frac{\sgn(d_1)\cdot d_1 - |d_1|}{2}  = 0.
\]

From the continuity of $f_{\sf C}$, we know that $-\left|\bm{y}'^\top \bm{e}_1\right|/2 > f_{\sf F}(\bm{y}') + f_{\sf F}(-\bm{y}')$ for any $\bm{y}'$ near $\bm{y}$, so that $f_{\sf C}(\bm{y}')=0$ and $f_{\sf C}$ is continuously differentiable at $\bm{y}$ with $\nabla f_{\sf C}(\bm{y}) = \bm{0}$. By homogeneity of $f_{\sf C}$ and considering $t \bm{y}$ with $t \searrow 0 $, we get $\bm{0} \in \partial f_{\sf C}(\bm{0})$.

(``Only if'') Conversely, if 3SAT cannot be satisfied, by the conclusion in the (``Only if'') proof of \Cref{lem:plt-np-hard}, we know that $\bm{0} \in \widehat{\partial} f_{\sf{F}}(\bm{0})$, which, by definition, indicates $0 \leq f_{\sf{F}}'(\bm{0}; \bm{d}) = f_{\sf F}(\bm{d})$ for any direction $\bm{d} \in \mathbb{R}^m$. We conclude that $f_{\sf F}(\cdot) = 0$. Thus, $f_{\sf C}(\bm{d}) = \frac{\bm{d}^\top \bm{e}_1}{2}$ always and 
\[
\dist\Big(\bm{0}, \partial f_{\sf C}(\bm{0}) \Big) = \frac{\|\bm{e}_1\|}{2} = \frac{1}{2} > \epsilon,
\]
as desired.
Hence, similarly, \Cref{prop:PLST} is automatically strongly $\cNP$-hard.
\end{proof}

Now, we are ready for the proof of \Cref{thm:hard-general}.

\begin{proof}[Proof of \Cref{thm:hard-general}]
	Using \Cref{lem:plt-np-hard} and \Cref{lem:plst-np-hard}, we only need to show the constructions $f_{\sf F}$ and $f_{\sf C}$ in \Cref{prob:plt} and \Cref{prop:PLST} have polynomial-size representations with polynomially bounded magnitude in \MaxMin{} form. Note the following elementary identity
	\[
	\sum_{i=1}^I \min_{1 \leq j \leq J} a_{i,j} = \min \left(S:=\left\{ \sum_{i=1}^I a_{i,j_i}: j_{i'} \in [J], i' \in [I]  \right\} \right).
	\]
	It is easy to check that $|S| \leq J^I$, which is polynomial with respect to $J$ when $I$ is a constant. Thus, $f_{\sf F}(\cdot)$ can be rewritten in the \MaxMin{} form with $l=n$, $k \leq 3n+1$, and $|\mathcal{M}_i| \leq 8$ for any $i \in [n]$. As for $f_{\sf C}(\cdot)$, we first note that 
	\[
	f_{\sf C}(\bm{d})= \max\left\{ \frac{\bm{d}^\top \bm{e}_1}{2} + f_{\sf F}(\bm{d}) + f_{\sf F}(-\bm{d}), \min\left\{\bm{d}^\top \bm{e}_1,0\right\} \right\}.
	\]
	Note that functions $\bm{d}\mapsto f_{\sf F}(-\bm{d})$ and $\bm{d}\mapsto \bm{d}^\top \bm{e}_1/2 + f_{\sf F}(\bm{d})$ have polynomial-size \MaxMin{} representations.
It remains to prove that the addition of two polynomial-size \MaxMin{} functions (i.e., $\bm{d}\mapsto \frac{\bm{d}^\top \bm{e}_1}{2} + f_{\sf F}(\bm{d}) + f_{\sf F}(-\bm{d})$) still has polynomial size. Let two \MaxMin{} functions $h$ and $g$ with data $\{(\bm{x}_j,a_j)\}_{j=1}^k$ and $\{(\bm{x}'_{j'},a'_{j'})\}_{j'=1}^{k'}$ be given, respectively. We compute
	\[
	\begin{aligned}
	h(\bm{w})+g(\bm{w})&=\left(\max_{1\leq i \leq l} \min_{j \in \mathcal{M}_i} \bm{w}^\top \bm{x}_j + a_j\right)+\left(\max_{1\leq i' \leq l'} \min_{j' \in \mathcal{M}'_{i'}} \bm{w}^\top \bm{x}'_{j'} + a'_{j'}\right)\\
	&=\max_{\substack{1\leq i \leq l \\ 1\leq i' \leq l'}}\left[ \left(\min_{j \in \mathcal{M}_i} \bm{w}^\top \bm{x}_j + a_j\right)+\left(\min_{j' \in \mathcal{M}'_{i'}} \bm{w}^\top \bm{x}'_{j'} + a'_{j'}\right) \right] \\
	&=\max_{\substack{1\leq i \leq l \\ 1\leq i' \leq l'}} \min_{\substack{j \in \mathcal{M}_i \\ j' \in \mathcal{M}'_{i'}}} \bm{w}^\top \left(\bm{x}_j + \bm{x}'_{j'}\right)+ a_j + a'_{j'},
	\end{aligned}
	\]
	which has polynomial size with respect to the size of $h$ and $g$ as desired. 
\end{proof}

\subsubsection{Membership}

To be specific, we write the \PA{} function $f:\mathbb{R}^d\to \mathbb{R}$ in \MaxMin{} form as
\[
		f(\bm{w}) = \max_{1\leq i \leq l} \min_{j \in \mathcal{M}_i} \bm{w}^\top \bm{x}_j + a_j,
		\]
where the data $(\bm{x}_j, a_j) \in \mathbb{R}^{d+1}$ for $j \in [k]$, and index sets $\mathcal{M}_i \subseteq [k]$ for  $i \in [l]$.

\begin{proof}[Proof of \Cref{thm:mem-general}]  To avoid introduce active sets and without loss of generality, we assume that $a_j=0$ for any $1\leq j\leq k$ and the vectors $\bm{x}_1, \dots, \bm{x}_k$ are non-redundant, i.e., $\bm{x}_i \neq \bm{x}_j$ if $i\neq j$.

(a)
Note that $\bm{0}\notin \widehat{\partial}f(\bm{0})$ if and only of there exists a vector $\bm{d}\in\mathbb{R}^d$ such that $f'(\bm{0}; \bm{d}) < 0$. Checking whether $f'(\bm{0}; \bm{d}) < 0$ can be done by computing 
\[
f'(\bm{0}; \bm{d})=\max_{1\leq i \leq l} \min_{j \in \mathcal{M}_i} \bm{d}^\top \bm{x}_j
\]
in polynomial time. Suppose that $\bm{0}\notin \widehat{\partial}f(\bm{0})$. By positive homogeneity, there exists $\bm{d}\in\mathbb{R}^d$ such that $f'(\bm{0}; \bm{d}) \leq  -1$. Thus, for any $i \in [l]$, we have $\min_{j \in \mathcal{M}_i} \bm{d}^\top \bm{x}_j \leq -1$. Let us define index sets $\mathcal{I}_i:=\arg\min_{j \in \mathcal{M}_i} \bm{d}^\top \bm{x}_j \subseteq [k]$ for all $i \in [l]$. Then, the following linear inequality system is consistent:
\[
\left\{ \begin{array}{rcl}
         \bm{d}^\top \bm{x}_{j^*} \leq -1 & \mbox{for}
         & i \in [l], j^* \in \mathcal{I}_i, \\ 
         \bm{d}^\top \bm{x}_{j^*} \leq \bm{d}^\top \bm{x}_{j}  & \mbox{for} & i \in [l], j^* \in \mathcal{I}_i, j \in \mathcal{M}_i\backslash\mathcal{I}_i.
                \end{array}\right.
\]
There exists a vertex solution, say $\bm{d}'$, of above inequality system with polynomial size.
By definition, we know that $f'(\bm{0}; \bm{d}') \leq -1$.
 Given certificates $\bm{d}' \in \mathbb{Q}^d$, %
 a verifier can check $f'(\bm{0}; \bm{d}')<0$ in polynomial time. Hence, the original problem is in the class co-$\cNP$.

(b) We will construct short certificates for $\bm{0}\in \partial f(\bm{0})$ using the set of essentially active indices; see \cite[p.~92]{scholtes2012introduction}. According to \cite[Propsition 4.3.1]{scholtes2012introduction}, we have the following characterization of $\partial f(\bm{0})$:
\[
\partial f(\bm{0})=\conv\left\{\bm{g}_i: i \in \mathcal{I}_{f}^e(\bm{0})\right\},
\]
where $\mathcal{I}_{f}^e(\bm{0})$ is the set of essentially active indices of function $f$ at point $\bm{0}$, defined by
\[
\mathcal{I}_{f}^e(\bm{0}):=\left\{ i \in [k]: \bm{0} \in \cl\left(\intt\left\{\bm{d}: f(\bm{d}) = \bm{d}^\top \bm{x}_i\right\}\right) \right\}.
\]
Suppose that $\bm{0} \in \partial f(\bm{0}).$ By Carath\'eodory's Theorem \cite[Theorem 17.1]{rockafellar1970convex}, there exist $d+1$ indices $i_1, \dots, i_{d+1} \in \mathcal{I}_{f}^e(\bm{0})$ such that $\bm{0} \in \conv\left\{\bm{x}_{i_1}, \dots, \bm{x}_{i_{d+1}}\right\}$. Thus, we only need to prove the existence of short certificates for $i_1, \dots, i_{d+1} \in \mathcal{I}_{f}^e(\bm{0})$. Now, we focus on $i_1$. By definition of $i_1 \in \mathcal{I}_{f}^e(\bm{0})$, it follows that
\[
\bm{0} \in  \cl\left(\intt\left\{\bm{d}: f(\bm{d}) = \bm{x}_{i_1}^\top \bm{d}\right\}\right).
\]
From the non-emptiness of above set and positive homogeneity of $f$, there exists $\bm{z}\in\mathbb{R}^d$ and $\eta>0$ such that $f(\bm{z}+\bm{d})=(\bm{z}+\bm{d})^\top \bm{x}_{i_1}$ and $f(\bm{z}+\bm{d}) = f'(\bm{0}; \bm{z}+\bm{d})$, for any $\|\bm{d}\| \leq \eta$. To proceed, for any $\bm{d} \in  \eta\mathbb{B}$, we compute
\[
\bm{d}^\top \bm{x}_{i_1} + \bm{z}^\top \bm{x}_{i_1} 
= f(\bm{z}+\bm{d})  %
= f'(\bm{0};\bm{z}+\bm{d})
= \max_{1\leq i \leq l} \min_{j \in \mathcal{M}_i} (\bm{z}+\bm{d})^\top \bm{x}_j.
\]
For any $i\neq j\in[k]$, we know that $(\bm{z}+\bm{d})^\top \bm{x}_i = (\bm{z}+\bm{d})^\top \bm{x}_j$ if and only if $\bm{z}+\bm{d}\in\spn\left(\bm{x}_i - \bm{x}_j\right)^\bot$. Similar to the proof of \Cref{thm:mem-dc}\ref{item:thm:mem-dc-b} and the illustration in \Cref{fig:mem-prf}, the interior of 
\[\textstyle
S:=(\bm{z}+\eta\mathbb{B})\backslash \left( \bigcup_{\bm{x}_i \neq \bm{x}_j} \spn\left(\bm{x}_i - \bm{x}_j\right)^\bot\right)\]
is non-empty, so that there exist $\bm{z}' \in \bm{z} + \eta\mathbb{B}$ and $0<\eta' < \eta$ such that $\bm{z}'+\eta'\mathbb{B}\subseteq\intt(S)$. For any $\bm{d}' \in \eta'\mathbb{B}$ and $i\in[l]$, consider
\[
\mathcal{I}_i:=\argmin_{j \in \mathcal{M}_i}{} (\bm{z}'+\bm{d}')^\top \bm{x}_j, \qquad\mathcal{I}:= \argmax_{1\leq i \leq l} \min_{j \in \mathcal{M}_i}{} (\bm{z}'+\bm{d}')^\top \bm{x}_j.
\]
By continuity and definition of $\intt(S)$, we get $|\mathcal{I}_i|=1$ for all $i\in[l]$ and $\mathcal{I}_j = \mathcal{I}_k = \{i_1\}$ for any $j,k \in \mathcal{I}$. 
Then, the following linear inequality system is consistent:
\begin{equation}\label{eq:proof-mem-maxmin}
\left\{ \begin{array}{rcl}
         \bm{d}^\top \bm{x}_{i_1} \geq \bm{d}^\top \bm{x}_{j} + 1 & \mbox{for}
         & j \in \cup_{i' \in [l]\backslash\mathcal{I}} \mathcal{I}_{i'},\\ 
         \bm{d}^\top \bm{x}_{i'} \leq \bm{d}^\top \bm{x}_{j}-1  & \mbox{for} & i \in [l], i' \in \mathcal{I}_i, j \in \mathcal{M}_i\backslash\mathcal{I}_i.
                \end{array}\right.
\end{equation}
There exists a vertex solution, say $\bm{d}^\triangle \in \mathbb{Q}^d$, of above inequality system \labelcref{eq:proof-mem-maxmin} with polynomial size.
By the definition of \nref{eq:proof-mem-maxmin} and $|\mathcal{I}_1|=1$, we know that, for all $i\in[l]$,
\[
\mathcal{I}_i=\argmin_{j \in \mathcal{M}_i}{} (\bm{d}^\triangle)^\top \bm{x}_j, \qquad\mathcal{I}:= \argmax_{1\leq i \leq l} \min_{j \in \mathcal{M}_i}{} (\bm{d}^\triangle)^\top \bm{x}_j.
\]
 Given rational certificates $\bm{d}^\triangle \in \mathbb{Q}^d, \mathcal{I}\subseteq [l],$ and $\mathcal{I}_1, \dots, \mathcal{I}_l \subseteq [k]$, a verifier can construct and check the consistence of the inequality system \nref{eq:proof-mem-maxmin} in polynomial time, which certifies $i_1 \in \mathcal{I}_f^e(\bm{0})$.
 This verification can be done by the verifier for all $i_1,\dots, i_{d+1}$ and, eventually, certify $\bm{0} \in \conv\left\{\bm{g}_{i_1}, \dots, \bm{g}_{i_{d+1}}\right\}$ by solving an LP in polynomial time. Therefore, the original problem is in the class $\cNP$.
\end{proof}

\section{Sum Rule, Compatibility, and Transversality}\label{sec:cr}
In this section, we study the validity of sum rule for the difference of two convex \PA{} functions and its interplay with two geometric conditions, i.e., compatibility and transversality. %

\subsection{Motivation} 
As a motivating example, we focus on the following DC function $f:=h'-g'$ with two convex \PA{} components $h',g':\mathbb{R}^d\to\mathbb{R}$:
\[
f( \bm{w})=\left(h'(\bm{w}):=h(\bm{A}\bm{x}) \right) - \left( g'(\bm{w}):= g(\bm{B}\bm{w}) \right),
\]
where the matrices $\bm{A} \in \mathbb{R}^{n\times d}, \bm{B}\in\mathbb{R}^{m\times d}$ are given, and $h:\mathbb{R}^n\to \mathbb{R}, g:\mathbb{R}^m \to \mathbb{R}$ are convex \PA{} functions.
One can immediately observe that above \PA{} function $f$ is typically subdifferentially irregular (see \Cref{def:regular}).
The choice of investigating DC form rather than the \MaxMin{} form is driven by applicability considerations. The combination and composition of summation and difference-of-convex structure is prevalent in modern machine learning and statistical applications; see \cite{cui2018composite} and \cite{nouiehed2019pervasiveness}.
Furthermore, DC functions find utility in the analysis of contemporary neural networks; see \cite{arora2018understanding,goodfellow2013maxout} and \Cref{sec:app} for relevant examples.

The computational hardness results in \Cref{sec:hardness-DC} for testing is partly due to the failure of exact subdifferential calculus rules. 
Thus, to enhance the tractability of stationarity testing, it is of interest to find out a condition, under which an equality-type calculus rule holds,  so that the subdifferential set of the function can be efficiently characterized as that of smooth functions. 
Specifically, we aim to investigate the validity of the following equation for subdifferential of the function $f:\mathbb{R}^d\to \mathbb{R}$ at a fixed point $\bm{w}\in\mathbb{R}^d$:
\begin{equation}\label{eq:subd-cr}
\partial f(\bm{w}) = \partial h'(\bm{w}) - \partial g'(\bm{w}) = \bm{A}^\top \partial h(\bm{Aw}) - \bm{B}^\top \partial g(\bm{Bw}).
\end{equation}
Given above exact sum rule in \nref{eq:subd-cr}, computing the quantity $\dist(\bm{0}, \partial f(\bm{w})) \in \mathbb{R}_+$ for a rational vector $\bm{w}$ is a convex quadratic optimization problem. It is well-known that such a problem can be approximately solved in polynomial time by, for example, IPMs \cite{nesterov1994interior}. Therefore, efficient $\epsilon$-stationarity detection for $f$ can be guaranteed by an exact subdifferential sum rule.

\begin{Example}
The exact sum rule in \nref{eq:subd-cr} fails on fairly simple examples. Consider a univariate convex \PA{} function $f:\mathbb{R}\to \mathbb{R}$ defined by $f(t):=2\max\{t,0\} - \max\{t,0\}$. The exact subdifferential sum rule does not hold when $t=0$, as evidenced by the strict set inclusion  $\partial f(0) = [0,1] \subsetneq [-1,2] = \partial [2\max\{\cdot, 0\}](0) - \partial [\max\{\cdot, 0\}](0)$.
\end{Example}

The goal of this section is to establish the necessary and sufficient condition for validating the sum rule in \nref{eq:subd-cr}. 
Before that, to appreciate the advancements in nonsmooth analysis and to contextualize our results, we will discuss several known sufficient conditions that validate \nref{eq:subd-cr}.%
\subsection{Prior Art on Validating the Sum Rule}\label{sec:known-cond}

Let us define a function $r:\mathbb{R}^{n} \times \mathbb{R}^{m} \to \mathbb{R}$ as
\[
r(\bm{u}, \bm{v}):= 
h(\bm{u}) - g(\bm{v}).
\]
Consider the following linear mapping $G: \mathbb{R}^d \to \mathbb{R}^{n} \times \mathbb{R}^{m}$ defined by
\[
G(\bm{w}):=\left(
\bm{A} \bm{w},
\bm{B} \bm{w}
     \right).
\]
Then, by definition, we have $f(\bm{w}) = (h'-g')(\bm{w}) = (r\circ G)(\bm{w})$ for any $\bm{w} \in \mathbb{R}^d$.

\paragraph{Clarke Regularity.} The arguably most well-known condition validating \nref{eq:subd-cr} is the regularity proposed by \citet{clarke1975generalized}; see \Cref{def:regular}.
\begin{Fact}[cf.~{\cite[p.~40, Corollary 3]{clarke1990optimization}}]\label{fct:regular-cr}
Suppose that $f=f_1 + f_2$ where $f_1$ and $f_2$ are locally Lipschitz functions. Given a point $\bm{x}\in\mathbb{R}^d$, if $f_1, f_2$ are regular at $\bm{x}$, then we have
\[
\partial f(\bm{x}) = \partial f_1 (\bm{x}) + \partial f_2(\bm{x}).
\]	
\end{Fact}
In particular, from \cite[Proposition 2.3.6]{clarke1990optimization}, smoothness or convexity implies Clarke regularity, which enables \Cref{fct:regular-cr} to encompass the gradient and convex subdifferential sum rule as special cases. For our \DC{} function $f=h'-g'$, where $h'$ is convex and thus naturally exhibits regularity, we observe that the function $-g'$ is regular at $\bm{w}$ if and only if $-g'$ simplifies to an affine function near $\bm{w}$ and the function $f$ is locally convex near $\bm{w}$; see \Cref{prop:all-are-transveral} below. This presents a rather strong requirement, making Clarke regularity too stringent an assumption for our purposes.

\paragraph{Separability.} Fix $t \in \{0, \dots, d\}$. Suppose that the input data $\bm{A} \in \mathbb{R}^{n\times d}$ and $\bm{B} \in \mathbb{R}^{m \times d}$ can be written in separable form as 
\begin{equation}\label{eq:motiv-separability}
\bm{A} = \left[ \begin{array}{c|c}
		\widetilde{\bm{A}} & \bm{0}_{n \times t}
		\end{array} \right]  \in \mathbb{R}^{n\times d}, \qquad 
		\bm{B} = 
		\left[ \begin{array}{c|c}
		\bm{0}_{m\times (d-t)} & \widetilde{\bm{B}}
		\end{array} \right]  \in \mathbb{R}^{m\times d}.
\end{equation}
Accordingly, we can write the vector $\bm{w}\in\mathbb{R}^d$ as $\bm{w}^\top = \left[ \begin{array}{c|c}
		\bm{w}_1^\top & \bm{w}_2^\top
		\end{array} \right]$, where $\bm{w}_1\in\mathbb{R}^{d-t}$ and $\bm{w}_2\in\mathbb{R}^{t}$. Then, we have $f(\bm{w})= h(\widetilde{\bm{A}}\bm{w}_1) - g(\widetilde{\bm{B}}\bm{w}_2)$.
It is clear that the function $(\bm{w}_1, \bm{w}_2)\mapsto h(\widetilde{\bm{A}}\bm{w}_1) - g(\widetilde{\bm{B}}\bm{w}_2)$  is  separable in its arguments. The following result by \citet{rockafellar1985extensions} gives another condition validating sum rule in \nref{eq:subd-cr}:
\begin{Fact}[cf.~{\cite[Proposition 2.5]{rockafellar1985extensions}}]
	Let $f(\bm{x},\bm{y})=f_1(\bm{x})+f_2(\bm{y})$. If $f_1$ and $f_2$ are locally Lipschitz at $\bm{x}$ and $\bm{y}$, respectively, then
	\[
	\partial f(\bm{x}, \bm{y}) = \partial f_1(\bm{x}) \times \partial f_2(\bm{y}).
	\]
\end{Fact}

\paragraph{Surjectivity.} 
Let us define a matrix
\begin{equation}\label{eq:motiv-surj}
\bm{D}=\left[ \begin{array}{c|c}
		\bm{A}^\top & \bm{B}^\top \end{array} \right] \in \mathbb{R}^{d \times (n+m)}.
		\end{equation}
Note that the linear mapping $G$ is surjective if and only if $\rnk(\bm{D}) = n+m$. The third condition, validating \nref{eq:subd-cr}, assumes the surjectivity of the mapping $G$:
\begin{Fact}[cf.~{\cite[Theorem 2.3.10]{clarke1990optimization}, \cite[Exercise 10.7]{rockafellar2009variational}}]
Let $G:\mathbb{R}^d \to \mathbb{R}^n$ be a strictly differentiable map and $r:\mathbb{R}^n\to \mathbb{R}$ be a locally Lipschitz function. If $r$ (or $-r$) is regular at $G(\bm{x})$ or Jacobian $\nabla G(\bm{x})$ has rank $n$, then
$\partial (r\circ G)(\bm{x}) = \nabla G(\bm{x})^\top \partial r(G(\bm{x}))$.
\end{Fact}

\paragraph{Discussion.}
Roughly speaking, these existing conditions inherently rely on certain ``separability''-type properties, either literally or in the spectral domain, to certify the validity of the sum rule in \nref{eq:subd-cr}. Under such assumptions, the interaction between the convex and concave parts is properly separated. 
Indeed, for the DC function $f=h'-g'$, aforementioned three conditions are all special case of the following transversality-type condition; see \Cref{sec:zono} for detailed discussion.
\begin{Proposition}\label{prop:all-are-transveral}
Assume that any one of the following conditions holds:
\begin{enumerate}[label=\textnormal{(\alph*)}]
	\item The \PA{} function $-g'$ is Clarke regular.
	\item The matrices $\bm{A}$ and $\bm{B}$ are separable in the sense of \nref{eq:motiv-separability}.
	\item The matrix $\bm{D}$ in \nref{eq:motiv-surj} has full column rank.
\end{enumerate}
Then, we have
\begin{equation}\label{eq:pre-transverality}
	\parr(\partial h'(\bm{w}))\cap \parr (\partial g'(\bm{w})) = \{\bm{0}\}.
\end{equation}
\end{Proposition}
\begin{proof}
	See \Cref{prf-prop:all-are-transveral}.
\end{proof}

While the three aforementioned conditions are special cases of \eqref{eq:pre-transverality}, the strength of \eqref{eq:pre-transverality} with respect to the validity of the sum rule remains unclear. Specifically, is it sufficient? If so, is it also necessary? Moreover, identifying the weakest possible condition that guarantees the sum rule in \nref{eq:subd-cr} is an intriguing question, which also motivates the following:
\begin{center}
\textit{Is a ``separability''-type condition necessary to validate the exact sum rule?}
\end{center}
The remainder of this section is devoted to addressing this question, and, perhaps surprisingly, the answer is both yes and no. See the discussion in \Cref{sec:zono}.

\subsection{Compatible Polytopes}\label{sec:compat}
Our development centers around two geometric notions. Here, we introduce the most important one, which concerns polytope pairs.

\begin{Definition}[Compatible polytopes]\label{def:comp-poly}
	Two polytopes $A$ and $B$ in $\mathbb{R}^d$ are called compatible if for any vectors $\bm{a}\in A$ and $\bm{b} \in B$ such that $\bm{a} - \bm{b} \in \ext(A-B)$, we have $\bm{a} + \bm{b} \in \ext(A+B)$. 
\end{Definition}

Here are some elementary properties regarding the compatibility of polytopes.
\begin{Proposition} \label{prop:comp-poly}
The following elementary properties hold.
\begin{enumerate}[label=\textnormal{(\roman*)}]
	\item Polytopes $A$ and $B$ are compatible if and only if $B$ and $A$ are compatible.	
	\item Polytopes $A$ and $B$ are compatible if and only if $-A$ and $-B$ are compatible.	
	\item For any given vector $\bm{v}$, polytopes $A+\{\bm{v}\}$ and $B$ are compatible if and only if $A$ and $B$ are compatible.
	\item Suppose that $X:=\sum_{p=1}^P X_p, Y:=\sum_{q=1}^Q Y_q$, where $X_p,Y_q$ are all polytopes. Then, $X$ and $Y$ are compatible if and only if for any subsets $S_1 \subseteq [P], S_2 \subseteq [Q]$, $\sum_{p\in S_1} X_p$ and $\sum_{q\in S_2} Y_q$ are compatible polytopes.
\end{enumerate}
\end{Proposition}
\begin{proof}
	The proof uses tools from nonsmooth analysis; see \Cref{sec:comp-properity}.
\end{proof}

We now provide a few examples to build intuitions.

\begin{figure}[t]
	\centering
	\includegraphics[width=.9\textwidth]{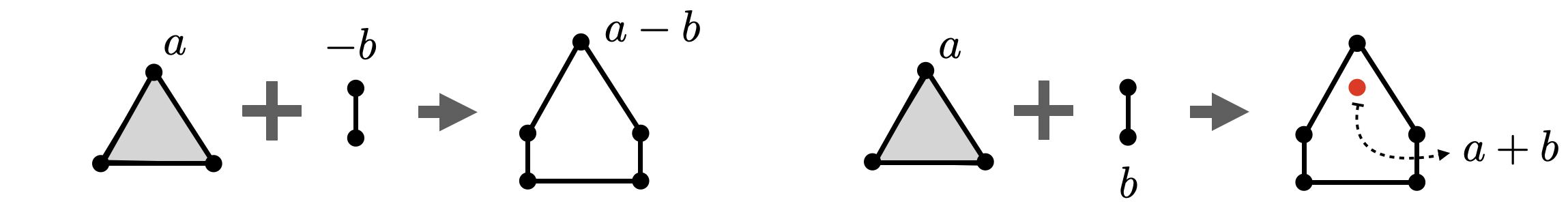}
	\caption{Incompatible polytopes in \Cref{exam:incomp}.}\label{fig:incomp}
\end{figure}

\begin{Example}[Compatible polytopes in $\mathbb{R}^2$]\label{exam:comp2}
	Consider polytopes
	$
	X := [\bm{0}, \bm{e}_1]$ and  $Y := [\bm{0}, \bm{e}_2]
	$ in $\mathbb{R}^2$.
	One can readily verify the compatibility between the polytopes $X$ and $Y$.
\end{Example}

\begin{Example}[Incompatible polytopes in $\mathbb{R}^2$]\label{exam:incomp}
	Let vectors $\bm{x}_1 := \bm{0}, \bm{x}_2 := (-1,-1), \bm{x}_3 := (1,-1), \bm{y}_1 := \bm{0}, $ and $\bm{y}_2 := (0, -1/2)$ in $\mathbb{R}^2$ be given. Consider polytopes
	\[
	X := \conv\{\bm{x}_i: i \in \{1,2,3\}\}, \qquad Y := [\bm{y}_1, \bm{y}_2].
	\]
	Let $\bm{a} := \bm{x}_1$ and $\bm{b} := \bm{y}_2$.
	Note that $\bm{a} - \bm{b} = (0,1/2) \in \ext(X-Y)$. However, it is easily seen that $\bm{a} + \bm{b}=(0,-1/2) \in \intt (X+Y)$, so that $\bm{a} + \bm{b}\notin \ext(X+Y)$. Therefore, polytopes $X$ and $Y$ are incompatible; see the illustration in  \Cref{fig:incomp}.
\end{Example}

The following is a nontrivial example in $\mathbb{R}^4$ and will be discussed in \Cref{sec:zono}.

\begin{Example}[Compatible polytopes in $\mathbb{R}^4$]\label{exam:comp}
Let vectors $\bm{y} := (1,1,-1,-1), \bm{e}_0 := \bm{0}$ in $\mathbb{R}^4$ be given. Define polytopes
\[
X:=\conv\{ \bm{e}_i : i \in \{0, \ldots, 4\}\}, \qquad Y := [\bm{0}, \bm{y}].
\]
By \Cref{lem:fukuda}, every extreme point of $X-Y$ can be written as $\bm{e}_{i'} - \mathbf{1}_{\eta > 0}\cdot \bm{y}$ for some $i' \in \{0,\dots,4\}$ and $\eta \in \{-1,1\}$, and there exists a vector $\bm{h} \in \mathbb{R}^4$ such that $-\eta\cdot \bm{h}^\top \bm{y} > 0$ and  $\bm{h}^\top \bm{e}_{i'} > \bm{h}^\top \bm{e}_i$  for any $i \in \{0,\dots, 4\}\backslash\{i'\}$. 
Let $\bm{a}:=\bm{e}_{i'}$ and $\bm{b}:=\mathbf{1}_{\eta > 0}\cdot \bm{y}$, hence $\bm{a}-\bm{b} \in \ext(X-Y)$.
Suppose that $\bm{a} + \bm{b} \notin \ext(X+Y)$.  According to \Cref{lem:fukuda}, there is no vector $\bm{h}$ such that $\eta\cdot \bm{h}^\top \bm{y} > 0$ and  $\bm{h}^\top \bm{e}_{i'} > \bm{h}^\top \bm{e}_i$  for any $i \in \{0,\dots, 4\}\backslash\{i'\}$. Hence, from \Cref{lem:gordan}, there exist $\bm{\theta}\in \mathbb{R}^5_+$ and $\mu \geq 0$ such that $(\bm{\theta}_{-i'-1}, \mu) \neq \bm{0}$ and
	\[
	-\mu\eta \cdot\bm{y} = \sum_{i \in \{0,\dots, 4\}\backslash\{i'\}} \theta_{i+1}\cdot (-\bm{e}_{i})+\left( \sum_{i \in \{0,\dots, 4\}\backslash\{i'\}} \theta_{i+1}\right)\cdot \bm{e}_{i'}.
	\]
	If $\mu = 0$, then we can only have $\bm{\theta}_{-i'-1} = \bm{0}_4$, contradicting the hypothesis.
Suppose $\mu\neq 0$. Note that $|\{i:\mu\eta \cdot y_i > 0\}|=2$. However, we have
\[
\left|\left\{i: p_i > 0, \bm{p} = \sum_{i \in \{0,\dots, 4\}\backslash\{i'\}} \theta_{i+1}\cdot (-\bm{e}_i)+\left( \sum_{i \in \{0,\dots, 4\}\backslash\{i'\}} \theta_{i+1}\right)\cdot \bm{e}_{i'} \right\}\right|
\leq 1,
\]
a contradiction. Therefore, $\bm{a} + \bm{b} \in \ext(X+Y)$, so that the polytopes $X$ and $Y$ are compatible.
\end{Example}

\subsection{Validity of Exact Sum Rule}\label{sec:validity}
Now, we are ready to present our main result on the validity of sum rule for a general \DC{} function with convex \PA{} components.

\begin{Theorem}[Exact sum rule]\label{thm:general-sum-Clarke-geometric}
Let any convex \PA{} functions $h,g:\mathbb{R}^d\to \mathbb{R}$ and a point $\bm{x}\in\mathbb{R}^d$ be given. The following are equivalent. 
\begin{enumerate}[label=\textnormal{(\alph*)}]
\item $\partial (h - g)(\bm{x})=\partial h(\bm{x}) - \partial g(\bm{x})$. \label{item:thm:general-sum-Clarke-geometric-a}
\item $\partial h(\bm{x})$ and $\partial g(\bm{x})$ are compatible polytopes. \label{item:thm:general-sum-Clarke-geometric-b}
\item It holds that
\[
	\{\bm{0}\} = \bigcup\left\{ \mathcal{T}_{\partial h(\bm{x})}(\bm{a})\cap \big( - \mathcal{T}_{\partial g(\bm{x})}(\bm{b}) \big) :\bm{a} - \bm{b} \in \ext(\partial h(\bm{x}) - \partial g(\bm{x}))\right\}.
\] \label{item:thm:general-sum-Clarke-geometric-c}
\end{enumerate}
\end{Theorem}

\begin{Remark}%
\Cref{thm:general-sum-Clarke-geometric} shows that compatibility (a geometric notion), as defined in \Cref{def:comp-poly}, is a necessary and sufficient condition for the validity of the exact sum rule (a nonsmooth analytic concept) for \PA{} functions. While some sufficient conditions for the equality-type calculus rules can be deduced from the literature (see \Cref{sec:known-cond}, \cite[Chapter 10]{rockafellar2009variational}, and \cite[Chapter 4]{Mordukhovich.18}), nontrivial necessary conditions are rarely found, let alone simultaneously necessary and sufficient conditions, as shown in our \Cref{thm:general-sum-Clarke-geometric}. This reveals a new connection between polyhedral geometry  and nonsmooth analysis for \PA{} functions.
\end{Remark}

\begin{Remark}\label{rmk:cr-exponential}
While the compatibility of two $\mathcal{V}$-polytopes can be verified efficiently, the time required to verify the compatibility of polytopes $\partial h(\bm{x})$ and $\partial g(\bm{x})$ in \Cref{thm:general-sum-Clarke-geometric}\ref{item:thm:general-sum-Clarke-geometric-b} can be exponential for $n$-\MC{} functions $h$ and $g$ for even fixed depth $n$. This is because the $\mathcal{V}$-representation of polytopes $\partial h(\bm{x})$ and $\partial g(\bm{x})$ may contain an exponentially large number of vertices, which results from Minkowski summation. This curse of dimensionality phenomenon also arises in algorithm design for piecewise smooth functions with a similar summation structure; see \cite[Section 7]{pang2017computing}.
\end{Remark}

\begin{Remark}
We invite the reader to compare the properties of polytope compatibility outlined in \Cref{prop:comp-poly} with the basic subdifferential calculus rules presented in \cite[Section 2.3]{clarke1990optimization}, especially regarding how they are used in the proof of \Cref{prop:comp-poly} in \Cref{sec:comp-properity}.
\end{Remark}

\begin{proof}[Proof of \Cref{thm:general-sum-Clarke-geometric}] We first make some preparatory moves.
Using \cite[Proposition 4.4.3(c)]{cui2021modern}, for any $\bm{w}$ near $\bm{0}$, we have
\[
h(\bm{w}) - h(\bm{0})= h'(\bm{0};\bm{w})=\max_{\bm{x} \in  \partial h(\bm{0})} \bm{w}^\top \bm{x}, \qquad g(\bm{w}) - g(\bm{0}) = g'(\bm{0}; \bm{w})=\max_{\bm{y} \in \partial g(\bm{0})} \bm{w}^\top \bm{y}.
\]
To simplify notation, we write $X:=\partial h(\bm{0})$ and $Y:=\partial g(\bm{0})$. 
As $h$ and $g$ are convex \PA{} functions, the set $X$ and $Y$ are polytopes.
Let $\ext(X)=:\{\bm{x}_i\}_{i=1}^n, \ext(Y)=:\{\bm{y}_j\}_{j=1}^m$. %
Define $f:=h'(\bm{0}; \cdot) - g'(\bm{0}; \cdot)$. We naturally have
\begin{equation}\label{eq:prf-sum-rule-f}
f(\bm{w}) = \max_{1\leq i \leq n} \bm{w}^\top \bm{x}_i - \max_{1\leq j \leq m} \bm{w}^\top \bm{y}_j,
\end{equation}
which is positively homogeneous.
Using \Cref{fct:clarke-dd} and $\Delta_t(h-g)(\bm{0})(\bm{w})=\Delta_t f(\bm{0})(\bm{w})$ for sufficiently small $\bm{w}$ and $t>0$, we conclude that $\partial (h-g)(\bm{0}) = \partial f(\bm{0})$.

((b)$\implies$(a)) 
By the fuzzy sum rule \cite[Proposition 2.3.3]{clarke1990optimization}, it holds that $\partial f(\bm{0}) =\partial (h-g)(\bm{0})\subseteq X-Y$.
To prove the other direction,
let $\bm{x} \in X, \bm{y} \in Y$ be such that $\bm{x}-\bm{y}\in \ext(X-Y)$. As polytopes $X$ and $Y$ are compatible, we know that $
\bm{x}+ \bm{y} \in  \ext(X+Y).$
By \Cref{lem:fukuda}, we know that $\bm{x} \in \ext(X), \bm{y} \in \ext(Y)$, and there exists a vector $\bm{w}$ such that
$
\bm{w}^\top \bm{x} > \bm{w}^\top \bm{x}'$ for all $\bm{x}' \in \ext(X)\backslash\{\bm{x}\}$ and $
\bm{w}^\top \bm{y} > \bm{w}^\top \bm{y}'$ for all $\bm{y}'\in\ext(Y)\backslash\{\bm{y}\}.
$

By \nref{eq:prf-sum-rule-f} and the continuity of the function $f$, we know that $f$ is continuously differentiable at any $\bm{w}'$ near $\bm{w}$ with 
\[
f(\bm{w}') = (\bm{x}-\bm{y})^\top \bm{w}',\qquad \nabla f(\bm{w}') = \bm{x}-\bm{y}. 
\]
Using \citep[Theorem 9.61]{rockafellar2009variational} on $t\bm{w}'$ with $t \searrow 0$ and by the positive homogeneity of $f$, we get
\[
 \bm{x}-\bm{y} = \lim_{t \searrow 0}\nabla f(t\bm{w}') \in \partial f(\bm{0}).
\]
By \Cref{lem:fukuda}, we know that $\ext(X-Y)=\{\bm{x} - \bm{y}: \bm{x} \in X, \bm{y} \in Y, \bm{x} - \bm{y} \in \ext(X-Y) \}$.
Therefore, it holds that $\ext (X-Y) \subseteq \partial f(\bm{0})$, which implies $X-Y \subseteq \partial f(\bm{0})=\partial (h-g)(\bm{0})$.%

($\neg$(b)$\implies \neg$(a)) 
Suppose that $\bm{x} \in X, \bm{y} \in Y, \bm{x}-\bm{y} \in \ext(X-Y)$ but  
$
\bm{x}+\bm{y}\notin \ext(X+Y).
$ 
Applying \Cref{lem:fukuda} to the relation $\bm{x}-\bm{y} \in \ext(X-Y)$, there exists a vector $\bm{h}$ such that%
\begin{equation}\label{eq:sum-rule-proof-goldan-h}
\bm{h}^\top \bm{x} > \bm{h}^\top \bm{x}', \forall \bm{x}'\in \ext(X)\backslash\{\bm{x}\}, \qquad
\bm{h}^\top \bm{y} < \bm{h}^\top \bm{y}', \forall \bm{y}'\in \ext(Y)\backslash\{\bm{y}\}.
\end{equation}
In a similar vein, as $\bm{x}+\bm{y} \notin \ext(X+Y)$, there is no vector $\bm{z}$ such that 
\begin{equation}\label{eq:sum-rule-proof-goldan}
\bm{z}^\top \bm{x} > \bm{z}^\top \bm{x}', \forall \bm{x}'\in \ext(X)\backslash\{\bm{x}\},\qquad
\bm{z}^\top \bm{y}> \bm{z}^\top \bm{y}', \forall \bm{y}'\in\ext(Y)\backslash\{\bm{y}\}.
\end{equation}
We prove $\neg (a)$ by contradiction. Suppose that $X-Y = \partial f(\bm{0})$. We have 
$
\ext(X-Y)= \ext \partial f(\bm{0}) \subseteq \partial f(\bm{0}).
$
By \cite[Proposition 4.3.1]{scholtes2012introduction}, we get a representation of $\partial f(\bm{0})$ as
\[
\partial f(\bm{0}) = \conv \left\{ \bm{x}_i-\bm{y}_j: (i,j) \in \mathcal{I}_f^e(\bm{0})  \right\},
\]
where the so-called \emph{essentially active set} $\mathcal{I}_f^e(\bm{0})$ is defined by (see \cite[Section 4.1]{scholtes2012introduction})
\[
\mathcal{I}_f^e(\bm{0}):=\left\{(i,j) \in [n]\times [m]: 
\bm{0} \in \cl \left( \intt \left\{ \bm{w}: \bm{w}^\top(\bm{x}_i- \bm{y}_j) = f(\bm{w}) \right\} \right)
\right\}.
\]
Due to $\bm{x}-\bm{y} \in \ext \partial f(\bm{0})$, by \citep[Corollary 18.3.1]{rockafellar1970convex}, there exists an index $(i,j) \in \mathcal{I}_f^e(\bm{0}) $ such that
$
\bm{x}_i - \bm{y}_j = \bm{x} - \bm{y}.
$
If $\bm{x}_i \neq \bm{x}$ (and consequently $\bm{y}_j \neq \bm{y}$), then using $\bm{x}_i \in \ext(X)$ and $\bm{y}_j \in \ext(Y)$, we obtain
\[
\bm{h}^\top(\bm{x}-\bm{y})=\bm{h}^\top\bm{x}_i-\bm{h}^\top\bm{y}_j \overset{\nref{eq:sum-rule-proof-goldan-h}}{<} \bm{h}^\top(\bm{x}-\bm{y}),
\]
which leads to a contradiction. Therefore, it holds that $\bm{x}_i = \bm{x}, \bm{y}_j = \bm{y}$, and
\[
\bm{0} \in \cl \left( \intt \left\{ \bm{w}: \bm{w}^\top(\bm{x}-\bm{y}) = f(\bm{w}) \right\} \right).
\]
By non-emptiness of above set, there exist $\epsilon > 0$ and $\bm{z}$ such that for any $\bm{w} \in \mathbb{B}_\epsilon(\bm{z})$, it holds 
\begin{equation}\label{eq:proof-CR-int-eq}
v(\bm{w}):=\max_{\bm{x}' \in \ext(X)} \bm{w}^\top \left(\bm{x}' - \bm{x} \right) = \max_{\bm{y}' \in \ext(Y)} \bm{w}^\top \left(\bm{y}' - \bm{y}\right)=:u(\bm{w}).
\end{equation}
Let functions $v, u:\mathbb{R}^d \to \mathbb{R}$ be defined above in \nref{eq:proof-CR-int-eq}.
Note that $v(\bm{z}) \geq 0$ and $u(\bm{z}) \geq 0$, as $\bm{x} \in \ext(X)$ and $\bm{y} \in \ext(Y)$.
Moreover, from \cite[Theorem B.3.1.2, Proposition C.2.1.3]{hiriart2004fundamentals}, the functions $v$ and $u$ are continuous on $\mathbb{R}^d$.
We will derive a contradiction to \nref{eq:proof-CR-int-eq} by considering the following two cases.
\begin{itemize}%
	\item $v(\bm{z}) > 0$: Let $t > 0$ be such that $\bm{z} + t\bm{h} \in \mathbb{B}_\epsilon(\bm{z})$ and $v(\bm{z} + t\bm{h}) > 0$, where the vector $\bm{h}$ is defined above \nref{eq:sum-rule-proof-goldan-h}. %
	For any $\bm{x}'' \in\ext(X)$ such that $v(\bm{z} + t\bm{h}) = (\bm{z}+t\bm{h})^\top \left(\bm{x}'' - \bm{x} \right)$, it follows that $\bm{x}'' \neq \bm{x}$, as otherwise $v(\bm{z} + t\bm{h})=0$.
	Then, we obtain
	\[
	v(\bm{z} + t\bm{h})=(\bm{z}+t\bm{h})^\top \left(\bm{x}'' - \bm{x} \right) \leq v(\bm{z}) + t\bm{h}^\top \left(\bm{x}'' - \bm{x} \right) \overset{\nref{eq:sum-rule-proof-goldan-h}}{<} v(\bm{z}).
	\]
	Similarly, for any $\bm{y}'' \in\ext(Y)$ such that $u(\bm{z}) = \bm{z}^\top \left(\bm{y}'' - \bm{y} \right)$, we have
	\[
	u(\bm{z} + t\bm{h})\geq (\bm{z}+t\bm{h})^\top \left(\bm{y}'' - \bm{y} \right) = u(\bm{z}) + t\bm{h}^\top \left(\bm{y}'' - \bm{y} \right) \overset{\nref{eq:sum-rule-proof-goldan-h}}{\geq} u(\bm{z}) \overset{\nref{eq:proof-CR-int-eq}}{=} v(\bm{z}) > v(\bm{z} + t\bm{h}),
	\]
	a contradiction to \nref{eq:proof-CR-int-eq}.
	\item $v(\bm{z}) = 0$: We first claim that there exist $\bm{x}'\in \ext(X)\backslash\{\bm{x}\}$ or $\bm{y}'\in \ext(Y)\backslash \{\bm{y}\}$ such that $\bm{z}^\top \left(\bm{x}' - \bm{x} \right)=0$ or $\bm{z}^\top \left(\bm{y}' - \bm{y} \right)=0$, as otherwise, we have a contradiction to \nref{eq:sum-rule-proof-goldan}. Without of loss of generality, we assume that there exists $\bm{x}''\in \ext(X)\backslash\{\bm{x}\}$ such that $\bm{z}^\top \big(\bm{x}'' - \bm{x} \big)=0$. Let $t > 0$ be such that $\bm{z} - t\bm{h} \in \mathbb{B}_\epsilon(\bm{z})$. It follows that
	\[
	v(\bm{z} + t\bm{h}) \geq (\bm{z}-t\bm{h})^\top \left(\bm{x}'' - \bm{x} \right) = v(\bm{z}) - t\bm{h}^\top \left(\bm{x}'' - \bm{x} \right) \overset{\nref{eq:sum-rule-proof-goldan-h}}{>} v(\bm{z})=0.
	\]
	Meanwhile, for any $\bm{y}'' \in \ext(Y)$ such that $u(\bm{z}-t\bm{h}) = (\bm{z}-t\bm{h})^\top \left(\bm{y}'' - \bm{y} \right)$, we have
	\[
	u(\bm{z}-t\bm{h}) = (\bm{z}-t\bm{h})^\top \left(\bm{y}'' - \bm{y} \right) \leq u(\bm{z}) - t\bm{h}^\top \left(\bm{y}'' - \bm{y} \right) \overset{\nref{eq:sum-rule-proof-goldan-h}}{\leq} u(\bm{z})\overset{\nref{eq:proof-CR-int-eq}}{=}0 < v(\bm{z} + t\bm{h}),
	\]
	a contradiction to \nref{eq:proof-CR-int-eq}.
\end{itemize}

((c)$\implies$(b)) 
Let $\bm{a}\in X, \bm{b} \in Y$ be such that $\bm{a}-\bm{b} \in \ext(X-Y)$. With \Cref{lem:fukuda}, we know there exist $i'\in[n]$ and $j'\in[m]$ such that $\bm{x}_{i'} = \bm{a},\bm{y}_{j'} = \bm{b}$. Note that
\[
\mathcal{T}_X(\bm{a}) =\cone\left\{X - \bm{a}\right\},\quad
\mathcal{T}_Y(\bm{b}) =\cone\left\{Y - \bm{b}\right\}.
\]
We claim that, for $\bm{\theta} \in \mathbb{R}_+^{n}$, if
$
\sum_{i\neq i'} \theta_{i}\cdot (\bm{x}_{i} - \bm{x}_{i'}) = \bm{0},
$
then $\theta_{i}=0$ for all $i\neq i'$. To see it, by \Cref{lem:fukuda}, there exists $\bm{c}\in\mathbb{R}^d$ such that
$
\bm{c}^\top \bm{x}_{i'} > \bm{c}^\top \bm{x}_{i}$ for all $i\in [n]\backslash\{i'\}.
$
Thus, if there exists $\theta_{i} > 0$ with $i \neq i'$, we get $0=\sum_{i\neq i'} \theta_{i}\cdot \bm{c}^\top(\bm{x}_{i} - \bm{x}_{i'})<0$, a contradiction. 
By symmetry, we have the similar claim for $\{\bm{y}_{j}:j\in[m]\}$.
This, combining with the condition (c), i.e., $\mathcal{T}_X(\bm{a})\cap (- \mathcal{T}_Y(\bm{b})) = \{\bm{0}\}$, yields that if
\begin{equation}\label{eq:sum-rule-proof-goldan-c2b-eq}
	\sum_{i \in [n]\backslash\{i'\}} \theta_{i}\cdot (\bm{x}_{i} - \bm{x}_{i'}) + \sum_{j \in [m]\backslash\{j'\}} \mu_{j}\cdot (\bm{y}_{j} - \bm{y}_{j'}) = \bm{0},
\end{equation}
with $\theta_{i}\geq 0, \mu_{j} \geq 0$, 
then
$\theta_{i}=\mu_{j}=0$ for any $i \in [n]\backslash \{i'\}$ and $j\in [m]\backslash \{j'\}$.
Applying \Cref{lem:gordan} to \nref{eq:sum-rule-proof-goldan-c2b-eq}, there exists a vector $\bm{w}$ such that %
\begin{equation}\label{eq:sum-rule-proof-goldan-c2b-ineq}
\bm{w}^\top \bm{x}_{i'} > \bm{w}^\top \bm{x}_{i}, \forall i \in [n]\backslash\{i'\}, \quad
\bm{w}^\top \bm{y}_{j'} > \bm{w}^\top \bm{y}_{j}, \forall j \in [m]\backslash\{j'\},
\end{equation}
which, by \Cref{lem:fukuda}, yields 
$
\bm{a} + \bm{b} = \bm{x}_{i'} + \bm{y}_{j'} \in \ext(X+Y). 
$
Therefore, polytopes $X$ and $Y$ are compatible.

((b)$\implies$(c)) 
Note that if $\bm{a} \notin X$ or $\bm{b} \notin Y$, we have $\mathcal{T}_X(\bm{a})\cap (- \mathcal{T}_Y(\bm{b})) = \emptyset$ by definition.
Let $\bm{a}\in X, \bm{b} \in Y$ such that $\bm{a}-\bm{b}  \in \ext(X-Y)$ be given. With \Cref{lem:fukuda}, we know there exist $i'\in[n]$ and $j'\in[m]$ such that $\bm{x}_{i'} = \bm{a},\bm{y}_{j'} = \bm{b}$. If $\bm{a} + \bm{b}  \in \ext(X+Y)$, then, by \Cref{lem:fukuda}, there exists a vector $\bm{w}$ such that \nref{eq:sum-rule-proof-goldan-c2b-ineq} holds. By \Cref{lem:gordan}, we know that, for $\theta_{i} \geq 0$ and $\mu_{j} \geq 0,$ if 
\[
\bm{v} := \sum_{i \in [n]\backslash\{i'\}} \theta_{i}\cdot (\bm{x}_{i} - \bm{x}_{i'})=- \sum_{j \in [m]\backslash\{j'\}} \mu_{j}\cdot (\bm{y}_{j} - \bm{y}_{j'}) \in \mathcal{T}_X(\bm{a})\cap (- \mathcal{T}_Y(\bm{b})), 
\]
then $\theta_{i}=\mu_{j}=0$ for any $i\in [n]\backslash \{i'\}$ and $j\in [m]\backslash \{ j'\}$, so that we obtain $\bm{v}=\bm{0}$.
The proof completes by noting that $\mathcal{T}_X(\bm{a})=\cone\{ \bm{x}_{i} - \bm{x}_{i'}, i\in[n]\}$ and $\mathcal{T}_Y(\bm{b})=\cone\{ \bm{y}_{j} - \bm{y}_{j'}, j\in[m]\}$.
\end{proof}

\subsection{Transversality vs Compatibility}\label{sec:zono}

As mentioned in \Cref{rmk:cr-exponential}, verifying the conditions in \Cref{thm:general-sum-Clarke-geometric} may require exponential time. Therefore, a polynomial-time verifiable sufficient condition would be of great practical interest. In this subsection, we introduce such a regularity condition, called transversality, which is closely related to a known condition for \emph{convex} composite functions.
\subsubsection{Definition}
For convex \PA{} functions $h$ and $g$, by \Cref{thm:general-sum-Clarke-geometric}, the compatibility of $\partial h(\bm{0})$ and $\partial g(\bm{0})$ is equivalent to the following regularity condition:
\[
	\{\bm{0}\} = \bigcup\left\{ \mathcal{T}_{\partial h(\bm{x})}(\bm{a})\cap \big( - \mathcal{T}_{\partial g(\bm{x})}(\bm{b}) \big) :\bm{a} - \bm{b} \in \ext(\partial h(\bm{x}) - \partial g(\bm{x}))\right\}.
\]
 One straightforward option is to modify the above condition concerning the intersection of tangent cones to allow for affine-type combinations, leading to the following geometric notion.

\begin{Definition}[Transversality]\label{def:transv}
Given two convex \PA{} functions $h,g:\mathbb{R}^d\to \mathbb{R}$, we say that the functions $h$ and $g$ are transversal at point $\bm{x} \in \mathbb{R}^d$ if 
	\[
	\parr(\partial h(\bm{x})) \cap \parr(\partial g(\bm{x})) = \{\bm{0}\}.
	\]
\end{Definition}

Recall that we have seen the transversality condition above in \Cref{prop:all-are-transveral}, where we showed that, when applied to the \PA{} function $h-g$, three existing conditions are all special cases of our transversality condition.
The subdifferentials of convex \PA{} functions, $h$ and $g$, are convex polytopes. Therefore, given a point $\bm{x}$, we can verify the transversality of $h$ and $g$ by solving an LP feasibility problem. For many cases, including all \MC{} functions, this can be done efficiently.
\begin{Proposition}
Given two $n$-\MC{} functions $h,g:\mathbb{R}^d \to \mathbb{R}$ with rational data and $\bm{w} \in \mathbb{Q}^d$, there exists an algorithm that verifies the transversality of $h$ and $g$ at $\bm{w}$ in polynomial time.
\end{Proposition}
\begin{proof}
	See \Cref{prop:nMC-polynomiality}.
\end{proof}

\begin{Remark}[Terminology]\label{rmk:transversality}
Our transversality condition in \Cref{def:transv} for \emph{nonconvex} \PA{} functions is closely related to existing constraint qualifications for \emph{convex-composite} functions; see \cite[Definition 4]{burke2020strong}.
Specifically, let $r:\mathbb{R}^n\to \mathbb{R}\cup\{\infty\}$ be a convex, closed, proper function, and $G:\mathbb{R}^d \to \mathbb{R}^n$ be a smooth mapping. A function $f:\mathbb{R}^d \to \mathbb{R}\cup \{\infty\}$ is called convex-composite if it can be written as $f=r\circ G$. The calculus rules for convex-composite functions require certain so-called constraint qualifications. Fix a point $\bm{w} \in \mathbb{R}^d$. The following \emph{transversality condition} is presented in \cite[Definition 4(TC)]{burke2020strong}:
\begin{equation}\label{eq:tc-kkt}
\mathop{\textnormal{ker}}\left(\nabla G(\bm{w})^\top \right)
\cap
\parr(\partial r(G(\bm{w})) = \{\bm{0}\},
\end{equation}
where $\mathop{\textnormal{ker}}(\bm{A})$ is the null space of the matrix $\bm{A}$.
To see the relationship between our condition in \Cref{def:transv} and \nref{eq:tc-kkt}, by slightly abusing notation, let $r:\mathbb{R}^d\times \mathbb{R}^d \to \mathbb{R}$ be such that $r(\bm{u}, \bm{v}):=h(\bm{u}) - g(\bm{v})$. Define a smooth mapping $G:\mathbb{R}^d \to \mathbb{R}^d \times \mathbb{R}^d$ as $G(\bm{w}):=(\bm{w}, \bm{w})$. By \cite[Proposition 2.5]{rockafellar1985extensions}, we obtain that $\partial r(G(\bm{w})) = \partial h(\bm{w}) \times \partial (-g)(\bm{w}).$  
Note that $\mathop{\textnormal{ker}}\left(\nabla G(\bm{w})^\top \right)=\{(\bm{u}, -\bm{u}):\bm{u} \in \mathbb{R}^d\}.$
One can easily verify that \nref{eq:tc-kkt} and \Cref{def:transv} are equivalent.
\end{Remark}

\subsubsection{Interrelation}

 As promised, we first show that the transversality of $h$ and $g$ at $\bm{x}$ is a sufficient condition for the compatibility of polytopes $\partial h(\bm{x})$ and $\partial g(\bm{x})$.
\begin{Proposition}[Transversality implies compatibility]\label{prop:par-sufficient}
Let two convex \PA{} functions $h,g:\mathbb{R}^d\to \mathbb{R}$ and a point $\bm{x}\in\mathbb{R}^d$ be given. If functions $h$ and $g$ are transversal at point $\bm{x} \in \mathbb{R}^d$, then
\[
\partial (h-g)(\bm{x}) = \partial h(\bm{x}) - \partial g(\bm{x}).
\]
\end{Proposition}
\begin{proof}
	For any $\bm{a} \in \partial h(\bm{x})$ and $\bm{b} \in \partial g(\bm{x})$, we have 
	\[
	\cone\{\partial h(\bm{x}) - \bm{a}\} \subseteq \parr (\partial h(\bm{x})), \quad\text{and}\quad  - \cone\{\partial g(\bm{x}) - \bm{b}\} \subseteq \parr (\partial g(\bm{x})).
	\]
	Note that $\bm{0} \in \mathcal{T}_{\partial h(\bm{x})}(\bm{a})=\limsup_{t\searrow 0}(\partial h(\bm{x}) - \bm{a})/t=\cone\{\partial h(\bm{x}) - \bm{a}\}$ and
	$\bm{0} \in \mathcal{T}_{\partial g(\bm{x})}(\bm{b})=\limsup_{t\searrow 0}(\partial g(\bm{x}) - \bm{b})/t=\cone\{\partial g(\bm{x}) - \bm{b}\}.$ 
	Hence, for any $\bm{a} \in \partial h(\bm{x})$ and $\bm{b} \in \partial g(\bm{x})$, we have 
	\[
	\{\bm{0}\}\subseteq \mathcal{T}_{\partial h(\bm{x})}(\bm{a})\cap\left(-\mathcal{T}_{\partial g(\bm{y})}(\bm{b})\right) \subseteq \parr(\partial h(\bm{x})) \cap \parr(\partial g(\bm{x})) = \{\bm{0}\},
	\]
	where the last equality is from transversality. 
	 The proof completes by using \Cref{thm:general-sum-Clarke-geometric}.
\end{proof}

\paragraph{Zonotopes.}
 While allowing for polynomial-time verification, this ``relaxed'' condition, i.e., transversality,  may appear excessively conservative. 
 A remarkable property of transversality, beyond the efficient verifiability, is its simultaneous necessity and sufficiency when the polytopes $\partial h(\bm{x})$ and $\partial g(\bm{x})$ are zonotopes; see \Cref{fig:zonotopes} for an illustration.

\begin{figure}[th]
	\centering
	\includegraphics[width=.25\textwidth]{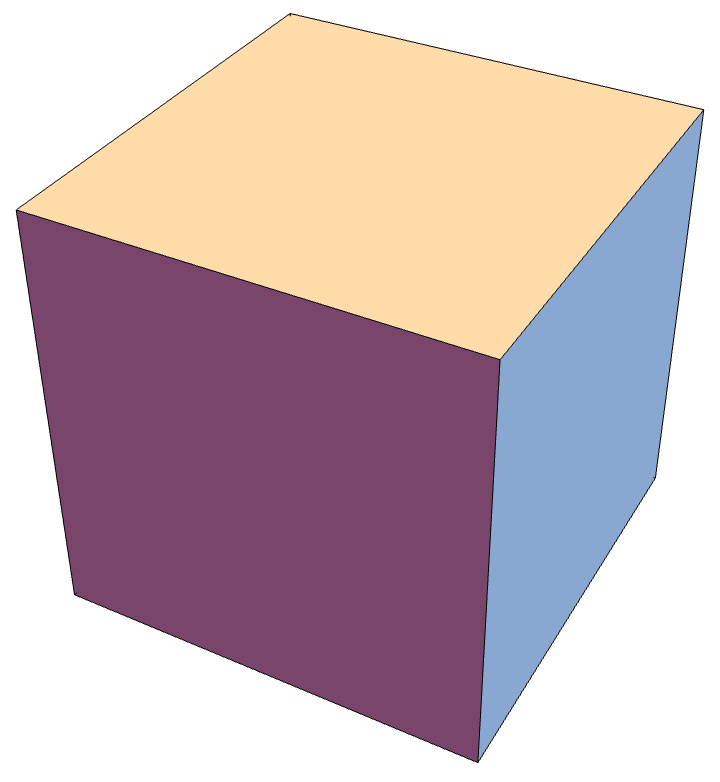}
	\hspace{1cm}
	\includegraphics[width=.25\textwidth]{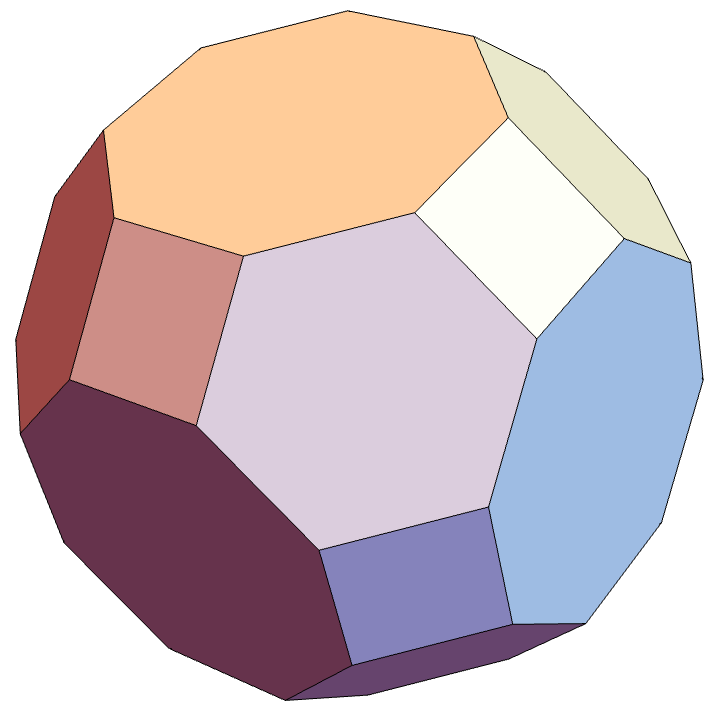}
	\hspace{1cm}
	\includegraphics[width=.25\textwidth]{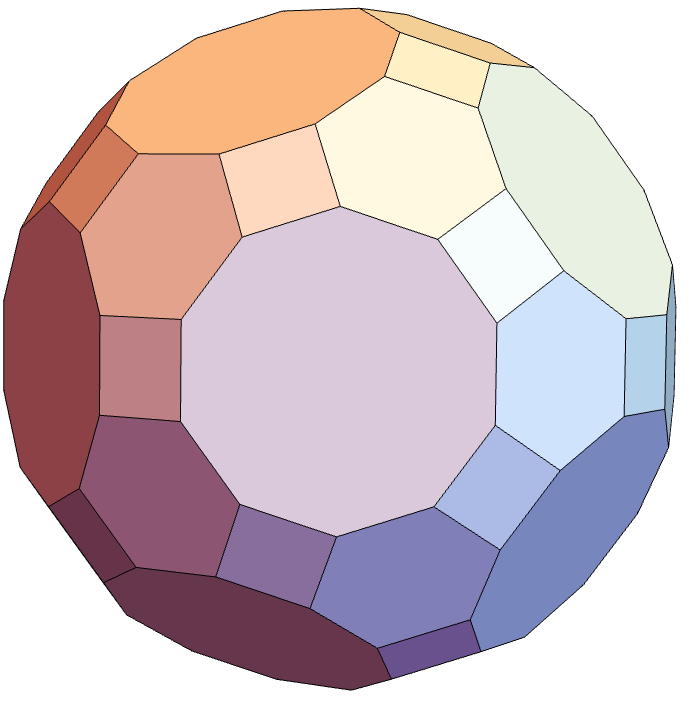}
	\caption{Illustration of zonotopes in $\mathbb{R}^3$.}\label{fig:zonotopes}
\end{figure}

\begin{Proposition}[Zonotopes; compatibility implies transversality]\label{prop:zonotop}
Given two convex \PA{} functions $h,g:\mathbb{R}^d\to \mathbb{R}$ and a point $\bm{x}\in\mathbb{R}^d$, suppose that the polytopes $\partial h(\bm{x})$ and $\partial g(\bm{x})$ are zonotopes. Then, the functions $h$ and $g$ are transversal at the point $\bm{x} \in \mathbb{R}^d$ if and only if
\[
\partial (h-g)(\bm{x}) = \partial h(\bm{x}) - \partial g(\bm{x}).
\]
\end{Proposition}
\begin{proof}
	The proof is a bit technical and is deferred to \Cref{sec:prf-zonotope}.
\end{proof}
\begin{Remark}
	It is known that zonotopes correspond to the extreme case in some variational problems involving polytopes; see \citep[Proposition 2.1.8]{gritzmann1993minkowski}. \Cref{prop:zonotop} provides another example of such a case, which is not only of theoretical interest but also has real-world applications. For instance, it can be used in the analysis of the loss function of $\rho$-margin loss SVM and the two-layer \textnormal{ReLU}-type neural networks; see \Cref{sec:svm,sec:2relu} for details.
\end{Remark}

\begin{Remark}
Consider the case where $h(\bm{w}):=\max\left\{\bm{w}^\top \bm{x}, 0\right\}$ and $g(\bm{w}):=\max\left\{\bm{w}^\top \bm{y}, 0\right\}$  in the context of \Cref{prop:zonotop}. Define line segments $X:=[\bm{x},\bm{0}]$ and $Y:=[\bm{y},\bm{0}]$ in $\mathbb{R}^d$ with the non-zero vectors $\bm{x}$ and $\bm{y}$. A straightforward computation shows that $X$ and $Y$ are compatible if and only if $(\parr\{\bm{x},\bm{0}\})\cap (\parr\{\bm{y},\bm{0}\}) = \{\bm{0}\}.$ \Cref{prop:zonotop} demonstrates that this simple observation can be extended, in a highly nontrivial manner, to general zonotopes.
\end{Remark}

In the following example, we illustrate that the sufficient condition presented in \Cref{prop:par-sufficient} may not be necessary in general. This specific example precisely corresponds to the compatible polytopes discussed in \Cref{exam:comp}.
\begin{Example}[$\mathbb{R}^4\to \mathbb{R}$ function]\label{exam:comp-subd}
 Let $\bm{y} := (1,1,-1,-1)\in\mathbb{R}^4$ and  $\bm{e}_0 := \bm{0}_4$ be given. Consider two convex \PA{} functions $h,g:\mathbb{R}^4\to \mathbb{R}$ defined by
\[
h(\bm{w}):=\max_{0\leq i \leq 4} \bm{w}^\top \bm{e}_i,\qquad  g(\bm{w}):= \max{}\{0, \bm{w}^\top \bm{y}\}.
\]
Recall that in \Cref{exam:comp}, polytopes $\partial h(\bm{0})=\conv\{\bm{e}_i:i \in \{0, \dots, 4\}\}$ and $\partial g(\bm{0})=[\bm{0}, \bm{y}]$ are shown to be compatible. By \Cref{thm:general-sum-Clarke-geometric}, we get that
\[
\partial (h-g)(\bm{0})=\partial h(\bm{0})-\partial g(\bm{0})=\conv\{\bm{e}_i\}_{i=0}^4 - [\bm{0}, \bm{y}]. 
\]
However, it is easily seen that $\parr(\partial h(\bm{0}))=\parr\{\bm{e}_i:i \in \{0, \dots, 4\}\} = \mathbb{R}^4  \supsetneq [\bm{0}, \bm{y}]=\partial g(\bm{0})$. Therefore, the exact sum rule holds due to ``compatibility'' rather than ``separability'' of the data.
\end{Example}

Summarizing, we have the following observations.
\begin{itemize}
	\item When polytopes representing subdifferentials are zonotopes, a ``separability''-type condition---namely, transversality---is necessary for the validity of the exact sum rule.
	\item However, in general, as demonstrated in \Cref{exam:comp-subd}, the validity of the exact sum rule relies on compatibility rather than transversality.
\end{itemize}

\subsubsection{Proof of \Cref{prop:zonotop}}\label{sec:prf-zonotope}
Here are some elementary properties of Minkowski sum of sets.
\begin{Lemma}\label{lem:localization}
	Let $A, B, C $ be sets in $\mathbb{R}^n$. Suppose further that $A$ is convex and closed, and $C$ is non-empty and bounded. If the strict inclusion $A \subsetneq B$ holds, then we can assert $A+C \subsetneq B + C$.
\end{Lemma}
\begin{proof}
	Let $\bm{x}_b \in B\backslash A$. The claim is trivial when $A=\emptyset$. Choose $\bm{x}_a' \in A$ and set $\delta\coloneqq \|\bm{x}_b-\bm{x}_a'\|$. As $A$ is closed, the following $\bm{x}_a$ is well-defined
	\[
	\bm{x}_a \coloneqq \argmin_{\bm{a} \in A} \|\bm{a} - \bm{x}_b \|=\argmin_{\bm{a} \in A\cap\mathbb{B}_{\delta}(\bm{x}_b)} \|\bm{a} - \bm{x}_b \|. 
	\]
	Let $\bm{d} \coloneqq \bm{x}_b - \bm{x}_a$. As $\bm{x}_b \notin A$ and $A$ is closed, we know $\|\bm{d}\| > 0$.
	By the optimality condition and convexity of $A$, we know $\langle \bm{a} - \bm{x}_a, \bm{d}\rangle \leq 0, \forall \bm{a}\in A$, which implies $\langle  \bm{d}, \bm{a} \rangle \leq \langle\bm{d}, \bm{x}_a \rangle, \forall \bm{a} \in A$. As $C$ is bounded, we know $\langle \bm{c}, \bm{d} \rangle \leq \|\bm{c}\|\cdot \|\bm{d}\| < \infty,\forall \bm{c} \in C$. Let 
	$
	\bm{x}_c \in \eargmax_{\bm{c} \in C}\ \langle \bm{c}, \bm{d} \rangle,
	$
	where $0<\epsilon<\|\bm{d}\|^2$.
	We claim $\bm{x}_b + \bm{x}_c \notin A + C$. Suppose not. Then, there exist $\bm{y}_a \in A, \bm{y}_c \in C$ such that $\bm{y}_a + \bm{y}_c = \bm{x}_b + \bm{x}_c$. However, we compute
	\[
		\langle \bm{d}, \bm{x}_b + \bm{x}_c \rangle 
		=  \langle \bm{d}, \bm{d} + \bm{x}_a \rangle + \langle \bm{d}, \bm{x}_c \rangle \\
		 \geq \|\bm{d}\|^2 + \langle\bm{d}, \bm{y}_a \rangle + \langle\bm{d}, \bm{y}_c \rangle - \epsilon\\
		 >  \langle\bm{d}, \bm{y}_a + \bm{y}_c \rangle,
	\]
	which gives the contradiction.
\end{proof}
\begin{Remark}
Though the claim seems straightforward,
\Cref{lem:localization} is indeed nontrivial. We record the following counterexamples when different conditions are removed.
\begin{itemize}[label=$\circ$]
	\item $C$ is empty: $A+C = B + C = \emptyset$. 
	\item $C$ is unbounded: if $C=\mathbb{R}^n$ and $A,B$ are non-empty, then $A+C = B + C = \mathbb{R}^n$.
	\item $A$ is nonconvex: if $A=\mathbb{B}\backslash\left(\frac{1}{4}\mathbb{B}\right), B=\mathbb{B}, C=\mathbb{B}$, then $A+C = B + C = 2\mathbb{B}$.
	\item $A$ is not closed: if $A=\intt(\mathbb{B}), B=\mathbb{B}, C = \intt(\mathbb{B})$, then $A+C = B + C = \intt(2\mathbb{B})$.
\end{itemize}	
\end{Remark}

\begin{Lemma}\label{lem:extremept}
Let $\{\bm{p}_i\}_{i=1}^n$ be linearly independent. Define a convex set $C = \sum_{i=1}^n \bm{p}_i\cdot[0,1]$, which is actually a parallelotope.
We have $\ext(C) = \sum_{i=1}^n \{0,1\}\cdot \bm{p}_i.$
\end{Lemma}
\begin{proof}
Note that $C = \conv \left(\sum_{i=1}^n \{0,1\}\cdot \bm{p}_i\right)$. By \cite[Corollary 18.3.1]{rockafellar1970convex}, we obtain $\ext(C) \subseteq \sum_{i=1}^n \{0,1\}\cdot \bm{p}_i.$ It suffices to only prove $\sum_{i=1}^n \{0,1\}\cdot \bm{p}_i \subseteq \ext(C)$.
	Suppose not. Let $\bm{p} := \sum_{i=1}^n s_i\cdot\bm{p}_i \notin \ext(C)$ with $s_i \in \{0,1\} $ for all $i\in[n]$. We write $\bm{p} = \frac{1}{2}\bm{x}_1 + \frac{1}{2}\bm{x}_2 $ where $ \bm{x}_1 := \sum_{i=1}^n\alpha_i\cdot \bm{p}_i \in C, \bm{x}_2 := \sum_{i=1}^n\beta_i\cdot \bm{p}_i \in C$, and $\bm{x}_1 \neq \bm{x}_2$. We know $\alpha_i \in [0,1]$ and $\beta_i \in [0,1]$ for any $i\in[n]$ by definition. Thus, it holds
	$
	\sum_{i=1}^n s_i \cdot \bm{p}_i = \sum_{i=1}^n \left(\frac{\alpha_i + \beta_i}{2} \right)\cdot \bm{p}_i.
	$
	As vectors $\{\bm{p}_i\}_{i=1}^n$ are linearly independent, for any $i \in [n]$, we obtain $s_i = \left(\frac{\alpha_i + \beta_i}{2} \right) \in \{0,1\}$. If $s_i = 0$, we have $\alpha_i = \beta_i = 0$. Meanwhile, we know $\alpha_i = \beta_i = 1$ if $s_i = 1$. Therefore, it holds $\bm{x}_1 = \bm{x}_2 = \bm{p}$, a contradiction.
\end{proof}

The proof of \Cref{prop:zonotop} relies on the following key technical lemma.
\begin{Lemma}\label{lem:s-nece} Given vectors $\bm{x}_i\neq \bm{0}$ and  $\bm{y}_j \neq \bm{0}$ in $\mathbb{R}^d$ for $i\in[n]$ and $j\in[m]$,
let 
\[
f_{\sf Z}(\bm{w}):=\sum_{i=1}^n \max \left\{\bm{w}^\top \bm{x}_i, 0 \right\} - \sum_{j=1}^m \max \left\{ \bm{w}^\top \bm{y}_j, 0 \right\}.
\]
	If there exists $\bm{v}\in\mathbb{R}^d$ such that
	\[
	\bm{0}\neq \bm{v}\in\spn \lB \bm{x}_i \rB_{i=1}^n  \cap 
	\spn \lB \bm{y}_j \rB_{j=1}^m,
	\]
	then $
	\partial f_{\sf{Z}} (\bm{0}) \subsetneq G_{\sf{Z}}:=\sum_{i=1}^n [\bm{x}_i, \bm{0}] - \sum_{j=1}^m [\bm{y}_j, \bm{0}].
	$
\end{Lemma}
\begin{proof} Our proof will be divided into two steps.

\textit{Step 1.}
	We first show that it suffices to consider only a properly chosen subset of $\{\bm{x}_i\}_{i=1}^n$ and $\{\bm{y}_j\}_{j=1}^{m}$. Let the index set $\mathcal{J}_x \subseteq [n]$  be a minimal selection from $\{\bm{x}_i\}_{i=1}^n$ such that $\lB \bm{x}_i \rB_{i\in\mathcal{J}_x}$ are linearly independent and satisfy
	$
	\bm{0}\neq \bm{v} \in \spn \lB \bm{x}_i \rB_{i\in\mathcal{J}_x}.%
	$
	Similarly, we define $\mathcal{J}_y \subseteq [m]$  for $\{\bm{y}_j\}_{j=1}^m$.
	Then, we write
	\[
f_{\sf{Z}}(\bm{w})= f_{\sf{Z}_1}(\bm{w}) + f_{\sf{Z}_2}(\bm{w}),
\]
where
\[
\begin{aligned}
f_{\sf{Z}_1}(\bm{w})&\coloneqq\sum_{i\in[n]\backslash\mathcal{J}_x} \max\lB \bm{x}_i^\top \bm{w},0 \rB - \sum_{j\in[m]\backslash\mathcal{J}_y } \max\lB \bm{y}_j^\top \bm{w}, 0 \rB,\\
f_{\sf{Z}_2}(\bm{w})&\coloneqq\sum_{i\in\mathcal{J}_x} \max\lB \bm{x}_i^\top \bm{w},0 \rB - \sum_{j\in\mathcal{J}_y } \max\lB \bm{y}_j^\top \bm{w}, 0 \rB.	
\end{aligned}
\]
By using the fuzzy sum rule \cite[Proposition 2.3.3]{clarke1990optimization} twice, we obtain
\[
\partial f_{\sf{Z}}(\bm{0}) \subseteq \partial f_{\sf{Z}_1}(\bm{0}) + \partial f_{\sf{Z}_2}(\bm{0}) \subseteq \partial f_{\sf{Z}_1}(\bm{0}) + G_{\sf{Z}_2} \subseteq G_{\sf{Z}},
\]
where we define set $G_{\sf{Z}_2}\coloneqq \sum_{i\in\mathcal{J}_x} [\bm{x}_i, \bm{0}] -\sum_{j\in\mathcal{J}_y}  [\bm{y}_j,\bm{0}]$.
Thus, to prove $\partial f_{\sf{Z}}(\bm{w}) \subsetneq G_{\sf{Z}}$, by \Cref{lem:localization} and \cite[Proposition 2.1.2(a)]{clarke1990optimization}, we only need to show $\partial f_{\sf{Z}_2}(\bm{w}) \subsetneq G_{\sf{Z}_2}$. By abuse of notation and focus on $f_{\sf{Z}_2}$, we can assume that $\mathcal{J}_x=[n]$ and $\mathcal{J}_y=[m]$. Hence, $\{\bm{x}_i\}_{i=1}^n$ are linearly independent, and $\{\bm{y}_j\}_{j=1}^m$ are also linearly independent. 

 Fix $\bm{y}_m$. We can further assume $\{\bm{x}_i\}_{i=1}^n \cup \{\bm{y}_j\}_{j=1}^{m-1}$ are linearly independent. To see this, we employ a procedure akin to Gaussian elimination. 
	
	\begin{description}
	\item[(Initialization)] As $\bm{v}\in\spn \lB \bm{x}_i \rB_{i=1}^n  \cap 
	\spn \lB \bm{y}_j \rB_{j=1}^m $, we write
	\begin{equation}\label{eq:prf-def-of-v}
			\bm{0}\neq \bm{v} = \sum_{i=1}^n a_i \cdot \bm{x}_i = \sum_{j=1}^m b_j \cdot \bm{y}_j.
	\end{equation}
	As index sets $\mathcal{J}_x$ and $\mathcal{J}_y$ are minimal, we get $a_i\neq 0, b_j \neq 0, \forall i \in [n],j\in[m]$.
		\item[(Elimination)]  Provided $\{\bm{x}_i\}_{i=1}^n \cup \{\bm{y}_j\}_{j=1}^{m-1}$ are linearly dependent, we can write
	$
	\bm{0} = \sum_{i=1}^n p_i\cdot \bm{x}_i + \sum_{j=1}^{m-1}q_j\cdot \bm{y}_j
	$ 
	with some $j' \in [m-1]$ such that $q_{j'}\neq 0$, as otherwise by linear independence of $\{\bm{x}_i\}_{i=1}^n$, it holds that $p_i=0$ for all $i\in[n]$, contradicting the linear dependence of $\{\bm{x}_i\}_{i=1}^n \cup \{\bm{y}_j\}_{j=1}^{m-1}$.
	As $q_{j'}\neq 0$, we have 
	\[
	\bm{y}_{j'} = -\sum_{i=1}^n (p_i/q_{j'})\cdot \bm{x}_i-\sum_{j\in [m-1]\backslash\{j'\}} (q_j/q_{j'})\cdot \bm{y}_j.
	\]
	Substituting it into \nref{eq:prf-def-of-v}, the vector $\bm{y}_{j'}$ is removed and we obtain
	\[
	\bm{v}' =\sum_{i=1}^n \left(a_i':=a_i+\frac{b_{j'}p_i}{q_{j'}}\right)\bm{x}_i = b_m \bm{y}_m + \sum_{j\in [m-1]\backslash\{j'\}} \left(b_j':=b_j-\frac{b_{j'}q_j}{q_{j'}}\right)\cdot \bm{y}_j,
	\]
	where $\bm{v}'=\bm{v} + \frac{b_{j'}}{q_{j'}} \sum_{i=1}^n p_i\cdot \bm{x}_i$.
	To see $\bm{v}'\neq \bm{0}$, suppose not. Then, we get 
	\[
	\bm{v} = -\frac{b_{j'}}{q_{j'}} \sum_{i=1}^n p_i\cdot \bm{x}_i=\frac{b_{j'}}{q_{j'}} \sum_{j=1}^{m-1} q_j\cdot \bm{y}_j \in \spn \lB \bm{y}_j \rB_{j=1}^{m-1},
	\]
	 so that the linear independence of $\lB \bm{y}_j \rB_{j=1}^{m}$ yields $b_m=0$, a contradiction. After that, we exam $\{a_i'\}_i$ and $\{b_j'\}_j$. We remove $\bm{x}_i$ if $a_i'=0$ and remove $\bm{y}_j$ if $b_j'=0$.
	\item[(Repeat)] 
 Due to $\bm{v}'\neq \bm{0}$, repeat this elimination procedure at most $m-1$ times and
 by abuse of notation, we obtain that $\{\bm{x}_i\}_{i=1}^n \cup \{\bm{y}_j\}_{j=1}^{m-1}$ are linearly independent. It is possible that all $\{\bm{y}_j\}_{j=1}^{m-1}$ are removed and we get $m=1$ with $\bm{y}_m \in \spn\{\bm{x}_i\}_{i=1}^n$. But as $\bm{y}_m \neq \bm{0}$, we always have $n\geq 1$. By \Cref{lem:localization}, it suffices to focus on such a subset of the original inputs.
		\end{description}	
	Therefore, we can write
	\begin{equation}\label{eq:ym}
			\bm{y}_m = \sum_{i=1}^n \alpha_i \bm{x}_i + \sum_{j=1}^{m-1} \beta_j \bm{y}_j,
	\end{equation}
	with $\{\bm{x}_i\}_{i=1}^n \cup \{\bm{y}_j\}_{j=1}^{m-1}$ are linearly independent and $\alpha_i \neq 0, \beta_j \neq 0$ for any $i \in [n], j \in [m-1]$. 

\textit{Step 2.}
	We proceed to show that polytopes $X:=\sum_{i=1}^n[\bm{0},\bm{x}_i]$ and $Y:=\sum_{j=1}^m [\bm{0}, \bm{y}_j]$ are incompatible. Let $\bm{\theta}\in\mathbb{R}_+^{n+m}$ and we define
	\[\theta_i :=     \left\{ \begin{array}{rcl}
         |\alpha_i| & \mbox{for}
         & 1\leq i \leq n \\ |\beta_{i-n}|  & \mbox{for} & n+1 \leq i \leq n+m-1, \\
         1 & \mbox{for} & i = m+n
                \end{array}\right. \quad
	\bm{A} := \left[ \begin{array}{c}
     \sgn(\alpha_1)\cdot \bm{x}_1^\top \\
     \vdots \\
     \sgn(\alpha_n)\cdot \bm{x}_n^\top \\
      \sgn(\beta_1)\cdot \bm{y}_1^\top \\
     \vdots \\
     \sgn(\beta_{m-1})\cdot \bm{y}_{m-1}^\top \\
     -\bm{y}_m^\top
     \end{array} \right]
     \in \mathbb{R}^{(n+m)\times d}.
     	\]
	Note that $\bm{\theta}\neq \bm{0}$ and 
	$
	\bm{A}^\top \bm{\theta} = \sum_{i=1}^n \alpha_i \bm{x}_i + \sum_{j=1}^{m-1} \beta_j \bm{y}_j - \bm{y}_m = \bm{0}.
	$
	By Gordan's Theorem in \Cref{lem:gordan}, there is no vector $\bm{d}\in\mathbb{R}^d$ such that
	\begin{equation}\label{eq:sgn-gordan-1}
	\left\{ \begin{array}{rcl}
         \sgn(-\alpha_i)\cdot \bm{d}^\top \bm{x}_i > 0 & \mbox{for}
         & i \in [n] \\
         \sgn(-\beta_j)\cdot \bm{d}^\top \bm{y}_j > 0 & \mbox{for}
         & j \in [m-1]\\
          \bm{d}^\top \bm{y}_m > 0 &
         & 
             \end{array}\right..
	\end{equation}
	Let $\bm{x}:=\sum_{i=1}^n \mathbf{1}_{\alpha_i<0}\cdot \bm{x}_i  \in X$ and $\bm{y}:=\sum_{j=1}^{m-1}\mathbf{1}_{\beta_j<0}\cdot \bm{y}_j + \bm{y}_m \in Y$.
	Another way of stating \nref{eq:sgn-gordan-1} is to say: there is no vector $\bm{d}$ such that $\bm{d}^\top\bm{x} > \bm{d}^\top \bm{z}_x$ and $\bm{d}^\top\bm{y} > \bm{d}^\top \bm{z}_y$ for any $\bm{z}_x \in X\backslash\{\bm{x}\}$ and $\bm{z}_y \in Y\backslash\{\bm{y}\}$. 
	By \Cref{lem:fukuda}, we have
	\[
	\bm{x}+\bm{y}=\sum_{i=1}^n \mathbf{1}_{\alpha_i<0}\cdot \bm{x}_i + \sum_{j=1}^{m-1}\mathbf{1}_{\beta_j<0}\cdot \bm{y}_j + \bm{y}_m \notin \ext(X+Y).
	\]	
	Let $\bm{\theta}' \in \mathbb{R}_+^{n+m}$ and consider the following equation:
	\[
	\sum_{i=1}^n \theta_i'\cdot\sgn(\alpha_i)\cdot \bm{x}_i + \sum_{j=1}^{m-1}\theta_{j+n}'\cdot\sgn(\beta_j)\cdot (-\bm{y}_j) + \theta_{n+m}'\cdot \bm{y}_m = \bm{0}.
	\]
	By linear independence of $\{\bm{x}_i\}_{i=1}^n\cup \{\bm{y}_j\}_{j=1}^{m-1}$ and $n\geq 1$, one can readily see that $\theta'_{n+m} \geq 0$ can only be zero, so that it holds $\bm{\theta}'=\bm{0}$. By Gordan's Theorem in \Cref{lem:gordan}, there exists a vector $\bm{d}\in\mathbb{R}^d$ such that
	\begin{equation}
	\left\{ \begin{array}{rcl}
         \sgn(-\alpha_i)\cdot \bm{d}^\top \bm{x}_i > 0 & \mbox{for}
         & i \in [n] \\
         \sgn(-\beta_j)\cdot \bm{d}^\top (-\bm{y}_j) > 0 & \mbox{for}
         & j \in [m-1]\\
          \bm{d}^\top (-\bm{y}_m) > 0 &
         & 
             \end{array}\right..
	\end{equation}
	Using \Cref{lem:fukuda}, we get
	\[
	\bm{x}-\bm{y}=\sum_{i=1}^n \mathbf{1}_{\alpha_i<0}\cdot \bm{x}_i - \sum_{j=1}^{m-1}\mathbf{1}_{\beta_j<0}\cdot \bm{y}_j - \bm{y}_m \in \ext(X-Y).
	\]	
	Therefore, we conclude that polytopes $X$ and $Y$ are incompatible,	which completes the proof by \Cref{thm:general-sum-Clarke-geometric}.
\end{proof}

Now, we are ready to prove \Cref{prop:zonotop}.
\begin{proof}[Proof of \Cref{prop:zonotop}] The ``only if'' part is a direct corollary of \Cref{prop:par-sufficient}. We only prove the ``if'' part. 
Without loss of generality, we assume $\bm{x}=\bm{0}$. As functions $h,g$ are convex piecewise affine, we can write
\[
h(\bm{d})= h(\bm{0})+ h'(\bm{0}; \bm{d}) = \max_{\bm{g}_h \in \partial h(\bm{0})} \bm{d}^\top \bm{g}_h, \qquad 
g(\bm{d})= g(\bm{0})+g'(\bm{0}; \bm{d}) = \max_{\bm{g}_g \in \partial g(\bm{0})} \bm{d}^\top \bm{g}_g.
\]
From the hypothesis,  polytopes $\partial h(\bm{0})$ and $\partial g(\bm{0})$ are zonotopes. Then, we can write them with their generators as $\partial h(\bm{0})=\sum_{p=1}^P [\bm{x}_{p,1}, \bm{x}_{p,2}]$ and $\partial g(\bm{0}) = \sum_{q=1}^Q [\bm{y}_{q,1}, \bm{y}_{q,2}]$ for some $P$ and $Q$.
Consequently, we can construct the following \PA{} function:
\[
f(\bm{w})=\sum_{p=1}^P \max \left\{\bm{w}^\top (\bm{x}_{p,1} - \bm{x}_{p,2}), 0 \right\} - \sum_{q=1}^Q \max \left\{ \bm{w}^\top (\bm{y}_{q,1} - \bm{y}_{q,2}), 0 \right\}.
\]
By the compatibility of $\partial h(\bm{0})$ and $\partial g(\bm{0})$ from hypothesis and \Cref{thm:general-sum-Clarke-geometric}, using \Cref{prop:comp-poly}(iii),
we know that polytopes $\sum_{p=1}^P [\bm{x}_{p,1}- \bm{x}_{p,2}, \bm{0}]$ and $\sum_{q=1}^Q [\bm{y}_{q,1}- \bm{y}_{q,2}, \bm{0}]$ are also compatible. 
Then, it follows from \Cref{thm:general-sum-Clarke-geometric} that
\[
\partial f(\bm{0})=\sum_{p=1}^P \conv \left\{\bm{x}_{p,1} - \bm{x}_{p,2}, \bm{0} \right\} - \sum_{q=1}^Q \conv \left\{ \bm{y}_{q,1} - \bm{y}_{q,2}, \bm{0} \right\}.
\]
It can readily be seen that $\parr(\partial h(\bm{0}))=\sum_{p=1}^P [\bm{x}_{p,1}, \bm{x}_{p,2}] = \spn\{\bm{x}_{p,1} - \bm{x}_{p,2}\}_{p=1}^P$ and $\parr(\partial g(\bm{0}))=\sum_{q=1}^Q[\bm{y}_{q,1}, \bm{y}_{q,2}] = \spn\{\bm{y}_{q,1} - \bm{y}_{q,2}\}_{q=1}^Q$. The proof completes by using the contraposition of \Cref{lem:s-nece} on $f$.
\end{proof}

\section{Rounding and Finite Termination}\label{sec:robust}

Our focus has been on exact $\epsilon$-stationarity testing, which checks whether $\bm{0} \in \partial f(\bm{w}) + \epsilon\mathbb{B}$ for a given point $\bm{w}$. Yet, for a numerical algorithm, reaching exact nondifferentiable points is virtually impossible due to algorithmic randomization and/or finite precision limitations. Therefore, a need arises for developing a \emph{robust} stationarity testing approach. This approach, when queried  at a point $\bm{w}$, certifies or refutes $\bm{0} \in \partial f(\bm{w} + \delta\mathbb{B}) + \epsilon\mathbb{B}$ for given precisions $\epsilon, \delta \geq 0$. In this section, we present the main algorithmic results of this paper regarding robust testing for the \PA{} function $h-g$, where the convex functions $h,g:\mathbb{R}^d \to \mathbb{R}$ are both given in the $n$-\MC{} form with the depth $n$ as an input.

\subsection{Testing Oracle and Problem Setup}\label{sec:robust-setup}

We begin by abstracting the procedure for exact $\epsilon$-stationarity testing into the following oracle.

\begin{Definition}[Oracle]\label{def:oracle}
Let convex \PA{} functions $h,g$ be given.
A stationarity test oracle for function $h-g$ is an oracle whose inputs are certain representation of functions $h$ and $g$, a point $\bm{w}\in\mathbb{Q}^d$, and a rational number $\epsilon \geq 0$. This oracle decides whether $\bm{0} \in  \partial (h-g)(\bm{w})+\epsilon\mathbb{B}$  or not.
\end{Definition}

\begin{Remark}
In our robust testing algorithm, we only need to call the exact $\epsilon$-stationarity testing procedure in a black-box manner. Thus, one can use any efficient implementation of the oracle defined above. 
Moreover, by our development in \Cref{sec:cr}, if polytopes $\partial h(\bm{w})$ and $\partial g(\bm{w})$ are compatible, 
then the oracle in \Cref{def:oracle} has polynomial-time implementation.
\end{Remark}

\paragraph{Information-Theoretic Hardness.}
Let an efficient implementation of the oracle be provided, as defined in \Cref{def:oracle}. Suppose the \DC{} function is given to us in a black-box manner, following the information-theoretic model \cite{nemirovskij1983problem}. 
In this setting, the efficiency of the algorithm is evaluated based on its use of information rather than the time it takes for computation.
An immediate question arises: Can we utilize the testing oracle in \Cref{def:oracle} to construct an algorithm for detecting an $(\epsilon,\delta)$-NAS point? This algorithm should only interact with the objective function via a local oracle (see, e.g., \cite[Section 2.3.1]{tian2022no}), which provides information about the function, such as function values and subdifferentials. The following proposition provides a negative answer.%

\begin{Proposition}[cf.~{\cite[Proposition 2]{tian2022no}}]\label{prop:info-hardness} Suppose that $0 < \epsilon,\delta < 1$. For any deterministic algorithm $\mathscr{A}$ interacting with a local oracle of the objection function, there exist two $4$-Lipschitz convex \PA{} functions $h,g:\mathbb{R}\to\mathbb{R}$ such that $\mathscr{A}$
 cannot decide whether $0 \in \partial (h-g)(\delta\mathbb{B}) + \epsilon\mathbb{B}$ with finite calls to the local oracle and the testing oracle in \Cref{def:oracle}.
\end{Proposition}

Therefore, to make progress, we consider a white-box setting as follows. %

\paragraph{Problem Setup.} 
The functions $h$ and $g$ in this section are given in $n$-\MC{} form with $n$ as an input.
We adopt a constructive viewpoint, wherein we certify the $(\epsilon,\delta)$-NAS status of a point $\bm{w}$ for a \PA{} function $h-g$ only when we find a certificate $\bm{w}^* \in \mathbb{B}_\delta(\bm{w})$ that satisfies $\bm{0} \in \partial (h-g)(\bm{w}^*)+ \epsilon\mathbb{B}.$
Note that if a point $\bm{w}^* \in \mathbb{B}_\delta(\bm{w})$ successfully passes the exact stationarity test using an implementation of the oracle in \Cref{def:oracle}, then $\bm{w}$ must indeed be an $(\epsilon,\delta)$-NAS point. In other words, there are no false positives in this test.
The question that arises is whether, if $\bm{w}$ is sufficiently close to some unknown $\epsilon$-stationary point, we can \emph{always} find a point $\bm{w}^*$ near $\bm{w}$ that is $\epsilon$-stationary. Namely, we aim to control the occurrence of false negatives in our robust test. Without leveraging structures of the objective function, finding such a point is generally impossible, as discussed in \Cref{prop:info-hardness}.
As a result, in contrast to the information-theoretic model \cite{nemirovskij1983problem}, we will employ an oracle Turing machine as our computational model.

\paragraph{DC with $n$-MC{} Components.} In this section, we study the robust testing problem for a \PA{} function represented by the difference of two $n$-\MC{} functions, $h,g:\mathbb{R}^d\to \mathbb{R}$. To fix notations, we explicitly express these two $n$-\MC{} functions as

\[
\begin{aligned}
h(\bm{w})&:=\sum_{1\leq j_n \leq J_n^h} \max_{1\leq i_{n} \leq I_n^h} \cdots \sum_{1\leq j_1 \leq J_1^h} \max_{1\leq i_1 \leq I_1^h} \Big(\bm{w}^\top \bm{x}_{i_1,j_1,\dots,i_n,j_n} + a_{i_1,j_1,\dots,i_n,j_n}\Big), \\
g(\bm{w})&:=\sum_{1\leq j_n \leq J_n^g} \max_{1\leq i_{n} \leq I_n^g} \cdots \sum_{1\leq j_1 \leq J_1^g} \max_{1\leq i_1 \leq I_1^g} \Big(\bm{w}^\top \bm{y}_{i_1,j_1,\dots,i_n,j_n} + b_{i_1,j_1,\dots,i_n,j_n}\Big).
\end{aligned}
\]
To analyze the \PA{} function $h-g$ with its multi-composite structure more succinctly, we introduce several value functions for the $n$-\MC{} functions $h$ and $g$, respectively. We begin by defining the value functions for $h$.
For any index $(i_1,j_1,\dots,i_n,j_n) \in \prod_{k=1}^n [I_k^h]\times [J_k^h]$, let  $v_{i_1,j_1,\dots,i_n,j_n}:\mathbb{R}^d \to \mathbb{R}$ denote the affine function
\[
v_{i_1,j_1,\dots,i_n,j_n}(\bm{w}) :=\bm{w}^\top \bm{x}_{i_1,j_1,\dots,i_n,j_n} + a_{i_1,j_1,\dots,i_n,j_n}.
\]
Then, for any $k \in [n-1]$, we define convex value functions $v_{j_k,\dots,i_n,j_n}$ and  $v_{i_{k+1},j_{k+1},\dots,i_n,j_n}$ as

\[
v_{j_k,\dots,i_n,j_n}(\bm{w}) :=\max_{1\leq i_k \leq I_k^h}v_{i_k,j_k,\dots,i_n,j_n} (\bm{w}), \qquad 
v_{i_{k+1}, j_{k+1}\dots,i_n,j_n}(\bm{w}) :=\sum_{1\leq j_k \leq J_k^h}v_{j_k, i_{k+1},\dots,i_n,j_n} (\bm{w}).
\]
Finally, for any $j_n \in [J_n^h]$, we define a value function $v_{j_n}$ as
\[
v_{j_n}(\bm{w}) := \max_{1\leq i_n \leq I_n^h} v_{i_n, j_n} (\bm{w}).
\]
Similar functions are also defined for the $n$-\MC{} function $g$ with $u_{i_k,j_k,\dots,i_n,j_n}$ and $u_{j_k,\dots,i_n,j_n}$ for any $k \in [n]$. Thus, we can rewrite the \PA{} function $h-g$ simply as
\[
(h-g)(\bm{w})= \left(\sum_{1\leq j_n \leq J_n^h} v_{j_n}(\bm{w})\right) - \left(\sum_{1\leq j_n \leq J_n^g} u_{j_n}(\bm{w})\right).
\]

\paragraph{Separation $\delta$-Function.} Suppose we have an $\epsilon$-stationary point $\bm{w}^* \in \mathbb{R}^d$. When a point $\bm{w}$ is sufficiently close (within a distance of $\delta$) to $\bm{w}^*$, our objective is to certify the condition: 
$
\bm{0} \in \partial (h-g)(\bm{w} + \delta\mathbb{B}) + \epsilon\mathbb{B}.
$ 
To quantify the required degree of proximity, we introduce the following separation constants for the \PA{} function $h-g$ with $n$-\MC{} components.
\begin{Definition}[Separation]\label{def:sep}
Define a polynomial-time computable constant $R$ as
\[
\begin{aligned}
R:=\max\left\{\sum_{1\leq j_n \leq J_n^h} \max_{1\leq i_{n} \leq I_n^h} \right.\cdots \sum_{1\leq j_1 \leq J_1^h} &\max_{1\leq i_1 \leq I_1^h} \| \bm{x}_{i_1,j_1,\dots,i_n,j_n} \|,  \\
&\left.\sum_{1\leq j_n \leq J_n^g} \max_{1\leq i_{n} \leq I_n^g} \cdots \sum_{1\leq j_1 \leq J_1^g} \max_{1\leq i_1 \leq I_1^g} \| \bm{y}_{i_1,j_1,\dots,i_n,j_n} \|
\right\}.
\end{aligned}
\]
Let  a point $\bm{w} \in \mathbb{R}^d$ be given. Let $\gamma^v_{\sf gap}(\bm{w})\in (0,\infty]$ be the supremum of  $\gamma > 0$ such that	 the following linear inequality system is satisfied:
\[
\left\{
\begin{aligned}
v_{j_n}(\bm{w}) &\geq v_{i_n, j_n}(\bm{w}) + \gamma, &\text{for any}& & v_{i_n, j_n}(\bm{w}) &\neq v_{j_n}(\bm{w}), \\
 v_{j_{n-1},i_n,j_n}(\bm{w}) &\geq v_{i_{n-1},j_{n-1},i_n,j_n}(\bm{w}) + \gamma, &\text{for any} & & v_{i_{n-1},j_{n-1},i_n,j_n}(\bm{w}) &\neq  v_{j_{n-1},i_n,j_n}(\bm{w}), \\
 &\ \ \vdots & & & &\ \ \vdots \\
  v_{j_{1},i_2,\dots,i_n,j_n}(\bm{w}) &\geq v_{i_{1},j_1,\dots,i_n,j_n}(\bm{w}) + \gamma, &\text{for any}& & v_{i_{1},j_1,\dots,i_n,j_n}(\bm{w}) &\neq  v_{j_{1},i_2,\dots,i_n,j_n}(\bm{w}).
\end{aligned}
\right.
\]
Likewise, a function $\gamma^u_{\sf gap} : \mathbb{R}^d \to (0,\infty]$ can be defined for value functions of $g$.
We define the extended-real-valued function $\delta_{\sf sep}:\mathbb{R}^d\to (0,\infty]$ as 
\[
\delta_{\sf sep}(\bm{w}):=\frac{\min\{\gamma^v_{\sf gap}(\bm{w}), \gamma^u_{\sf gap}(\bm{w})\}}{12R}.
\]
\end{Definition}

\begin{Remark}
For rational input data, all constants (and functions) in \Cref{def:sep} can be computed (and evaluated) in polynomial time. Notably, the constant $\delta_{\textnormal{sep}}(\bm{w})$ exhibits similarity with quantities \cite[Equation (7.3.9)]{wright1997primal} and  \cite[Definition 4]{spielman2003smoothed} in the context of solving LP with IPMs. These quantities are closely related to finite-termination strategies for IPMs; see also \cite{ye1992finite}.	However, our approach to test stationarity differs significantly from that of terminating IPMs in the following two key aspects:
\begin{enumerate}[label=\textnormal{(\roman*)}]
	\item Linear programming is a convex problem, and IPMs are guaranteed to converge to a global optimal solution. In contrast, our interest lies in terminating an algorithm that converges to a stationary point of a nonconvex objective function. Hence, no global property, like convexity, can be exploited in algorithm design for finite termination.
	\item The existing termination techniques in \cite{ye1992finite,spielman2003smoothed} crucially rely on the strict complementarity of the solution that many IPMs converge to (see \cite{guler1993convergence}). However, in our stationarity setting, we make no assumptions about the stationary point to which the algorithm is converging. In particular, our guarantee (in \Cref{thm:robust-test}) is algorithm-independent.
\end{enumerate}
\end{Remark}
We can now define the goal of this section as solving the following robust stationarity testing problem for a \DC{} function with $n$-\MC{} components.
\begin{tcolorbox}[sharp corners,colback=white, colframe=black,boxrule=.2mm]
\begin{center}
	\sf Robust Stationarity Testing for DC (RST-DC)
\end{center}
\textbf{Instance.} An integer $n\in\mathbb{N}$. Two $n$-\MC{} functions $h$ and $g$ with rational data. A point $\bm{w} \in \mathbb{Q}^d$. The precision parameters $0 < \epsilon, \delta\in \mathbb{Q}$.

\vspace{3mm}
\textbf{Question.} Either
\begin{description}[leftmargin=!,labelwidth=\widthof{\bfseries Falseea},font=\normalfont\space]
	\item[\textsf{(True)}] find a vector $\bm{w}^* \in \mathbb{Q}^d$ such that 
$
\|\bm{w} - \bm{w}^*\| \leq \delta$ and $\bm{0} \in \partial (h-g)(\bm{w}^*) + \epsilon\mathbb{B}$, or
\item[\textsf{(False)}] assert that any  vector $\bm{w}^* \in \mathbb{R}^d$ with $\bm{0} \in \partial (h-g)(\bm{w}^*) + \epsilon\mathbb{B}$ satisfies 
\[\|\bm{w} - \bm{w}^*\| > \min\{\delta, \delta_{\sf sep}(\bm{w}^*)\}.\]
\end{description}
\end{tcolorbox}

\begin{Remark}\label{rmk:uniform-lb}
The constant $\delta_{\sf sep}(\bm{w}^*)>0$ in the right-hand side of the {\sf{(False)}} output option above depends on an unknown point $\bm{w}^*$. While this does not affect its usage as an algorithm-independent finite-termination criterion, considering the infimum of  $\delta_{\sf sep}(\bm{w}^*)$ over the set of $\epsilon$-stationary points $S_\epsilon:=\{\bm{x}:\bm{0}\in\partial (h-g)(\bm{x})+\epsilon\mathbb{B}\}$ reveals that this infimum could be zero. To obtain a uniform positive lower bound that only depends on $h,g, \epsilon$, and boundedness, as promised in \Cref{sec:postive}, we show that a topological argument suffices. To this end, first note that $S_\epsilon$ is closed. To see it, consider a sequence $\bm{x}_k \to \bm{x}$ with $\bm{x}_k \in S_\epsilon$. By the outer semicontinuous of the set-valued mapping $\partial(h-g)$ at $\bm{x}$ \cite[Proposition 2.1.5(d)]{clarke1990optimization} and compactness of $\epsilon\mathbb{B}$, we get $\bm{x} \in S_\epsilon$. Consider the compact set $S^t_\epsilon:=S_\epsilon\cap t\mathbb{B}$ with a large $t \in [0,\infty)$ such that $\bm{w} \in ( t-2\delta )\mathbb{B}$. For any $\bm{x} \in S^t_\epsilon$, by \Cref{def:sep}, we have $\delta_{\sf sep}(\bm{x})>0$. Thus, the set $O_{\epsilon, t}:=\bigcup_{\bm{x} \in S_\epsilon^t} \intt\left(\bm{x}+\delta_{\sf sep}(\bm{x})\mathbb{B}\right)$ is an open cover of the compact $S^t_\epsilon$. The assertion in option {\sf (False)} essentially says that  
\[
\bm{w} \notin \bigcup\left\{ \bm{w}^* + \min\{\delta, \delta_{\sf sep}(\bm{w}^*)\}\mathbb{B}: \bm{w}^* \in [\partial(h-g)]^{-1}(\epsilon\mathbb{B})\right\} =: F_{\epsilon,\delta};
\]
see \Cref{fig:delta} for an illustration. Note that
$ \cl(O_{\epsilon, t})\cap(S_\epsilon^t+\delta\mathbb{B}) \subseteq F_{\epsilon,\delta},$ hence $\bm{w} \notin \cl(O_{\epsilon, t})\cap(S_\epsilon^t+\delta\mathbb{B})$.
Let $\delta^*:=\sup\{r\geq 0: S^t_\epsilon+r\mathbb{B}\subseteq O_{\epsilon, t}\}$, which is a constant that only depends on $h,g,\epsilon$, and $t$. We show that $\delta^* > 0$. Conversely, suppose $\delta^* = 0$. Then, for any $n\in\mathbb{N}$, there exists $\bm{x}_n \in S^t_\epsilon$ such that $\bm{x}_n+\frac{1}{n}\mathbb{B} \not\subseteq O_{\epsilon, t}$. From the compactness of $S^t_\epsilon$, by taking a subsequence if necessary, we get $\bm{x}_n \to \bm{x}'$. But $\bm{x}_n+\frac{1}{n}\mathbb{B} \subseteq \bm{x}'+\intt\left(\delta_{\sf sep}(\bm{x}')\mathbb{B}\right) \subseteq O_{\epsilon, t}$ for sufficiently large $n \in \mathbb{N}$, which is a contradiction. Thus, $\bm{w} \notin \cl(O_{\epsilon, t})\cap(S_\epsilon^t+\delta\mathbb{B})$ implies 
\[
\bm{w} \notin S^t_\epsilon + \min\{\delta, \delta^*\}\mathbb{B} = S_\epsilon\cap t\mathbb{B}+ \min\{\delta, \delta^*\}\mathbb{B}.
\]
By definition, we have $\bm{w} \notin \mathbb{R}^d\backslash(t-2\delta)\mathbb{B}$, so that $\bm{w} \notin S_\epsilon + \min\{\delta, \delta^*\}\mathbb{B}$ as desired.
\end{Remark}

\begin{figure}[t]
	\centering
	\includegraphics[width=.58\textwidth]{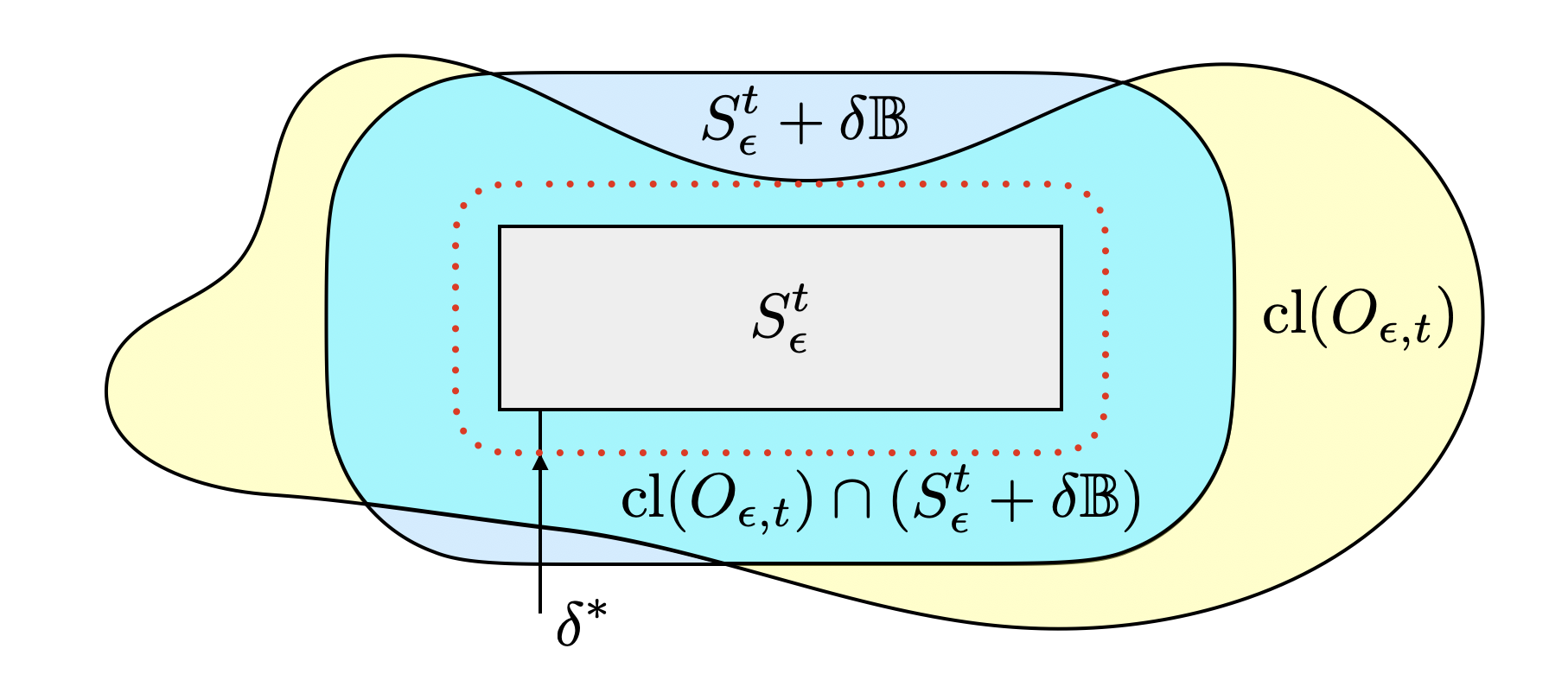}
	\caption{Illustration of sets in \Cref{rmk:uniform-lb}.}\label{fig:delta}
\end{figure}

\begin{Remark}
We emphasize that the positive radius $\min\{\delta, \delta_{\sf sep}(\bm{w}^*)\}$ is independent of the point $\bm{w}$. Such independence is crucial for terminating an algorithm (e.g., subgradient method) in finite time. Indeed, if $\bm{0} \notin \partial (h-g)(\bm{w})$, it is easy to certify that there is a positive number $\eta_{\bm{w}}$ dependent on $\bm{w}$, such that $\bm{0} \notin \partial (h-g)(\bm{w} + \eta_{\bm{w}}\mathbb{B})$. However, as $\bm{w} \to \bm{w}^* \in [\partial (h-g)]^{-1}(\bm{0})$, we may find that  
\[
\bm{w} \notin S_\epsilon\cap t\mathbb{B}+ \min\{\delta, \delta^*\}\mathbb{B}.
\]
with $\eta_{\bm{w}} \searrow 0$, providing little information about $\bm{w}^*$ and rendering it unsuitable as a stopping rule.
\end{Remark}

\subsection{A Convex Toy Example}
Before presenting the formal construction for the general case, let us first examine a relatively simple toy example to build some intuition. Specifically, we focus on testing $(\epsilon,\delta)$-NAS  for the following convex \PA{} function:
\[
h(\bm{w}):=\max_{1 \leq i \leq n} \bm{w}^\top \bm{x}_i + a_i.
\]
Although this is a somewhat trivial problem, it allows us to introduce some very natural constructions and ideas that will later be proven generalizable to much more complex nonconvex scenarios.

Fix $\bm{w} \in \mathbb{R}^d$ and suppose that there is an $\epsilon$-stationary point $\bm{w}^*$ near $\bm{w}$. It is easy to see that the variational properties of $h$ at $\bm{w}$ are fully determined by the following active index set:
\[
\mathcal{I}(\bm{w}):=\left\{ i \in [n]: h(\bm{w})= \bm{w}^\top \bm{x}_i + a_i\right\}.
\]
To explore the neighborhood of the point of interest, $\bm{w}$, a natural idea is to consider the following $\delta_1$-approximate active index set:
\[
\mathcal{I}^{\delta_1}(\bm{w}):=\left\{ i \in [n]: h(\bm{w}) - \delta_1 \leq  \bm{w}^\top \bm{x}_i + a_i\right\}.
\]
If the parameter $\delta_1 > 0$ is chosen properly, we can reasonably hope to ``capture'' the active pattern of the unknown $\epsilon$-stationary point $\bm{w}^*$ using $\mathcal{I}^{\delta_1}(\bm{w})$; i.e., \(\mathcal{I}^{\delta_1}(\bm{w}) = \mathcal{I}(\bm{w}^*)\).
A natural approach to leverage this identified pattern is to characterize the subdivision of \(\mathbb{R}^d\) such that the function \(h\) exhibits the desired active pattern on that subdivision.

To this end, let us consider the following polyhedron parameterized by $\delta_1, \delta_2, \delta_3 > 0$:
\[
H^{\delta_1,\delta_2,\delta_3}:=
\left\{\bm{z}:
\begin{alignedat}{3}
\bm{z}^\top \bm{x}_i + a_i &= \bm{z}^\top \bm{x}_j + a_j , && \forall i,j \in \mathcal{I}^{\delta_1}(\bm{w}), && \qquad(a)\\
  \bm{z}^\top \bm{x}_i + a_i &\geq h(\bm{w}) - \delta_2,  && \forall i \in \mathcal{I}^{\delta_1}(\bm{w}), && \qquad(b)\\
			\bm{z}^\top \bm{x}_i + a_i&\leq h(\bm{w}) - \delta_3,\ \  && \forall i \in [n]\backslash\mathcal{I}^{\delta_1}(\bm{w}), && \qquad(c)\\
\end{alignedat}
\right\},
\]
where $\delta_2 \in (0,\delta_3)$. The set $H^{\delta_1,\delta_2,\delta_3}$ will serve as a subset of the aforementioned subdivision.
To understand the motivation behind the construction of $H^{\delta_1,\delta_2,\delta_3}$, a few comments are in order:
\begin{itemize}
\item The equalities in (a) are designed to ensure that the identified pieces in $\mathcal{I}^{\delta_1}(\bm{w})$ have the same objective value. These equality constraints, along with the inequalities in (b) and (c), may cause the system to become infeasible, but for now, let us assume they are consistent.
\item As we are only interested in the neighborhood of $\bm{w}$, and motivated by the Lipschitz continuity of $h$, the inequalities in (b) ensure that the common objective value of the identified pieces in (a) is not significantly lower than $h(\bm{w})$ by a parameter $\delta_2 > 0$.
	\item The inequalities in (c) artificially create a positive objective value gap $\delta_3-\delta_2 > 0$ between the pieces in $\mathcal{I}^{\delta_1}(\bm{w})$ and those in $[n]\backslash \mathcal{I}^{\delta_1}(\bm{w})$.
\end{itemize}
By construction, we observe that $\mathcal{I}(\bm{z}) = \mathcal{I}^{\delta_1}(\bm{w})$ for all $\bm{z} \in H^{\delta_1, \delta_2, \delta_3}$. Using \cite[Corollary D.4.3.2]{hiriart2004fundamentals}, if $\mathcal{I}^{\delta_1}(\bm{w}) = \mathcal{I}(\bm{w}^*)$ for a properly chosen $\delta_1$, then for any $\bm{z} \in H^{\delta_1, \delta_2, \delta_3}$, we have $\partial h(\bm{z}) = \partial h(\bm{w}^*)$. The next step is to check whether $\dist(\bm{w}, H^{\delta_1, \delta_2, \delta_3}) \leq \delta$. If this condition is satisfied, the projection of $\bm{w}$ onto $H^{\delta_1, \delta_2, \delta_3}$ provides an explicit certificate for $\bm{w}$ being an $(\epsilon, \delta)$-NAS point.

However, there are several technical challenges in making the above rough ideas for the simple convex \PA{} function \(h\) rigorous and applicable to the general \PA{} function \(h-g\) with multi-composite structures. Here, we highlight some of these questions:  
\begin{itemize}  
    \item How should the parameters \(\delta_1\), \(\delta_2\), and \(\delta_3\) be set? Note that a large \(\delta_1\) will make the polyhedron \(H^{\delta_1,\delta_2,\delta_3}\) empty, while a small \(\delta_1\) may fail to capture the active pattern of the desired \(\mathcal{I}(\bm{w}^*)\).  
    \item How can we address the complicated composite structure in the multi-composite representation of \(h\) and \(g\)? A trivial extension may encounter issues of computability and nonconvexity.  
    \item In the above discussion, we assume the existence of an \(\epsilon\)-stationary point \(\bm{w}^*\). How can we provably and efficiently certify the absence of \(\epsilon\)-stationary points in the neighborhood of \(\bm{w}\)?  
\end{itemize}

These questions, among others, will be resolved by a more dedicated construction presented in the remainder of this section.

\subsection{The Butterfly Net Algorithm}

We start by describe informally the intuition behind the the algorithm.

\paragraph{Intuition.} 
Let us informally sketch the idea behind our new robust testing algorithm for solving the \textsf{RST-DC} problem. Let convex \PA{} functions $h,g:\mathbb{R}^d\to \mathbb{R}$ and a point $\bm{w} \in \mathbb{R}^d$ be given. Fix precision $\epsilon \geq 0$ and $\delta > 0$. Our goal is to identify an unknown stationary point $\bm{w}^* \in \mathbb{R}^d$, whose existence is uncertain, near the point $\bm{w}$. To convey the intuition informally, imagine standing at the point $\bm{w}$, swinging a butterfly net in an attempt to capture the elusive $\bm{w}^*$. The length of the net is adjustable, set by a positive $\delta_0:=\delta$, and the shape of its opening is also adjustable, polyhedral, and defined by $\delta_0$. Let polyhedron $P^{\delta_0} \subseteq \mathbb{R}^d$ represent the opening, whose formal definition is presented in \Cref{sec:poly-opening}. We swing the net with $\delta_0$ by projecting $\bm{w}$ onto the polyhedron $P^{\delta_0}$ if the polyhedron is non-empty. Suppose the projection  is $\widehat{\bm{w}} \in \mathbb{R}^d$. If $\|\bm{w} - \widehat{\bm{w}}\| \leq \delta$ and $\bm{0} \in \partial (h-g)(\widehat{\bm{w}}) + \epsilon\mathbb{B}$, we succeed and return {\sf (True)} by confirming $\bm{w}$ as an $(\epsilon,\delta)$-\NAS{} point with an explicit certificate $\widehat{\bm{w}}$. If not, we halve $\delta_0$ to $\delta_0/2$ and repeat the capturing process. This loop will not continue indefinitely. We halt and assert {\sf (False)} if $\delta_0 \leq \delta_{\textnormal{sep}}(\bm{w})$, relying on the observation (see \Cref{prop:termination}) that $\bm{w} \in P^{\delta_0}$ for all $0 <\delta_0 \leq \delta_{\textnormal{sep}}(\bm{w})$. The crux of the correctness of our algorithm  (see \Cref{thm:robust}) is that if $0 < \|\bm{w} - \bm{w}^*\| \leq \delta_0 \leq 2\delta_{\textnormal{sep}}(\bm{w}^*)$, then we are guaranteed to have ``captured'' the point $\bm{w}^*$, with $\bm{w}^* \in P^{\delta_0}$. The projection $\widehat{\bm{w}}$ also satisfies $\partial (h-g)(\widehat{\bm{w}}) = \partial (h-g)(\bm{w}^*)$. This enables us to verify the $(\epsilon,\delta)$-\NAS{} status of $\bm{w}$ without precisely pinpointing the original target $\bm{w}^*$. Ultimately, the correctness is demonstrated through contraposition.

\subsubsection{A Family of Polyhedra}\label{sec:poly-opening}

Fix $\bm{w} \in \mathbb{R}^d$ and $\delta > 0$.
For any $1\leq k \leq n$ and index $(j_k,\dots,i_n,j_n) \in [J_k^h]\times \prod_{p=k+1}^n ([I_p^h]\times [J_p^h])$, we consider an index set $\mathcal{I}_{j_k,\dots,i_n,j_n}^{\bm{w},\delta,h} \subseteq [I_k^h]$:
\[
\mathcal{I}_{j_k,\dots,i_n,j_n}^{\bm{w},\delta,h}:=\left\{i_k\in[I_k^h]:v_{i_k,j_k,\dots,i_n,j_n}(\bm{w}) \geq v_{j_k,\dots,i_n,j_n}(\bm{w}) - 3R\delta \right\}.
\]
We define a potentially \emph{nonconvex} set $V_{j_k,\dots,i_n,j_n}^{\bm{w},\delta,h}$ in $\mathbb{R}^d$ as
\[
V_{j_k,\dots,i_n,j_n}^{\bm{w},\delta,h}:=
\left\{\bm{z}:
\begin{alignedat}{2}
v_{i_k,j_k,\dots,i_n,j_n}(\bm{z})&=v_{j_k,\dots,i_n,j_n}(\bm{z}), && \forall i_k \in \mathcal{I}^{\bm{w},\delta,h}_{j_k,\dots,i_n,j_n} \\
  v_{j_k,\dots,i_n,j_n}(\bm{z}) &\geq v_{j_k,\dots,i_n,j_n}(\bm{w}) - 2R\delta,  && \\
			v_{i_k,j_k,\dots,i_n,j_n}(\bm{z})&\leq v_{j_k,\dots,i_n,j_n}(\bm{w}) - 4R\delta,\ \  && \forall i_k \in [I_k^h]\backslash \mathcal{I}^{\bm{w},\delta,h}_{j_k,\dots,i_n,j_n}\\
\end{alignedat}
\right\},
\]
where we use the superscript $\delta,h,\bm{w}$ to emphasize the dependence in the definition. While the value function $v_{j_k, \dots, i_n, j_n}$ is convex, its upper level set is not necessarily convex. Hence, the set $V_{j_k, \dots, i_n, j_n}^{\bm{w}, \delta, h}$ may be nonconvex. The intersection of all these sets across indices is denoted by $V^{\bm{w}, \delta, h}$ and defined as follows:
\[
	V^{\bm{w},\delta,h}:= \bigcap \left\{ V_{j_k,\dots,i_n,j_n}^{\bm{w},\delta,h}: 1\leq k \leq n,(j_k,\dots,i_n,j_n) \in {\textstyle [J_k^h]\times \prod_{p=k+1}^n ([I_p^h]\times [J_p^h])}\right\}.
\]

Similarly, for any $1\leq k \leq n$ and index $(j_k,\dots,i_n,j_n) \in [J_k^g]\times \prod_{p=k+1}^n ([I_p^g]\times [J_p^g])$, we define an index set $\mathcal{I}_{j_k,\dots,i_n,j_n}^{\bm{w},\delta,g} \subseteq [J_k^g]$ as:
\[
\mathcal{I}_{j_k,\dots,i_n,j_n}^{\bm{w},\delta,g}:=\Big\{i_k\in[J_k^g]:u_{i_k,j_k,\dots,i_n,j_n}(\bm{w}) \geq u_{j_k,\dots,i_n,j_n}(\bm{w}) - 3R\delta \Big\}.
\]
A potentially \emph{nonconvex} set 
$U_{j_k,\dots,i_n,j_n}^{\bm{w},\delta,g}$ in $\mathbb{R}^d$ can be defined as
\[
U_{j_k,\dots,i_n,j_n}^{\bm{w},\delta,g}:=
\left\{\bm{z}:
\begin{alignedat}{2}
u_{i_k,j_k,\dots,i_n,j_n}(\bm{z})&=u_{j_k,\dots,i_n,j_n}(\bm{z}), && \forall i_k \in \mathcal{I}_{j_k,\dots,i_n,j_n}^{\bm{w},\delta,g}\\
  u_{j_k,\dots,i_n,j_n}(\bm{z}) &\geq u_{j_k,\dots,i_n,j_n}(\bm{w}) - 2R\delta,  && \\
			u_{i_k,j_k,\dots,i_n,j_n}(\bm{z})&\leq v_{j_k,\dots,i_n,j_n}(\bm{w}) - 4R\delta,\ \  && \forall i_k \in [J_k^g]\backslash \mathcal{I}_{j_k,\dots,i_n,j_n}^{\bm{w},\delta,g} \\
\end{alignedat}
\right\}.
\]
The following set is defined accordingly:
\[
U^{\bm{w},\delta,g}:= \bigcap \left\{ U_{j_k,\dots,i_n,j_n}^{\bm{w},\delta,g}: 1\leq k \leq n, (j_k,\dots,i_n,j_n) \in {\textstyle [J_k^g]\times \prod_{p=k+1}^n ([I_p^g]\times [J_p^g])}\right\}.
\]
The key object (i.e.~``the net'') of our algorithm is the following set in $\mathbb{R}^d$:
\begin{equation}\label{eq:defofP}
	P^{\bm{w},\delta,h,g}:=V^{\bm{w},\delta,h}\cap U^{\bm{w},\delta,g}.
\end{equation}

\subsubsection{The Algorithm} 
Our main algorithmic results in this paper are the robust testing algorithm presented in \Cref{alg:ls} and a ``rounding'' subprocedure in \Cref{alg:rounding}. 
The idea of the scheme roughly follows the steps we informally described at the beginning of this section.

\begin{algorithm}[h]  
	\caption{Robust Stationarity Test for (\sf{RST-DC})}  
	\label{alg:ls}
	\begin{algorithmic}[1]  
		\Procedure{RST}{$\bm{w}, \epsilon, \delta,h,g$} %
				\State $k\leftarrow 0$;
		\Repeat 
		\State $\widehat{\bm{w}}\leftarrow\textsc{Rnd}(\bm{w}, 2^{-k}\delta, h,g)$; \Comment{\Cref{alg:rounding}}
		\If{$\widehat{\bm{w}}$ is not the flag {\sf (Infeasible)}, $\dist(\bm{0}, \partial (h-g)(\widehat{\bm{w}})) \leq \epsilon$, and $\|\widehat{\bm{w}} - \bm{w}\| \leq \delta$}\label{item:alg1-oracle}
		\State \Return{{\sf (True)} and $\widehat{\bm{w}}$;} \label{item:alg1-true}
		\EndIf
		\State $k \leftarrow k+1$;
		\Until{$2^{-k-2}\delta \leq \delta_{\sf sep}(\bm{w})$}
		\State \Return {\sf (False)};
		\EndProcedure
	\end{algorithmic}  
\end{algorithm}

\begin{algorithm}[h]  
	\caption{Rounding}  
	\label{alg:rounding}
	\begin{algorithmic}[1]  
		\Procedure{Rnd}{$\bm{w},\delta,h,g$} %
		\If{$P^{\bm{w},\delta,h,g}=\emptyset$}
		\State \Return{a flag {\sf (Infeasible)};}
		\EndIf
		\State Solve the following convex $\ell_2$-projection problem:%
		\[
			\widehat{\bm{w}} := \argmin_{ \bm{z} \in P^{\bm{w},\delta,h,g}} \ \frac{1}{2}\| \bm{z} - \bm{w} \|^2_2 \tag{see \nref{eq:defofP} for $P^{\bm{w},\delta,h,g}$}
		\]
		\State \Return{$\widehat{\bm{w}}$;}
		\EndProcedure
	\end{algorithmic}  
\end{algorithm} 

\paragraph{Exact Testing Oracle.} In line~\ref{item:alg1-oracle} of \Cref{alg:ls}, for a candidate certificate $\widehat{\bm{w}}$, we need to verify whether $\dist(\bm{0}, \partial (h-g)(\widehat{\bm{w}})) \leq \epsilon$ by calling an exact testing oracle, as described in \Cref{def:oracle}.
We emphasize that while the assumed oracle has the capability to solve an $\cNP$-hard problem (as shown in \Cref{thm:hard-dc}), we do not exploit its computational power in other parts of the algorithm. Indeed, as will be highlighted shortly in \Cref{rmk:other-stationarity}, if the user is willing to accept a weaker notion of stationarity, such as DC-criticality, or possesses additional information that ensures the validity of the sum rule in \Cref{sec:compat}, then the exact testing oracle can be implemented in polynomial time. The main purpose of introducing such an ideal oracle is to decouple the challenges of exact and robust testing and to demonstrate that these two tasks can, to some extent, be addressed independently.

\paragraph{Projection onto $P^{\bm{w},\delta,h,g}$.}\label{sec:representation-P}
In our algorithm, we need to find the projection of the candidate point $\bm{w}$ onto the set $P^{\bm{w},\delta,h,g}$, which is defined as the intersection of many potentially \emph{nonconvex} set. Here, we demonstrate that the sets $V^{\bm{w},\delta,h}, U^{\bm{w},\delta,g}$, and $P^{\bm{w},\delta,h,g}$ are all convex.%

Without of loss of generality, we focus on showing the convexity of $V^{\bm{w},\delta,h}$. 
Note that, for any $k \in [n]$ and any fixed index $(j_k,\dots,i_n,j_n) \in [J_k^h]\times \prod_{p=k+1}^n ([I_p^h]\times [J_p^h])$, the set $V_{j_k,\dots,i_n,j_n}^{\bm{w},\delta,h}$ can be rewritten as the following redundant form:
\[
V_{j_k,\dots,i_n,j_n}^{\bm{w},\delta,h}=
\left\{\bm{z}:
\begin{alignedat}{2}
v_{i_k,j_k,\dots,i_n,j_n}(\bm{z})&=v_{i_k',j_k,\dots,i_n,j_n}(\bm{z}), && \forall i_k,i_k' \in \mathcal{I}^{\bm{w},\delta,h}_{j_k,\dots,i_n,j_n} \\
  v_{i_k,j_k,\dots,i_n,j_n}(\bm{z}) &\geq v_{j_k,\dots,i_n,j_n}(\bm{w}) - 2R\delta,  && \forall i_k \in \mathcal{I}^{\bm{w},\delta,h}_{j_k,\dots,i_n,j_n}\\
			v_{i_k,j_k,\dots,i_n,j_n}(\bm{z})&\leq v_{j_k,\dots,i_n,j_n}(\bm{w}) - 4R\delta,\ \  && \forall i_k \in [I_k^h]\backslash \mathcal{I}^{\bm{w},\delta,h}_{j_k,\dots,i_n,j_n}\\
\end{alignedat}
\right\}.
\]
We observe that $v_{j_k,\dots,i_n,j_n}(\bm{z}) = v_{i_k,j_k,\dots,i_n,j_n}(\bm{z})$ for any $\bm{z} \in V_{j_k,\dots,i_n,j_n}^{\bm{w},\delta,h}$ and any $i_k \in \mathcal{I}^{\bm{w},\delta,h}_{j_k,\dots,i_n,j_n}$.

When $k=1$, for any fixed index $(j_1,\dots,i_n,j_n) \in [J_1^h]\times \prod_{p=2}^n ([I_p^h]\times [J_p^h])$, the convex function $v_{i_1,j_1,\dots,i_n,j_n}(\bm{z}) = \bm{z}^\top \bm{x}_{i_1,j_1,\dots,i_n,j_n} + b_{i_1,j_1,\dots,i_n,j_n}$ is affine on $\mathbb{R}^d$ for any $i_1 \in [I_1^h]$, so that the set $V_{j_1,\dots,i_n,j_n}^{\bm{w},\delta,h}$ is the intersection of halfspaces and hyperplanes, and thus convex.
Moreover, the convex function $v_{j_1,\dots,i_n,j_n}$ is actually affine on the convex set $V_{j_1,\dots,i_n,j_n}^{\bm{w},\delta,h}$.
 In what follows, we demonstrate that similar claims hold for any $1 \leq k \leq n$ by induction.
 
 	When $k=2$, for any fixed index $(j_2,\dots,i_n,j_n) \in [J_2^h]\times \prod_{p=3}^n ([I_p^h]\times [J_p^h])$, the convex function $v_{i_2, j_2,\dots,i_n,j_n} = \sum_{j_1 \in [J^h_1]} v_{j_1,\dots,i_n,j_n}$ is affine on the convex set
\begin{equation}\label{eq:represent-P-k=2}
\bigcap_{j_{1} \in [J_{1}^h]} V_{j_{1},\dots,i_n,j_n}^{\bm{w},\delta,h}.
\end{equation}
Therefore, the intersection of $V_{j_2,\dots,i_n,j_n}^{\bm{w},\delta,h}$ and the above set in \eqref{eq:represent-P-k=2}
 is the intersection of halfspaces and hyperplanes, and is thus convex.

 In the same vein, it is easily seen that the set
\[
\bigg(\bigcap \left\{ V_{j_k,\dots,i_n,j_n}^{\bm{w},\delta,h}: 1\leq k < k',(j_k,\dots,i_n,j_n) \in {\textstyle [J_k^h]\times \prod_{p=k+1}^n ([I_p^h]\times [J_p^h])}\right\} \bigg)\cap V_{j_{k'},\dots,i_n,j_n}^{\bm{w},\delta,h} 
\]
is convex for any $2\leq k' \leq n$, so that the sets $V^{\bm{w},\delta,h}, U^{\bm{w},\delta,g}$, and $P^{\bm{w},\delta,h,g}$ are all convex.
Using standard LP reformulation techniques, it can be shown that the convex polyhedron $P^{\bm{w},\delta,h,g}$ is representable by a polynomial size LP. Consequently, the projection of the point $\bm{w}$ onto $P^{\bm{w},\delta,h,g}$ can be computed in polynomial time.

\subsubsection{Correctness and Efficiency}
In the following result, we show that the new algorithm will correctly answer the {\sf{RST-DC}} question in oracle-polynomial time.

\begin{Theorem}\label{thm:robust-test}
	Let two $n$-\MC{} functions $h,g:\mathbb{R}^d\to \mathbb{R}$, a point $\bm{w}\in\mathbb{Q}^d$, and rational precisions $\epsilon \geq 0, \delta > 0$ be given. Suppose that we have an implementation of the stationarity test oracle in \Cref{def:oracle} for function $h-g$. \Cref{alg:ls} will terminate in oracle-polynomial time and provide a correct answer to \textnormal{Problem} \sf{RST-DC}.
\end{Theorem}

\begin{Remark}[Other solution concepts]\label{rmk:other-stationarity}
While the stationarity mentioned in \Cref{def:oracle} is stated using Clarke subdifferential for simplicity, 
the conclusion of \Cref{thm:robust-test} and the applicability of \Cref{alg:ls} actually hold for other exact testing oracle beyond Clarke stationarity (see the statement of \Cref{thm:robust}).
 Indeed, our robust testing approach works for any locally defined solution notions (e.g., local minimum, Fr\'echet and limiting subdifferential-based stationarities; see \cite[Chapter 6]{cui2021modern}) and \DC{}-decomposition dependent solution notion (e.g., DC-criticality), as long as an implementable exact $\epsilon$-stationarity testing oracle is provided.
\end{Remark}

For any algorithm that computes stationary points asymptotically, an efficient stopping rule can be constructed using \Cref{alg:ls}, which results in a new algorithm that terminates in finite time. 
The following immediate corollary of \Cref{thm:robust-test} highlights this.

\begin{Corollary}[Stopping rule]\label{coro:stop}
	Let $n$-\MC{} functions $h,g:\mathbb{R}^d\to \mathbb{R}$ be given. Fix any precisions $0< \epsilon, \delta \in \mathbb{Q}.$ Suppose that a polynomial-time implementation of some (generalized) $\epsilon$-stationarity test oracle is provided. For any sequence $\bm{w}_t \to \bm{w}^* \in (\partial f)^{-1}(\epsilon\mathbb{B})$, there exists a finite $T \in \mathbb{N}$ such that for all $t \geq T$, 
the outputs of calling procedure 
\[
 \textsc{RST}(\bm{w}_t, \epsilon, \delta, h,g)
 \]
 defined by \Cref{alg:ls}
are {\sf (True)} and a certificate $\widehat{\bm{w}}_t \in \mathbb{R}^d$ such that 
\[
\bm{0} \in \partial (h-g)(\widehat{\bm{w}}_t)+\epsilon\mathbb{B},\qquad\text{and}\qquad\|\bm{w}_t-\widehat{\bm{w}}_t\|\leq \delta.
\] 
Moreover, for any $t\in\mathbb{N}$, calling procedure $\textsc{RST}(\bm{w}_t, \epsilon, \delta, h,g)$ terminates in polynomial time.
\end{Corollary}

\subsection{Proof of \Cref{thm:robust-test}}

We first remark on the $R$-Lipschitz continuity of value functions of $h$ and $g$.
\begin{Lemma}[Lipschitz continuity]
	For any $k \in [n]$, value functions $v_{i_k, j_k,\dots,i_n,j_n}, u_{i_k, j_k,\dots,i_n,j_n}$, $v_{j_{k},\dots,i_n,j_n}$, and $u_{j_{k},\dots,i_n,j_n}$ are all $R$-Lipschitz. %
\end{Lemma}
\begin{proof}
	The proof is elementary and omitted for brevity.
\end{proof}

The correctness of \Cref{alg:ls} hinges on the following key technical lemma.

\begin{Lemma}\label{thm:robust}
Fix a point $\bm{w}^*\in\mathbb{R}^d$. For any $\bm{w}\in\mathbb{R}^d$ and $\delta > 0$ such that
\[
\|\bm{w} - \bm{w}^*\| \leq  \delta \leq 2\delta_{\textnormal{sep}}(\bm{w}^*),
\] let $\widehat{\bm{w}} $ be the output of calling procedure 
$
 \textsc{Rnd}(\bm{w}, \delta, h,g)
$
defined by \Cref{alg:rounding}.
Then, the output $\widehat{\bm{w}}$ is not the flag {\sf (Infeasible)} and satisfies $\|\widehat{\bm{w}} - \bm{w}\| \leq \delta$. Moreover, for any sufficiently small $\tau > 0$ and any point $\bm{h} \in \mathbb{R}^d$ sufficiently close to $\bm{0}$, the difference quotient functions of $h$ and $g$ satisfy
\[
\Delta_\tau h(\widehat{\bm{w}}+\bm{h})(\bm{d}) = \Delta_\tau h(\bm{w}^{*}+\bm{h})(\bm{d}), \qquad \Delta_\tau g(\widehat{\bm{w}}+\bm{h})(\bm{d}) = \Delta_\tau g(\bm{w}^{*}+\bm{h})(\bm{d})%
\]
and thus,
\[
\Delta_\tau (h-g)(\widehat{\bm{w}}+\bm{h})(\bm{d}) = \Delta_\tau (h-g)(\bm{w}^{*}+\bm{h})(\bm{d}),%
\]
for any direction $\bm{d} \in \mathbb{R}^d$. Consequently, by \Cref{fct:clarke-dd}, we have
$
\partial (h-g)(\widehat{\bm{w}}) = \partial (h-g)(\bm{w}^*).%
$
\end{Lemma}

\begin{Remark}
We highlight two key points regarding \Cref{thm:robust} as follows:
\begin{itemize}
\item The reference point $\bm{w}^*$ is almost arbitrary and need not be an ($\epsilon$-)stationary point.
\item The information about the point $\bm{w}^*$ is not utilized by \Cref{alg:rounding}. However, as the conclusion of \Cref{thm:robust} indicates, the output of \Cref{alg:rounding} (a point $\widehat{\bm{w}}$) is closely tied to the variational properties of $h-g$ evaluated at the (unknown) point $\bm{w}^*$.
\end{itemize}
\end{Remark}

\begin{proof}
The proof falls naturally into two parts.

\textbf{Step 1.}
We first prove that under the hypothesis, $\bm{w}^* \in P^{\bm{w},\delta,h,g}.$
Note that $0 < \delta \leq 2\delta_{\sf sep}(\bm{w}^*)$ and $\| \bm{w} - \bm{w}^* \| \leq \delta$.
Fix any $k \in [n]$ and any index $(j_k,\dots,i_n,j_n) \in [J_k^h]\times \prod_{p=k+1}^n ([I_p^h]\times [J_p^h])$. For any $i_k \in [I_k^h]$ such that $v_{i_k,j_k,\dots,i_n,j_n}(\bm{w}^*)< v_{j_k,\dots,i_n,j_n}(\bm{w}^*)$, we compute
\begin{align*}
v_{i_k,j_k,\dots,i_n,j_n}(\bm{w})
&= v_{i_k,j_k,\dots,i_n,j_n}(\bm{w}^*) - (v_{i_k,j_k,\dots,i_n,j_n}(\bm{w}^*) - v_{i_k,j_k,\dots,i_n,j_n}(\bm{w}))\\
&\leq v_{i_k,j_k,\dots,i_n,j_n}(\bm{w}^*) + R\cdot \|\bm{w}^* - \bm{w}\| \tag{$R$-Lipschitz of $v_{i_k,j_k,\dots,i_n,j_n}$} \\
&\leq v_{j_k,\dots,i_n,j_n}(\bm{w}^*)- 12R\delta_{\sf sep}(\bm{w}^*) + R\delta \tag{definition of $\delta_{\sf sep}(\bm{w}^*)$} \\
 &\leq v_{j_k,\dots,i_n,j_n}(\bm{w}^*)- 5R\delta \tag{$\delta \leq 2\delta_{\sf sep}(\bm{w}^*)$}\\
 &\leq  v_{j_k,\dots,i_n,j_n}(\bm{w})- 4R\delta. \tag{$R$-Lipschitz of $v_{j_k,\dots,i_n,j_n}$}
\end{align*}
Therefore, we conclude that 
\begin{equation}\label{eq:robust-prf-wstar-leq}
\left\{i_k \in [I_k^h]:v_{i_k,j_k,\dots,i_n,j_n}(\bm{w}^*)< v_{j_k,\dots,i_n,j_n}(\bm{w}^*)\right\} \subseteq [I_k^h]\backslash \mathcal{I}^{\bm{w},\delta,h}_{j_k,\dots,i_n,j_n}. %
\end{equation}
Moreover, %
for any $i_k \in [I_k^h]$ such that $v_{i_k,j_k,\dots,i_n,j_n}(\bm{w}^*)= v_{j_k,\dots,i_n,j_n}(\bm{w}^*)$, it holds
\begin{align*}
v_{i_k,j_k,\dots,i_n,j_n}(\bm{w})
&= v_{i_k,j_k,\dots,i_n,j_n}(\bm{w}^*) - (v_{i_k,j_k,\dots,i_n,j_n}(\bm{w}^*)-v_{i_k,j_k,\dots,i_n,j_n}(\bm{w})) \\
&\geq v_{i_k,j_k,\dots,i_n,j_n}(\bm{w}^*) - R\cdot \|\bm{w}^* - \bm{w}\| \tag{$R$-Lipschitz of $v_{i_k,j_k,\dots,i_n,j_n}$} \\
&\geq v_{j_k,\dots,i_n,j_n}(\bm{w}^*) - R\delta \\
&\geq v_{j_k,\dots,i_n,j_n}(\bm{w}) - 2R\delta. \tag{$R$-Lipschitz of $v_{j_k,\dots,i_n,j_n}$}
\end{align*}
We obtain that 
\begin{equation}\label{eq:robust-prf-wstar-eq}
\left\{i_k \in [I_k^h]:v_{i_k,j_k,\dots,i_n,j_n}(\bm{w}^*) = v_{j_k,\dots,i_n,j_n}(\bm{w}^*)\right\} \subseteq \mathcal{I}^{\bm{w},\delta,h}_{j_k,\dots,i_n,j_n}.%
\end{equation}
As index sets
$\mathcal{I}^{\bm{w},\delta,h}_{j_k,\dots,i_n,j_n}$ and $[I_k^h]\backslash\mathcal{I}^{\bm{w},\delta,h}_{j_k,\dots,i_n,j_n}$ are disjoint by definition, %
we have that above two set inclusions in \nref{eq:robust-prf-wstar-leq,eq:robust-prf-wstar-eq} are actually equalities, so that
\begin{equation}\label{eq:prf-rounding-correct-Ih}
\left\{i_k \in [I_k^h]:v_{i_k,j_k,\dots,i_n,j_n}(\bm{w}^*) = v_{j_k,\dots,i_n,j_n}(\bm{w}^*)\right\} = \mathcal{I}^{\bm{w},\delta,h}_{j_k,\dots,i_n,j_n}.%
\end{equation}
Besides, for any $i_k\in \mathcal{I}^{\bm{w},\delta,h}_{j_k,\dots,i_n,j_n}$, we have
\begin{equation}\label{eq:robust-prf-wstar-in-P}
v_{i_k,j_k,\dots,i_n,j_n}(\bm{w}^*) \overset{\eqref{eq:prf-rounding-correct-Ih}}{=} v_{j_k,\dots,i_n,j_n}(\bm{w}^*) \geq v_{j_k,\dots,i_n,j_n}(\bm{w}) - R\delta > v_{j_k,\dots,i_n,j_n}(\bm{w}) - 2R\delta. 
\end{equation}
For any $i_k\in [I_k^h]\backslash\mathcal{I}^{\bm{w},\delta,h}_{j_k,\dots,i_n,j_n}$, it holds
\begin{align}
	v_{i_k,j_k,\dots,i_n,j_n}(\bm{w}^*) &\leq  v_{j_k,\dots,i_n,j_n}(\bm{w}^*) - 6R\delta \tag{definition of $\delta_{\sf sep}(\bm{w}^*)$} \\
	 &\leq v_{j_k,\dots,i_n,j_n}(\bm{w}) - 5R\delta < v_{j_k,\dots,i_n,j_n}(\bm{w}) - 4R\delta. \label{eq:robust-prf-wstar-in-P2}
\end{align}
Combining \nref{eq:prf-rounding-correct-Ih,eq:robust-prf-wstar-in-P,eq:robust-prf-wstar-in-P2}, we have $\bm{w}^* \in V_{j_k,\dots,i_n,j_n}^{\bm{w},\delta,h}$. As $ k \in [n]$ and index $(j_k,\dots,i_n,j_n)$ are arbitrary, it follows immediately that $\bm{w}^* \in V^{\bm{w},\delta,h}$.

The same reasoning applies to the sets $\mathcal{I}^{\bm{w},\delta,g}_{j_k,\dots,i_n,j_n}$ and $[I_k^g]\backslash\mathcal{I}^{\bm{w},\delta,g}_{j_k,\dots,i_n,j_n}$ for any $k \in [n]$ and index $(j_k,\dots,i_n,j_n) \in [J_k^g]\times \prod_{p=k+1}^n ([I_p^g]\times [J_p^g])$.
Thus, it follows that
\begin{equation}\label{eq:prf-rounding-correct-Ig}
\Big\{i_k \in [I_k^g]:v_{i_k,j_k,\dots,i_n,j_n}(\bm{w}^*) = u_{j_k,\dots,i_n,j_n}(\bm{w}^*)\Big\} = \mathcal{I}^{\bm{w},\delta,g}_{j_k,\dots,i_n,j_n}.%
\end{equation}
Consequently, we obtain
\[
\bm{w}^* \in V^{\bm{w},\delta,h}\cap U^{\bm{w},\delta,g} = P^{\bm{w}, \delta,h,g} \neq \emptyset,
\]
 hence that the output $\widehat{\bm{w}}$ will not be the flag {\sf (Infeasible)}.

\textbf{Step 2.} We next show the equivalence related to the difference quotient functions of $h,g,$ and $h-g$.
 By the optimality of projection onto $P^{\bm{w}, \delta,h,g}$, we have $\|\widehat{\bm{w}} - \bm{w}\| \leq \|\bm{w}^* - \bm{w}\| \leq \delta$.
For any $k \in [n]$ and any $(j_k,\dots,i_n,j_n) \in [J_k^h]\times \prod_{p=k+1}^n ([I_p^h]\times [J_p^h])$, $i_k\in \mathcal{I}^{\bm{w},\delta,h}_{j_k,\dots,i_n,j_n}$, and $i_k'\in [I_k^h]\backslash\mathcal{I}^{\bm{w},\delta,h}_{j_k,\dots,i_n,j_n}$, from the feasibility $\widehat{\bm{w}}\in P^{\bm{w},\delta,h,g}$,  it follows that 
\[
v_{i_k,j_k,\dots,i_n,j_n}(\widehat{\bm{w}}) \geq v_{j_k,\dots,i_n,j_n}(\bm{w}) - 2R\delta > v_{j_k,\dots,i_n,j_n}(\bm{w}) - 4R\delta \geq v_{i_k',j_k,\dots,i_n,j_n}(\widehat{\bm{w}}).
\] 
Hence
\begin{equation}\label{eq:prf-rounding-correct-Ih-out}
\left\{i_k \in [I_k^h]:v_{i_k,j_k,\dots,i_n,j_n}(\widehat{\bm{w}}) = v_{j_k,\dots,i_n,j_n}(\widehat{\bm{w}})\right\} = \mathcal{I}^{\bm{w},\delta,h}_{j_k,\dots,i_n,j_n}.%
\end{equation}
In a similar vein, for any $k \in [n]$ and index $(j_k,\dots,i_n,j_n) \in [J_k^g]\times \prod_{p=k+1}^n ([I_p^g]\times [J_p^g])$, one can readily show that
\begin{equation}\label{eq:prf-rounding-correct-Ig-out}
\Big\{i_k \in [I_k^g]:u_{i_k,j_k,\dots,i_n,j_n}(\widehat{\bm{w}}) = u_{j_k,\dots,i_n,j_n}(\widehat{\bm{w}})\Big\} = \mathcal{I}^{\bm{w},\delta,g}_{j_k,\dots,i_n,j_n}.%
\end{equation}
It follows from \nref{eq:prf-rounding-correct-Ih,eq:prf-rounding-correct-Ig,eq:prf-rounding-correct-Ih-out,eq:prf-rounding-correct-Ig-out} that, for any direction $\bm{d}\in\mathbb{R}^d$, we have
\begin{subequations}
\begin{align}
h'(\widehat{\bm{w}}; \bm{d})&=\sum_{1\leq j_n \leq J_n^h} \max_{i_{n} \in \mathcal{I}^{\bm{w},\delta,h}_{j_n}} \cdots \sum_{1\leq j_1 \leq J_1^h} \max_{i_1 \in \mathcal{I}^{\bm{w},\delta,h}_{j_1,i_2,\dots,i_n,j_n}} \bm{d}^\top \bm{x}_{i_1,j_1,\dots,i_n,j_n} = h'(\bm{w}^*; \bm{d}), \label{eq:proof-rounding-hg-dd-h}\\
g'(\widehat{\bm{w}}; \bm{d})&=\sum_{1\leq j_n \leq J_n^g} \max_{i_{n} \in \mathcal{I}^{\bm{w},\delta,g}_{j_n}} \cdots \sum_{1\leq j_1 \leq J_1^g} \max_{i_1 \in \mathcal{I}^{\bm{w},\delta,g}_{j_1,i_2,\dots,i_n,j_n}} \bm{d}^\top \bm{x}_{i_1,j_1,\dots,i_n,j_n} = g'(\bm{w}^*; \bm{d}). \label{eq:proof-rounding-hg-dd-g}
\end{align}
\end{subequations}
Therefore, for any $\bm{v} \in \mathbb{R}^d$ sufficiently near $\bm{0}$, we have
\begin{align*}
h(\widehat{\bm{w}}+\bm{v}) - h(\widehat{\bm{w}}) &\overset{\textnormal{(\Cref{fct:pa-dd}})}{=} h'(\widehat{\bm{w}}; \bm{v}) \overset{\eqref{eq:proof-rounding-hg-dd-h}}{=} h'(\bm{w}^*; \bm{v}) \overset{\textnormal{(\Cref{fct:pa-dd}})}{=} h(\bm{w}^*+\bm{v}) - h(\bm{w}^*), \\
g(\widehat{\bm{w}}+\bm{v}) - g(\widehat{\bm{w}}) &\overset{\textnormal{(\Cref{fct:pa-dd}})}{=} g'(\widehat{\bm{w}}; \bm{v}) \overset{\eqref{eq:proof-rounding-hg-dd-g}}{=} g'(\bm{w}^*; \bm{v}) \overset{\textnormal{(\Cref{fct:pa-dd}})}{=} g(\bm{w}^*+\bm{v}) - g(\bm{w}^*).
\end{align*}
As $\tau > 0$ and $\|\bm{h}\|$ are both sufficiently small, we conclude that 
\[
\begin{aligned}
	\Delta_\tau h(\widehat{\bm{w}}+\bm{h})(\bm{d}) &= 
\frac{\big(h(\widehat{\bm{w}}+\bm{h} + \tau \bm{d}) - h(\widehat{\bm{w}})\big) - \big(h(\widehat{\bm{w}}+\bm{h})-h(\widehat{\bm{w}})\big) }{\tau} \\
&= \frac{\big(h(\bm{w}^*+\bm{h} + \tau \bm{d}) - h(\bm{w}^*)\big) - \big(h(\bm{w}^*+\bm{h})-h(\bm{w}^*)\big) }{\tau} = \Delta_\tau h(\bm{w}^*+\bm{h})(\bm{d}).
\end{aligned}
\]
Similar arguments apply to $\Delta_\tau g$ and $\Delta_\tau (h-g)$.
\end{proof}

To analyze the efficiency of \Cref{alg:ls}, we first observe the following finite-termination property of searching the parameter $\delta > 0$.

\begin{Corollary}[Finite termination]\label{prop:termination}
Fix $\bm{w}\in\mathbb{R}^d$. For any $0 < \delta \leq 2\delta_{\textnormal{sep}}(\bm{w})$, let $\widehat{\bm{w}} $ be the output of calling procedure 
\[
 \textsc{Rnd}(\bm{w}, \delta, h, g)
 \]
defined by \Cref{alg:rounding}. Then we have that $\widehat{\bm{w}}$ is not the flag {\sf (Infeasible)} and satisfies $\widehat{\bm{w}}=\bm{w}$.	
\end{Corollary}
\begin{proof}
Let $\bm{w}^*:=\bm{w}$.
From the first step in the proof of \Cref{thm:robust}, we know that 
 $\bm{w} \in P^{\bm{w},\delta,h,g}$ whenever $0 < \delta \leq 2\delta_{\textnormal{sep}}(\bm{w})$.
	Therefore, the convex projection problem in \Cref{alg:rounding} is feasible with the output $\widehat{\bm{w}}=\bm{w}$ as the unique optimal solution.
\end{proof}

Now, we are ready for the proof of \Cref{thm:robust-test}.

\begin{proof}[Proof of \Cref{thm:robust-test}]
The proof will be divided into two steps.

\textbf{Step 1.}
	We first prove the correctness of \Cref{alg:ls} for solving problem {\sf{RST-DC}}. Here, we do not need the rationality of inputs. Note that \Cref{alg:ls} never makes false positive error, in the sense that when \Cref{alg:ls} returns {\sf (True)} and gives an answer to {\sf{RST-DC}}, it is always correct. We prove the {\sf (False)} case by contrapositive.
	Suppose that there exists a vector $\bm{w}^*\in\mathbb{R}^d$ such that $\dist\big(\bm{0}, \partial f(\bm{w}^*)\big) \leq \epsilon$ and $\|\bm{w} - \bm{w}^*\| \leq \min\{ \delta, \delta_{\sf sep}(\bm{w}^*)\}.$ We consider the following three cases.
	\begin{itemize}
		\item $0 < \delta \leq 2\delta_{\sf sep}(\bm{w}^*)$. Let $\widehat{\bm{w}}$ be the output of calling
		\[
		\textsc{Rnd}(\bm{w}, \delta, h,g)
		\]
		for the first time ($k=0$).
		Note that by the hypothesis, we have $\|\bm{w} - \bm{w}^*\| \leq \min\{ \delta, \delta_{\sf sep}(\bm{w}^*)\} \leq \delta \leq 2 \delta_{\sf sep}(\bm{w}^*)$.
		By \Cref{thm:robust}, we know that the output $\widehat{\bm{w}}$ is not the flag {\sf (Infeasible)}. Besides, we have  $\|\widehat{\bm{w}} - \bm{w}\| \leq \delta$ and
	$
	\dist\big(\bm{0}, \partial f(\widehat{\bm{w}}) \big) = \dist\big(\bm{0},\partial f(\bm{w}^*)\big) \leq \epsilon,
	$ so that \Cref{alg:ls} will return {\sf (True)} and $\widehat{\bm{w}}$ in line~\ref{item:alg1-true}.
	\item $\delta > 2\delta_{\sf sep}(\bm{w}^*)$ and $\delta_{\sf sep}(\bm{w}^*) \geq \delta_{\sf sep}(\bm{w})$. There exists some $k \in \mathbb{N}$ such that $\delta_{\sf sep}(\bm{w}) \leq\delta_{\sf sep}(\bm{w}^*) \leq 2^{-k}\delta \leq 2\delta_{\sf sep}(\bm{w}^*)$ before escaping the loop. Let $\widehat{\bm{w}}$ be the output of calling 
	\[
	\textsc{Rnd}(\bm{w}, 2^{-k}\delta, h,g).
	\]
	Note that by the hypothesis, we have $\|\bm{w} - \bm{w}^*\| \leq \delta_{\sf sep}(\bm{w}^*) \leq 2^{-k}\delta$. 
	By \Cref{thm:robust}, we know that the output $\widehat{\bm{w}}$ is not the flag {\sf (Infeasible)}, $\|\widehat{\bm{w}} - \bm{w}\| \leq 2^{-k}\delta \leq \delta$ and $
	\dist\big(\bm{0}, \partial f(\widehat{\bm{w}}) \big) = \dist\big(\bm{0},\partial f(\bm{w}^*)\big) \leq \epsilon,
	$ %
	so that \Cref{alg:ls} will return {\sf (True)} and $\widehat{\bm{w}}$ in line~\ref{item:alg1-true}.
	\item $\delta > 2\delta_{\sf sep}(\bm{w}^*)$ and $\delta_{\sf sep}(\bm{w}^*) < \delta_{\sf sep}(\bm{w})$. Let $\widehat{\bm{w}}$ be the output of calling 
	\[
	\textsc{Rnd}(\bm{w}, 2^{-k}\delta, h,g)
	\]
	for the last time before breaking the loop; i.e., with $\delta_{\sf sep}(\bm{w}) < 2^{-k}\delta \leq 2\delta_{\sf sep}(\bm{w})$. 
	By \Cref{prop:termination}, we know that the output $\widehat{\bm{w}}$ is not the flag {\sf (Infeasible)} with $\widehat{\bm{w}} = \bm{w}$. Assume, only for theoretical analysis, that the loop never ends. Then, there exists some $k' \geq k$ such that $\delta_{\sf sep}(\bm{w}^*) \leq 2^{-k'}\delta \leq 2\delta_{\sf sep}(\bm{w}^*)$. By \Cref{prop:termination}, we know that 
	\[
	\widehat{\bm{w}} = \bm{w} = \textsc{Rnd}(\bm{w}, 2^{-k'}\delta, h,g).
	\]
	Note that by hypothesis, we have $\|\bm{w} - \bm{w}^*\| \leq \delta_{\sf sep}(\bm{w}^*) \leq 2^{-k'}\delta\leq \delta$. 
 By \Cref{thm:robust}, we conclude that $
	\dist\big(\bm{0}, \partial f(\widehat{\bm{w}}) \big) = \dist\big(\bm{0}, \partial f(\bm{w}) \big) = \dist\big(\bm{0},\partial f(\bm{w}^*)\big) \leq \epsilon,
	$ so there is no loss in stopping \Cref{alg:ls} at iteration $k$ rather than at $k'$. Thus, without the assumption that the loop never stop, \Cref{alg:ls} will still return {\sf (True)} and $\widehat{\bm{w}}$ in line~\ref{item:alg1-true}.
	\end{itemize}
	Hence, \Cref{alg:ls} will return {\sf (True)} and $\widehat{\bm{w}}$ correctly under all aforementioned three cases. By contrapositive, if \Cref{alg:ls} returns {\sf (False)} eventually, then for all vector $\bm{w}^*\in\mathbb{R}^d$, we have either $\dist\big(\bm{0}, \partial f(\bm{w}^*)\big) > \epsilon$ or $\|\bm{w} - \bm{w}^*\| > \min\{ \delta, \delta_{\sf sep}(\bm{w}^*)\},$ which gives the assertion in the {\sf (False)} part of {\sf{RST-DC}} correctly.
	
	\textbf{Step 2.}
	Next, we consider the efficiency of \Cref{alg:ls}, where the rationality of data becomes crucial. Let the bit size of the inputs be $L \in \mathbb{N}$. Note that the only nontrivial step is calling \Cref{alg:rounding} in Line 4, whose computational complexity is polynomial in $L$ by the classic work of \citet{kozlov1980polynomial} and the discussion in \Cref{sec:representation-P}. 
	If $\delta_{\sf sep}(\bm{w}) = \infty$, then we only call \Cref{alg:rounding} once in \Cref{alg:ls}  before termination.
	If $\delta_{\sf sep}(\bm{w}) \in [1, \infty)$, then the iteration number is upper bounded by $1+\lceil |\log_2(\delta)| \rceil \leq O(L)$. Now suppose that $\delta_{\sf sep}(\bm{w}) \in [0, 1)$.
	In this case, it is easy to see that we will call \Cref{alg:rounding}  at most 
	\[
	1+\lceil |\log_2(\delta)| \rceil + \lceil |\log_2(\delta_{\sf sep}(\bm{w}))|\rceil
	\] times.
	It is clear that the bit size of $R$ and $\delta$ is upper bounded by $O(L)$. So we only bound the term $\lceil |\log_2(\delta_{\sf sep}(\bm{w}))|\rceil$.
	Note that, from \Cref{def:sep}, the positive quantity $12R\cdot\delta_{\sf sep}(\bm{w})$ is the minimal non-zero gap in a polynomial-sized rational linear inequality system. It is tedious but routine to show that $12R\cdot\delta_{\sf sep}(\bm{w}) \geq 2^{-O(L)}$; see, e.g., \cite[Lemma 3.1]{wright1997primal}.
	Thus, the overall computational complexity is polynomial in the number of input bits $L$.
\end{proof}

\section{Applications}\label{sec:app}
In this section,
we discuss some applications of the presented techniques to structured piecewise smooth functions from statistics and machine learning. Although these functions may not be \PA{}, they include \PA{} components in the analysis of their variational properties.
To fix notation, let labeled data $\{(\widetilde{\bm{x}} _i, y_i)\in\mathbb{R}^{d+1}:i \in [n]\}$ be given. We usually consider augmented data vectors $\bm{x}_i \in \mathbb{R}^{d+1}$ for all $i\in[n]$ defined by $\bm{x}_i^\top := \left[ \begin{array}{c|c}
		\widetilde{\bm{x}}_i^\top & 1 \end{array} \right]$.
\subsection{$\rho$-Margin Loss SVM}\label{sec:svm}
The problem for solving SVM with a $\rho$-margin loss has been widely recognized in operations research \cite{brooks2011support,zhang2018robust}, statistics \cite{shen2003psi}, and machine learning \cite{huang2014ramp,tian2022computing} communities because it provides better robustness against data outliers than the vanilla SVM. The $\rho$-margin loss (also known as ramp-loss) can be also seen from the learning theory textbook by \citet[Corollary 5.11]{mohri2018foundations}, in which the problem serves as the $\rho$-margin generalization bound for linear hypothesis set. For this classification problem, we further assume that the labels $y_i \in \{-1,1\}$ for all $i \in [n].$

Formally, the objective function of $\rho$-margin loss SVM problem can be written as
\[
f_{\sf{SVM}}(\bm{w}):= \lambda r(\bm{w}) + \sum_{i=1}^n \min\left\{ 1, \max\left\{ 1-\frac{y_i\bm{x}_i^\top \bm{w} }{\rho} \right\} \right\},
\]
where $\rho >0$ and $r:\mathbb{R}^{d+1}\to \mathbb{R}$ is a continuously differentiable regularization function.
We can rewrite $f_{\sf{SVM}}$ as $f_{\sf{SVM}}(\bm{w}) = \frac{\lambda}{2}r(\bm{w}) + \widetilde{f}_{\sf{SVM}}(\bm{w})$, where the \PA{} function $\widetilde{f}_{\sf{SVM}}(\cdot)$ is defined by
\[
\widetilde{f}_{\sf{SVM}}(\bm{w}):= \sum_{i=1}^n \max\left\{ 1-\frac{y_i\bm{x}_i^\top \bm{w}}{\rho}, 0 \right\} - \sum_{j=1}^n\max\left\{ \frac{-y_j\bm{x}_j^\top \bm{w}}{\rho},0 \right\}.
\]
By \cite[Exercise 8.8(c)]{rockafellar2009variational}, we obtain
\[
\partial f_{\sf{SVM}}(\bm{w}) = \lambda \nabla r(\bm{w}) + \partial\widetilde{f}_{\sf{SVM}}(\bm{w}).
\]
Thus, the only task remaining concerns the variational behavior of $\widetilde{f}_{\sf{SVM}}$, which is evidently a \DC{} function with $1$-\MC{} components. 
Define sets $\mathcal{I}_1(\bm{w}) := \{i\in[n]: y_i\bm{x}_i^\top \bm{w} = \rho \}$ and $\mathcal{I}_2(\bm{w}) := \{j\in[n]: y_j\bm{x}_j^\top \bm{w} = 0 \}$. 
Besides, the convex subdifferentials of DC components of $\widetilde{f}_{\sf{SVM}}$ are zonotopes by definition.
By \Cref{prop:zonotop}, we have the following subdifferential formula:
\[
\partial f_{\sf{SVM}}(\bm{w}) =
\lambda \nabla r(\bm{w}) -\sum_{i=1}^n \frac{y_i\bm{x}_i}{\rho}\partial [\max\{\cdot, 0\}]\left(1-y_i\bm{x}_i^\top \bm{w}/\rho\right) + \sum_{j=1}^n\frac{y_j\bm{x}_j}{\rho}\partial [\max\{\cdot, 0\}]\left(-y_j\bm{x}_j^\top \bm{w}/\rho\right)
\]
if and only if the following transversality condition holds at the point $\bm{w}$:
\[%
	\spn \left\{\bm{x}_i: i \in \mathcal{I}_1(\bm{w})\right\} \cap  \spn \left\{\bm{x}_j: j \in \mathcal{I}_2(\bm{w})\right\}= \{\bm{0}\},
	\]
which can be implied by a ``general position''-type condition on the data points $\{\widetilde{\bm{x}}_i\}_i$; see \Cref{sec:2relu} for more details.
Our other results in \Cref{sec:cr,sec:robust} can be similarly applied.
\subsection{Continuous Piecewise Affine Regression}
A conventional statistical estimation model is typically defined by a linear combination of input features, resulting in the formation of a linear regression model. Recently, there has been a surge of interest in research exploring the application of piecewise affine models \cite{mazumder2019computational, hahn2017parameter}, as discussed in the work \cite{cui2018composite}.
Simply put, a continuous piecewise affine regression problem replaces the linear model with a general piecewise affine function. Formally, the objective is to minimize the following piecewise differentiable function:

\[
f_{\sf{PAR}}\big((\bm{a}_j)_{j\in[m_1]}, (\bm{b}_k)_{k\in[m_2]}\big):=\sum_{i=1}^n \ell_i \left( \max_{1 \leq j \leq m_1} \bm{a}_j^\top \bm{x}_i - \max_{1 \leq k \leq m_2} \bm{b}_k^\top \bm{x}_i \right),
\]
where the loss function $\ell_i:\mathbb{R} \to \mathbb{R}$ for the $i$-th data point is assumed to be continuously differentiable.
For instance, we may choose $\ell_i(t):=\tfrac{1}{2}(t-y_i)^2,$ which leads to a least squares problem.
For any $i \in [n]$ and $j \in [m_1]$, let us define
\[
\bm{\theta}_a:=\left[ \begin{array}{c}
     \bm{a}_1 \\ \vdots \\\bm{a}_{m_1} \end{array} \right] \in \mathbb{R}^{dm_1}, \qquad
     \bm{x}^a_{i,j} := \left[ \begin{array}{c}
     \bm{0}_{d(j-1)} \\ \bm{x}_i \\\bm{0}_{d(m_1-j)} \end{array} \right] \in \mathbb{R}^{dm_1}.
\]
We can define vectors $\bm{\theta}_b \in \mathbb{R}^{dm_2}$ and $\bm{x}^b_{i,k}\in \mathbb{R}^{dm_2}$ for any $i \in [n]$ and $k \in [m_2]$ in a similar way. This gives a reformulation of the function $f_{\sf{PAR}}$ as
\[
f_{\sf{PAR}'}(\bm{\theta}_a, \bm{\theta}_b):=\sum_{i=1}^n \ell_i \left( \max_{1 \leq j \leq m_1} \bm{\theta}_a^\top \bm{x}^a_{i,j} - \max_{1 \leq k \leq m_2} \bm{\theta}_b^\top \bm{x}^b_{i,k} \right).
\]
An immediate difficulty in analyzing the differential structure of the function $f_{\sf{PAR}'}$ is its non-separable form, which is due to the nonlinear functions $\ell_i:\mathbb{R}\to \mathbb{R}$. 
Indeed, it is legitimate to consider only a separable, partially linearized version of $f_{\sf{PAR}'}$.
The following result can be seen as a corollary of a more general theorem in \cite[Theorem 6.5]{mordukhovich1996nonsmooth} (see also \cite[Theorem 4.5(ii)]{Mordukhovich.18}). Here, we provide a self-contained, elementary proof.
\begin{Proposition}[Partial linearization]\label{prop:linearization}
Let a point $\bm{x}\in\mathbb{R}^d$ and a locally Lipschitz  $f:\mathbb{R}^d \rightarrow \mathbb{R}$ be given in form of composition $f= h\circ G$, where function $h:\mathbb{R}^n \rightarrow \mathbb{R}$ is differentiable with locally Lipschitz continuous gradient near $G(\bm{x})$ and the mapping $G:\mathbb{R}^d\rightarrow\mathbb{R}^n$ is locally Lipschitz near $\bm{\bm{x}}$. Then, we have
	\[
	\partial f(\bm{x}) = \partial \Big[\left\langle \nabla h\big(G(\bm{x})\big), G(\cdot ) \right\rangle \Big](\bm{x}).
	\]
\end{Proposition}
\begin{proof}
Let $\bar{f}:=\left\langle \nabla h\big(G(\bm{x})\big), G(\cdot ) \right\rangle$.
We show $f^\circ(\bm{x}; \bm{v}) = \bar{f}^\circ(\bm{x}; \bm{v})$ for any $\bm{v}\in\mathbb{R}^d$. Note that the Clarke generalized subderivative can be written as
\[
f^\circ(\bm{x};\bm{v}) = \limsup_{\substack{\bm{x}'\rightarrow \bm{x} \\ t \searrow 0}} \Delta_t f(\bm{x}')(\bm{v})
= \lim_{\epsilon\searrow 0} \sup_{\|\bm{x}'- \bm{x}\|\leq \epsilon} \sup_{0 < t < \epsilon} \Delta_t f(\bm{x}')(\bm{v}).
\]
We assume $h$ is $L_h$-smooth near $g(\bm{x})$ and $G$ is $L_G$-Lipschitz near $\bm{x}$. We will use the following estimation (see \cite[Lemma 1.2.3]{nesterov2003introductory}) if $h$ is $L_h$-smooth at $\bm{z} \in \mathbb{R}^n$:
\[
-\frac{L_h}{2}\|\bm{z}'-\bm{z}\|^2 \leq h(\bm{z}') - h(\bm{z})- \left\langle \nabla h(\bm{z}), \bm{z}'-\bm{z}\right\rangle
\leq \frac{L_h}{2}\|\bm{z}'-\bm{z}\|^2.
\]
To prove $f^\circ(\bm{x};\bm{v}) \geq \bar{f}^\circ(\bm{x};\bm{v})$, we compute as follows
\begin{align*}
	\Delta_t f(\bm{x}')(\bm{v})&= \frac{h\big(G(\bm{x}'+t\bm{v})\big) - h\big(G(\bm{x}')\big)}{t} \\
	&\geq \frac{1}{t}\left\langle \nabla h\big(G(\bm{x}')\big), G(\bm{x}'+t\bm{v}) - G(\bm{x}') \right\rangle -\frac{L_h}{2t} \left\| G(\bm{x}'+t\bm{v}) - G(\bm{x}') \right\|^2 \\
	&\geq \frac{1}{t}\left\langle \nabla h\big(G(\bm{x})\big), G(\bm{x}'+t\bm{v}) - G(\bm{x}') \right\rangle  -\frac{L_h L_g^2}{2}\left\| \bm{v} \right\|^2 \cdot t 
	 -L_hL_g^2 \|\bm{v}\|\cdot \|\bm{x}-\bm{x}'\| \\
	 &= \Delta_t \bar{f}(\bm{x}')(\bm{v})  -\frac{L_h L_g^2}{2}\left\| \bm{v} \right\|^2 \cdot t 
	 -L_hL_g^2 \|\bm{v}\|\cdot \|\bm{x}-\bm{x}'\|.
\end{align*}
Therefore, for any $\bm{v}\in\mathbb{R}^d$, we know 
\[
\begin{aligned}
f^\circ(\bm{x};\bm{v}) &= \lim_{\epsilon\searrow 0} \sup_{\|\bm{x}'- \bm{x}\|\leq \epsilon} \sup_{0 < t < \epsilon} \Delta_t f(\bm{x}')(\bm{v}), \\
&\overset{(i)}{\geq} \lim_{\epsilon\searrow 0} \sup_{\|\bm{x}'- \bm{x}\|\leq \epsilon} \sup_{0 < t < \epsilon} \Delta_t \bar{f}(\bm{x}')(\bm{v})  - \left(\lim_{\epsilon\searrow 0} \frac{L_h L_g^2}{2}\left\| \bm{v} \right\|^2 \cdot \epsilon\right)
	 -\left(\lim_{\epsilon\searrow 0} L_hL_g^2 \|\bm{v}\|\cdot \epsilon\right) \\
&=\lim_{\epsilon\searrow 0} \sup_{\|\bm{x}'- \bm{x}\|\leq \epsilon} \sup_{0 < t < \epsilon} \Delta_t \bar{f}(\bm{x}')(\bm{v}) \\
&=\bar{f}^\circ(\bm{x};\bm{v}),
\end{aligned}
\]
where in $(i)$ we use $\sup f-g \geq \sup f - \sup g$.
For the converse direction $f^\circ(\bm{x};\bm{v}) \leq \bar{f}^\circ(\bm{x};\bm{v})$, we just compute similarly.
We have proved $f^\circ(\bm{x};\bm{v}) = \bar{f}^\circ(\bm{x};\bm{v}), \forall \bm{v}\in\mathbb{R}^d$. The claim follows from the correspondence between sublinear $f^\circ$ and convex set $\partial f(\bm{x})$ \cite[Proposition 2.1.5]{clarke1990optimization}.	
\end{proof}

Fix $(\bm{\theta}_a, \bm{\theta}_b)$.
Now, we consider a partially linearized version of the function $f_{\sf{PAR}'}$:
\[
\widetilde{f}_{\sf{PAR}'}(\bm{\theta}_a', \bm{\theta}_b'):=\left( \widetilde{f}_{\sf{PARa}}(\bm{\theta}_a'):=\sum_{i=1}^n p_i\cdot \max_{1 \leq j \leq m_1} \bm{\theta}_a'^\top \bm{x}^a_{i,j} \right) - \left( \widetilde{f}_{\sf{PARb}}(\bm{\theta}_b'):=\sum_{i=1}^n p_i\cdot\max_{1 \leq k \leq m_2} \bm{\theta}_b'^\top \bm{x}^b_{i,k} \right),
\]
where $p_i = \ell_i'\left( \max_{1 \leq j \leq m_1} \bm{\theta}_a^\top \bm{x}^a_{i,j} - \max_{1 \leq k \leq m_2} \bm{\theta}_b^\top \bm{x}^b_{i,k} \right) \in \mathbb{R}$ for any $i \in [n]$.
By \cite[Proposition 2.5]{rockafellar1985extensions} and \Cref{prop:linearization}, we have
\[
\partial f_{\sf{PAR}'}(\bm{\theta}_a, \bm{\theta}_b) = \partial \widetilde{f}_{\sf{PAR}'}(\bm{\theta}_a, \bm{\theta}_b) = \partial \widetilde{f}_{\sf{PARa}}(\bm{\theta}_a) \times [-\partial \widetilde{f}_{\sf{PARb}}(\bm{\theta}_b)],
\]
where functions $\widetilde{f}_{\sf{PARa}}$ and $\widetilde{f}_{\sf{PARb}}$ are both DC functions with $1$-\MC{} components; so that our results in \Cref{sec:cr,sec:robust} can be applied accordingly.

\subsection{Shallow ReLU-Type Neural Networks}\label{sec:2relu}

We discuss the stationarity testing problem for a two-layer neural network with a ReLU-type activation function, which is a task already considered by \citet{yun2018efficiently}. For a ``ReLU-type'' activation function, we mean a univariate function representable by maximizing two affine univariate functions, e.g., ReLU and Leaky ReLU. Formally, we are interested in the following loss function $f_{\sf{2NN}}:\mathbb{R}^{h(d+1)}\to \mathbb{R}$ for a two-layer network with $h$ hidden units:

\[
 f_{\sf{2NN}}\big((\bm{w}_k, u_k)_{k\in[h]}\big):=\sum_{i=1}^n \ell_i\left(\sum_{k=1}^h u_k\cdot \max\left\{\bm{w}_k^\top \bm{x}_i, 0 \right\} \right) + r\big((\bm{w}_k, u_k)_{k\in[h]}\big),
\]
where the loss function $\ell_i:\mathbb{R} \to \mathbb{R}$ for the $i$ data point and the regularization term $r$ are assumed to be continuously differentiable. Our analysis relies on the following corollary of \Cref{prop:linearization}.

\begin{Corollary}\label{coro:prod}
	Let the function $f:\mathbb{R}^d\times \mathbb{R} \rightarrow \mathbb{R}$ be $(\bm{x}, u) \mapsto u\cdot g(\bm{x})$, where $g:\mathbb{R}^d\rightarrow\mathbb{R}$ is a Lipschitz function. Then, we have $\partial f(\bm{x},u) = \partial [u\cdot g](\bm{x}) \times \{g(\bm{x})\}$.
\end{Corollary}
\begin{proof}
	Let $h:\mathbb{R}\times \mathbb{R}\rightarrow \mathbb{R}$ be $h(a,b)=a \cdot b$. It is easy to see $h$ is smooth at any $(a,b)$. Let $C(\bm{x}, u) = (g(\bm{x}),u)$. As $f(\bm{x},u) = (h\circ C)(\bm{x},u)$, by \Cref{prop:linearization}, we know
	\[
	\partial f(\bm{x},u) = \partial \big[(\bm{x}',u')\mapsto g(\bm{x})\cdot u' + u\cdot g(\bm{x}')\big](\bm{x},u) = \partial [u\cdot g](\bm{x}) \times \{g(\bm{x})\},
	\]
	as required.
\end{proof}
Fix $(\bm{w}_k, u_k)_{k\in[h]}$.
By virtue of \Cref{coro:prod}, consider the following partially linearized loss function of $f_{\sf{2NN}}$:
\[
\begin{aligned}
 \widetilde{f}_{\sf{2NN}}\big((\bm{w}_k', u_k')_{k\in[h]}\big):=&\ \sum_{i=1}^n p_i \cdot\left(\sum_{k=1}^h u_k'\cdot \max\left\{\bm{w}_k'^\top \bm{x}_i, 0 \right\} \right) + r\big((\bm{w}'_k, u'_k)_{k\in[h]}\big) \\
 =&\ \sum_{k=1}^h u_k'\cdot\left(g_k(\bm{w}_k'):=  \sum_{i=1}^n p_i\cdot\max\left\{\bm{w}_k'^\top \bm{x}_i, 0 \right\} \right) + r\big((\bm{w}_k', u_k')_{k\in[h]}\big),
 \end{aligned}
\]
where $p_i:=\ell_i'\left(\sum_{k=1}^h u_k\cdot \max\left\{\bm{w}_k^\top \bm{x}_i, 0 \right\} \right) \in \mathbb{R}$ for any $i\in[n]$.
By \Cref{prop:linearization}, \cite[Proposition 2.5]{rockafellar1985extensions}, and \Cref{coro:prod}, we have
\[
\begin{aligned}
\partial f_{\sf{2NN}}\big((\bm{w}_k, u_k)_{k\in[h]}\big) &= \partial \widetilde{f}_{\sf{2NN}}\big((\bm{w}_k, u_k)_{k\in[h]}\big) \\
&= \nabla r\big((\bm{w}_k, u_k)_{k\in[h]}\big)+\prod_{k=1}^h \partial[u_k\cdot g_k](\bm{w}_k) \times \{g_k(\bm{w}_k)\}.
\end{aligned}
\]
Note that the function $\bm{w} \mapsto u_k\cdot g_k(\bm{w})$ is a DC function with $1$-\MC{} components for any $k \in [h]$.
Let us define sets $\mathcal{I}^+_k = \{i \in [n]: u_kp_i> 0, \bm{w}_k^\top \bm{x}_i = 0 \}$ and $\mathcal{I}^-_k = \{i \in [n]: u_kp_i< 0, \bm{w}_k^\top \bm{x}_i = 0 \}$ for any $k \in [h]$.
 By \cite[Exercise 8.8(c)]{rockafellar2009variational} and \Cref{prop:zonotop}, we have the validity of the following natural subdifferential formula:
\begin{equation}\label{eq:subdiff-2relu}
\partial f_{\sf{2NN}}\big((\bm{w}_k, u_k)_k\big)=\nabla r\big((\bm{w}_k, u_k)_k\big)+\prod_{k=1}^h  \left(u_k\sum_{i=1}^n p_i\bm{x}_i \cdot \partial[\max\{\cdot, 0\}]\left(
\bm{w}_k^\top \bm{x}_i\right)\right)\times \{g_k(\bm{w}_k)\}
\end{equation}
if and only if the following qualification on data $\{\bm{x}_i\}_{i=1}^n$ holds
\begin{equation}\label{eq:2relu-kq}
\bigcup_{k\in[h]} \Big( \spn \left\{\bm{x}_i: i \in \mathcal{I}_k^+ \right\}\cap \spn \left\{\bm{x}_j: j \in \mathcal{I}_k^- \right\}\Big)= \{\bm{0}\}.
\end{equation}

An immediate question arises regarding when the condition in \nref{eq:2relu-kq} is satisfied. A sufficient condition, as presented in \cite[Assumption 2]{yun2018efficiently}, is that the data $\{\widetilde{\bm{x}}_i \in \mathbb{R}^d:i\in[n]\}$ should be in ``general position,'' a concept which is formally defined below.

\begin{Definition}[General position of points; cf.~{\cite[p.~57]{grunbaum2003convex}}]\label{def:gp}
	A finite subset $\{\widetilde{\bm{x}}_i \in \mathbb{R}^d:i\in[n]\}$ is said to be in general position if no $d+1$ points among them lie on the same affine hyperplane.
\end{Definition}

We are now in a position to compare the weakest possible regularity condition presented in \nref{eq:2relu-kq}, which is deduced from \Cref{prop:zonotop}, with existing conditions found in the literature.%
\begin{Proposition}\label{prop:relation-kq}
For the conditions concerning the validity of the subdifferential formula of $\partial f_{\sf{2NN}}$ in \eqref{eq:subdiff-2relu}, the following relation is observed:
\[
\text{``general position'' in \Cref{def:gp}} \implies \text{surjectivity in \Cref{sec:known-cond}} \implies \text{\nref{eq:2relu-kq}}.
\]	
\end{Proposition}
\begin{proof}
	The proof is by \cite[Lemma 1]{yun2018efficiently} and \Cref{prop:zonotop}.
\end{proof}

In the following, we present two examples to demonstrate that the one-sided implications in \Cref{prop:relation-kq} are indeed strict.
\begin{Example}[{\nref{eq:2relu-kq}} $\notimplies$ surjectivity]
Let the function $f:\mathbb{R}^4\rightarrow \mathbb{R}$ be given as
\[
f(x,y,z,b)\coloneqq\max\{2y+b,0\}+\max\{2x+2z+b,0\}+\max\{x+y+z+b,0\}-\max\{x-z+b,0\}.
\]
Consider $(x,y,z,b)=\bm{0}$. One can verify that \nref{eq:2relu-kq} is satisfied but not surjectivity in \Cref{sec:known-cond}.
\end{Example}

\begin{Example}[Surjectivity $\notimplies$ ``general position'']
Let the function $f:\mathbb{R}^3\rightarrow \mathbb{R}$ be given as
\[
f(x,y,b)\coloneqq\max\{-2y+b,0\}+\max\{-y+b,0\}+\max\{x+b,0\}-\max\{y+b,0\}.
\]
Consider $(x,y,b) = (1,1,-1)$. Surjectivity is satisfied, but the data is not in general position.
\end{Example}

\begin{Remark}
In practice, if the features of data include a discrete-valued component, e.g., $\{-1,1\}$, then the points $\{\widetilde{\bm{x}}_i\}_{i=1}^n$ are rarely in general position because at least half of them must lie in the same affine hyperplane (e.g.,  $\{\bm{y}:y_1 = 1\}$ or $\{\bm{y}:y_1 = -1\}$). The qualification deduced in \nref{eq:2relu-kq} establishes a weaker condition---indeed, the weakest one---compared to the concept of ``general position'' for ensuring the correctness of any testing strategy relying on the validity of \eqref{eq:subdiff-2relu}.
\end{Remark}

\subsection{Penalized Deep  Neural Networks}
In this subsection, we show that the analysis used in \Cref{sec:2relu} can be applied to nonsmooth deep neural networks. However, the multi-composite form of the vanilla deep networks does not exhibit a piecewise affine structure with respect to the parameters. Therefore, we shift our focus to a penalized form of deep networks using a technique known as the ``pull-out'' method according to \cite[Section 9.3]{cui2021modern}. This penalization formulation has been widely adopted in the machine learning and optimization community to design provably convergent algorithms for deep neural networks; refer to \cite{zeng2019global, zhang2017convergent, taylor2016training, liu2023inexact} and references therein for more details.

Formally, the training of a deep neural network can be expressed as the following constrained nonconvex nonsmooth optimization problem:
\begin{alignat*}{2}
	\min_{\substack{(\bm{w}_l)_{l\in[L]}, (\bm{v}_{i,0})_{i\in[n]}\\ (\bm{v}_{i,l})_{i\in[n],l\in[L]}}} &\ \sum_{i=1}^n \ell_i (\bm{v}_{i,L}) + r\big((\bm{w}_l)_l\big)
 \\
	\text{s.t.}\qquad\ \ &\ \bm{v}_{i,l} = \psi^{l}(\bm{w}_l, \bm{v}_{i-1}), \qquad && \forall i\in[n], l\in[L] \\ 
	&\ \bm{v}_{i,0} = \bm{x}_i, && \forall i\in[n],
\end{alignat*}
where the loss function $\ell_i$ for the $i$-th data, the activation function $\psi^l$ of the $l$-th layer, and the regularization term $r$ are both assumed to be either continuously differentiable or of DC-type with $m$-\MC{} components, where $m$ is fixed.
Penalization proves to be a highly effective method for incorporating complex constraints into the objective function by introducing a penalty function and a scalar penalty parameter.
Existing works in the literature on neural network optimization have explored the following two penalized formulations.%
\paragraph{Quadratically Penalized Formulation.}  A classical method involves penalization by a quadratic function, widely employed in designing convergent algorithms for optimizing deep neural networks. For further details, see \cite[Equation (2.3)]{zeng2019global}, \cite[Equation (4)]{zhang2017convergent}, and \cite[Equation (4)]{taylor2016training}. Specifically, the penalized objective function is expressed as:
\[
 f_{\sf{Q}}\big((\bm{w}_l)_{l}, (\bm{v}_{i,l})_{i,l}\big):=\sum_{i=1}^n \left(\ell_i\left(\bm{v}_{i,L} \right) + \rho\sum_{l=1}^L \left\| \bm{v}_{i,l} - \psi^{l}(\bm{w}_l, \bm{v}_{i-1}) \right\|_2^2 \right) + r\big((\bm{w}_l)_l\big),
 \]
 where $\rho \geq 0$ is the penalty parameter.

 \paragraph{Exactly Penalized Formulation.} The exact penalty method produces a feasible solution to the constrained problem with a finite penalty parameter by employing a nondifferentiable penalty function. This approach, featured in \cite[Section 9]{cui2021modern}, has gained recent popularity in computing sharper stationary points for deep neural networks; see \cite[Equation (2.5)]{cui2020multicomposite} and \cite[Problem (PP)]{liu2023inexact}. %
 In particular, the $\ell_1$-penalized formulation in \cite[Equation (2.5)]{cui2020multicomposite} can be written as:
 \[
 f_{\sf{E}}\big((\bm{w}_l)_{l}, (\bm{v}_{i,l})_{i,l}\big):=\sum_{i=1}^n \left(\ell_i\left(\bm{v}_{i,L} \right) + \rho\sum_{l=1}^L \left\| \bm{v}_{i,l} - \psi^{l}(\bm{w}_l, \bm{v}_{i-1}) \right\|_1 \right) + r\big((\bm{w}_l)_l\big) 
 \]
 with $\rho \geq 0$ as the exact penalty parameter.
 
 One can readily verify that by using the procedures and techniques presented in \Cref{sec:2relu}, the results in \Cref{sec:cr,sec:robust} can be applied to the partly linearized versions of $f_{\sf{Q}}$ and $f_{\sf{E}}$ accordingly. The details are omitted for brevity.

\section{Closing Remarks}\label{sec:concl}

We explore the computational complexity, regularity conditions, and the development of robust algorithms for testing approximate stationary points in the context of piecewise affine functions. Our findings reveal the computational hardness in testing various first-order approximate stationarity concepts. We establish the first necessary and sufficient condition for the validity of an equality-type (Clarke) subdifferential sum rule, which applies to a specific representation of arbitrary piecewise affine functions. Additionally, we introduce the first oracle-polynomial-time algorithm designed to detect near-approximate stationary points for piecewise affine functions, which offers an efficient algorithm-independent stopping rule. These results are complemented with applications to various structured piecewise smooth functions. 
As for open research directions, for piecewise affine functions with a fixed input dimension, it is unclear whether the hardness results still hold. It would be interesting to study the necessary and sufficient conditions for the validity of calculus rules for a broader class of nonsmooth functions. %
Additionally, developing an algorithm-independent, potentially non-constructive, and robust testing approach for more general piecewise smooth functions would be an intriguing direction for further exploration.

\section*{Acknowledgement}
LT would like to thank Jiewen Guan (CUHK) for his careful reading of a previous version of the manuscript.

\appendix
\addtocontents{toc}{\protect\setcounter{tocdepth}{1}}
\crefalias{section}{appendix}

\section{Complexity Issues of $n$-MC{} Functions}%
\label{sec:nMC}
In this section, we will discuss the complexity issues related to the \MC{} representation for convex  \PA{} functions.
In particular, we consider the following convex function:
\[
h(\bm{w}):=\sum_{1\leq j_n \leq J_n} \max_{1\leq i_{n} \leq I_n} \cdots \sum_{1\leq j_1 \leq J_1} \max_{1\leq i_1 \leq I_1} \Big(\bm{w}^\top \bm{x}_{i_1,j_1,\dots,i_n,j_n} + a_{i_1,j_1,\dots,i_n,j_n}\Big).
\]
\paragraph{Encoding Length.}
The required input data for representing the $n$-\MC{} function $h:\mathbb{R}^d\to \mathbb{R}$ is formulated as follows:
\begin{itemize}
	\item The level of composition: $n \in \mathbb{N}$.
	\item The width parameters: $\{(I_p, J_p) \in \mathbb{N} \times \mathbb{N}: 1\leq p \leq n\}$.
	\item The dimension of input: $d \in \mathbb{N}$.
	\item The data: $\{(\bm{x}_{i_1,j_1,\dots,i_n,j_n}, a_{i_1,j_1,\dots,i_n,j_n}) \in \mathbb{R}^d \times \mathbb{R}: (i_1,j_1,\dots,i_n,j_n) \in \prod_{p=1}^n ([I_p]\times [J_p])\}$.
\end{itemize}
Following the notation in \cite[Section 1.3]{grotschel2012geometric}, we use $\langle a \rangle, \langle \bm{a} \rangle$, and $\langle \bm{A} \rangle$ to denote the encoding length (or bit/input size) of scalar $a$, vector $\bm{a}$, and matrix $\bm{A}$, respectively.
Therefore, the encoding length of the function $h$, denoted by $\langle h \rangle \in \mathbb{N}$, is defined as
\[
\begin{aligned}
\langle h \rangle:=\sz{n} + \sz{d} &+ \sum_{p=1}^n \left(\sz{I_p} + \sz{J_p}\right) \\
&+ 
\sum_{1 \leq j_n \leq J_n}
\sum_{1 \leq i_n \leq I_n}
\cdots
\sum_{1 \leq j_1 \leq J_1}
\sum_{1 \leq i_1 \leq I_1}
\Big(\sz{\bm{x}_{i_1,j_1,\dots,i_n,j_n}} + \sz{a_{i_1,j_1,\dots,i_n,j_n}}\Big).
\end{aligned}
\]

\paragraph{Value Functions.}
To analyze the  multi-composite structure of the \PA{} function $h$ more succinctly, we introduce several auxiliary value functions. 
For any index $(i_1,j_1,\dots,i_n,j_n) \in \prod_{k=1}^n [I_k]\times [J_k]$, let  $v_{i_1,j_1,\dots,i_n,j_n}:\mathbb{R}^d \to \mathbb{R}$ denote the affine function
\[
v_{i_1,j_1,\dots,i_n,j_n}(\bm{w}) :=\bm{w}^\top \bm{x}_{i_1,j_1,\dots,i_n,j_n} + a_{i_1,j_1,\dots,i_n,j_n}.
\]
Then, for any $k \in [n-1]$, we define convex value functions $v_{j_k,\dots,i_n,j_n}$ and  $v_{i_{k+1},j_{k+1},\dots,i_n,j_n}$ as

\[
v_{j_k,\dots,i_n,j_n}(\bm{w}) :=\max_{1\leq i_k \leq I_k}v_{i_k,j_k,\dots,i_n,j_n} (\bm{w}), \qquad 
v_{i_{k+1}, j_{k+1}\dots,i_n,j_n}(\bm{w}) :=\sum_{1\leq j_k \leq J_k}v_{j_k, i_{k+1},\dots,i_n,j_n} (\bm{w}).
\]
Finally, for any $j_n \in [J_n]$, we define value functions $v_{j_n}$ as
\[
v_{j_n}(\bm{w}) := \max_{1\leq i_n \leq I_n} v_{i_n, j_n} (\bm{w}).
\]
Thus, we can write the \PA{} function $h$ as a simple summation:
\[
h(\bm{w})= \sum_{1\leq j_n \leq J_n} v_{j_n}(\bm{w}).
\]

\paragraph{Convex Subdifferential.}
We now present a convenient representation of the convex subdifferential of the function $h$.
The value functions will be used to describe the active set of elemental functions of $h$.
Note that the value functions $v_{j_k,\dots,i_n,j_n}(\bm{w})$ and  $v_{i_k,j_k,\dots,i_n,j_n}(\bm{w})$ can be evaluated in polynomial time with respect to the input size $\sz{h}+\sz{\bm{w}}$. The total number of value functions is upper bounded by
\[
\sum_{k=1}^n \prod_{p=k}^n I_p J_p \leq 
n\prod_{p=1}^n I_p J_p,
\]
which is polynomial in $\sz{h}$.
For any integer $k \in [n]$ and index $(j_k,\dots,i_n,j_n) \in [J_k]\times \prod_{p=k+1}^n ([I_p]\times [J_p])$, we define the active set $\mathcal{A}_{j_k,\dots,i_n,j_n}(\bm{w})$ as
\[
\mathcal{A}_{j_k,\dots,i_n,j_n}(\bm{w}):=\left\{ i \in [I_k]: v_{j_k,\dots,i_n,j_n}(\bm{w}) = v_{i_k,j_k,\dots,i_n,j_n}(\bm{w})\right\},
\]
which is polynomial time computable.
From \cite[Theorem D.4.1.1, Corollary D.4.3.2]{hiriart2004fundamentals}, we have
\begin{equation}\label{eq:cvx-sub-nMC-geometric}
\begin{aligned}
\partial h(\bm{w}) &=  \sum_{1\leq j_n \leq J_n} \conv \bigcup_{i_{n} \in \mathcal{A}_{j_n}(\bm{w})} \partial v_{i_n,j_n}(\bm{w})\\
&=
\sum_{1\leq j_n \leq J_n} \conv \bigcup_{i_{n} \in \mathcal{A}_{j_n}(\bm{w})} \cdots \sum_{1\leq j_1 \leq J_1} \conv \bigcup_{i_1 \in \mathcal{A}_{j_1,\dots,i_n,j_n}(\bm{w})} \{ \bm{x}_{i_1,j_1,\dots,i_n,j_n}\}.
\end{aligned}
\end{equation}
As the main focus of this paper is on identifying polynomial-time computability, the following reformulation, while somewhat rough and potentially redundant, is sufficient for our purpose.
Using standard LP reformulation techniques, any vector $\bm{g} \in \mathbb{R}^d$ in the polytope $\partial h(\bm{w})$ can be written as
\begin{equation}\label{eq:cvx-sub-nMC-g}
\bm{g} := 
\sum_{1\leq j_n \leq J_n} \quad \sum_{i_{n} \in \mathcal{A}_{j_n}(\bm{w})} \cdots \sum_{1\leq j_1 \leq J_1} \quad \sum_{i_1 \in \mathcal{A}_{j_1,\dots,i_n,j_n}(\bm{w})} t_{i_1,j_1,\dots,i_n,j_n}\cdot \bm{x}_{i_1,j_1,\dots,i_n,j_n},
\end{equation}
where the coefficients $\{t_{i_1,j_1,\dots,i_n,j_n} \}_{i_1,j_1,\dots,i_n,j_n}$ are linearly constrained as follows.
For any integer $k \in [n-1]$ and index $(j_k,\dots,i_n,j_n) \in [J_k]\times \prod_{p=k+1}^n ([I_p]\times [J_p])$, we require
\begin{equation}\label{eq:cvx-sub-nMC-t-general}
\sum_{i_k \in \mathcal{A}_{j_k,\dots,i_n,j_n}(\bm{w})} t_{i_k,j_k,\dots,i_n,j_n} = t_{i_{k+1},j_{k+1},\dots,i_n,j_n}, \quad\text{and}\quad  t_{i_k,j_k,\dots,i_n,j_n} \geq 0 \text{ for any } i_k \in [I_k].
\end{equation}
Meanwhile, for any $j_n \in [J_n]$, we require
\begin{equation}\label{eq:cvx-sub-nMC-t-final}
\sum_{i_n \in \mathcal{A}_{j_n}(\bm{w})} t_{i_n,j_n} = 1, \quad\text{and}\quad t_{i_n,j_n} \geq 0, \forall i_n \in [I_n].
\end{equation}
We can now express the subdifferential of $h$ at point $\bm{w}$ as 
\[
\partial h(\bm{w})=\left\{\bm{g}: \text{linear (in)equalities } \text{\labelcref{eq:cvx-sub-nMC-g,eq:cvx-sub-nMC-t-general,eq:cvx-sub-nMC-t-final} hold} \right\}.
\]
In the sequel, we show that the polytope $\partial h(\bm{w})$ is representable by a polynomial-size LP.
Let $\bm{a}\in\mathbb{R}^{\prod_{p=1}^n I_pJ_p}$ be the vectorization of coefficients $\{t_{i_1,j_1,\dots,i_n,j_n}, (i_1,j_1,\dots,i_n,j_n) \in \prod_{p=1}^n ([I_p]\times [J_p])\}$. Arrange the input data $\{\bm{x}_{i_1, j_1, \dots, i_n,j_n}\}_{i_1, j_1, \dots, i_n,j_n}$ as a matrix $\bm{X} \in \mathbb{R}^{d \times \left(\prod_{p=1}^n I_pJ_p\right)}$ as 
\[
\bm{X} := 
\begin{bNiceMatrix}
\vert & \cdots & \vert \\
\bm{x}_{1,1, \cdots, 1,1}' & \cdots & \bm{x}_{I_1,J_1, \cdots, I_n,J_n}' \\
\vert & \cdots & \vert \\
\end{bNiceMatrix},
\]
where $\bm{x}_{i_1, j_1, \dots, i_n,j_n}' = \bm{x}_{i_1, j_1, \dots, i_n,j_n}$ if, for all $k \in [n]$, it holds that $i_k \in \mathcal{A}_{j_k,\dots,i_n,j_n}(\bm{w})$; otherwise,  $\bm{x}_{i_1, j_1, \dots, i_n,j_n}'$ is set to $\bm{0}$.
Thus, we can rewrite \labelcref{eq:cvx-sub-nMC-g} in a compact form as $\bm{g} = \bm{X}\bm{a}.$
Let $\bm{b} \in \mathbb{R}^{\sum_{k=2}^{n-1}\prod_{p=k}^n I_pJ_p}$ 
be the vectorization of coefficients $t_{i_k,j_k,\dots,i_n,j_n}$ for all $2 \leq k \leq n-1$ and $(i_k,j_k,\dots,i_n,j_n) \in \prod_{p=k}^n ([I_p]\times [J_p])$. 
Let $\bm{c} \in \mathbb{R}^{I_nJ_n}$ be the vectorization of coefficients $t_{i_n, j_n}$ for all $(i_n, j_n) \in [I_n]\times [J_n]$.
Then, we can define matrices 
\[
\begin{aligned}	
\bm{B}_1 &\in \{0,1\}^{\left(\sum_{k=2}^{n}\prod_{p=k}^n I_pJ_p\right)\times\left(\prod_{p=1}^n I_pJ_p\right)}, \\
 \quad \bm{B}_2 &\in \{-1,0,1\}^{\left(\sum_{k=2}^{n}\prod_{p=k}^n I_pJ_p\right)\times\left(\sum_{k=2}^{n-1}\prod_{p=k}^n I_pJ_p\right)}\text{, and} \\
  \quad \bm{C}_1 &\in \{-1,0\}^{\left(\sum_{k=2}^{n}\prod_{p=k}^n I_pJ_p\right)\times I_nJ_n}
\end{aligned}
\]
such that \labelcref{eq:cvx-sub-nMC-t-general} can be rewritten as
\[
\bm{B}_1 \bm{a} + \bm{B}_2 \bm{b} + \bm{C}_1\bm{c} = \bm{0}, \qquad \bm{a}\geq \bm{0}, \bm{b} \geq \bm{0}.
\]
Similarly, we define a matrix $\bm{C}_2 \in \{0, 1\}^{J_n\times I_nJ_n}$ such that \labelcref{eq:cvx-sub-nMC-t-final} can be written as
\[
\bm{C}_2 \bm{c} = \bm{1}_{J_n}, \qquad \bm{c} \geq \bm{0}.
\]
Then, we consider the following polyhedron $H$ in the lifted space:
\begin{equation}\label{eq:nMC-H}
H: = \left\{ 
\begin{bNiceMatrix}[margin]
\bm{g} \\ \hline
\bm{a} \\
\bm{b} \\
\bm{c} 
\end{bNiceMatrix}
:
\begin{bNiceArray}{c|ccc}
-\bm{I}_d & \bm{X} & \bm{0} & \bm{0} \\
\bm{0} & \bm{B}_1 & \bm{B}_2 & \bm{C}_1 \\
\bm{0} & \bm{0} & \bm{0} & \bm{C}_2 \\
\end{bNiceArray}
\begin{bNiceMatrix}[margin]
\bm{g} \\ \hline
\bm{a} \\
\bm{b} \\
\bm{c} 
\end{bNiceMatrix}
=
\begin{bNiceMatrix}
\bm{0} \\
\bm{0} \\
\bm{1}_{J_n} 
\end{bNiceMatrix}, 
\begin{bNiceMatrix}
\bm{a} \\
\bm{b} \\
\bm{c} 
\end{bNiceMatrix}
\geq \bm{0}
\right\}.
\end{equation}
To ease the notation, let us define
\[
\bm{t}:=\begin{bNiceMatrix}
\bm{a} \\
\bm{b} \\
\bm{c} 
\end{bNiceMatrix},\quad
\bm{A}_1:=
\begin{bNiceMatrix}
-\bm{I}_d  \\
\bm{0}  \\
\bm{0}  \\
\end{bNiceMatrix},
\quad
\bm{A}_2:=
\begin{bNiceMatrix}
 \bm{X} & \bm{0} & \bm{0} \\
 \bm{B}_1 & \bm{B}_2 & \bm{C}_1 \\
 \bm{0} & \bm{0} & \bm{C}_2 \\
\end{bNiceMatrix},
\quad
\bm{r}:=
\begin{bNiceMatrix}
\bm{0} \\
\bm{0} \\
\bm{1}_{J_n} 
\end{bNiceMatrix},
\]
where the dimension of matrix $
\begin{bNiceArray}{c|c}
\bm{A}_1 & \bm{A}_2
\end{bNiceArray}$ is $(d+J_n+\sum_{k=2}^{n}\prod_{p=k}^n I_pJ_p)\times (d+\sum_{k=1}^{n}\prod_{p=k}^n I_pJ_p)$. The encoding length of the matrix  $
\begin{bNiceArray}{c|c}
\bm{A}_1 & \bm{A}_2
\end{bNiceArray}$ is upper bounded by
\[
\sz{\bm{X}} + 
2\left(d+n\prod_{p=1}^n I_pJ_p\right)^2,
\]
which, although it is a fairly rough upper bound, is still polynomial with respect to the input size of the function $h$. We can rewrite the polynomial-size polyhedron $H$ as
$
\left\{ 
(\bm{g}, \bm{t})
:
\bm{A}_1 \bm{g} +  \bm{A}_2 \bm{t} 
= \bm{r},
\bm{t}
\geq \bm{0}
\right\}.
$
Consequently, the polytope $\partial h(\bm{w})$ can be represented as the projection of $H$ by
\begin{equation}\label{eq:nMC-standard-LP}
\partial h(\bm{w}) = \{\bm{g}: \exists (\bm{g}, \bm{t}) \in H\}=
\left\{ 
\bm{g}
:
\bm{A}_1 \bm{g} +  \bm{A}_2 \bm{t} 
= \bm{r},
\bm{t}
\geq \bm{0}
\right\}.
\end{equation}
It is also evident that $H$ can be written as a polynomial-size $\mathcal{H}$-polyhedron in the standard form.

\paragraph{Decision Problems.}
The main goal of this section is to prove the following claims:
\begin{Proposition}\label{prop:nMC-polynomiality}
Let two convex \PA{} functions $h,g:\mathbb{R}^d\to \mathbb{R}$ with rational data be given in $n$-\MC{} form. For a given point $\bm{w}\in\mathbb{Q}^d$, the following problems are in the class $\cP$.
\begin{enumerate}[label=\textnormal{(\alph*)}]
	\item Given $\bm{g} \in \mathbb{Q}^d$, determine whether $\bm{g} \in \partial h(\bm{w})$.%
	\item Given $\epsilon \in [0, \infty)\cap\mathbb{Q}$, determine whether $\bm{0} \in \partial h(\bm{w}) - \partial g(\bm{w}) + \epsilon\mathbb{B}$.%
	\item Determine whether $\parr(\partial h(\bm{w})) \cap \parr (\partial g(\bm{w})) = \{\bm{0}\}$.%
\end{enumerate}	
\end{Proposition}
\begin{proof}
 The claim in (a) and (b) directly follow from the LP representation in \eqref{eq:nMC-standard-LP} and the polynomiality of solving LPs \cite{vavasis1996primal} and QPs \cite{kozlov1980polynomial}.
 In the sequel, we will only prove (c).
 Note that, with a slight abuse of notation, $H$ in \eqref{eq:nMC-H} can be represented in $\mathcal{H}$-form as $\{(\bm{g}, \bm{t}): \bm{M}^\top \bm{g} + \bm{N}^\top \bm{t} \leq \bm{c}\},$ where matrices $\bm{M} \in \mathbb{Q}^{d \times m}, \bm{N} \in \mathbb{Q}^{d\times t}$, and a vector $\bm{c} \in \mathbb{Q}^m$ all have polynomial size. 
 Using an idea similar to \cite[Section 5.2]{Jones:169769} (see also \cite[Proposition 4.4]{paffenholz2010polyhedral}), we have
 \[
 \parr(H)=\left\{(\bm{g}, \bm{t}): \bm{m}_i^\top \bm{g} + \bm{n}_i^\top \bm{t} \leq \bm{0}, i \in \mathcal{E}\right\},
 \]
 where $\bm{m}_i, \bm{n}_i$ are columns of $\bm{M},\bm{N}$ and $\mathcal{E}:=\left\{i \in [m]: \bm{m}_i^\top \bm{g}+ \bm{n}_i^\top \bm{t} = c_i \text{ for all } (\bm{g}, \bm{t}) \in H \right\}$ is the so-called \emph{equality set} of $H$ (see \cite[Definition 4.1]{paffenholz2010polyhedral}). The equality set $\mathcal{E} \subseteq [m]$ is computable in polynomial time by solving at most $m$ LPs (see \cite[Section 5.2]{Jones:169769} and \cite[Theorem 3.2]{edmonds1982brick}).
 From the non-emptiness of $\partial h(\bm{0})$ \cite[Proposition 2.1.2(a)]{clarke1990optimization}, we have $H\neq \emptyset$.
 Let $(\bm{g}_0, \bm{t}_0) \in H$, so that $\bm{g}_0 \in \partial h(\bm{0})$. We compute
\begin{align*}
\left\{\bm{g}: \bm{m}_i^\top \bm{g} + \bm{n}_i^\top \bm{t} \leq \bm{0}, i \in \mathcal{E} \right\} &=\{\bm{g}: (\bm{g}, \bm{t}) \in \parr(H)\} \\
&= \{\bm{g}: (\bm{g}, \bm{t}) \in \aff(H)\} - \bm{g}_0 \\
& =\aff(\partial h(\bm{0})) - \bm{g}_0 \tag{\cite[Lemma 3.21]{Jones:169769}} \\
&=\parr(\partial h(\bm{0})),
\end{align*}
so that the set $\parr(\partial h(\bm{0}))$ is efficiently LP representable. For all $i \in [d]$, consider the following set
\[
P_i:=\left\{\bm{g}: \bm{g}^\top \bm{e}_i \geq 1\right\}\cap \parr(\partial h(\bm{w})) \cap \parr (\partial g(\bm{w})).
\] 
It is evident that $\parr(\partial h(\bm{w})) \cap \parr (\partial g(\bm{w})) = \{\bm{0}\}$ if and only if $P_i = \emptyset$ for all $i \in [d]$, which can be verified or refuted by solving at most $d$ linear feasibility problems.
Consequently, (c) can be determined in polynomial time.
\end{proof}

\section{Missing Proofs from \Cref{sec:hardness}}
\subsection{Proof of \Cref{coro:dc-critical} (DC-Criticality)}\label{sec:prf-dc-critical}
\begin{proof}[Proof of \Cref{coro:dc-critical}]
	Note that in the definition of \Cref{prop:DCC}, for the convex \PA{} functions $h_{\sf C}$ and $g_{\sf C}$, we have $\bm{0} \in \partial h_{\sf C}(\bm{0})$ and $\bm{0} \in \partial g_{\sf C}(\bm{0})$. Thus, it always holds that $\bm{0} \in \partial h_{\sf C}(\bm{0}) - \partial g_{\sf C}(\bm{0})$, i.e., the point $\bm{0}$ is always a \DC{}-critical point of $h_{\sf C}-g_{\sf C}$. However, by \Cref{lem:dc-np-hard-C}, determining whether $\bm{0} \in \partial (h_{\sf C} - g_{\sf C})(\bm{0})$ is strongly $\cNP$-hard. As the convex functions $h_{\sf C}$ and $g_{\sf C}$ are both $2$-\MC{} functions, from \Cref{thm:mem-dc}, we conclude that determining whether $\bm{0} \in \partial (h_{\sf C} - g_{\sf C})(\bm{0})$ is $\cNP$-hard. 
\end{proof}

\subsection{Proof of \Cref{coro:abs-norm-hard} (Abs-Norm)}\label{sec:prf-abs}

We briefly review the abs-normal representation of a subclass of piecewise differentiable functions. For more details, we refer the reader to \cite{griewank2013stable,griewank2016first}.

The abs-normal representation is a piecewise linearization scheme concerning a certain subclass of piecewise differentiable functions in the sense of \citet{scholtes2012introduction}. In this subclass, functions are defined as compositions of smooth functions and the absolute value function. By identities $\max\{a,b\}=(a+b)/2+|a-b|/2, \min\{a,b\}=(a+b)/2-|a-b|/2,$ and $\max\{x,0\}=x/2+|x|/2$, composition with these nonsmooth elemental functions can also be represented in the abs-normal form.
Let $\varphi:\mathbb{R}^d\rightarrow\mathbb{R}$ be a function in such subclass. By numbering all inputs to the absolute value functions in the evaluation order as ``switching variables'' $z_i$ for $i\in [s]$, the function $\bm{x}\mapsto y=\varphi(\bm{x})$ can be written in the following abs-normal form:
\[
\bm{z} = F(\bm{x}, \bm{p}),\qquad y = f(\bm{x}, \bm{p}),
\]
where $\bm{x}\in\mathbb{R}^d, \bm{p}\in\mathbb{R}_+^s$, the smooth mapping $F:\mathbb{R}^d\times \mathbb{R}^s_+ \rightarrow \mathbb{R}^s$, and the smooth function $f:\mathbb{R}^d\times \mathbb{R}^s \rightarrow \mathbb{R}$. As the numbering of $\{z_i\}_{i=1}^s$ is in the evaluation order, $z_i$ is a function of $z_j$ only if $j < i$. In sum, we have
$
y = \varphi(\bm{x}) = f(\bm{x}, |\bm{z}(\bm{x})|),
$
where $\bm{z}(\bm{x})$ is a successive evaluation of $\{z_i\}_{i=1}^s$ with given $\bm{x}$. To see such an evaluation of $\bm{z}(\bm{x})$ is well-defined, note that $z_1 = F_1(\bm{x})$ and for any $1 < i \leq s$,
$ 
z_i = F_i(\bm{x}, |z_1|, \dots, |z_{i-1}|).
$
We remark that, similar to the \DC{} decomposition in DC programming, the function $\varphi$ may have many different abs-normal decomposition, hence has different abs-normal forms. The following vectors and matrices are useful when study the function in abs-normal form:
\begin{alignat*}{2}
\bm{a} &\coloneqq \frac{\partial}{\partial \bm{x}} f(\bm{x},\bm{p}) \in \mathbb{R}^d, \qquad
&&\bm{Z} \coloneqq \frac{\partial}{\partial \bm{x}} F(\bm{x},\bm{p}) \in \mathbb{R}^{s\times d}, \\
\bm{b} &\coloneqq \frac{\partial}{\partial \bm{p}} f(\bm{x},\bm{p}) \in \mathbb{R}^s, 
&&\bm{L} \coloneqq \frac{\partial}{\partial \bm{p}} F(\bm{x},\bm{p}) \in \mathbb{R}^{s\times s}.
\end{alignat*}
For any $\bm{\sigma} \in \{-1,1\}^s$, we will denote by $\bm{\Sigma } \coloneqq \diag(\bm{\sigma }) \in \{-1,0,1\}^{s\times s}$. We can define (see also \cite[Equation (11)]{griewank2016first}) a matrix
$
\nabla \bm{z}^\sigma \coloneqq (\bm{I} - \bm{L\Sigma})^{-1}\bm{Z} \in \mathbb{R}^{s\times d},
$
which will play a key role in the definition of linear independence kink qualification (LIKQ) \cite[Definition 2.12]{walther2019characterizing}.

\begin{proof}[Proof of \Cref{coro:abs-norm-hard}]
We first show that $f_{\sf{F}}(\cdot)$ in \Cref{prob:plt} can be written in the abs-normal form in polynomial time.
For ease of notation, let $q_i(\bm{d}) \coloneqq -\sum_{j=1}^3 \max\left\{ \bm{d}^\top \bm{y}_{3(i-1)+j},0 \right\}$ for any $i \in [n]$. Then, we can rewrite every $q_i$ in the abs-linear form as
\begin{alignat*}{2}
z_i(\bm{d}) &= \bm{y}_i^\top \bm{d}, &\forall i \in [n]. \\	
q_i(\bm{d}, \bm{p}) &= -\frac{1}{2}\sum_{j=1}^3\bm{d}^\top \bm{y}_{3(i-1)+j} - \frac{1}{2}\sum_{j=1}^3 p_{3(i-1)+j}, \qquad &\forall i \in [n].
\end{alignat*}
Note that the function $f_{\sf{F}}$ can be expressed as
\[
y\coloneqq f_{\sf{F}}(\bm{d})=\max_{1 \leq i \leq n} q_i(\bm{d}, |\bm{z}|) = \max\{\dots, \max\{q_1(\bm{d},|\bm{z}|),q_2(\bm{d},|\bm{z}|)\},\dots,q_n(\bm{d},|\bm{z}|)\},
\]
which can be written in abs-normal form as
\begin{alignat*}{2}
	z_i &= F_i(\bm{q}, |\bm{z}|) = \frac{1}{2^{i-2}}\cdot q_1 + \sum_{t=2}^{i-1}\big( q_t + p_{3n+t-1} \big) - q_i, &\qquad\forall 3n+1 \leq i \leq 4n-1,\\
	y&=f(\bm{q}, \bm{p}) = \frac{1}{2^{n-1}}\cdot q_1 + \sum_{t=2}^n\big( q_t + p_{3n+t-1} \big).
\end{alignat*}
In sum, we have
\[
z_i = \left\{ \begin{array}{rcl}
         \displaystyle \bm{y}_i^\top \bm{w} & \mbox{for}
         & 1 \leq i \leq 3n \\ 
         \displaystyle\frac{1}{2^{i-3n-1}}\cdot q_1+\sum_{t=2}^{i-3n-1} \frac{1}{2^{i-3n-t}}\cdot \big( q_t + p_{3n+t-1} \big)  & \mbox{for} & 3n+1 \leq i \leq 4n-1                \end{array}\right..
\]
Then, we know
\[
f_{\sf{F}}(\bm{d})=f(\bm{d}, |\bm{z}(\bm{d})|) = \frac{1}{2^{n-1}}\cdot q_1 + \sum_{t=2}^n\big( q_t + |z_{3n+t-1}(\bm{d})| \big).
\]
Then, the matrices $\bm{L}, \bm{Z}, \bm{a}, \bm{b}$ can be computed in polynomial time.
We note that $\bm{0}\in\widehat{\partial} f_{\sf{F}}(\bm{0})$ if and only 
if the function $f_{\sf{F}}$ is first-order minimal in abs-normal form and this is shown in the discussion below \cite[Equation (2)]{griewank2019relaxing} (see also \cite[p.~3]{griewank2016first}). 
To see the verification of FOM is in $\cNP$, for any given $\bm{\sigma} \in \{-1,1\}^s$, the computation of the vector $\bm{a}^\top + \bm{b}^\top \big(\diag(\bm{\sigma}) - \bm{L}\big)^{-1}\bm{Z}$ and the matrix $\big(\diag(\bm{\sigma}) - \bm{L}\big)^{-1}\bm{Z}$ can be done in polynomial time. Then, verification of FOM for a given $\bm{\sigma}$ reduces to check the infeasibility of a linear system, which is in $\cP$. In sum, we have shown verification of FOM is $\cNP$-complete, which implies a general test of FOM without kink qualification in \cite[Theorem 4.1]{griewank2019relaxing} is co-$\cNP$-complete.
\end{proof}

\subsection{Proof of \Cref{coro:nnhard} (CNNs)}\label{sec:cnn-hard}
We formulate the following problem concerning the detection of $\epsilon$-stationary points for a CNN.
\begin{Problem}\label{prob:nnt} Fix $\epsilon \in [0, 1/2)$.
	Let the input data $\{\bm{y}_i\}_{i=1}^{3n} \subseteq \{-1,0,1\}^{m}$	
		be given. Let us define a gadget function $f_{\sf{NF}}:\mathbb{R}^{3}\times\mathbb{R}^m \rightarrow \mathbb{R}$ as
	\[
	f_{\sf{NF}}(\bm{u}, \bm{w})\coloneqq \max_{1\leq i \leq n} \sum_{j=1}^3 u_{j} \cdot \max\left\{ \bm{w}^\top \bm{y}_{3(i-1)+j},0 \right\}.
	\]
	Then, consider the following function $f_{\sf{NC}}:\mathbb{R}^{3}\times\mathbb{R}^{3}\times\mathbb{R}^m \rightarrow \mathbb{R}$:
	\[
	f_{\sf NC}(\bm{a}, \bm{b}, \bm{d}):=\frac{\bm{d}^\top \bm{e}_1}{2} + \max\left\{ f_{\sf{NF}}(\bm{a}, \bm{d}) + f_{\sf NF}(\bm{b}, -\bm{d}), a_1 \max\left\{\frac{\bm{d}^\top \bm{e}_1}{2},0\right\}+b_1 \max\left\{-\frac{\bm{d}^\top \bm{e}_1}{2},0\right\}  \right\}.
	\]
	Determine whether  
	$
	\bm{0} \in \partial f_{\sf{NC}}(-\bm{1}_3,-\bm{1}_3, \bm{0}_m) + \epsilon\mathbb{B}.
	$
\end{Problem}

\begin{Lemma}\label{lem:nnt-co-np-hard}
	\Cref{prob:nnt} is $\cNP$-hard.%
\end{Lemma}
\begin{proof}
The proof is similar in spirit to that of \Cref{lem:plst-np-hard}. 
For any given instance of 3SAT, we construct the input data $\{\bm{y}_i\}_{i=1}^{3n}$ according to the proof of \Cref{lem:plt-np-hard}. Fix $\epsilon \in [0,1/2)$. We show that $\bm{0} \in \partial f_{\sf{NC}}(-\bm{1}_3,-\bm{1}_3, \bm{0}_m) + \epsilon\mathbb{B}$ if and only if the instance of 3SAT can be satisfied. 

(``If'') 
Note that $f_{\sf{C}}(\bm{y})=f_{\sf{NC}}(-\bm{1}_3,-\bm{1}_3, \bm{y})$ for any $\bm{y} \in \mathbb{R}^m$.
If the instance of 3SAT in the reduction presented in the proof of \Cref{lem:plt-np-hard} can be satisfied, by the argument in the (``If'') part of $\Cref{lem:plst-np-hard}$ and continuity, there exists $\bm{y}\in \{-1,1\}^m$ such that for any $\bm{y}' \in \mathbb{R}^m,$ near $\bm{y}$, it holds that
\[
f_{\sf{NF}}(-\bm{1}_3, \bm{y}') + f_{\sf{NC}}(-\bm{1}_3, -\bm{y}') < - \left| \frac{\bm{d}^\top \bm{e}_1}{2} \right| =  - \max\left\{\frac{\bm{d}^\top \bm{e}_1}{2},0\right\} -\max\left\{-\frac{\bm{d}^\top \bm{e}_1}{2},0\right\}.
\]
Therefore, by continuity, one can readily check that there exists $\bm{y}\in \{-1,1\}^m$ such that 
\[
f_{\sf{NC}}(\bm{a}',\bm{b}',\bm{y}')=\frac{\bm{y}'^\top \bm{e}_1}{2}+a'_1\cdot\max\left\{\frac{\bm{y}'^\top \bm{e}_1}{2},0\right\}+b_1'\cdot \max\left\{-\frac{\bm{y}'^\top \bm{e}_1}{2},0\right\}
\] 
for any $(\bm{a}',\bm{b}',\bm{y}')$ near $(-\bm{1}_3,-\bm{1}_3,\bm{y})$. Without loss of generality, we assume $y_1 = 1$. Then, the above equation simplifies to
$
f_{\sf{NC}}(\bm{a}',\bm{b}',\bm{y}')=(1+a'_1)y_1'/2,
$ 
so that the function $f_{\sf{NC}}$ is continuously differentiable near the point $(-\bm{1}_3,-\bm{1}_3,\bm{y})$ with gradient $\nabla f_{\sf{NC}}(-\bm{1}_3,-\bm{1}_3,\bm{y}) = \bm{0}.$ Consider sequence $\{(-\bm{1}_3,-\bm{1}_3,\frac{1}{n}\bm{y})\}_n\to  (-\bm{1}_3,-\bm{1}_3,\bm{0}_m)$ as $n\to \infty$. By homogeneity with respect to $\bm{y}$, it holds  $\nabla f_{\sf{NC}}(-\bm{1}_3,-\bm{1}_3,\frac{1}{n}\bm{y}) = \bm{0}$ for any $n \geq 1$. We conclude that $\bm{0} \in \partial f_{\sf{NC}}(-\bm{1}_3,-\bm{1}_3,\bm{0}_m)$ as desired.

(``Only if'') If the instance of the 3SAT cannot be satisfied, by virtue of the proof of \Cref{lem:plt-np-hard}, one can easily check that $f_{\sf{NF}}(\bm{u},\bm{y}) = 0$ for any $\bm{u}$ near $-\bm{1}_3$ and any $\bm{y}\in\mathbb{R}^m$. Consequently, $f_{\sf{NC}}(\bm{a},\bm{b},\bm{y}) = \frac{\bm{y}^\top \bm{e}_1}{2}$ for any $(\bm{a}, \bm{b})$ near $(-\bm{1}_3,-\bm{1}_3)$ and any $\bm{y}\in\mathbb{R}^m$, so that 
\[
\dist\Big(\bm{0}, \partial f_{\sf{NC}}(-\bm{1}_3,-\bm{1}_3, \bm{0}_m)\Big) = \frac{1}{2} > \epsilon,
\]
as desired.
\end{proof}

\begin{proof}[Proof of \Cref{coro:nnhard}]
We now show that the function $f_{\sf{NC}}$ in \Cref{prob:nnt} can be represented by the loss of a convolutional neural network.
First, we observe
\[
f_{\sf{NF}}(\bm{a}, \bm{d}) + f_{\sf NF}(\bm{b}, -\bm{d}) = \max_{1 \leq i, j \leq n} 
\sum_{k=1}^3 \Big( a_{k} \cdot \max\left\{ \bm{d}^\top \bm{y}_{3(i-1)+k},0 \right\}+
b_k \cdot \max\left\{ \bm{d}^\top \bm{y}_{3(j-1)+k},0 \right\}\Big).
\]
The construction is by training the network on a single data point with the loss function locally being $t\mapsto t$ and a linear regularization term $\bm{d}\mapsto \frac{\bm{d}^\top\bm{e}_1}{2}$. The architecture of our network is formally described as follows. 
\begin{description}
\item[(input data)] The data point is organized into the following matrix 
\[\bm{X}:=\left[ \begin{array}{c|c|c|c|c}
		\bm{y}_1 & \cdots & \bm{y}_{3n} & \frac{\bm{e}_1}{2} & -\frac{\bm{e}_1}{2}\end{array} \right]\subseteq \mathbb{Q}^{m\times (3n+2)}.
		\]
\item[(conv. layer)]
Define the output $\bm{h}^{(1)}:=\bm{X}^\top \bm{d} \in \mathbb{R}^{3n+2}$ with network parameters $\bm{d}\in\mathbb{R}^m$.
\item[(ReLU)] Apply ReLU activation function element-wise to obtain $\bm{h}^{(2)}:= \max\{\bm{h}^{(1)},0\} \in \mathbb{R}^{3n+2}$.
\item[(conv. layer)] Define the convolutional operation, with network parameters $\bm{a}\in\mathbb{R}^3$ and $\bm{b}\in\mathbb{R}^3$, on the vector $\bm{h}^{(2)}$ to obtain a vector $\bm{h}^{(3)} \in \mathbb{R}^{n^2+1}$ as
\[
h^{(3)}_p:=
\left\{ \begin{array}{rcl}
         \sum_{k=1}^3 \Big( a_{k} \cdot h^{(2)}_{3\lfloor (p-1)/n \rfloor+k}+
b_{k} \cdot  h^{(2)}_{3(p-1+\lfloor (p-1)/n \rfloor)+k}\Big) & \mbox{for}
         & 1 \leq p \leq n^2 \\
          a_1 \cdot h^{(2)}_{3n+1} + b_1 \cdot h^{(2)}_{3n+2} & \mbox{for} & p=n^2+1. 
                \end{array}\right.
\]
\item[(max-pooling)] Define the output layer as $h^{(\text{out})}:= \max_{i \in [n^2+1]} h_i^{(3)}$.
\end{description}
In sum, we obtain the equality $\frac{\bm{d}^\top \bm{e}_1}{2} + h^{(\text{out})}(\bm{a}, \bm{b}, \bm{d})=f_{\sf{NC}}(\bm{a},\bm{b},\bm{d})$ as required. The proof completes by using \Cref{lem:nnt-co-np-hard}.
\end{proof}

\section{Missing Proofs from \Cref{sec:cr}}

\subsection{Proof of \Cref{prop:comp-poly}}\label{sec:comp-properity}

\begin{proof}[Proof of \Cref{prop:comp-poly}]
Given the polytopes $A$ and $B$ in $\mathbb{R}^d$, let us define two convex \PA{} functions $h,g:\mathbb{R}^d\to \mathbb{R}$ as 
\[
h(\bm{w}):= \max_{\bm{a} \in A}\ \bm{w}^\top \bm{a}, \qquad
g(\bm{w}):= \max_{\bm{b} \in B}\ \bm{w}^\top \bm{b}.
\]
It holds that $A=\partial h(\bm{0})$ and $B=\partial g(\bm{0})$.
Using \Cref{thm:general-sum-Clarke-geometric}, we know that the polytopes $A$ and $B$ are compatible if and only if 
\[
\partial (h-g)(\bm{0}) = \partial h(\bm{0})-\partial g(\bm{0}).
\] 
Therefore, we prove the geometric claims in \Cref{prop:comp-poly} using tools from nonsmooth analysis as follows:
\begin{itemize}
	\item (i) and (ii) follows from $\partial(-f) = -\partial f$; see \cite[Proposition 2.3.1]{clarke1990optimization}.
	\item (iii) is a direct corollary of  \cite[Proposition 2.3.3]{clarke1990optimization}.
	\item (iv) follows from \Cref{lem:localization} and \cite[Proposition 2.3.3]{clarke1990optimization}.
\end{itemize}
This completes the proof.
\end{proof}
\subsection{Proof of \Cref{prop:all-are-transveral}}\label{prf-prop:all-are-transveral}

\begin{proof}[Proof of {\Cref{prop:all-are-transveral}}]
	Evidently, either statement (b) or (c) implies \nref{eq:pre-transverality}. We will only prove that (a) implies \nref{eq:pre-transverality}.
	If the function $-g'$ is Clarke regular at $\bm{w}$, then we have $-\partial g'(\bm{w}) = \partial (-g')(\bm{w}) = \widehat{\partial} (-g')(\bm{w})$. We claim that the function $g'$ can only be an affine function near $\bm{w}$, so that \nref{eq:pre-transverality} holds since $\parr (\partial g'(\bm{w})) = \{\bm{0}\}$. Suppose conversely that the function $g'$ is not affine near $\bm{w}$. By \cite[Proposition 2.2.4]{clarke1990optimization}, there exist at least two distinct subgradients $\bm{g}_1, \bm{g}_2 \in \partial g'(\bm{w})$. By Clarke regularity, we obtain $-\bm{g}_1, -\bm{g}_2 \in \widehat{\partial} (-g')(\bm{w})$. Note that there must exists an element in $\{-\bm{g}_1, -\bm{g}_2\}$, denoted by $-\bm{g}$ for simplicity, that is not equal to $\bm{g}_1$. By the definition of the Fr\'echet and convex subdifferentials, for any $\bm{z} \in \mathbb{R}^d$, we obtain
	\begin{align*}
			g'(\bm{z}) &\geq g'(\bm{w}) + \bm{g}_1^\top (\bm{z} - \bm{w}), \tag{$\bm{g}_1 \in \partial g'(\bm{w})$ and \Cref{fct:clarke-dd}} \\
			-g'(\bm{z}) &\geq -g'(\bm{w}) - \bm{g}^\top (\bm{z} - \bm{w}) + o(\|\bm{z} - \bm{w}\|). \tag{$-\bm{g} \in \widehat{\partial} (-g')(\bm{w})$ and \cite[Definition 8.3(a)]{rockafellar2009variational}} \\
	\end{align*}
	Summing up, we get $(\bm{g}_1 - \bm{g})^\top (\bm{z} - \bm{w}) \leq o(\|\bm{z} - \bm{w}\|)$ for any $\bm{z} \in \mathbb{R}^d$. Let $\bm{z}:=\bm{w}+t(\bm{g}_1 - \bm{g}) \to \bm{w}$ as $t\searrow 0$. It follows that
	\[
	\lim_{t \searrow 0} \frac{(\bm{g}_1 - \bm{g})^\top (\bm{z} - \bm{w})}{\|\bm{z} - \bm{w}\|} = \|\bm{g}_1 - \bm{g}\| \neq 0,
	\]
	a contradiction. 
\end{proof}

\bibliography{ref}
\bibliographystyle{abbrvnat}

\end{document}

%% file: fast-draft-macro.tex
\usepackage[pdfencoding=auto, psdextra, colorlinks,citecolor=blue!90!black, urlcolor= black]{hyperref}       %
\usepackage{booktabs}       %
\usepackage{amsfonts}       %
\usepackage{nicefrac}       %
\usepackage{mathrsfs}
\usepackage{graphicx}
\usepackage{amsmath,amssymb,amsthm}
\usepackage[sans]{dsfont}
\usepackage[numbers,square]{natbib}
\usepackage{cleveref}
\usepackage{pifont}
\usepackage{algorithm,algorithmicx,algpseudocode}
\usepackage{thmtools}
\usepackage{thm-restate}
\usepackage{nicematrix}
\def\O#1{\text{\ding{\the\numexpr#1+171}}}
\declaretheorem[name=Theorem,within=section,Refname={Theorem,Theorems}]{Theorem}
\declaretheorem[name=Theorem,numbered=no]{nTheorem}
\declaretheorem[name=Problem,sibling=Theorem,Refname={Problem,Problems}]{Problem}
\declaretheorem[name=Lemma,sibling=Theorem,Refname={Lemma,Lemmas}]{Lemma}
\declaretheorem[name=Proposition,sibling=Theorem,Refname={Proposition,Propositions}]{Proposition}
\declaretheorem[name=Proposition,numbered=no]{nProposition}
\declaretheorem[name=Fact,sibling=Theorem]{Fact}
\declaretheorem[name=Fact,numbered=no]{nFact}
\declaretheorem[name=Corollary,sibling=Theorem]{Corollary}
\declaretheorem[name=Corollary,numbered=no]{nCorollary}
\declaretheorem[name=Example,sibling=Theorem]{Example}
\declaretheorem[name=Remark,sibling=Theorem]{Remark}
\declaretheorem[name=Definition,sibling=Theorem]{Definition}
\declaretheorem[name=Definition,numbered=no]{nDefinition}

\usepackage{bm}

%% file: arXiv.bbl
\begin{thebibliography}{85}
\providecommand{\natexlab}[1]{#1}
\providecommand{\url}[1]{\texttt{#1}}
\expandafter\ifx\csname urlstyle\endcsname\relax
  \providecommand{\doi}[1]{doi: #1}\else
  \providecommand{\doi}{doi: \begingroup \urlstyle{rm}\Url}\fi

\bibitem[Ahmadi and Zhang(2022{\natexlab{a}})]{ahmadi2022AIM}
A.~A. Ahmadi and J.~Zhang.
\newblock Complexity aspects of local minima and related notions.
\newblock \emph{Advances in Mathematics}, 397:\penalty0 108119, 2022{\natexlab{a}}.

\bibitem[Ahmadi and Zhang(2022{\natexlab{b}})]{ahmadi2022complexity}
A.~A. Ahmadi and J.~Zhang.
\newblock On the complexity of finding a local minimizer of a quadratic function over a polytope.
\newblock \emph{Mathematical Programming}, 195\penalty0 (1-2):\penalty0 783--792, 2022{\natexlab{b}}.

\bibitem[Arora et~al.(2018)Arora, Basu, Mianjy, and Mukherjee]{arora2018understanding}
R.~Arora, A.~Basu, P.~Mianjy, and A.~Mukherjee.
\newblock Understanding deep neural networks with rectified linear units.
\newblock In \emph{International Conference on Learning Representations}, 2018.

\bibitem[Bena{\"\i}m et~al.(2005)Bena{\"\i}m, Hofbauer, and Sorin]{benaim2005stochastic}
M.~Bena{\"\i}m, J.~Hofbauer, and S.~Sorin.
\newblock Stochastic approximations and differential inclusions.
\newblock \emph{SIAM Journal on Control and Optimization}, 44\penalty0 (1):\penalty0 328--348, 2005.

\bibitem[Bertsimas and Tsitsiklis(1997)]{bertsimas1997introduction}
D.~Bertsimas and J.~N. Tsitsiklis.
\newblock \emph{Introduction to Linear Optimization}, volume~6.
\newblock Athena Scientific Belmont, MA, 1997.

\bibitem[Bodlaender et~al.(1990)Bodlaender, Gritzmann, Klee, and Van~Leeuwen]{bodlaender1990computational}
H.~L. Bodlaender, P.~Gritzmann, V.~Klee, and J.~Van~Leeuwen.
\newblock Computational complexity of norm-maximization.
\newblock \emph{Combinatorica}, 10:\penalty0 203--225, 1990.

\bibitem[Bolte et~al.(2022)Bolte, Boustany, Pauwels, and Pesquet-Popescu]{bolte2022complexity}
J.~Bolte, R.~Boustany, E.~Pauwels, and B.~Pesquet-Popescu.
\newblock On the complexity of nonsmooth automatic differentiation.
\newblock In \emph{International Conference on Learning Representations}, 2022.

\bibitem[Brooks(2011)]{brooks2011support}
J.~P. Brooks.
\newblock Support vector machines with the ramp loss and the hard margin loss.
\newblock \emph{Operations Research}, 59\penalty0 (2):\penalty0 467--479, 2011.

\bibitem[Burke and Engle(2020)]{burke2020strong}
J.~V. Burke and A.~Engle.
\newblock Strong metric (sub) regularity of {K}arush--{K}uhn--{T}ucker mappings for piecewise linear-quadratic convex-composite optimization and the quadratic convergence of {N}ewton's method.
\newblock \emph{Mathematics of Operations Research}, 45\penalty0 (3):\penalty0 1164--1192, 2020.

\bibitem[Clarke(1975)]{clarke1975generalized}
F.~H. Clarke.
\newblock Generalized gradients and applications.
\newblock \emph{Transactions of the American Mathematical Society}, 205:\penalty0 247--262, 1975.

\bibitem[Clarke(1990)]{clarke1990optimization}
F.~H. Clarke.
\newblock \emph{Optimization and Nonsmooth Analysis}.
\newblock SIAM, 1990.

\bibitem[Cui and Pang(2021)]{cui2021modern}
Y.~Cui and J.-S. Pang.
\newblock \emph{Modern Nonconvex Nondifferentiable Optimization}.
\newblock SIAM, 2021.

\bibitem[Cui et~al.(2018)Cui, Pang, and Sen]{cui2018composite}
Y.~Cui, J.-S. Pang, and B.~Sen.
\newblock Composite difference-max programs for modern statistical estimation problems.
\newblock \emph{SIAM Journal on Optimization}, 28\penalty0 (4):\penalty0 3344--3374, 2018.

\bibitem[Cui et~al.(2020)Cui, He, and Pang]{cui2020multicomposite}
Y.~Cui, Z.~He, and J.-S. Pang.
\newblock Multicomposite nonconvex optimization for training deep neural networks.
\newblock \emph{SIAM Journal on Optimization}, 30\penalty0 (2):\penalty0 1693--1723, 2020.

\bibitem[Davis and Drusvyatskiy(2019)]{davis2019stochastic}
D.~Davis and D.~Drusvyatskiy.
\newblock Stochastic model-based minimization of weakly convex functions.
\newblock \emph{SIAM Journal on Optimization}, 29\penalty0 (1):\penalty0 207--239, 2019.

\bibitem[Davis et~al.(2020)Davis, Drusvyatskiy, Kakade, and Lee]{davis2020stochastic}
D.~Davis, D.~Drusvyatskiy, S.~Kakade, and J.~D. Lee.
\newblock Stochastic subgradient method converges on tame functions.
\newblock \emph{Foundations of Computational Mathematics}, 20\penalty0 (1):\penalty0 119--154, 2020.

\bibitem[Davis et~al.(2022)Davis, Drusvyatskiy, Lee, Padmanabhan, and Ye]{davis2021gradient}
D.~Davis, D.~Drusvyatskiy, Y.~T. Lee, S.~Padmanabhan, and G.~Ye.
\newblock A gradient sampling method with complexity guarantees for {L}ipschitz functions in high and low dimensions.
\newblock In \emph{Advances in Neural Information Processing Systems}, 2022.

\bibitem[de~Oliveira(2020)]{de2020abc}
W.~de~Oliveira.
\newblock The {ABC} of {DC} programming.
\newblock \emph{Set-Valued and Variational Analysis}, 28:\penalty0 679--706, 2020.

\bibitem[Edmonds et~al.(1982)Edmonds, Pulleyblank, and Lov{\'a}sz]{edmonds1982brick}
J.~Edmonds, W.~Pulleyblank, and L.~Lov{\'a}sz.
\newblock Brick decompositions and the matching rank of graphs.
\newblock \emph{Combinatorica}, 2:\penalty0 247--274, 1982.

\bibitem[Fearnley et~al.(2022)Fearnley, Goldberg, Hollender, and Savani]{fearnley2022complexity}
J.~Fearnley, P.~Goldberg, A.~Hollender, and R.~Savani.
\newblock The complexity of gradient descent: {CLS} $=$ {PPAD} $\cap$ {PLS}.
\newblock \emph{Journal of the ACM}, 70\penalty0 (1):\penalty0 1--74, 2022.

\bibitem[Fukuda(2004)]{fukuda2004zonotope}
K.~Fukuda.
\newblock From the zonotope construction to the {M}inkowski addition of convex polytopes.
\newblock \emph{Journal of Symbolic Computation}, 38\penalty0 (4):\penalty0 1261--1272, 2004.

\bibitem[Garey and Johnson(1978)]{garey1978strong}
M.~R. Garey and D.~S. Johnson.
\newblock ``{S}trong'' {NP}-completeness results: Motivation, examples, and implications.
\newblock \emph{Journal of the ACM}, 25\penalty0 (3):\penalty0 499--508, 1978.

\bibitem[Garey and Johnson(1979)]{garey1979computers}
M.~R. Garey and D.~S. Johnson.
\newblock \emph{Computers and Intractability}, volume 174.
\newblock 1979.

\bibitem[Goodfellow et~al.(2013)Goodfellow, Warde-Farley, Mirza, Courville, and Bengio]{goodfellow2013maxout}
I.~Goodfellow, D.~Warde-Farley, M.~Mirza, A.~Courville, and Y.~Bengio.
\newblock Maxout networks.
\newblock In \emph{International Conference on Machine Learning}, pages 1319--1327. PMLR, 2013.

\bibitem[Griewank(2013)]{griewank2013stable}
A.~Griewank.
\newblock On stable piecewise linearization and generalized algorithmic differentiation.
\newblock \emph{Optimization Methods and Software}, 28\penalty0 (6):\penalty0 1139--1178, 2013.

\bibitem[Griewank and Walther(2008)]{griewank2008evaluating}
A.~Griewank and A.~Walther.
\newblock \emph{Evaluating Derivatives: Principles and Techniques of Algorithmic Differentiation}.
\newblock SIAM, 2008.

\bibitem[Griewank and Walther(2016)]{griewank2016first}
A.~Griewank and A.~Walther.
\newblock First-and second-order optimality conditions for piecewise smooth objective functions.
\newblock \emph{Optimization Methods and Software}, 31\penalty0 (5):\penalty0 904--930, 2016.

\bibitem[Griewank and Walther(2019)]{griewank2019relaxing}
A.~Griewank and A.~Walther.
\newblock Relaxing kink qualifications and proving convergence rates in piecewise smooth optimization.
\newblock \emph{SIAM Journal on Optimization}, 29\penalty0 (1):\penalty0 262--289, 2019.

\bibitem[Gritzmann and Sturmfels(1993)]{gritzmann1993minkowski}
P.~Gritzmann and B.~Sturmfels.
\newblock Minkowski addition of polytopes: {C}omputational complexity and applications to {G}r{\"o}bner bases.
\newblock \emph{SIAM Journal on Discrete Mathematics}, 6\penalty0 (2):\penalty0 246--269, 1993.

\bibitem[Gr{\"o}tschel et~al.(2012)Gr{\"o}tschel, Lov{\'a}sz, and Schrijver]{grotschel2012geometric}
M.~Gr{\"o}tschel, L.~Lov{\'a}sz, and A.~Schrijver.
\newblock \emph{Geometric Algorithms and Combinatorial Optimization}, volume~2.
\newblock Springer Science \& Business Media, 2012.

\bibitem[Gr{\"u}nbaum(2003)]{grunbaum2003convex}
B.~Gr{\"u}nbaum.
\newblock \emph{Convex Polytopes}, volume 221.
\newblock Springer Science \& Business Media, 2003.

\bibitem[G{\"u}ler and Ye(1993)]{guler1993convergence}
O.~G{\"u}ler and Y.~Ye.
\newblock Convergence behavior of interior-point algorithms.
\newblock \emph{Mathematical Programming}, 60\penalty0 (1):\penalty0 215--228, 1993.

\bibitem[Hahn et~al.(2017)Hahn, Banerjee, and Sen]{hahn2017parameter}
G.~Hahn, M.~Banerjee, and B.~Sen.
\newblock Parameter estimation and inference in a continuous piecewise linear regression model.
\newblock \emph{Manuscript, Department of Statistics, Columbia University}, 32\penalty0 (2):\penalty0 407--451, 2017.

\bibitem[Hare and Lewis(2004)]{hare2004identifying}
W.~L. Hare and A.~S. Lewis.
\newblock Identifying active constraints via partial smoothness and prox-regularity.
\newblock \emph{Journal of Convex Analysis}, 11\penalty0 (2):\penalty0 251--266, 2004.

\bibitem[Hiriart-Urruty and Lemar{\'e}chal(2004)]{hiriart2004fundamentals}
J.-B. Hiriart-Urruty and C.~Lemar{\'e}chal.
\newblock \emph{Fundamentals of Convex Analysis}.
\newblock Springer Science \& Business Media, 2004.

\bibitem[Hollender and Zampetakis(2023)]{hollender2023computational}
A.~Hollender and E.~Zampetakis.
\newblock The computational complexity of finding stationary points in non-convex optimization.
\newblock In \emph{The Thirty Sixth Annual Conference on Learning Theory}, pages 5571--5572. PMLR, 2023.

\bibitem[Huang et~al.(2014)Huang, Shi, and Suykens]{huang2014ramp}
X.~Huang, L.~Shi, and J.~A. Suykens.
\newblock Ramp loss linear programming support vector machine.
\newblock \emph{Journal of Machine Learning Research}, 15\penalty0 (1):\penalty0 2185--2211, 2014.

\bibitem[Jones(2005)]{Jones:169769}
C.~Jones.
\newblock Polyhedral tools for control.
\newblock \emph{PhD Thesis, University of Cambridge}, 2005.

\bibitem[Jordan et~al.(2023)Jordan, Kornowski, Lin, Shamir, and Zampetakis]{jordan2023deterministic}
M.~Jordan, G.~Kornowski, T.~Lin, O.~Shamir, and M.~Zampetakis.
\newblock Deterministic nonsmooth nonconvex optimization.
\newblock In \emph{The Thirty Sixth Annual Conference on Learning Theory}, pages 4570--4597. PMLR, 2023.

\bibitem[Kong and Lewis(2023)]{kong2023cost}
S.~Kong and A.~S. Lewis.
\newblock The cost of nonconvexity in deterministic nonsmooth optimization.
\newblock \emph{Mathematics of Operations Research}, 2023.

\bibitem[Kornowski and Shamir(2022)]{kornowski2021oracle}
G.~Kornowski and O.~Shamir.
\newblock Oracle complexity in nonsmooth nonconvex optimization.
\newblock \emph{Journal of Machine Learning Research}, 23\penalty0 (314):\penalty0 1--44, 2022.

\bibitem[Kozlov et~al.(1980)Kozlov, Tarasov, and Khachiyan]{kozlov1980polynomial}
M.~K. Kozlov, S.~P. Tarasov, and L.~G. Khachiyan.
\newblock The polynomial solvability of convex quadratic programming.
\newblock \emph{USSR Computational Mathematics and Mathematical Physics}, 20\penalty0 (5):\penalty0 223--228, 1980.

\bibitem[Kripfganz and Schulze(1987)]{kripfganz1987piecewise}
A.~Kripfganz and R.~Schulze.
\newblock Piecewise affine functions as a difference of two convex functions.
\newblock \emph{Optimization}, 18\penalty0 (1):\penalty0 23--29, 1987.

\bibitem[Le~Thi and Pham~Dinh(2018)]{le2018dc}
H.~A. Le~Thi and T.~Pham~Dinh.
\newblock {DC} programming and {DCA}: {T}hirty years of developments.
\newblock \emph{Mathematical Programming}, 169\penalty0 (1):\penalty0 5--68, 2018.

\bibitem[Lewis and Wright(2011)]{lewis2011identifying}
A.~S. Lewis and S.~J. Wright.
\newblock Identifying activity.
\newblock \emph{SIAM Journal on Optimization}, 21\penalty0 (2):\penalty0 597--614, 2011.

\bibitem[Li et~al.(2020)Li, So, and Ma]{li2020understanding}
J.~Li, A.~M.-C. So, and W.-K. Ma.
\newblock Understanding notions of stationarity in nonsmooth optimization: A guided tour of various constructions of subdifferential for nonsmooth functions.
\newblock \emph{IEEE Signal Processing Magazine}, 37\penalty0 (5):\penalty0 18--31, 2020.

\bibitem[Liu et~al.(2023)Liu, Liu, and Chen]{liu2023inexact}
W.~Liu, X.~Liu, and X.~Chen.
\newblock An inexact augmented lagrangian algorithm for training leaky {ReLU} neural network with group sparsity.
\newblock \emph{Journal of Machine Learning Research}, 24\penalty0 (212):\penalty0 1--43, 2023.

\bibitem[Majewski et~al.(2018)Majewski, Miasojedow, and Moulines]{majewski2018analysis}
S.~Majewski, B.~Miasojedow, and E.~Moulines.
\newblock Analysis of nonsmooth stochastic approximation: The differential inclusion approach.
\newblock \emph{arXiv preprint arXiv:1805.01916}, 2018.

\bibitem[Mazumder et~al.(2019)Mazumder, Choudhury, Iyengar, and Sen]{mazumder2019computational}
R.~Mazumder, A.~Choudhury, G.~Iyengar, and B.~Sen.
\newblock A computational framework for multivariate convex regression and its variants.
\newblock \emph{Journal of the American Statistical Association}, 114\penalty0 (525):\penalty0 318--331, 2019.

\bibitem[Mehrotra and Ye(1993)]{mehrotra1993finding}
S.~Mehrotra and Y.~Ye.
\newblock Finding an interior point in the optimal face of linear programs.
\newblock \emph{Mathematical Programming}, 62\penalty0 (1):\penalty0 497--515, 1993.

\bibitem[Melzer(1986)]{melzer1986expressibility}
D.~Melzer.
\newblock On the expressibility of piecewise-linear continuous functions as the difference of two piecewise-linear convex functions.
\newblock \emph{Mathematical Programming Studies}, 29:\penalty0 118--134, 1986.

\bibitem[Mohri et~al.(2018)Mohri, Rostamizadeh, and Talwalkar]{mohri2018foundations}
M.~Mohri, A.~Rostamizadeh, and A.~Talwalkar.
\newblock \emph{Foundations of Machine Learning}.
\newblock MIT press, 2018.

\bibitem[Mordukhovich(2018)]{Mordukhovich.18}
B.~S. Mordukhovich.
\newblock \emph{Variational Analysis and Applications}.
\newblock Springer, 2018.

\bibitem[Mordukhovich and Shao(1996)]{mordukhovich1996nonsmooth}
B.~S. Mordukhovich and Y.~Shao.
\newblock Nonsmooth sequential analysis in {A}splund spaces.
\newblock \emph{Transactions of the American Mathematical Society}, 348\penalty0 (4):\penalty0 1235--1280, 1996.

\bibitem[Murty and Kabadi(1987)]{murty1987some}
K.~G. Murty and S.~N. Kabadi.
\newblock Some {NP}-complete problems in quadratic and nonlinear programming.
\newblock \emph{Mathematical Programming}, 39\penalty0 (2):\penalty0 117--129, 1987.

\bibitem[Nemirovski and Yudin(1983)]{nemirovskij1983problem}
A.~Nemirovski and D.~Yudin.
\newblock \emph{Problem Complexity and Method Efficiency in Optimization}.
\newblock Wiley-Interscience, 1983.

\bibitem[Nesterov(2003)]{nesterov2003introductory}
Y.~Nesterov.
\newblock \emph{Introductory Lectures on Convex Optimization: A Basic Course}, volume~87.
\newblock Springer Science \& Business Media, 2003.

\bibitem[Nesterov(2013)]{nesterov2013gradient}
Y.~Nesterov.
\newblock Gradient methods for minimizing composite functions.
\newblock \emph{Mathematical Programming}, 140:\penalty0 125--161, 2013.

\bibitem[Nesterov and Nemirovskii(1994)]{nesterov1994interior}
Y.~Nesterov and A.~Nemirovskii.
\newblock \emph{Interior-Point Polynomial Algorithms in Convex Programming}.
\newblock SIAM, 1994.

\bibitem[Nouiehed et~al.(2019)Nouiehed, Pang, and Razaviyayn]{nouiehed2019pervasiveness}
M.~Nouiehed, J.-S. Pang, and M.~Razaviyayn.
\newblock On the pervasiveness of difference-convexity in optimization and statistics.
\newblock \emph{Mathematical Programming}, 174\penalty0 (1-2):\penalty0 195--222, 2019.

\bibitem[Paffenholz(2010)]{paffenholz2010polyhedral}
A.~Paffenholz.
\newblock Polyhedral geometry and linear optimization.
\newblock \emph{Unpublished Lecture Notes.}, 2010.

\bibitem[Pang et~al.(2017)Pang, Razaviyayn, and Alvarado]{pang2017computing}
J.-S. Pang, M.~Razaviyayn, and A.~Alvarado.
\newblock Computing {B}-stationary points of nonsmooth {DC} programs.
\newblock \emph{Mathematics of Operations Research}, 42\penalty0 (1):\penalty0 95--118, 2017.

\bibitem[Pardalos and Schnitger(1988)]{pardalos1988checking}
P.~M. Pardalos and G.~Schnitger.
\newblock Checking local optimality in constrained quadratic programming is {NP}-hard.
\newblock \emph{Operations Research Letters}, 7\penalty0 (1):\penalty0 33--35, 1988.

\bibitem[Pardalos and Vavasis(1992)]{pardalos1992open}
P.~M. Pardalos and S.~A. Vavasis.
\newblock Open questions in complexity theory for numerical optimization.
\newblock \emph{Mathematical Programming}, 57\penalty0 (1-3):\penalty0 337--339, 1992.

\bibitem[Rockafellar(1970)]{rockafellar1970convex}
R.~T. Rockafellar.
\newblock \emph{Convex Analysis}, volume~18.
\newblock Princeton university press, 1970.

\bibitem[Rockafellar(1979)]{rockafellar1979directionally}
R.~T. Rockafellar.
\newblock Directionally {L}ipschitzian functions and subdifferential calculus.
\newblock \emph{Proceedings of the London Mathematical Society}, 3\penalty0 (2):\penalty0 331--355, 1979.

\bibitem[Rockafellar(1985)]{rockafellar1985extensions}
R.~T. Rockafellar.
\newblock Extensions of subgradient calculus with applications to optimization.
\newblock \emph{Nonlinear Analysis: Theory, Methods \& Applications}, 9\penalty0 (7):\penalty0 665--698, 1985.

\bibitem[Rockafellar and Wets(2009)]{rockafellar2009variational}
R.~T. Rockafellar and R.~J.-B. Wets.
\newblock \emph{Variational Analysis}, volume 317.
\newblock Springer Science \& Business Media, 2009.

\bibitem[Scholtes(2012)]{scholtes2012introduction}
S.~Scholtes.
\newblock \emph{Introduction to Piecewise Differentiable Equations}.
\newblock Springer Science \& Business Media, 2012.

\bibitem[Shen et~al.(2003)Shen, Tseng, Zhang, and Wong]{shen2003psi}
X.~Shen, G.~C. Tseng, X.~Zhang, and W.~H. Wong.
\newblock On $\psi$-learning.
\newblock \emph{Journal of the American Statistical Association}, 98\penalty0 (463):\penalty0 724--734, 2003.

\bibitem[Spielman and Teng(2003)]{spielman2003smoothed}
D.~A. Spielman and S.-H. Teng.
\newblock Smoothed analysis of termination of linear programming algorithms.
\newblock \emph{Mathematical Programming}, 97:\penalty0 375--404, 2003.

\bibitem[Taylor et~al.(2016)Taylor, Burmeister, Xu, Singh, Patel, and Goldstein]{taylor2016training}
G.~Taylor, R.~Burmeister, Z.~Xu, B.~Singh, A.~Patel, and T.~Goldstein.
\newblock Training neural networks without gradients: A scalable {ADMM} approach.
\newblock In \emph{International Conference on Machine Learning}, pages 2722--2731. PMLR, 2016.

\bibitem[Tian and So(2022)]{tian2022computing}
L.~Tian and A.~M.-C. So.
\newblock Computing d-stationary points of $\rho$-margin loss {SVM}.
\newblock In \emph{International Conference on Artificial Intelligence and Statistics}, pages 3772--3793. PMLR, 2022.

\bibitem[Tian and So(2024)]{tian2022no}
L.~Tian and A.~M.-C. So.
\newblock No dimension-free deterministic algorithm computes approximate stationarities of {L}ipschitzians.
\newblock \emph{Mathematical Programming}, 208:\penalty0 51--74, 2024.

\bibitem[Tian et~al.(2022)Tian, Zhou, and So]{tian2022finite}
L.~Tian, K.~Zhou, and A.~M.-C. So.
\newblock On the finite-time complexity and practical computation of approximate stationarity concepts of {L}ipschitz functions.
\newblock In \emph{International Conference on Machine Learning}, pages 21360--21379. PMLR, 2022.

\bibitem[Vavasis and Ye(1996)]{vavasis1996primal}
S.~A. Vavasis and Y.~Ye.
\newblock A primal-dual interior point method whose running time depends only on the constraint matrix.
\newblock \emph{Mathematical Programming}, 74\penalty0 (1):\penalty0 79--120, 1996.

\bibitem[Walther and Griewank(2019)]{walther2019characterizing}
A.~Walther and A.~Griewank.
\newblock Characterizing and testing subdifferential regularity in piecewise smooth optimization.
\newblock \emph{SIAM Journal on Optimization}, 29\penalty0 (2):\penalty0 1473--1501, 2019.

\bibitem[Wright(1997)]{wright1997primal}
S.~J. Wright.
\newblock \emph{Primal-Dual Interior-Point Methods}.
\newblock SIAM, 1997.

\bibitem[Ye(1992)]{ye1992finite}
Y.~Ye.
\newblock On the finite convergence of interior-point algorithms for linear programming.
\newblock \emph{Mathematical Programming}, 57\penalty0 (1-3):\penalty0 325--335, 1992.

\bibitem[Yun et~al.(2019)Yun, Sra, and Jadbabaie]{yun2018efficiently}
C.~Yun, S.~Sra, and A.~Jadbabaie.
\newblock Efficiently testing local optimality and escaping saddles for {ReLU} networks.
\newblock In \emph{International Conference on Learning Representations}, 2019.

\bibitem[Zeng et~al.(2019)Zeng, Lau, Lin, and Yao]{zeng2019global}
J.~Zeng, T.~T.-K. Lau, S.~Lin, and Y.~Yao.
\newblock Global convergence of block coordinate descent in deep learning.
\newblock In \emph{International Conference on Machine Learning}, pages 7313--7323. PMLR, 2019.

\bibitem[Zhang et~al.(2018)Zhang, Pham, Fu, and Liu]{zhang2018robust}
C.~Zhang, M.~Pham, S.~Fu, and Y.~Liu.
\newblock Robust multicategory support vector machines using difference convex algorithm.
\newblock \emph{Mathematical Programming}, 169:\penalty0 277--305, 2018.

\bibitem[Zhang et~al.(2020)Zhang, Lin, Jegelka, Jadbabaie, and Sra]{zhang2020complexity}
J.~Zhang, H.~Lin, S.~Jegelka, A.~Jadbabaie, and S.~Sra.
\newblock Complexity of finding stationary points of nonsmooth nonconvex functions.
\newblock In \emph{International Conference on Machine Learning}, pages 11173--11182, 2020.

\bibitem[Zhang and Brand(2017)]{zhang2017convergent}
Z.~Zhang and M.~Brand.
\newblock Convergent block coordinate descent for training {T}ikhonov regularized deep neural networks.
\newblock \emph{Advances in Neural Information Processing Systems}, 30, 2017.

\bibitem[Ziegler(2012)]{ziegler2012lectures}
G.~M. Ziegler.
\newblock \emph{Lectures on Polytopes}, volume 152.
\newblock Springer Science \& Business Media, 2012.

\end{thebibliography}
